\documentclass{amsart}[11pt]

\usepackage[margin=1in]{geometry}

\usepackage{amsmath,amsthm,amssymb, amsfonts,mathtools}
\usepackage{latexsym}
\usepackage{url}
\usepackage{xcolor}
\usepackage{graphicx}
\usepackage{float}
\usepackage{caption}
\usepackage{subcaption}
\usepackage{verbatim}
\usepackage{enumitem}
\usepackage{amsrefs}
\usepackage{array}
\usepackage{booktabs}
\usepackage{color}
\usepackage[utf8]{inputenc}
\usepackage{tcolorbox}
\usepackage{bbm}
\usepackage{commath}
\usepackage{enumitem}
\usepackage{hyperref}
\usepackage{bm}

\newcommand{\N}{\mathbb{N}}

\newcommand{\R}{\mathbb{R}}
\newcommand{\E}{\mathbb{E}}
\newcommand{\F}{\mathcal{F}}
\newcommand{\Prob}{\mathbb{P}}
\newcommand{\1}{\mathbbm{1}}
\newcommand{\bb}[1]{\mathbb{#1}}
\newcommand{\br}[1]{\lbrace #1 \rbrace}

\newcommand{\toinf}{\rightarrow\infty}
\newcommand{\tto}{\rightarrow}

\newcommand{\W}{\Omega}
\newcommand{\mc}[1]{\mathcal{ #1}}

\DefineSimpleKey{bib}{arxiveprint}
\DefineSimpleKey{bib}{arxivid}
\DefineSimpleKey{bib}{arxivclass}
\makeatletter
\catcode`\'=11 
  \def\arxiveprint{%
    \resolve@inner{\bib@arxiveprint}
  }
  \def\bib@arxiveprint#1{%
    \begingroup
        #1\relax
        \bib@resolve@xrefs
        \bib@field@patches
        \bib'setup
        \let\PrintPrimary\@empty
        {%
          \IfEmptyBibField{arxivid}{\url{https://arxiv.org/}}
          {%
            \href{https://arxiv.org/abs/\bib'arxivid}{\nolinkurl{arXiv:\bib'arxivid}}%
            \IfEmptyBibField{arxivclass}{}{~\nolinkurl{[\bib'arxivclass]}}
          }
        }\bib'transition
        \setbib@@
    \endgroup
  }
\catcode`\'=12 
\makeatother
\BibSpec{arxiv}{%
    +{}  {\PrintAuthors}                {author}
    +{,} { \textit}                     {title}
    +{.} { }                            {part}
    +{:} { \textit}                     {subtitle}
    +{,} { \PrintContributions}         {contribution}
    +{.} { \PrintPartials}              {partial}
    +{,} { }                            {journal}
    +{}  { \textbf}                     {volume}
    +{}  { \PrintDatePV}                {date}
    +{,} { \issuetext}                  {number}
    +{,} { \eprintpages}                {pages}
    +{,} { }                            {status}
    +{,} { \PrintDOI}                   {doi}
    +{,} { available at \eprint}        {eprint}
    +{,} { available at \arxiveprint}   {arxiveprint}
    +{}  { \parenthesize}               {language}
    +{}  { \PrintTranslation}           {translation}
    +{;} { \PrintReprint}               {reprint}
    +{.} { }                            {note}
    +{.} {}                             {transition}
    +{}  {\SentenceSpace \PrintReviews} {review}
}

\newtheorem{thm}{Theorem}[section]
\newtheorem{theo}[thm]{Theorem}
\newtheorem{corollary}[thm]{Corollary}
	
\newtheorem{proposition}[thm]{Proposition}

\newtheorem{defi}[thm]{Definition}

\theoremstyle{remark}
\newtheorem{remark}[thm]{Remark}
\newtheorem{example}[thm]{Example}

\newtheorem{lemma}[thm]{Lemma}

\author{Z.W. Bezemek, and K. Spiliopoulos}
\address{Boston University, Department of Mathematics and Statistics\\ 111 Cummington Mall, Boston, MA 02215, USA}
\email[Zachary William Bezemek]{bezemek@bu.edu}
\email[Konstantinos Spiliopoulos]{kspiliop@bu.edu}
\thanks{This work has been partially supported by the National Science Foundation (DMS 2107856) and Simons Foundation Award  672441. The authors of the paper would like to thank both reviewers for a very careful and constructive review of this article.}
\title{Moderate deviations for fully coupled multiscale weakly interacting particle systems}
\date{\today}

\allowdisplaybreaks
\begin{document}

\begin{abstract}
We consider a collection of fully coupled weakly interacting diffusion processes moving in a two-scale environment. We study the moderate deviations principle of the empirical distribution of the particles' positions in the combined limit as the number of particles grow to infinity  and the time-scale separation parameter goes to zero simultaneously. We make use of weak convergence methods,  which provide a convenient representation for the moderate deviations rate function in a variational form in terms of an effective mean field control problem. We rigorously obtain equivalent representation for the moderate deviations rate function in an appropriate ``negative Sobolev" form, proving their equivalence, which is reminiscent of the large deviations rate function form for the empirical measure of weakly interacting diffusions obtained in the 1987 seminal paper by Dawson-G\"{a}rtner. In the course of the proof we obtain related ergodic theorems and we consider the regularity of Poisson type of equations associated to McKean-Vlasov problems, both of which are topics of independent interest. A novel ``doubled corrector problem" is introduced in order to control derivatives in the measure arguments of the solutions to the related Poisson equations used to control behavior of fluctuation terms.
\end{abstract}

\maketitle
\section{Introduction}
The purpose of this paper is to study the moderate deviations principle (MDP) for slow-fast interacting particle systems.   In particular, we consider the system
\begin{align}\label{eq:slowfast1-Dold}
dX^{i,\epsilon,N}_t &= \biggl[\frac{1}{\epsilon}b(X^{i,\epsilon,N}_t,Y^{i,\epsilon,N}_t,\mu^{\epsilon,N}_t)+ c(X^{i,\epsilon,N}_t,Y^{i,\epsilon,N}_t,\mu^{\epsilon,N}_t) \biggr]dt + \sigma(X^{i,\epsilon,N}_t,Y^{i,\epsilon,N}_t,\mu^{\epsilon,N}_t)dW^i_t\\
dY^{i,\epsilon,N}_t & = \frac{1}{\epsilon}\biggl[\frac{1}{\epsilon}f(X^{i,\epsilon,N}_t,Y^{i,\epsilon,N}_t,\mu^{\epsilon,N}_t)+ g(X^{i,\epsilon,N}_t,Y^{i,\epsilon,N}_t,\mu^{\epsilon,N}_t) \biggr]dt \nonumber\\
&+ \frac{1}{\epsilon}\biggl[\tau_1(X^{i,\epsilon,N}_t,Y^{i,\epsilon,N}_t,\mu^{\epsilon,N}_t)dW^i_t+\tau_2(X^{i,\epsilon,N}_t,Y^{i,\epsilon,N}_t,\mu^{\epsilon,N}_t)dB^i_t\biggr]\nonumber\\
(X^{i,\epsilon,N}_0,Y^{i,\epsilon,N}_0)& = (\eta^{x},\eta^{y})\nonumber
{}\end{align}
on a filtered probability space $(\W,\F,\Prob,\br{\F_t})$ with $\br{\F_t}$ satisfying the usual conditions, where $b,c,\sigma,f,g,\tau_1,\tau_2:\R\times\R\times\mc{P}_2(\R)\tto \R$, $B^i_t,W^i_t$ are independent standard 1-D $\F_t$-Brownian motions for $i=1,...,N$, and $(\eta^{x},\eta^{y})\in\R^2$.   Here and throughout $\mc{P}_2(\R)$ denotes the space of probability measures on $\R$ with finite second moment, equipped with the 2-Wasserstein metric (see Appendix \ref{Appendix:LionsDifferentiation}). $\mu^{\epsilon,N}$ is defined by
\begin{align}
\label{eq:empiricalmeasures}
\mu^{\epsilon,N}_t = \frac{1}{N}\sum_{i=1}^N \delta_{X^{i,\epsilon,N}_t},t\in [0,T].
\end{align}

In (\ref{eq:slowfast1-Dold}), $X^{i,\epsilon,N}$ and $Y^{i,\epsilon,N}$ represent the slow and fast motion respectively of the $i^{\text{th}}$ component. Note that classical models of interacting particles in a two-scale potential, see \cites{BS,Dawson,delgadino2020,GP}, can be thought of  as special cases of (\ref{eq:slowfast1-Dold}) with $Y^{i,\epsilon,N}=X^{i,\epsilon,N}/\epsilon$.

Assume that $\epsilon(N)\tto 0$ as  $N\toinf$. In our case, moderate deviations amounts to studying the behavior of the empirical measure of the particles, i.e., of $\mu^{\epsilon,N}$  in the regime between fluctuations and large deviations behavior.
In particular, if we denote by $\mc{L}(X)$ the process at which $\mu^{\epsilon,N}$ converges to (the law of the averaged McKean-Vlasov Equation \eqref{eq:LLNlimitold}) and consider the moderate deviation scaling sequence $\br{a(N)}_{N\in\bb{N}}$ such that $a(N)>0,\forall N\in\bb{N}$ with  $a(N)\tto 0$ and $a(N)\sqrt{N}\tto \infty$ as $N\toinf$,
the moderate deviations process is defined to be
\begin{align}\label{eq:fluctuationprocess}
Z^N_t \coloneqq a(N)\sqrt{N}(\mu^{\epsilon,N}_t-\mc{L}(X_t)),t\in [0,T].
\end{align}

The goal of this paper is to derive the large deviations principle with speed $a^{-2}(N)$ for the process $Z^N_t$, which is the moderate deviations principle for the measure-valued process $\mu^{\epsilon,N}_t$. Notice that if $a(N)=1$ then we get the standard fluctuations process whose limiting behavior amounts to fluctuations around the law of large numbers, $\mc{L}(X_t)$, whereas if $a(N)=1/\sqrt{N}$ then we would be in the large deviations regime.

We remark here that due to the effect of multiple scales, it turns out that a relation between $\epsilon$ and $N$ is needed. So beyond requiring $a(N)\tto 0$ and $a(N)\sqrt{N}\tto \infty$, we also require that there exists $\rho\in (0,1)$ and $\lambda \in (0,\infty]$ such that $a(N)\sqrt{N}\epsilon(N)^\rho \tto \lambda$ as $N\toinf$. 
Note that this should be viewed as a restriction on the scaling sequence $a(N)$, not on the relationship between $\epsilon$ and $N$, and in some regimes we expect this assumption can be weakened. See Remark \ref{remark:onthescalingofa(N)}.

The presence of multiple scales is a common feature in a range of models used in various disciplines ranging from climate modeling to chemical physics to finance, see for example \cites{BryngelsonOnuchicWolynes, feng2012small,jean2000derivatives, HyeonThirumalai, majda2008applied,
 Zwanzig} for a representative, but by no means complete list of references. Interacting diffusions have also been the central topic of study in science and engineering, see for example \cites{BinneyTremaine,Garnier1, Garnier2, IssacsonMS,Lucon2016,MotschTadmor2014} to name a few.
 In the absence of multiple scales, i.e., when $\epsilon=1$, law of large numbers, fluctuations and large deviations behavior as $N\rightarrow\infty$ has been studied in the literature, see \cites{Dawson,DG,BDF}. Analogously, in the case of $N=1$, the behavior as $\epsilon\downarrow 0$, have been extensively studied in the literature, see for example \cites{Baldi,DS,FS, GaitsgoryNguyen,JS,MS,MSImportanceSampling,Lipster,PV1,PV2, Spiliopoulos2013a, Spiliopoulos2014Fluctuations, Spiliopoulos2013QuenchedLDP, Veretennikov, VeretennikovSPA2000}.

 Homogenization of McKean-Vlasov equations (equations that are the limit of $N\rightarrow\infty$ with $\epsilon$ fixed) has also been recently studied in the literature, see e.g., \cites{BezemekSpiliopoulosAveraging2022,HLL,KSS,RocknerMcKeanVlasov}. These results can be thought of as looking at the limit of the system (\ref{eq:slowfast1-Dold}) when first $N\rightarrow\infty$ and then $\epsilon\rightarrow 0$. Large deviations for a special case of (\ref{eq:slowfast1-Dold}) has been recently established in \cite{BS} and in the absence of multiple scales in \cite{BDF}. In \cite{Orrieri} the author studies large deviations for interacting particle systems in the absence of multiple scales but in the joint mean-field and small-noise limit. In the absence of multiple scales, i.e., when $\epsilon=1$, moderate deviations for interacting particle systems have been studied in \cite{BW}.

The contributions of this work are fourfold. Firstly, we investigate the combined limit $N\rightarrow\infty$ and $\epsilon\rightarrow0$ for the fully coupled interacting particle system of McKean-Vlasov type (\ref{eq:slowfast1-Dold}) through the lens of moderate deviations. In order to do so, we use the weak convergence methodology developed in \cite{DE} which leads to the study of (appropriately linearized) optimal stochastic control problems of McKean-Vlasov type, see for example \cites{CD, Lacker,Fischer}. The first main result of this paper is Theorem \ref{theo:MDP} which provides a variational representation of the moderate deviations rate function.

Secondly, we rigorously re-express the obtained variational form of the rate function in the ``negative Sobolev'' form given in Theorem 5.1 of the seminal paper by Dawson-G\"{a}rtner \cite{DG} in the absence of multiple scales. Hence, we rigorously establish the equivalence of the two formulations in the moderate deviations setting, see Proposition \ref{prop:DGformofratefunction}. A connection of this form was recently established rigorously for the first time in the large deviations setting in \cite{BS}.

Thirdly, in the process of establishing the MDP, we derive related ergodic theorems for multiscale interacting particle systems that are of independent interest. Due to the nature of moderate deviations, we need to consider certain solutions of Poisson equations whose properties are considered for the first time in this paper. In particular, we must control a term involving a derivative in the measure argument of the solution to the Poisson Equation \eqref{eq:cellproblemold} (known as the Cell-Problem in the periodic setting). Such terms are unique to slow-fast interacting particle systems and slow-fast McKean-Vlasov SDEs, and thus do not appear whatsoever in proofs of averaging in the one-particle setting. Thus, the ``doubled corrector problem'' construction, \eqref{eq:doublecorrectorproblem}, and the method of proof of Proposition \ref{prop:purpleterm1} are novel ideas here, see also \cite{BezemekSpiliopoulosAveraging2022}.

Fourthly, in contrast to \cite{BW}, in this paper the coefficients of the model need not depend on the measure parameter in an affine way. We allow the coefficients of the interacting particle system \eqref{eq:slowfast1-Dold} to have any dependence on the measure $\mu$, so long that it is sufficiently smooth- see Corollary \ref{cor:mdpnomulti} and Remark \ref{remark:BWextension}. This is thanks to Lemma \ref{lemma:rocknersecondlinfunctderimplication}, which is inspired by Lemma 5.10 in \cite{DLR}, and allows us to see that with sufficient regularity of a functional on $\mc{P}_2(\R)$, the $L^2$ error of that functional evaluated at the empirical measure of $N$ IID random variables and the Law of those random variables is $\mc{O}(1/N)$ as $N\toinf$.

In Section \ref{SS:Examples}, we make our general results concrete for a popular model of interacting particles in a two-scale potential, see also \cite{GP} for a motivating example in this direction. In addition, we present in Section \ref{subsec:suffconditionsoncoefficients}  a number of concrete examples where the conditions of this paper hold.

The identification of the optimal change of measure in the moderate deviations lower bound through feedback controls together with the equivalence proof between the variational formulation and the ``negative Sobolev'' form of the rate function, open the door to a rigorous study of provably-efficient accelerated Monte-Carlo schemes for rare events computation, analogous to what has been accomplished in the one particle case, see e.g., \cites{DSW,MSImportanceSampling}. Exploring this is beyond the scope of this work and will be addressed elsewhere.

In addition, \cite{Dawson} remarks that phase transitions can occur at the level of fluctuations for interacting particle systems. Since the moderate deviations principle is essentially a large deviations statement around the fluctuations, the results obtained in this paper can potentially be related to phase transitions and allow to characterize them further. This dynamical systems direction is left for future work as it is also outside the scope of this paper.

In contrast to large deviations, the main difficulty with moderate deviations lies in the tightness proof, where we use an appropriate coupling argument, as well as in the fact that the space of signed measures is not completely metrizable in the topology of weak convergence (see \cite{DBDG} Remark 1.2, as well as \cite{RV} Remarks 2.2 and 2.3 for further discussion on related issues). Thus, as we will see, we will have to study $Z^N$ as a distribution-valued process on a suitable weighted Sobolev space. In addition, the presence of the multiple scales complicates the required estimates because the ergodic behavior needs to be accounted for as well. The coupling argument used in the proof of tightness is non-standard in that the IID particle system used as an intermediary process between the empirical measure $\mu^{\epsilon,N}$ from Equation \eqref{eq:empiricalmeasures} and its homogenized McKean-Vlasov limit $\mc{L}(X)$ from Equation \eqref{eq:LLNlimitold} is not equal in distribution to $X$. Instead, it is an IID system of slow-fast McKean Vlasov SDEs - see Equation \eqref{eq:IIDparticles}. Thus our proof of tightness is in a sense relying on the fact that the limits $N\toinf$ and $\epsilon\downarrow 0$ for the empirical measure \eqref{eq:empiricalmeasures} commute at the level of the law of large numbers. For a further discussion of this, see Remark \ref{remark:ontheiidsystem} and the discussion at the beginning of Section \ref{sec:tightness}.

The rest of the paper is organized as follows. In Section \ref{subsec:notationandtopology}, we introduce the appropriate topology for $Z^N$ and lay out our main assumptions. We also introduce a quite useful multi-index notation that will allow us to circumvent notational difficulties with various combinations of mixed derivatives that appear throughout the paper. The derivation of the moderate deviations principle is based on the weak convergence approach of \cite{DE} which converts the large deviations problem to weak convergence of an appropriate stochastic control problem.  The main result is presented in Section \ref{sec:mainresults}, Theorem \ref{theo:MDP}.  In Section \ref{sec:formofratefunction}
 we prove an alternative form of the rate function. This form provides a rigorous connection in moderate deviations between the ``variational form'' of the rate function for the empirical measure of weakly interacting particle systems proved in Theorem \ref{theo:MDP} to the ``negative Sobolev'' form given in Theorem 5.1 of the seminal paper by Dawson-G\"{a}rtner \cite{DG}. Corollaries \ref{cor:mdpnomulti} and \ref{corollary:dawsongartnerformnomulti} specialize the discussion to the setting without multiscale structure and thus generalizes the results of \cite{BW}. Specific examples are presented in Subsection \ref{SS:Examples}. Section \ref{S:ControlSystem} formulates the appropriate stochastic control problem.

Sections \ref{sec:ergodictheoremscontrolledsystem}-\ref{sec:lowerbound} are devoted to the proof of Theorem \ref{theo:MDP}.  Due to the presence of the multiple scales, ergodic theorems are needed  to characterize the behavior as $\epsilon\downarrow 0$ of certain functionals of interest for the controlled multiscale interacting particle system; this is the content of Section \ref{sec:ergodictheoremscontrolledsystem}. Tightness of the controlled system is proven in Section \ref{sec:tightness}. In Section \ref{sec:identificationofthelimit} we establish the limiting behavior of the controlled system. The Laplace principle lower bound is proven in Section \ref{sec:upperbound/compactnessoflevelsets}. Section \ref{sec:lowerbound} contains the proof of the Laplace principle upper bound as well as compactness of level sets. Conclusions and directions for future work are in Section \ref{S:Conclusions}. Appendix \ref{sec:notationlist} provides a list of technical notation used throughout the manuscript for convenience. A number of key technical estimates are presented in the remainder of the appendix. In particular, Appendix \ref{sec:aprioriboundsoncontrolledprocess} contains moments bounds for the controlled system. Appendix \ref{sec:regularityofthecellproblem} presents regularity results for the Poisson equation needed to study the fluctuations. Even though related results exist in the literature, the fully coupled McKean-Vlasov case is not covered by the existing results, and therefore Appendix \ref{sec:regularityofthecellproblem} contains the appropriate discussion of the necessary extensions. Lastly, Appendix \ref{Appendix:LionsDifferentiation} contains necessary results on differentiation of functions on spaces of measures.

\section{Notation, Topologies, and Assumptions}\label{subsec:notationandtopology}

In order to construct an appropriate topology for the process $Z^N$ from Equation \eqref{eq:fluctuationprocess}, we follow the method of \cites{BW,HM,KX}. Denote by $\mc{S}$ the space of functions $\phi:\R\tto \R$ which are infinitely differentiable and satisfy $|x|^m\phi^{(k)}(x)\tto 0$ as $|x|\toinf$ for all $m,k\in\bb{N}$. On $\mc{S}$, consider the sequence of inner products $(\cdot,\cdot)_n$ and $\norm{\cdot}_n$ defined by
\begin{align}\label{eq:familyofhilbertnorms}
(\phi,\psi)_n&\coloneqq \sum_{k=0}^n \int_\R (1+x^2)^{2n}\phi^{(k)}(x)\psi^{(k)}(x)dx,
\qquad \norm{\phi}_n\coloneqq \sqrt{(\phi,\phi)_n}
\end{align}
for each $n\in\bb{N}$. As per \cite{GV} p.82 (this specific example on p.84), this sequence of seminorms induces a nuclear Fr\'echet topology on $\mc{S}$. Let $\mc{S}_n$ be the completion of $\mc{S}$ with respect to $\norm{\cdot}_n$ and $\mc{S}_{-n}=\mc{S}_n'$ the dual space of $\mc{S}_n$. We equip $\mc{S}_{-n}$ with dual norm $\norm{\cdot}_{-n}$ and corresponding inner product $(\cdot,\cdot)_{-n}$. Then $\br{\mc{S}_{n}}_{n\in\bb{Z}}$ defines a sequence of nested Hilbert spaces with $\mc{S}_m\subset \mc{S}_n$ for $m\geq n$. In addition we have for each $n\in\bb{N}$, there exists $m>n$ such that the canonical embedding $\mc{S}_{-n}\tto \mc{S}_{-m}$ is Hilbert-Schmidt. In particular, this holds for $m$ sufficiently large that $\sum_{j=1}^N\norm{\phi^m_j}_n<\infty$, where $\br{\phi^m_j}_{j\in\bb{N}}$ is a complete orthonormal system of $\mc{S}_m$. This allows us to use the results of \cite{Mitoma} to see that $\br{Z^N}_{N\in\bb{N}}$ is tight as a sequence of $C([0,T];S_{-m})$-valued random variables for sufficiently large $m$.
In particular, we will require $m>7$ to be sufficiently large so that the canonical embedding
\begin{align}\label{eq:mdefinition}
\mc{S}_{-7}\tto \mc{S}_{-m}\text{ is Hilbert-Schmidt.}
\end{align}
In the proof of the Laplace Principle, we will also make use of $w>9$ such that
\begin{align}\label{eq:wdefinition}
\mc{S}_{-m-2}\tto \mc{S}_{-w}\text{ is Hilbert-Schmidt.}
\end{align}
When proving compactness of level sets of the rate function, we will in addition make use of $r>11$ sufficiently large that the canonical embedding
\begin{align}\label{eq:rdefinition}
\mc{S}_{-w-2}\tto \mc{S}_{-r}\text{ is Hilbert-Schmidt.}
\end{align}

It will be useful to consider another system of seminorms on $\mc{S}$ given, for each $n\in\bb{N}$, by
\begin{align}\label{eq:boundedderivativesseminorm}
|\phi|_n \coloneqq \sum_{k=0}^n \sup_{x\in\R}|\phi^{(k)}(x)|
{}\end{align}
Via a standard Sobolev embedding argument, one can show that for each $n\in\bb{N}$, there exists $C(n)$ such that:
\begin{align}\label{eq:sobolembedding}
|\phi|_n\leq C(n)\norm{\phi}_{n+1},\forall \phi\in\mc{S}.
\end{align}

Let $\bm{X}$ and $\bm{Y}$ be Polish spaces, and $(\tilde\W,\tilde\F,\mu)$ be a measure space. We will denote by $\mc{P}(\bm{X})$ the space of probability measures on $\bm{X}$ with the topology of weak convergence, $\mc{P}_2(\bm{X})\subset \mc{P}(\bm{X})$ the space of square integrable probability measures on $\bm{X}$ with the 2-Wasserstein metric (see Definition \ref{def:lionderivative}), $\mc{B}(\bm{X})$ the Borel $\sigma$-field of $\bm{X}$, $C(\bm{X};\bm{Y})$ the space of continuous functions from $\bm{X}$ to $\bm{Y}$, $C_b(\bm{X})$ the space of bounded, continuous functions from $\bm{X}$ to $\R$ with norm $\norm{\cdot}_\infty$, and $L^p(\tilde\W,\tilde\F,\mu)$ the space of $p$-integrable functions on $(\tilde\W,\tilde\F,\mu)$ (where if $\tilde\W =\bm{X}$ and no $\sigma$-algebra is provided we assume it is $\mc{B}(\bm{X})$). For $\mu\in\mc{P}(\bm{X}),\nu\in\mc{P}(\bm{Y})$, we will denote the product measure induced by $\mu$ and $\nu$ on $\bm{X}\times \bm{Y}$ by $\mu\otimes \nu$. We will at times denote $L^2(\bm{X}\times \bm{X},\mu\otimes \mu)$ by $L^2(\bm{X},\mu)\otimes L^2(\bm{X},\mu)$. We will denote by $L^1_{\text{loc}}(\bm{X},\mu)$ the space of locally integrable functions on $\bm{X}$. For $U\subseteq \R^d$ open, we will denote by $C^\infty_c(U)$ the space of smooth, compactly supported functions on $U$. $C^k_b(\R)$ for $k\in\bb{N}$ will note the space of functions with $k$ continuous and bounded derivatives on $\R$, with norm $|\cdot|_k$ as in Equation \eqref{eq:boundedderivativesseminorm}, and $C^{1,k}_b([0,T]\times\R)$ will denote continuous functions $\psi$ on $[0,T]\times\R$ with a continuous, bounded time derivative on $(0,T)$, denoted $\dot{\psi}$, such that $\norm{\psi}_{C^{1,k}_b([0,T]\times\R)}\coloneqq \sup_{t\in[0,T],x\in\R} |\dot{\psi}(t,x)|+\sup_{t\in[0,T]}\norm{\psi(t,\cdot)}_{C^k_b(\R)}<\infty$. $C^k_{b,L}(\R)\subset C^k_b(\R)$ is the space of functions in $C^k_b(\R)$ such that all $k$ derivatives are Lipschitz continuous. For $\phi \in L^1(\bm{X},\mu),\mu\in\mc{P}(\R)$ we define $\langle \mu,\phi\rangle \coloneqq \int_{\bm{X}}\phi(x)\mu(dx)$. Similarly, for $Z\in\mc{S}_{-p},\phi\in\mc{S}_p$, we will denote the action of $Z$ on $\phi$ by $\langle Z,\phi\rangle$. For $a,b\in \R$, we will denote $a\vee b = \max\br{a,b}$ and $a\wedge b = \min\br{a,b}$. $C$ will be used for a constant which may change from line to line throughout, and when there are parameters $a_1,...,a_n$ which $C$ depends on in an important manner, will denote this dependence by $C(a_1,...,a_n)$. For all function spaces, the codomain is assumed to be $\R$ unless otherwise denoted.

In the construction of the controlled system, we will also make use of the space of measures on $\R^d\times [0,T]$ such that $Q(\R^d\times[0,t]) = t,\forall t\in [0,T]$. We will denote this space $M_T(\R^d)$. We equip $M_T(\R^d)$ with the topology of weak convergence of measures (thus making $M_T(\R^d)$ a Polish space by \cite{DE} Theorem A.3.3). See also the proof of Lemma 3.3.1 in \cite{DE} for the fact that $M_T(\R^d)$ is a closed subset of finite positive Borel measures on $\R^d\times [0,T]$). For when dealing with the occupation measures as defined in Equation (\ref{eq:occupationmeasures}), we will in particular take $d=4$ and will interpret $Q(dx,dy,dz,dt)$ as $x$ denoting variable representing the first coordinate in $\R^4$, $y$ the second, and $z$ the third and fourth.

For a mapping $\vartheta:[0,T]\tto \mc{P}(\R^d)$, it will be useful to define an element of $M_T(\R^d)$ induced by $\vartheta$ by
\begin{align}\label{eq:rigorousLotimesdt}
\nu_{\vartheta}(A\times[0,t])\coloneqq \int_0^t \vartheta(s)[A]ds,\forall t\in [0,T],A\in\mc{B}(\R^d).
\end{align}

Due to the nature of the space we consider the sequence $\br{Z^N}_{N\in\bb{N}}$ to live on, it is natural that we will have to restrict the growth of the coefficients which appear in Equations \eqref{eq:slowfast1-Dold} and \eqref{eq:LLNlimitold} in $x$. We will also need to ensure that the derivatives of the Poisson Equation which appear in the definition the limiting coefficients in Equation \eqref{eq:limitingcoefficients} exist and that the homogenized drift and diffusion coefficients in Equation \eqref{eq:averagedlimitingcoefficients}, which determine the limiting McKean-Vlasov Equation $X_t$ from Equation \eqref{eq:LLNlimitold}, are well-defined. In doing so, will be controlling many mixed derivatives of functions in the Lions sense \cite{NotesMFG} and in the standard sense, it will be useful for us to borrow the multi-index notation proposed in \cite{CM} and employed in \cite{CST}. For the reader's convenience, we have included in Appendix \ref{Appendix:LionsDifferentiation} a brief review on differentiation of functions on spaces of measures. For a more comprehensive exposition on this, we refer the interested reader to \cite{CD} Chapter 5.

Furthermore, since we prove the moderate deviations principle via use of the controlled particle system \eqref{eq:controlledslowfast1-Dold}, we will only have up to second moments of the controlled fast system (see Appendix \ref{sec:aprioriboundsoncontrolledprocess}). It will be important to make sure that terms which the controlled fast process enters in the intermediate proofs of tightness, so naturally we will need some assumptions on the rate of polynomial growth in $y$ of the coefficients which appear in Equations \eqref{eq:slowfast1-Dold} and \eqref{eq:LLNlimitold} (See Remark \ref{remark:ontheassumptions}). We thus extend the multi-index notation from the aforementioned papers to track specific collections of mixed partial derivatives, and to give us a clean way of tracking the rate of polynomial growth in $y$ for those mixed partials in the coming definitions.

\begin{defi}\label{def:multiindexnotation}
Let $n,l$ be non-negative integers and $\bm{\beta}=(\beta_1,...,\beta_n)$ be an $n-$dimensional vector of non-negative integers. We call any ordered tuple of the form $(n,l,\bm{\beta})$ a \textbf{multi-index}. For a function $G:\R\times\mc{P}_2(\R)\tto \R$, we will denote for a multi-index $(n,l,\bm{\beta})$, if this derivative is well defined,
\begin{align*}
D^{(n,l,\bm{\beta})}G(x,\mu)[z_1,...,z_n] = \partial_{z_1}^{\beta_1}...\partial_{z_n}^{\beta_n}\partial_x^l \partial_\mu^n G(x,\mu)[z_1,...,z_n],
\end{align*}
 As noted in the Remark \ref{remark:thirdLionsDerivative}, for such a derivative to be well defined we require for it to be jointly continuous in $x,\mu,z_1,...,z_n$ where the topology used in the measure component is that of $\mc{P}_2(\R)$.

We also define $\bm{\delta}^{(n,l,\bm{\beta})}G(x,\mu)[z_1,...,z_n]$ in the exact same way, with the Lions derivatives $\partial_\mu$ replaced by linear functional derivatives $\frac{\delta}{\delta m}$; see Appendix \ref{Appendix:LionsDifferentiation} for differentiation of functions on spaces of measures.
\end{defi}
\begin{defi}\label{def:completemultiindex}
For $\bm{\zeta}$ a collection of multi-indices of the form $(n,l,\bm{\beta})\in\bb{N}\times\bb{N}\times\bb{N}^n$, we will call $\bm{\zeta}$ a \textbf{complete} collection of multi-indices if for any $(n,l,\bm{\beta})\in \bm{\zeta}$, $\br{(k,j,\bm{\alpha}(k))\in \bb{N}\times\bb{N}\times\bb{N}^k:j\leq l,k\leq n,\bm{\alpha}(k)=(\alpha_1,...,\alpha_k),\exists \bm{\beta}(k)=(\beta(k)_1,...,\beta(k)_k)\in \binom{\bm{\beta}}{k} \text{ such that } \alpha_p \leq \beta(k)_p,\forall p=1,...,k }\subset \bm{\zeta}$. Here for a vector of positive integers $\beta=(\beta_1,...,\beta_n)$ and $k\in\bb{N},k\leq n$, we are using the notation $\binom{\bm{\beta}}{k}$ to represent the set of size $\binom{n}{k}$ containing all the $k$-dimensional vectors of positive integers which can be obtained from removing $n-k$ entries from $\bm{\beta}$.
\end{defi}
\begin{remark}\label{remark:oncompletecollectiondefinition}
Definition \ref{def:completemultiindex} is enforcing that if collection of multi-indices contains a multi-index representing some mixed derivative in $(x,\mu,z)$ as per Definition \ref{def:multiindexnotation}, then it also contains all lower-order mixed derivatives of the same type. For example, if $\bm{\zeta}$ is a collection of multi-indices containing $(2,0,(1,1))$ (corresponding to $\partial_{z_1}\partial_{z_2}\partial^2_\mu G(x,\mu)[z_1,z_2]$) then, in order to be complete, it must also contain the terms $(2,0,(1,0)),(2,0,(0,1)),(2,0,0),(1,0,1),(1,0,0),$ and $(0,0,0)$ (corresponding to the terms $\partial_{z_1}\partial^2_\mu G(x,\mu)[z_1,z_2]$, $\partial_{z_2}\partial^2_\mu G(x,\mu)[z_1,z_2],\partial^2_\mu G(x,\mu)[z_1,z_2],\partial_{z}\partial_\mu G(x,\mu)[z]$, $\partial_\mu G(x,\mu)[z]$, and $G(x,\mu)$ respectively). This is a technical requirement used in order to state the results in Appendix \ref{subsection:regularityofthe1Dpoissoneqn} in a way that allows the inductive arguments used therein to go through.
\end{remark}

Using this multi-index notation, it will be useful to define some spaces regarding regularity of functions in regards to these mixed derivatives. We thus make the following modifications to Definition 2.13 in \cite{CST}:
\begin{defi}\label{def:lionsderivativeclasses}
For $\bm{\zeta}$ a collection of multi-indices of the form $(n,l,\bm{\beta})\in\bb{N}\times\bb{N}\times\bb{N}^n$, we define $\mc{M}_b^{\bm{\zeta}}(\R\times\mc{P}_2(\R))$ to be the class of functions $G:\R\times\mc{P}_2(\R)\tto \R$ such that $D^{(n,l,\bm{\beta})}G(x,\mu)[z_1,...,z_n]$ exists and satisfies
\begin{align}
\label{eq:LionsClassBoundedness} \norm{G}_{\mc{M}_b^{\bm{\zeta}}(\R\times\mc{P}_2(\R))}\coloneqq \sup_{(n,l,\bm{\beta})\in \bm{\zeta}}\sup_{x,z_1,...,z_n\in\R,\mu\in\mc{P}_2(\R)}|D^{(n,l,\bm{\beta})}G(x,\mu)[z_1,...,z_n]|&\leq C .
\end{align}
We denote the class of functions $G\in \mc{M}_b^{\bm{\zeta}}(\R\times\mc{P}_2(\R))$ such that:
\begin{align}
\label{eq:LionsClassLipschitz}|D^{(n,l,\bm{\beta})}G(x,\mu)[z_1,...,z_n]- D^{(n,l,\bm{\beta})}G(x',\mu')[z_1',...,z_n']|&\leq C_L\biggl(|x-x'|+\sum_{i=1}^N|z_i-z'_i|+\bb{W}_2(\mu,\mu') \biggr)
\end{align}
for all $(n,l,\bm{\beta})\in \bm{\zeta}$ and $x,x',z_1,...,z_n,z_1',...,z_n'\in\bb{\R},\mu,\mu'\in\mc{P}_2(\R)$ by $\mc{M}_{b,L}^{\bm{\zeta}}(\R\times\mc{P}_2(\R))$. We define $\mc{M}_b^{\bm{\zeta}}(\mc{P}_2(\R))$ and $\mc{M}_{b,L}^{\bm{\zeta}}(\mc{P}_2(\R))$ analogously, where instead here $\bm{\zeta}$ is a collection of multi-indices of the form $(n,\bm{\beta})\in\bb{N}\times\bb{N}^n$, and we take the $l=0$ in the above multi-index notation for the derivatives.

We will also make use of the class of functions $\mc{M}_{p}^{\bm{\zeta}}(\R\times\R\times \mc{P}_2(\R))$ which contains $G:\R\times\R\times\mc{P}_2(\R)\tto \R$ such that $G(\cdot,y,\cdot)\in \mc{M}_{b}^{\bm{\zeta}}(\R\times \mc{P}_2(\R))$ for all $y\in\R$, with all derivatives appearing in the definition of $\mc{M}_{b}^{\bm{\zeta}}(\R\times \mc{P}_2(\R))$ jointly continuous in $(x,y,\bb{W}_2)$, and for each multi-index $(n,l,\bm{\beta})\in\bm{\zeta}$,
\begin{align}\label{eq:newqnotation}
\sup_{x,z_1,...,z_n\in\R,\mu\in\mc{P}_2(\R)}|D^{(n,l,\bm{\beta})}G(x,y,\mu)[z_1,...,z_n]|&\leq C(1+|y|)^{q_G(n,l,\bm{\beta})},
\end{align}
where $q_G(n,l,\bm{\beta})\in\R$.
Similarly, $\mc{M}_{p,L}^{\bm{\zeta}}(\R\times \R\times \mc{P}_2(\R))$ is defined as $G\in \mc{M}_{p}^{\bm{\zeta}}(\R\times\R\times \mc{P}_2(\R))$ such that Equation \eqref{eq:LionsClassLipschitz} holds for $G(\cdot,y,\cdot)$ for each $y\in \R$, where $C_L(y)$ grows at most polynomially in $y$.

We also define $\mc{M}_b^{\bm{\zeta}}([0,T]\times\R\times\mc{P}_2(\R))$ to be the class of functions $G:[0,T]\times\R\times \mc{P}_2(\R)\tto \R$ such that $G(\cdot,x,\mu)$ is continuously differentiable on $(0,T)$ for all $x\in\R,\mu\in\mc{P}_2(\R)$ with time derivative denoted by $\dot{G}(t,x,\mu)$, $G(t,\cdot,\cdot) \in \mc{M}_b^{\bm{\zeta}}(\R\times\mc{P}_2(\R))$ for all $t\in [0,T]$, with \eqref{eq:LionsClassBoundedness} holding uniformly in $t$, and $G,\dot{G},$ and all derivatives involved in the definition of $\mc{M}_b^{\bm{\zeta}}(\R\times\mc{P}_2(\R))$ are jointly continuous in time, measure, and space. We define for $G\in \mc{M}_b^{\bm{\zeta}}([0,T]\times\R\times\mc{P}_2(\R))$
\begin{align*}
\norm{G}_{\mc{M}_b^{\bm{\zeta}}([0,T]\times\R\times\mc{P}_2(\R))}\coloneqq \sup_{t\in[0,T]}\norm{G(t,\cdot)}_{\mc{M}_b^{\bm{\zeta}}(\R\times\mc{P}_2(\R))}+\sup_{t\in[0,T],x\in\R,\mu\in\mc{P}_2(\R)}|\dot{G}(t,x,\mu)|.
\end{align*}
We denote the class of functions $G\in \mc{M}_b^{\bm{\zeta}}([0,T]\times\R\times\mc{P}_2(\R))$ such that \eqref{eq:LionsClassLipschitz} holds uniformly in $t$ by $\mc{M}_{b,L}^{\bm{\zeta}}([0,T]\times\R\times\mc{P}_2(\R))$. Again, we define $\mc{M}_b^{\bm{\zeta}}([0,T]\times\mc{P}_2(\R))$, $\mc{M}_{b,L}^{\bm{\zeta}}([0,T]\times\mc{P}_2(\R))$, and $\mc{M}_p^{\bm{\zeta}}([0,T]\times\R\times\R\times\mc{P}_2(\R))$ analogously.

At times we will want to consider Lions Derivatives bounded in $L^2(\R,\mu)$ rather than uniformly in $z$. Thus we define $\tilde{\mc{M}}_b^{\bm{\zeta}}(\R\times\mc{P}_2(\R))$ to be the class of functions $G:\R\times\mc{P}_2(\R)\tto \R$ such that $D^{(n,l,\bm{\beta})}G(x,\mu)[z_1,...,z_n]$ exists and satisfies
\begin{align}
\label{eq:LionsClassL2Boundedness} \norm{G}_{\tilde{\mc{M}}_b^{\bm{\zeta}}(\R\times\mc{P}_2(\R))}&\coloneqq \sup_{(n,l,\bm{\beta})\in \bm{\zeta}}\sup_{x\in\R,\mu\in\mc{P}_2(\R)}\norm{D^{(n,l,\bm{\beta})}G(x,\mu)[\cdot]}_{L^2(\mu,\R)^{\otimes n}} \\
&=\sup_{(n,l,\bm{\beta})\in \bm{\zeta}}\sup_{x\in\R,\mu\in\mc{P}_2(\R)}\biggl(\int_{\R}...\int_{\R} |D^{(n,l,\bm{\beta})}G(x,\mu)[z_1,...,z_n]|^2\mu(dz_1)...\mu(dz_n)  \biggr)^{1/2}  \leq C. \nonumber
\end{align}

We also define $\tilde{\mc{M}}_b^{\bm{\zeta}}([0,T]\times\R\times\mc{P}_2(\R))$ analogously to $\mc{M}_b^{\bm{\zeta}}([0,T]\times\R\times\mc{P}_2(\R))$,  $\tilde{\mc{M}}_{p}^{\bm{\zeta}}(\R\times\R\times \mc{P}_2(\R))$ analogously to $\mc{M}_{p}^{\bm{\zeta}}(\R\times\R\times \mc{P}_2(\R))$, and $\tilde{\mc{M}}_{p}^{\bm{\zeta}}([0,T]\times\R\times\R\times \mc{P}_2(\R))$ analogously to $\mc{M}_{p}^{\bm{\zeta}}([0,T]\times\R\times\R\times \mc{P}_2(\R))$. We will denote the polynomial growth rate for $G\in \tilde{\mc{M}}_{p}^{\bm{\zeta}}(\R\times\R\times \mc{P}_2(\R))$ and $(n,l,\bm{\beta})\in\bm{\zeta}$ as in Equation \eqref{eq:newqnotation} but with the $L^2(\R,\mu)^{\otimes n}$-norm by $\tilde{q}(n,l,\bm{\beta})\in\R$ to avoid confusion with polynomial growth of the derivatives in the uniform norm. That is:
\begin{align}\label{eq:tildeq}
\sup_{x\in\R,\mu\in\mc{P}_2(\R)}\biggl(\int_{\R}...\int_{\R} |D^{(n,l,\bm{\beta})}G(x,\mu)[z_1,...,z_n]|^2\mu(dz_1)...\mu(dz_n)  \biggr)^{1/2}\leq C(1+|y|)^{\tilde{q}_G(n,l,\bm{\beta})}.
\end{align}
Lastly, we define $\mc{M}_{\bm{\delta},b}^{\bm{\zeta}}(\R\times \mc{P}_2(\R))$ and $\mc{M}_{\bm{\delta},p}^{\bm{\zeta}}(\R\times\R\times \mc{P}_2(\R))$ in the same way as $\mc{M}_{b}^{\bm{\zeta}}(\R\times \mc{P}_2(\R))$ and $\mc{M}_{p}^{\bm{\zeta}}(\R\times\R\times \mc{P}_2(\R))$ respectively, but with with $D^{(n,l,\bm{\beta})}$ replaced by $\bm{\delta}^{(n,l,\bm{\beta})}$. We also extend this in the natural way when the spatial components are in higher dimensions (i.e. taking gradients and using norms in $\R^d$).
\end{defi}

Let us now introduce the main assumptions that are needed for the work of this paper to go through. 
\begin{enumerate}[label={A\arabic*)}]
{}\item \label{assumption:uniformellipticity} $0<\lambda_-\leq \tau_1^2(x,y,\mu)+\tau_2^2(x,y,\mu)\leq \lambda_+<\infty$, $\forall x,y\in\R,\mu\in\mc{P}_2(\R)$, and $\tau_1,\tau_2$ have two uniformly bounded derivatives in $y$ and which are jointly continuous in $(x,y,\bb{W}_2).$
\item \label{assumption:retractiontomean} 
There exists $\beta>0$ and $\kappa>0$ such that:
\begin{align}\label{eq:fnearOU}
f(x,y,\mu) = -\kappa y + \eta(x,y,\mu)
\end{align}
where $\eta$ is uniformly bounded in $x$ and $\mu$, and Lipschitz in the sense of \ref{assumption:uniformLipschitzxmu} in $x$, $\mu$, and $y$ with
\begin{align*}
|\eta(x,y_1,\mu)-\eta(x,y_2,\mu)|\leq L_{\eta}|y_1-y_2|,\forall x\in\R,\mu\in\mc{P}_2(\R)
\end{align*}
for $L_\eta$ such that $L_\eta - \kappa<0$, and
\begin{align}\label{eq:rocknertyperetractiontomean}
2(f(x,y_1,\mu)-f(x,y_2,\mu))(y_1-y_2)+3|\tau_1(x,y_1,\mu)-\tau_1(x,y_2,\mu)|^2 &+3|\tau_2(x,y_1,\mu)-\tau_2(x,y_2,\mu)|^2 \leq -\beta |y_1-y_2|^2,\\
&\forall x,y_1,y_2\in \R,\mu\in\mc{P}(\R).\nonumber
\end{align}
\end{enumerate}

Let $a(x,y,\mu)=\frac{1}{2}[\tau_1^2(x,y,\mu)+\tau_2^2(x,y,\mu)]$. For $x\in \R,\mu\in \mc{P}_2(\R)$, we define the differential operator $L_{x,\mu}$ acting on $\phi \in C_b^2(\R)$ by
\begin{align}\label{eq:frozengeneratormold}
L_{x,\mu}\phi(y) = f(x,y,\mu)\phi'(y)+a(x,y,\mu)\phi''(y).
\end{align}

Note that under assumptions \ref{assumption:uniformellipticity} and \ref{assumption:retractiontomean},  there is a constant $C$ independent of $x,y,\mu$ such that:
\begin{align}\label{eq:fdecayimplication}
2f(x,y,\mu)y+3|\tau_1(x,y,\mu)|^2+3|\tau_2(x,y,\mu)|^2 \leq -\frac{\beta}{2} |y|^2 +C,\forall x,y\in\R,\mu\in\mc{P}_2(\R).
\end{align}
Thus by \cite{PV1} Proposition 1 (see also \cite{Veretennikov1987}), there exists a $\pi(\cdot;x,\mu)$ which is the unique measure solving
\begin{align}\label{eq:invariantmeasureold}
L_{x,\mu}^*\pi=0.
\end{align}

Moreover, for all $k>0$, there is $C_k\geq 0$ such that $\sup_{x\in \R,\mu\in\mc{P}_2(\R)}\int_\R |y|^k \pi(dy;x,\mu)\leq C_k$.

\begin{enumerate}[label={A\arabic*)}]\addtocounter{enumi}{2}
\item \label{assumption:centeringcondition} For $\pi$ as in Equation \eqref{eq:invariantmeasureold},
\begin{align}\label{eq:centeringconditionold}
\int_\R b(x,y,\mu)\pi(dy;x,\mu)=0,\forall x\in\R,\mu\in\mc{P}_2(\R),
\end{align}
$b$ is jointly continuous in $(x,y,\bb{W}_2)$, grows at most polynomially in $y$ uniformly in $x\in\R,\mu\in\mc{P}_2(\R)$.
\end{enumerate}

Having introduced the notation above, we can now present the law of large numbers for the empirical measure $\mu^{\epsilon,N}$ from Equation \ref{eq:empiricalmeasures} in the joint limit as $\epsilon\downarrow 0,N\toinf$. Under assumptions \ref{assumption:uniformellipticity}-\ref{assumption:centeringcondition}, by Lemma \ref{lemma:Ganguly1DCellProblemResult} we consider $\Phi$ the unique classical solution to:
\begin{align}\label{eq:cellproblemold}
L_{x,\mu}\Phi(x,y,\mu) &= -b(x,y,\mu),\qquad
\int_{\R}\Phi(x,y,\mu)\pi(dy;x,\mu)=0.
{}\end{align}

Let us define the functions
\begin{align}\label{eq:limitingcoefficients}
\gamma(x,y,\mu)& \coloneqq \gamma_1(x,y,\mu)+c(x,y,\mu)\\
\gamma_1(x,y,\mu)&\coloneqq b(x,y,\mu)\Phi_x(x,y,\mu)+g(x,y,\mu)\Phi_y(x,y,\mu)+\sigma(x,y,\mu)\tau_1(x,y,\mu)\Phi_{xy}(x,y,\mu) \nonumber \\
D(x,y,\mu) & \coloneqq D_1(x,y,\mu)+\frac{1}{2}\sigma^2(x,y,\mu) \nonumber\\
D_1(x,y,\mu)& = b(x,y,\mu)\Phi(x,y,\mu)+\sigma(x,y,\mu)\tau_1(x,y,\mu)\Phi_{y}(x,y,\mu). \nonumber
\end{align}
and
\begin{align}\label{eq:averagedlimitingcoefficients}
\bar{\gamma}(x,\mu) &\coloneqq \biggl[\int_\R \gamma(x,y,\mu) \pi(dy;x,\mu)\biggr],
\qquad
\bar{D}(x,\mu) \coloneqq\biggl[\int_\R D(x,y,\mu) \pi(dy;x,\mu)\biggr].
\end{align}

Then, by essentially the same arguments as in \cite{BS}, under the conditions outlined below, $\mu^{\epsilon,N}$ converges in distribution to the deterministic limit  $\mc{L}(X)$ where $X$ satisfies the averaged McKean-Vlasov SDE
\begin{align}\label{eq:LLNlimitold}
dX_t &= \bar{\gamma}(X_t,\mc{L}(X_t))dt+\sqrt{2\bar{D}(X_t,\mc{L}(X_t))}dW^2_t\quad
X_0 = \eta^x.
\end{align}

Here $W^2_t$ is a Brownian motion on another filtered probability space satisfying the usual conditions.
In fact, we see here in Lemma \ref{lemma:W2convergenceoftildemu} that in fact this convergence occurs in $\mc{P}_2(\R)$ for each $t\in [0,T].$
\begin{remark}\label{remark:alternateformofdiffusion}
Using an integration-by-parts argument, one can find that the diffusion coefficient $\bar{D}$ can be written in the alternative form
\begin{align}\label{eq:alternativediffusion}
\bar{D}(x,\mu) & = \frac{1}{2}\int_{\R} \left([\tau_2(x,y,\mu)\Phi_y(x,y,\mu)]^2 + [\sigma(x,y,\mu)+\tau_1(x,y,\mu)\Phi_y(x,y,\mu)]^2\right)\pi(dy;x,\mu),
\end{align}
and hence is non-negative. See \cite{Bensoussan} Chapter 3 Section 6.2 for a similar computation.
\end{remark}

We now introduce the remaining assumptions. Since we are dealing with fluctuations, we will need to be able to obtain rates of averaging, and thus there are several auxiliary Poisson equations involved in the proof of tightness. When there is more specific structure to the system of equations \eqref{eq:slowfast1-Dold}, these assumptions may be able to be verified on a case-by-case basis. In Subsection \ref{subsec:suffconditionsoncoefficients} we provide concrete examples for which all of the conditions imposed in the paper hold.  Remark \ref{remark:choiceof1Dparticles} and mainly Remark \ref{remark:ontheassumptions} discuss the meaning of these assumptions more thoroughly.
In doing so, it will be useful to define the following complete collections of multi-indices in the sense of Definitions \ref{def:multiindexnotation} and \ref{def:completemultiindex}:
\begin{align}\label{eq:collectionsofmultiindices}
\hat{\bm{\zeta}} &\ni \br{(0,j_1,0),(1,j_2,j_3),(2,j_4,(j_5,j_6)),(3,0,(j_7,0,0)):j_1\in\br{0,1,...4},j_2+j_3\leq 4,j_4+j_5+j_6\leq 2,j_7=0,1 }\\
\tilde{\bm{\zeta}}&\ni \br{(0,j_1,0),(1,j_2,j_3),(2,0,0):j_1=0,1,2,j_2+j_3 \leq 1}\nonumber\\
\tilde{\bm{\zeta}}_1&\ni \br{(j_1,j_2,0):j_1+j_2 \leq 1}\nonumber\\
\tilde{\bm{\zeta}}_2&\ni \br{(j,0,0):j=0,1}\nonumber\\
\tilde{\bm{\zeta}}_3&\ni \br{(0,j_1,0),(1,0,j_2):j_1=0,1,2,j_2=0,1}\nonumber \\
\bm{\zeta}_{x,l}&\ni \br{(0,j,0):j=0,1,..,l},l\in\bb{N}\nonumber\\
\bar{\bm{\zeta}} &\ni \br{(j,0,0):j=0,1,2}\nonumber\\
\bar{\bm{\zeta}}_l&\ni \br{(0,0,0),(1,0,j):j=0,1,...,l},l\in\bb{N}. \nonumber
\end{align}
In the following set of assumptions, recall that for $G:\R\times\R\times\mc{P}_2(\R)\tto \R$ and a multi-index $(n,l,\beta)$, $\tilde{q}_G(n,l,\beta)$ denotes the rate of polynomial growth in $y$ of the mixed derivative of $G$ corresponding to $(n,l,\beta)$ as per Equation \eqref{eq:tildeq} in Definition \ref{def:lionsderivativeclasses}. Recall also the spaces of functions of measures from Definition \ref{def:lionsderivativeclasses}.
\begin{enumerate}[label={A\arabic*)}]\addtocounter{enumi}{3}
\item \label{assumption:strongexistence} Strong existence and uniqueness holds for the system of SDEs \eqref{eq:slowfast1-Dold} for all $N\in\bb{N}$, for the Slow-Fast McKean-Vlasov SDEs \eqref{eq:IIDparticles}, and for the limiting McKean-Vlasov SDE \eqref{eq:LLNlimitold}. 
\item \label{assumption:gsigmabounded} $g$ and $\sigma$ are uniformly bounded, and $c,b$ grow at most linearly in $y$ uniformly in $x\in\R,\mu\in\mc{P}_2(\R)$. All coefficients are jointly continuous in $(x,y,\bb{W}_2).$ 
\item \label{assumption:multipliedpolynomialgrowth}  There exists a unique strong solution $\Phi\in \tilde{\mc{M}}_{p}^{\tilde{\bm{\zeta}}}(\R\times\R\times \mc{P}_2(\R))$ to Equation \eqref{eq:cellproblemold} with $\tilde{q}_{\Phi}(n,l,\bm{\beta})\leq 1,\forall (n,l,\bm{\beta})\in \tilde{\bm{\zeta}}$, and $\Phi_y\in \tilde{\mc{M}}_{p}^{\tilde{\bm{\zeta}}_2}(\R\times\R\times \mc{P}_2(\R))$, with $\tilde{q}_{\Phi_y}(n,l,\bm{\beta})\leq 1,\forall (n,l,\bm{\beta})\in \tilde{\bm{\zeta}_2}$. In addition, this can be strengthened to $\tilde{q}_\Phi(0,k,0)\leq 0,k=0,1$ and $\tilde{q}_{\Phi_y}(0,0,0)\leq 0$. (For Proposition \ref{prop:fluctuationestimateparticles1}  and Theorem \ref{theo:mckeanvlasovaveraging}).
\item \label{assumption:qF2bound}  There exists a unique strong solution $\chi\in \tilde{\mc{M}}_{p}^{\tilde{\bm{\zeta}}}(\R^2\times\R^2\times \mc{P}_2(\R))$ to Equation \eqref{eq:doublecorrectorproblem} with $\tilde{q}_{\chi}(n,l,\bm{\beta})\leq 1,\forall (n,l,\bm{\beta})\in \tilde{\bm{\zeta}}$, and $\chi_y\in \tilde{\mc{M}}_{p}^{\tilde{\bm{\zeta}}_1}(\R^2\times\R^2\times \mc{P}_2(\R))$, $\chi_{yy} \in \tilde{\mc{M}}_{p}^{(0,0,0)}(\R^2\times\R^2\times \mc{P}_2(\R))$ with $\tilde{q}_{\chi_y}(n,l,\bm{\beta})\leq 1,\forall (n,l,\bm{\beta})\in \tilde{\bm{\zeta}_1}$, $\tilde{q}_{\chi_{yy}}(0,0,0)\leq 1$. In addition, this can be strengthened to $\tilde{q}_\chi(0,k,0)\leq 0,k=0,1$ and $\tilde{q}_{\chi_y}(0,0,0)\leq 0$. (For Proposition \ref{prop:purpleterm1} and Theorem \ref{theo:mckeanvlasovaveraging}).
\item \label{assumption:forcorrectorproblem} For $F=\gamma,D,$ or $\sigma\psi_1+[\tau_1\psi_1+\tau_2\psi_2]\Phi_y$ for any $\psi_1,\psi_2\in C^\infty_c([0,T]\times\R\times \R)$, there exists a unique strong solution $\Xi\in \tilde{\mc{M}}_{p}^{\tilde{\bm{\zeta}}}([0,T]\times\R\times\R\times \mc{P}_2(\R))$ to Equation \eqref{eq:driftcorrectorproblem} with each of these choices of $F$, $\tilde{q}_{\Xi}(n,l,\bm{\beta})\leq 2,\forall (n,l,\bm{\beta})\in \tilde{\bm{\zeta}}$, and $\Xi_y\in \tilde{\mc{M}}_{p}^{\tilde{\bm{\zeta}}_1}([0,T]\times\R\times\R\times \mc{P}_2(\R))$ with $\tilde{q}_{\Xi_y}(n,l,\bm{\beta})\leq 2,\forall (n,l,\bm{\beta})\in \tilde{\bm{\zeta}}_1$. Moreover, we assume for all choices of $F$, this can be strengthened to $\tilde{q}_{\Xi}(n,l,\bm{\beta})\leq 1,\forall (n,l,\bm{\beta})\in \tilde{\bm{\zeta}}_1$ and $\tilde{q}_{\Xi_{y}}(0,0,0)\leq 1$. (For Propositions \ref{prop:llntypefluctuationestimate1}/ \ref{prop:LPlowerbound} and Theorem \ref{theo:mckeanvlasovaveraging}).
\item \label{assumption:uniformLipschitzxmu} For $F = \gamma,\sigma+\tau_1\Phi_y,\tau_2\Phi_y,\tau_1,\tau_2$:
\begin{align*}
|F(x_1,y_1,\mu_1)-F(x_2,y_2,\mu_2)|\leq C(|x_1-x_2|+|y_1-y_2|+\bb{W}_2(\mu_1,\mu_2)),\forall x_1,x_2,y\in\R,\mu_1,\mu_2\in \mc{P}_2(\R).
\end{align*}
(Lemmas \ref{lemma:tildeYbarYdifference} and \ref{lemma:XbartildeXdifference}).
\item \label{assumption:limitingcoefficientsLionsDerivatives} $\bar{\gamma},\bar{D}^{1/2}\in \mc{M}_{b,L}^{\hat{\bm{\zeta}}}(\R\times\mc{P}_2(\R))$. (For Theorem \ref{theo:mckeanvlasovaveraging}).
\item \label{assumption:tildechi} Consider the Poisson equation
\begin{align}\label{eq:tildechi}
L^2_{x,\bar{x},\mu}\tilde{\chi}(x,\bar{x},y,\bar{y},\mu) &= -b(x,y,\mu)\Phi(\bar{x},\bar{y},\mu),
\quad
\int_{\R}\int_{\R}\tilde{\chi}(x,\bar{x},y,\bar{y},\mu)\pi(dy;x,\mu)\pi(d\bar{y},\bar{x},\mu)=0.
\end{align}
where $L^2_{x,\bar{x},\mu}$ is as in Equation \eqref{eq:2copiesgenerator}. Assume there exists a unique strong solution $\tilde{\chi}\in \tilde{\mc{M}}_{p}^{\tilde{\bm{\zeta}}_3}(\R^2\times\R^2\times \mc{P}_2(\R))$ and $\tilde{\chi_y}\in \tilde{\mc{M}}_{p}^{\bm{\zeta}_{x,1}}(\R^2\times\R^2\times \mc{P}_2(\R))$  to Equation \eqref{eq:tildechi}.     (For Theorem \ref{theo:mckeanvlasovaveraging}).
\item \label{assumption:2unifboundedlinearfunctionalderivatives} $\tau_1,\tau_2,f,\gamma,\sigma+\tau_1\Phi_y,\tau_2\Phi_y \in \mc{M}^{\bar{\bm{\zeta}}}_{\bm{\delta},p}(\R\times\R\times\mc{P}_2(\R)).$ (For Lemmas \ref{lemma:tildeYbarYdifference} and \ref{lemma:XbartildeXdifference}). 
\item \label{assumption:limitingcoefficientsregularity} For $w$ as in Equation \eqref{eq:wdefinition} and $\bar{\gamma}$,$\bar{D}$ as in Equation \eqref{eq:averagedlimitingcoefficients}, $\bar{\gamma},\bar{D}\in\mc{M}_b^{\bm{\zeta}_{x,w+2}}(\R\times\mc{P}_2(\R))\cap \mc{M}_{\bm{\delta},b}^{\bar{\bm{\zeta}}_{w+2}}(\R\times\mc{P}_2(\R))$, and
\begin{align*}
\sup_{x\in\R,\mu\in\mc{P}_2(\R)}\norm{\frac{\delta}{\delta m}\bar{\gamma}(x,\mu)[\cdot]}_{w+2}+\sup_{x\in\R,\mu\in\mc{P}_2(\R)}\norm{\frac{\delta}{\delta m}\bar{D}(x,\mu)[\cdot]}_{w+2}<\infty.
\end{align*}
(For Lemmas \ref{lemma:Lnu1nu2representation}, \ref{lem:barLbounded}, \ref{lemma:4.32BW} and Proposition \ref{prop:limitsatisfiescorrectequations}).
\end{enumerate}

\begin{remark}\label{remark:choiceof1Dparticles}
 There is a current gap in the literature regarding rates of polynomial growth for derivatives of solutions to Poisson Equations of the form \eqref{eq:cellproblemold}, as outlined in \cite{GS} Remark A.1. Though in Proposition A.2 they state a result partially amending this issue, the bounds provided likely are not tight. In particular, under the assumption \ref{assumption:retractiontomean} which we require for moment bounds of the fast process (and hence slow) process in Section \ref{sec:aprioriboundsoncontrolledprocess}, their result cannot provide boundedness of derivatives in $y$ of $\Phi$ from \eqref{eq:cellproblemold}, or any of the other auxiliary Poisson equations which we consider. This in turn also makes it difficult to gain good rates of polynomial growth for derivatives in the parameters $x$ and $\mu$. We need strict control of these rates of growth, for the reasons outlined in Remark \ref{remark:ontheassumptions}. Stronger bounds are derived in the 1-D case in Proposition A.4 of \cite{GS}, so this makes gaining the necessary control much easier in the current setting (see the results contained in Subsection \ref{subsection:regularityofthe1Dpoissoneqn} in the Appendix). Note also the much stricter assumptions imposed when handling the multi-dimensional cell problem in Lemma \ref{lemma:extendedrocknermultidimcellproblem} (which is required to establish sufficient conditions for \ref{assumption:qF2bound}).
\end{remark}

\begin{remark}\label{remark:ontheassumptions}
Assumptions \ref{assumption:uniformellipticity} and \ref{assumption:retractiontomean} are used in tandem for the existence and uniqueness of the invariant measure $\pi$ from Equation \eqref{eq:invariantmeasureold}. Such an invariant measure exists under weaker recurrence conditions on $f$ (see, e.g. \cite{PV1} Proposition 1 ), but we use the near-Ornstein–Uhlenbeck structure assumed in \eqref{eq:fnearOU} and the form of the retraction to the mean \eqref{eq:rocknertyperetractiontomean} in order to prove certain moment bounds on the controlled fast process in the Appendix \ref{sec:aprioriboundsoncontrolledprocess}, and \eqref{eq:rocknertyperetractiontomean} is also used in order to gain sufficient conditions for the required regularity of the Poisson Equations in Assumptions \ref{assumption:multipliedpolynomialgrowth}- \ref{assumption:limitingcoefficientsregularity} in Appendix \ref{sec:regularityofthecellproblem}. In particular, \eqref{eq:fnearOU} is inspired by Assumption 4.1 (iii) in \cite{JS} and is needed for Lemma \ref{lemma:ytildeexpofsup}, and  \eqref{eq:rocknertyperetractiontomean} is a standard assumption for control of moments of SDEs over infinite time horizons and for controlling solutions of related Cauchy problems (see e.g. \cite{RocknerMcKeanVlasov} Assumption A.1 Equation (2.3)).

The centering condition \ref{assumption:centeringcondition} is standard in the theory of stochastic homogenization. Assumption \ref{assumption:strongexistence} is required in order to apply the weak-convergence approach to large deviations. In particular, it ensures that the prelimit control representation \eqref{eq:varrepfunctionalsBM} holds. This is known to hold, for example, under global Lipschitz assumptions on all the coefficients (see, e.g. \cite{Wang} Theorem 2.1 and Section 6.1 in \cite{RocknerMcKeanVlasov}), though can also be proved under much weaker assumptions. These two assumptions, along with existence and uniqueness of the invariant measure $\pi$ from Equation \eqref{eq:invariantmeasureold} and the Poisson Equation $\Phi$ from Equation \eqref{eq:cellproblemold}, can be seen as the crucial hypothesis of this paper. The rest of the assumptions are technical and essentially used to have sufficient conditions for tightness of the controlled fluctuations processes $\tilde{Z}^N$ from Equation (\eqref{eq:controlledempmeasure}) (and, in the case of Assumption \ref{assumption:limitingcoefficientsregularity}, to have uniqueness of solutions to its limit \eqref{eq:MDPlimitFIXED}).

The boundedness and linear growth of the coefficients from Assumption \ref{assumption:gsigmabounded} are used to restrict the growth of the coefficients so that second moments of the controlled fast process $\tilde{Y}^{i,\epsilon,N}$ from Equation \eqref{eq:controlledslowfast1-Dold} can be proved in Appendix \ref{sec:aprioriboundsoncontrolledprocess}, and to ensure that only knowing these second moment bounds are sufficient for boundedness of the remainder terms in e.g. the ergodic-type theorems of Section \ref{sec:ergodictheoremscontrolledsystem}. The joint continuity assumption is used to ensure that integrating the coefficients is a continuous function on the space of measures.

The Assumptions \ref{assumption:multipliedpolynomialgrowth}- \ref{assumption:limitingcoefficientsregularity} are listed in terms of the Poisson Equations and averaged coefficients (and hence implicitly in terms of $\Phi$ from Assumption \ref{assumption:multipliedpolynomialgrowth}) because these assumptions can be verified on a case-by-case basis when the differential operator \eqref{eq:frozengeneratormold} or the inhomogeneities considered have some special structure. See the Examples provided in Appendix \ref{subsec:suffconditionsoncoefficients}.

The growth required by the specific derivatives listed in Assumptions \ref{assumption:multipliedpolynomialgrowth} - \ref{assumption:forcorrectorproblem} are imposed in order to ensure that the remainder terms resulting form It\^o's formula in the Ergodic-Type Theorems in Section \ref{sec:ergodictheoremscontrolledsystem} 
are bounded. In particular, in Section \ref{sec:ergodictheoremscontrolledsystem}, we are dealing with the controlled slow-fast system \eqref{eq:controlledslowfast1-Dold}, which due to the controls a priori being at best $L^2$ integrable (see the bound \eqref{eq:controlassumptions}), we are only able to show that we have 2 bounded moments of the fast component (see Appendix \ref{sec:aprioriboundsoncontrolledprocess}). This is limiting, since the terms which show up in the Ergodic-Type Theorems are products of derivatives of the Poisson equation with the coefficients of the system \eqref{eq:slowfast1-Dold}, of which $c$ and $b$ may grow linearly as per assumption \ref{assumption:gsigmabounded}, and with the $L^2$ controls.

Using Assumption \ref{assumption:multipliedpolynomialgrowth} as an example and unpacking the multi-index notation, we are requiring $\Phi,\Phi_x,\Phi_y$ are bounded, and $\Phi_{xx},\partial_\mu \Phi,\partial_\mu\Phi_x,\partial_\mu \Phi_y,\partial_z\partial_\mu \Phi,\partial^2_\mu \Phi$ grow at most linearly in $y$ in their appropriate norms. Looking at the proof of Proposition \ref{prop:fluctuationestimateparticles1}, since we are taking the $L^2$ norm of the remainder terms $\tilde{B}_1-\tilde{B}_8$, we are essentially ensuring all the products showing up in these terms are $L^2$ bounded. In particular, in $\tilde{B}_7$, the controls are multiplied by $\Phi$ and $\Phi_x$, which is why we end up needing those derivatives to be bounded. $\Phi_y$ being bounded is needed elsewhere for essentially the same reason - see, e.g. the proof of Proposition \ref{prop:goodratefunction}, where we use that $B^N_t$ is bounded in $L^2$. The reasoning behind the Assumptions \ref{assumption:qF2bound} and \ref{assumption:forcorrectorproblem} are the exact same, with additional regularity of $\chi_y$ and $\Xi_y$ (replacing $\tilde{\bm{\zeta}}_2$ by $\tilde{\bm{\zeta}}_1$ means we are requiring $\chi_y$ and $\Xi_y$ have an $x$ derivative which grows at most linearly in addition to a $\mu$ derivative) and $\chi_{yy}$ required due to those additional terms showing up in $\bar{B}_2$ in Proposition \ref{prop:purpleterm1}, $C_2$ in Proposition \ref{prop:llntypefluctuationestimate1}, and $\bar{B}_{13}$ in Proposition \ref{prop:purpleterm1} respectively.

The Lipschitz continuity imposed in Assumption \ref{assumption:uniformLipschitzxmu} and the existence of two linear functional derivatives which grow at most polynomially in $y$ uniformly in $x,\mu$ imposed in Assumption \ref{assumption:2unifboundedlinearfunctionalderivatives} are used to couple the controlled particles \eqref{eq:controlledslowfast1-Dold} to the auxiliary IID particles \eqref{eq:IIDparticles} in Subsection \ref{subsec:couplingcontrolledtoiid}. In particular, the terms required to be Lipschitz are those which show up in the drift and diffusion of the processes which result from applying Proposition \ref{prop:fluctuationestimateparticles1} 
 to the controlled system and IID system respectively. The use of a Lipschitz property in such a coupling argument is standard - see, e.g. Lemma 1 in \cite{HM}. Since we don't assume that the coefficients have linear interaction with the measure, Assumption \ref{assumption:2unifboundedlinearfunctionalderivatives} is being used to apply Lemma \ref{lemma:rocknersecondlinfunctderimplication} to the listed functions. The result of that Lemma is essentially the Assumption (S3) made in \cite{KX}, which we are using in essentially the same manner that they are in their coupling argument in Theorem 2.4.

Assumption \ref{assumption:limitingcoefficientsLionsDerivatives} is tailored to ensure enough regularity of the coefficients of the Cauchy Problem on Wasserstein Space for  Theorem \ref{theo:mckeanvlasovaveraging} to hold- see \cite{BezemekSpiliopoulosAveraging2022} (in particular Lemma 5.1 therein). Assumption \ref{assumption:tildechi} is used to apply the same result, and requires the introduction of the additional auxiliary Poisson equation \ref{assumption:tildechi} which is defined similarly to $\chi$ from Assumption \ref{assumption:qF2bound} but with a different inhomogeneity due to an additional term which arises in \cite{BezemekSpiliopoulosAveraging2022} Proposition 4.4 due to the McKean-Vlasov dynamics. The use of this specific control over the derivatives of $\tilde{\chi}$ is discussed after the statement of Theorem \ref{theo:mckeanvlasovaveraging}.

Finally, Assumption \ref{assumption:limitingcoefficientsregularity} is needed for well-definedness/uniqueness of the limiting Equation \eqref{eq:MDPlimitFIXED}. See the analogous Assumptions 2.2/2.3 in \cite{BW}.

\end{remark}
\section{Main Results}\label{sec:mainresults}
We are now ready to state our main result, which takes the form of Theorem \ref{theo:MDP} below. These results will be applied to a concrete class of examples of interacting particle systems of the form \eqref{eq:slowfast1-Dold} in Subsection \ref{SS:Examples}.

We prove the large deviations principle for fluctuations process $\br{Z^N}$ from Equation \eqref{eq:fluctuationprocess} via means of the Laplace Principle. In other words, in Theorem \ref{theo:MDP}, we identify the rate function $I:C([0,T];\mc{S}_{-r})\tto [0,+\infty]$ such that for $w$ as in Equation \eqref{eq:wdefinition}:
\begin{align}\label{eq:laplaceprincipledefinition}
\lim_{N\toinf}-a^2(N)\log \E \exp\biggl(-\frac{1}{a^2(N)}F(Z^N) \biggr) = \inf_{Z\in C([0,T];S_{-w})}\br{I(Z)+F(Z)}
\end{align}
for all $F\in C_b(C([0,T];\mc{S}_{-\tau}))$, for any $\tau\geq w$. In particular, this holds for all $F\in C_b(C([0,T];\mc{S}_{-r}))$ for $r>w+2$ as in Equation \eqref{eq:rdefinition}, and for such $F$ the right hand side is equal to $\inf_{Z\in C([0,T];S_{-r})}\br{I(Z)+F(Z)}$ by construction of $I$ (see Theorem \ref{theo:MDP}). The equality \eqref{eq:laplaceprincipledefinition} along with compactness of level sets of $I$ implies that $\br{Z^N}$ satisfies the large deviations principle with speed $a^{-2}(N)$ and rate function $I$ via e.g. Theorem 1.2.3 in \cite{DE}.

In order to show \eqref{eq:laplaceprincipledefinition}, we will show in Section \ref{sec:upperbound/compactnessoflevelsets} that the Laplace principle Lower Bound:
\begin{align}\label{eq:LPupperbound}
\liminf_{N\toinf}-a^2(N)\log \E \exp\biggl(-\frac{1}{a^2(N)}F(Z^N) \biggr) \geq \inf_{Z\in C([0,T];S_{-w})}\br{I(Z)+F(Z)},\forall F\in C_b(C([0,T];\mc{S}_{-\tau}))
\end{align}
for any $\tau\geq w$, with $w$ as in Equation \eqref{eq:wdefinition}, holds. 

Then, in Section \ref{sec:lowerbound} we will prove the Laplace principle Upper Bound:
\begin{align}\label{eq:LPlowerbound}
\limsup_{N\toinf}-a^2(N)\log \E \exp\biggl(-\frac{1}{a^2(N)}F(Z^N) \biggr) \leq \inf_{Z\in C([0,T];S_{-w})}\br{I(Z)+F(Z)},\forall F\in C_b(C([0,T];\mc{S}_{-\tau}))
\end{align}
for any $\tau\geq w$ holds and compactness of level sets of $I$ in $C([0,T];\mc{S}_{-r})$, at which point the moderate deviations principle of Theorem \ref{theo:MDP} will be established.

We now formulate the rate function. Consider the controlled limiting equation:
\begin{align}\label{eq:MDPlimitFIXED}
\langle Z_t,\phi\rangle &= \int_0^t \langle Z_s,\bar{L}_{\mc{L}(X_s)}\phi(\cdot)\rangle ds+\int_{\R\times\R\times\R^2\times[0,t]} \sigma(x,y,\mc{L}(X_s))z_1 \phi'(x)Q(dx,dy,dz,ds)\\
&+\int_{\R\times\R\times\R^2\times[0,t]} [\tau_1(x,y,\mc{L}(X_s))z_1+\tau_2(x,y,\mc{L}(X_s))z_2]\Phi_y(x,y,\mc{L}(X_s))\phi'(x)Q(dx,dy,dz,ds)\nonumber\\
\bar{L}_{\nu}\phi(x) & \coloneqq \bar{\gamma}(x,\nu)\phi'(x)+\bar{D}(x,\nu)\phi''(x)+\int_\R \left(\frac{\delta}{\delta m}\bar{\gamma}(z,\nu)[x]\phi'(z)+ \frac{\delta}{\delta m}\bar{D}(z,\nu)[x]\phi''(z)\right)\nu(dz),\nu\in\mc{P}(\R)\nonumber.
\end{align}
for all $\phi\in C^\infty_c(\R)$. Here we recall the limiting coefficients $\bar{\gamma},\bar{D}$ from Equation \eqref{eq:averagedlimitingcoefficients}, the limiting McKean-Vlasov Equation $X_t$ from Equation \eqref{eq:LLNlimitold}, and the linear functional derivative $\frac{\delta}{\delta m}$ from Definition \ref{def:LinearFunctionalDerivative}.

\begin{thm}\label{theo:Laplaceprinciple}
Let assumptions \ref{assumption:uniformellipticity} - \ref{assumption:limitingcoefficientsregularity} hold. Then $\br{Z^N}_{N\in\bb{N}}$ satisfies the Laplace principle \eqref{eq:laplaceprincipledefinition} with rate function $I$ given by
\begin{align}\label{eq:proposedjointratefunction}
I(Z)=\inf_{Q\in P^*(Z)}\biggl\lbrace\frac{1}{2}\int_{\R\times\R\times\R^2\times [0,T]} \left(z_1^2+z_2^2\right) Q(dx,dy,dz,ds)\biggr\rbrace
{}\end{align}
where $Q\in M_T(\R^4)$ (recall this space from above Equation \eqref{eq:rigorousLotimesdt}) is in $P^*(Z)$ if:
\begin{enumerate}[label={($P^*$\arabic*)}]
\item \label{PZ:limitingequation}$(Z,Q)$ satisfies Equation \eqref{eq:MDPlimitFIXED}
\item \label{PZ:L2contolbound} $\int_{\R\times\R\times\R^2\times [0,T]} \left(z_1^2+z_2^2\right) Q(dx,dy,dz,ds)<\infty$
\item \label{PZ:secondmarginalinvtmeasure} Disintegrating $Q(dx,dy,dz,ds) = \kappa(dz;x,y,s)\lambda(dy;x,s)Q_{(1,4)}(dx,ds)$, $\lambda(dy;x,s) = \pi(dy;x,\mc{L}(X_s))$ $\nu_{\mc{L}(X_\cdot)}$-almost surely, where $\pi$ is as in Equation \eqref{eq:invariantmeasureold} and $\nu_{\mc{L}(X_\cdot)}$ is as in Equation \eqref{eq:rigorousLotimesdt}.
\item \label{PZ:fourthmarginallimitnglaw} $Q_{(1,4)}= \nu_{\mc{L}(X_\cdot)}$.
\end{enumerate}
Here we use the convention that $\inf\br{\emptyset}=+\infty$.
\end{thm}

Replacing assumption \ref{assumption:limitingcoefficientsregularity} by the following:
\begin{enumerate}[label={A'\arabic*)}]\addtocounter{enumi}{12}
\item \label{assumption:limitingcoefficientsregularityratefunction} For $r$ as in Equation \eqref{eq:rdefinition} and $\bar{\gamma}$,$\bar{D}$ as in Equation \eqref{eq:averagedlimitingcoefficients}, $\bar{\gamma},\bar{D}\in\mc{M}_b^{\bm{\zeta}_{x,r+2}}(\R\times\mc{P}_2(\R))\cap \mc{M}_{\bm{\delta},b}^{\bar{\bm{\zeta}}_{r+2}}(\R\times\mc{P}_2(\R))$ (recalling these spaces from Definition \ref{def:lionsderivativeclasses} and these collections of multi-indices from Equation \eqref{eq:collectionsofmultiindices}), and
\begin{align*}
\sup_{x\in\R,\mu\in\mc{P}(R)}\norm{\frac{\delta}{\delta m}\bar{\gamma}(x,\mu)[\cdot]}_{r+2}+\sup_{x\in\R,\mu\in\mc{P}(R)}\norm{\frac{\delta}{\delta m}\bar{D}(x,\mu)[\cdot]}_{r+2}<\infty.
\end{align*}
\end{enumerate}
we can in addition prove compactness of level sets of the rate function given in \eqref{eq:proposedjointratefunction} by extending it to a larger space. For a discussion of the necessity of this extension, see the comments below Equation (2.10) and below Equation (4.33) in \cite{BW}. This yields the main result:
\begin{thm}\label{theo:MDP}
Let assumptions \ref{assumption:uniformellipticity} - \ref{assumption:2unifboundedlinearfunctionalderivatives} and \ref{assumption:limitingcoefficientsregularityratefunction} hold. Then $\br{Z^N}_{N\in\bb{N}}$ from Equation \eqref{eq:fluctuationprocess} satisfies the large deviation principle on the space $C([0,T];\mc{S}_{-r})$, with $r$  as in Equation \eqref{eq:rdefinition}, speed $a^{-2}(N)$ and good rate function $I$ given as in Equation \eqref{eq:proposedjointratefunction}. 
Here we use the convention that $\inf\br{\emptyset}=+\infty$, and also impose that $I(Z)=+\infty$ for $Z\in C([0,T];\mc{S}_{-r})\setminus C([0,T];\mc{S}_{-w})$.
\end{thm}

As is typically the case when using the weak convergence approach of \cite{DE} to prove a large deviations principle, the rate function \eqref{eq:proposedjointratefunction} can also be characterized by controls in feedback form:
\begin{corollary}\label{cor:ordinarycontrolratefunction}
In the setting of Theorem \ref{theo:Laplaceprinciple}, we can alternatively characterize the rate function as:
\begin{align}\label{eq:proposedjointratefunctionordinary}
I^o(Z)=\inf_{h\in P^o(Z)}\biggl\lbrace\frac{1}{2}\int_{0}^T \E\biggl[\int_\R |h(s,X_s,y)|^2 \pi(dy;X_s,\mc{L}(X_s)) \biggr]ds\biggr\rbrace
\end{align}
where $h:[0,T]\times\R\times\R\tto \R^2$ is in $P^o(Z)$ if:
\begin{enumerate}[label={($P^o$\arabic*)}]
\item \label{Po:controlledlimitingeqn} $(Z,h)$ satisfies Equation \eqref{eq:MDPlimitFIXEDordinary} for all $t\in[0,T]$ and $\phi\in C_c^\infty(\R)$
\item $\int_{0}^T \E\biggl[\int_\R |h(s,X_s,y)|^2 \pi(dy;X_s,\mc{L}(X_s)) \biggr]ds<\infty.$
\end{enumerate}
Here we define:
\begin{align}\label{eq:MDPlimitFIXEDordinary}
\langle Z_t,\phi\rangle &= \int_0^t \langle Z_s,\bar{L}_{\mc{L}(X_s)}\phi(\cdot)\rangle ds+\int_0^t \E\biggl[\int_\R  \sigma(X_s,y,\mc{L}(X_s))h_1(s,X_s,y) \phi'(X_s)\pi(dy;X_s,\mc{L}(X_s))\biggr]ds\\
&+\int_0^t \E\left[\int_\R  [\tau_1(X_s,y,\mc{L}(X_s))h_1(s,X_s,y)+\tau_2(X_s,y,\mc{L}(X_s))h_2(s,X_s,y)]\times\right.\nonumber\\
&\hspace{7cm}\left.\times\Phi_y(X_s,y,\mc{L}(X_s))\phi'(X_s) \pi(dy;X_s,\mc{L}(X_s))\right.\biggr]ds\nonumber
\end{align}
Again, we use the convention that $\inf\br{\emptyset}=+\infty$. In the setting of Theorem \ref{theo:MDP}, we also impose that $I^o(Z)=+\infty$ for $Z\in C([0,T];\mc{S}_{-r})\setminus C([0,T];\mc{S}_{-w})$.
\end{corollary}

\begin{proof}
This follows from Jensen's inequality and the affine dependence of the coefficients on the controls. The details are omitted for brevity given that the argument is standard, e.g., see Section 5 in \cite{DS}.
\end{proof}

In addition, as a corollary to the proof of Theorem \ref{theo:MDP}, we extend the results from \cite{BW} as follows:
\begin{corollary}\label{cor:mdpnomulti}(MDP without Multiscale Structure)
Suppose that $b=f=g=\tau_1=\tau_2\equiv 0$ and $c(x,y,\mu)= c(x,\mu),\sigma(x,y,\mu)= \sigma(x,\mu)$. Let $v>4$ be sufficiently large that the canonical embedding $\mc{S}_{-4}\tto \mc{S}_{-v}$ is Hilbert-Schmidt, $\rho>6$ be sufficiently large that the canonical embedding $\mc{S}_{-v-2}\tto \mc{S}_{-\rho}$ is Hilbert-Schmidt, and $\bar{\bm{\zeta}}$ as in \eqref{eq:collectionsofmultiindices}. Assume also that $\sigma,c\in \mc{M}_{\bm{\delta},b}^{\bm{\zeta}}(\R\times \mc{P}_2(\R))$ and for $F(x,\mu) = c(x,\mu)$ or $\sigma(x,\mu)$:
\begin{enumerate}
\item $\sup_{\mu\in\mc{P}_2(\R)}|F(\cdot,\mu)|_{\rho+2}<\infty$
\item $\sup_{x\in\R,\mu\in\mc{P}_2(\R)}\norm{\frac{\delta}{\delta m}F(x,\mu)[\cdot]}_{\rho+2}<\infty$.
\end{enumerate}
Here we recall the space $\mc{M}_{\bm{\delta},b}$ from Definition \ref{def:lionsderivativeclasses}, the collection of multi-indices $\bm{\zeta}$ from Equation \eqref{eq:collectionsofmultiindices}, and the norms on $\mc{S}$ defined in Equations \eqref{eq:familyofhilbertnorms} and \eqref{eq:boundedderivativesseminorm}.
Then $\br{Z^N}_{N\in\bb{N}}$ satisfies a large deviation principle on the space $C([0,T];\mc{S}_{-\rho})$ with speed $a^{-2}(N)$ and good rate function $\tilde{I}^o$ given by
\begin{align}\label{eq:proposedjointratefunctionordinarynomultiscale}
\tilde{I}^o(Z)=\inf_{h\in \tilde{P}^o(Z)}\biggl\lbrace\frac{1}{2}\int_{0}^T \E\biggl[|h(s,X_s)|^2  \biggr]ds\biggr\rbrace
\end{align}
where $h:[0,T]\times\R\tto \R$ is in $\tilde{P}^o(Z)$ if:
\begin{enumerate}[label={($P^o$\arabic*)}]
\item \label{Ponomulti:limitingeqn} $(Z,h)$ satisfies Equation \eqref{eq:MDPlimitFIXEDordinarynomultiscale} for all $t\in[0,T]$ and $\phi\in C_c^\infty(\R)$
\item \label{Ponomulti:L2control} $\int_{0}^T \E\biggl[ |h(s,X_s)|^2 \biggr]ds<\infty$
\end{enumerate}
and $\inf\br{\emptyset}=+\infty$, $I(Z)=+\infty$ for $Z\in C([0,T];\mc{S}_{-\rho})\setminus C([0,T];\mc{S}_{-v})$.
Here we define:
\begin{align}\label{eq:MDPlimitFIXEDordinarynomultiscale}
\langle Z_t,\phi\rangle &= \int_0^t \langle Z_s,\tilde{L}_{\mc{L}(\tilde{X}_s)}\phi(\cdot)\rangle ds+\int_0^t \E\biggl[\sigma(\tilde{X}_s,\mc{L}(\tilde{X}_s)) h(s,\tilde{X}_s) \phi'(\tilde{X}_s)\biggr]ds\\
\tilde{L}_{\nu}\phi(x) & = c(x,\nu)\phi'(x)+\frac{\sigma^2(x,\nu)}{2}\phi''(x)+\int_\R \left(\frac{\delta}{\delta m}c(z,\nu)[x]\phi'(z) +\frac{1}{2}\frac{\delta}{\delta m}[\sigma^2(z,\nu)[x]]\phi''(z)\right) \nu(dz)\nonumber\\
\tilde{X}_t & = \eta^x + \int_0^t c(\tilde{X}_s,\mc{L}(\tilde{X}_s))ds + \int_0^t \sigma(\tilde{X}_s,\mc{L}(\tilde{X}_s))dW_s. \nonumber
\end{align}
\begin{proof}
This follows from Theorem \ref{theo:MDP}. The assumptions needed are vastly simplified due to the absence of multiscale structure. In particular, we have no need for the results from Section \ref{sec:ergodictheoremscontrolledsystem} and Subsection \ref{sec:averagingfullycoupledmckeanvlasov}. The rate function can be posed on a smaller space $C([0,T];\mc{S}_{-\rho})$ (agreeing with that of Theorem 2.1 in \cite{BW}) as opposed to the larger $C([0,T];\mc{S}_{-r})$ of our Theorem \ref{theo:MDP} due to the IID system \eqref{eq:IIDparticles} not depending on $\epsilon$ in this regime. In particular, this means $\bar{X}^\epsilon_t\overset{d}{=}\tilde{X}_t$ in the proof of Lemma \ref{lemma:Zboundbyphi4}, and hence the result is improved $C(T)|\phi|^2_1$ instead of $C(T)|\phi|^2_4$. Similarly, in the result of Lemma \ref{lemma:Lnu1nu2representation}, the bound on $R^N_t(\phi)$ can be improved from $\bar{R}(N,T)|\phi|_4$ to $\bar{R}(N,T)|\phi|_3$ using Lemma \ref{lemma:XbartildeXdifference} and the proof method of Proposition 4.2 in \cite{BW}. At this point tightness of $\br{\tilde{Z}^N}_{N\in\bb{N}}$ from Equation \eqref{eq:controlledempmeasure} can be proved in Proposition \ref{prop:tildeZNtightness}, but with the uniform 7-continuity of Equation \eqref{eq:implies7cont} improved to uniform 4-continuity, and hence the result holds with $w$ replaced by $v$. The remainder of the proofs in the paper found in Subsection \ref{SS:QNtightness} and Sections \ref{sec:identificationofthelimit}, and \ref{sec:upperbound/compactnessoflevelsets} then go through verbatim with $m$ and $w$ replaced by $v$ and $r$ replaced with $\rho$, but with the simplifications assumed on the coefficients allowing us to set many terms equal to $0$. In particular, in the controlled particle Equation \eqref{eq:controlledslowfast1-Dold}, we can set $\tilde{Y}^{i,\epsilon,N}\equiv 0$, and throughout the invariant measure $\pi$ from Equation \eqref{eq:invariantmeasureold} can be set to $\delta_0$, which makes dealing with the second marginals of the occupation measures $\br{Q^N}_{N\in\bb{N}}$ from Equation \eqref{eq:occupationmeasures} trivial. Lastly, in Section \ref{sec:lowerbound}, due to the lack of multiscale structure, there is no need for an approximation argument in the proof of Proposition \ref{prop:LPlowerbound}, and hence existence of solutions to \eqref{eq:MDPlimitFIXEDordinarynomultiscale} can be established in $C([0,T];\mc{S}_{-v})$ and compactness of level sets established in $C([0,T];\mc{S}_{-\rho})$ exactly as in Subsections 4.4 and 4.5 of \cite{BW}.
\end{proof}
\end{corollary}
\begin{remark}\label{remark:BWextension}
Note that, in contrast to \cite{BW}, which assumes a linear-in-measure form of the coefficients of Equation \eqref{eq:slowfast1-Dold} (without multiscale structure), i.e. that there are $\beta,\alpha:\R^2\tto\R$ such that $c(x,\mu) = \int_\R \beta(x,z)\mu(dz),\sigma(x,\mu)=\int_\R \alpha(x,z)\mu(dz)$, we do not suppose any particular form of $c(x,\mu)$, $\sigma(x,\mu)$ other than that they have sufficient regularity for the proof of tightness and existence/uniqueness of the limiting equation. We are able to do so via the use of Lemma \ref{lemma:rocknersecondlinfunctderimplication} (which holds also in the case without dependence of the function $p$ on $y$) and the assumption that $\sigma,c\in \mc{M}_{\bm{\delta},b}^{\bm{\zeta}}(\R\times \mc{P}_2(\R))$. For the specific linear form of $c$ and $\sigma$ assumed by \cite{BW}, $\frac{\delta}{\delta m}c(x,\mu)[z]=\beta(x,z)$ and $\frac{\delta}{\delta m}\sigma(x,\mu)[z]=\alpha(x,z)$, so the condition (2) from Corollary \ref{cor:mdpnomulti} in fact implies $\sigma,c\in \mc{M}_{\bm{\delta},b}^{\bm{\zeta}}(\R\times \mc{P}_2(\R))$. In addition, (1) and (2) are exactly the assumptions (a) and (b) from Condition 2.3 of \cite{BW} in this subcase, so indeed Corollary \ref{cor:mdpnomulti} provides a strict generalization of their result. See also Corollary \ref{corollary:dawsongartnerformnomulti} where we further extend this result to get an alternate form of the rate function analogous to that of Dawson-G\"artner \cite{DG}.
\end{remark}

It is also useful to characterize the way that the limiting equations \eqref{eq:MDPlimitFIXED},\eqref{eq:MDPlimitFIXEDordinary}, and \eqref{eq:MDPlimitFIXEDordinarynomultiscale} act on functions which depend both on time and space. Hence we make the following remark:
\begin{remark}\label{remark:limitingequationactingontimedependantfunctions}
We can alternatively characterize the controlled limiting Equation \eqref{eq:MDPlimitFIXED} (and analogously the ordinary controlled limiting Equations \eqref{eq:MDPlimitFIXEDordinary} and \eqref{eq:MDPlimitFIXEDordinarynomultiscale}) in terms of how the $Z$ acts on $\psi\in C^\infty_c(U\times\R)$, where $U$ is an open interval containing $[0,T]$. For Equation \eqref{eq:MDPlimitFIXEDordinary}, this characterization is:
\begin{align}\label{eq:contolledequationactingontimedependantfunctions}
&\langle Z_T,\psi(T,\cdot)\rangle = \int_0^T \langle Z_s,\dot{\psi} (s,\cdot)\rangle ds + \int_0^T \langle Z_s,\bar{L}_{\mc{L}(X_s)}\psi(s,\cdot)\rangle ds \\
&+\int_0^T \E\biggl[\int_\R  \sigma(X_s,y,\mc{L}(X_s))h_1(s,X_s,y) \psi_x(s,X_s)\pi(dy;X_s,\mc{L}(X_s))\biggr]ds \nonumber\\
&+\int_0^T \E\biggl[\int_\R  [\tau_1(X_s,y,\mc{L}(X_s))h_1(s,X_s,y)+\tau_2(X_s,y,\mc{L}(X_s))h_2(s,X_s,y)]\Phi_y(X_s,y,\mc{L}(X_s))\psi_x(s,X_s) \pi(dy;X_s,\mc{L}(X_s))\biggr]ds\nonumber\\
&Z_0=0.\nonumber
\end{align}
This is analogous to the form of the limiting equation seen in \cite{Orrieri} (Remark 2.9) and \cite{BW} (Remark 2.2).

\end{remark}

\section{On the form of the rate function}\label{sec:formofratefunction}
\subsection{Statement and Proof of Equivalent forms of the Rate Function}
Here we prove an alternative form of the moderate deviations rate function \eqref{eq:proposedjointratefunction}, which is analogous to the ``negative Sobolev'' form of the large deviations rate function for the empirical measure of weakly interacting diffusions found in Theorem 5.1 of the classical work of Dawson-G\"artner \cite{DG}. This is the first time such a form of the rate function has been provided in the moderate deviations setting, both with and without multiscale structure. The result for the specialized case without multiscale structure can be found as Corollary \ref{corollary:dawsongartnerformnomulti} below.

A direct connection between the variational form of the large deviations rate function from \cite{BDF} and the ``negative Sobolev'' form of \cite{DG} was recently made for the first time in \cite{BS} Section 5.2. In contrast to the large deviations setting, in the moderate deviations rate function \eqref{eq:proposedjointratefunctionordinary}, we already know the controls $h$ are in feedback form, but rather than being feedback controls of the limiting controlled processes $Z$ in Equation \eqref{eq:MDPlimitFIXEDordinary}, they are feedback controls of the law of large numbers $\mc{L}(X)$ from Equation \eqref{eq:LLNlimitold}. Moreover, contrast to in the large deviations setting of \cite{BS}, here we handle the dependence of the controls $h$ on the parameter $y$ do to the multiscale structure and obtaining the ``negative Sobolev'' form of the rate function uniformly.

In order to state the alternate form of the rate function we first need to recall the following definition:
\begin{defi}\label{def:absolutelycontinuous}(Definition 4.1 in \cite{DG})
For a compact set $K\subset \R$, we will denote the subspace of $C^\infty_c(\R)$ which have compact support contained in $K$ by $\mc{S}_K$.
Let $I$ be an interval on the real line. A map $Z:I\tto \mc{S}'$ is called absolutely continuous if for each compact set $K\subset \R$, there exists a neighborhood of $0$ in $\mc{S}_K$ and an absolutely continuous function $H_K:I\tto \R$ such that
\begin{align*}
|\langle Z(u),\phi\rangle - \langle Z(v),\phi\rangle | \leq |H_K(u)-H_K(v)|
\end{align*}
for all $u,v\in I$ and $\phi\in U_K$.
\end{defi}

It is also useful to recall the following result:
\begin{lemma}\label{lemma:DG4.2}(Lemma 4.2 in \cite{DG})
Assume the map $Z:I\tto \mc{S}'$ is absolutely continuous. Then the real function $\langle Z,\phi\rangle$ is absolutely continuous for each $\phi\in C^\infty_c(\R)$ and the derivative in the distribution sense
\begin{align*}
\dot{Z}(t)\coloneqq \lim_{h\downarrow 0}h^{-1}[Z(t+h)-Z(t)]
\end{align*}
exists for Lebesgue almost-every $t\in I$.
\end{lemma}

Now we are ready to state the equivalent form of the rate function: 
\begin{proposition}\label{prop:DGformofratefunction}
Let assumptions \ref{assumption:uniformellipticity} - \ref{assumption:2unifboundedlinearfunctionalderivatives} and \ref{assumption:limitingcoefficientsregularityratefunction} hold. Assume also that $\bar{D}(x,\mu)>0$ for all $x\in\R,\mu\in\mc{P}_2$, where $\bar{D}$ is as in Equation \eqref{eq:averagedlimitingcoefficients}. Let $r$ be as in Equation \eqref{eq:rdefinition}. Consider $I^{DG}:  C([0,T];\mc{S}_{-r})\tto [0,+\infty]$ given by:
\begin{align}\label{eq:DGratefunctionmultiscaleFIX}
I^{DG}(Z)& = \frac{1}{4}\int_0^T \sup_{\phi\in C^\infty_c(\R):\E[\bar{D}(X_t,\mc{L}(X_t))|\phi'(X_t)|^2]\neq 0}\frac{\biggl|\langle \dot{Z}_t-\bar{L}^*_{\mc{L}(X_t)}Z_t,\phi\rangle\biggr|^2}{\E\biggl[\bar{D}(X_t,\mc{L}(X_t))|\phi'(X_t)|^2\biggr]}dt
\end{align}
if $Z(0)=0$, $Z$ is absolutely continuous in the sense if Definition \ref{def:absolutelycontinuous}, and $Z\in C([0,T];\mc{S}_{-w})$, and $I^{DG}(Z)=+\infty$ otherwise. Here $X_t$ is as in Equation \eqref{eq:LLNlimitold}, $\dot{Z}$ is the time derivative of $Z$ in the distribution sense from Lemma \ref{lemma:DG4.2} and $\bar{L}^*_{\mc{L}(X_s)}:\mc{S}_{-w}\tto \mc{S}_{-(w+2)}$ is the adjoint of $\bar{L}_{\mc{L}(X_s)}:\mc{S}_{w+2}\tto \mc{S}_w$ given in Equation \eqref{eq:MDPlimitFIXED} (using here Lemma \ref{lem:barLbounded}).

Then $\br{Z^N}_{N\in\bb{N}}$ from Equation \eqref{eq:fluctuationprocess} satisfies a large deviation principle on the space $C([0,T];\mc{S}_{-r})$ with speed $a^{-2}(N)$ and good rate function $I^{DG}$.
\end{proposition}
\begin{remark}\label{remark:barDnondegenerate}
Note that the assumption that $\bar{D}(x,\mu)>0$ for all $x\in\R$ and $\mu\in\mc{P}_2(\R)$ is not very restrictive. In particular, via the representation for the density of the invariant measure $\pi$ given in Equation \eqref{eq:explicit1Dpi}, we know it is strictly positive for all $x,\mu$. Then via the representation for $\bar{D}(x,\mu)$ given in Equation \eqref{eq:alternativediffusion}, we have that if there is $x,\mu$ such that $\bar{D}(x,\mu)=0$, then for that $x,\mu$, we must have
\begin{align*}
[\tau_2(x,y,\mu)\Phi_y(x,y,\mu)]^2 + [\sigma(x,y,\mu)+\tau_1(x,y,\mu)\Phi_y(x,y,\mu)]^2 =0
\end{align*}
for Lebesgue-almost every $y\in\R$. This will only happen if $\sigma$ has a very specific relation to $f,\tau_1,\tau_2,b$ and hence $\Phi_y$.

\end{remark}

In order to prove Proposition \ref{prop:DGformofratefunction}, we first prove the following Lemma, which gives us a form of the rate function analogous to Equation (4.21) in \cite{DG}:
\begin{lemma}\label{lemma:Eqn4.21DGform}
Assume the same setup as Proposition \ref{prop:DGformofratefunction}. For $\psi \in C^\infty_c(U\times\R)$ and $Z\in S_{-w}$, define
\begin{align}\label{eq:FZ}
F_{Z}(\psi) & = \langle Z_T,\psi(T,\cdot)\rangle - \int_0^T \langle Z_s,\dot{\psi} (s,\cdot)\rangle ds -\int_0^T\langle Z_s , \bar{L}_{\mc{L}(X_s)}\psi(s,\cdot)\rangle ds
\end{align}
and consider $J:  C([0,T];\mc{S}_{-\rho})\tto [0,+\infty]$ given by:
\begin{align}
J(Z)& = \sup_{\psi\in C^\infty_c(U\times\R)}\biggl\lbrace F_Z(\psi)-\int_0^T\E\biggl[\bar{D}(X_t,\mc{L}(X_t))|\psi_x(t,X_t)|^2\biggr]dt \biggr\rbrace
\end{align}
if $Z_0=0$ and $Z\in C([0,T];\mc{S}_{-w})$, and $J(Z)=+\infty$ otherwise. Here $U$ is an open interval containing $[0,T]$. Then $\br{Z^N}_{N\in\bb{N}}$ satisfies a large deviation principle on the space $C([0,T];\mc{S}_{-r})$ with speed $a^{-2}(N)$ and good rate function $J$.
\end{lemma}

\begin{proof}
Since by Theorem 1.3.1 in \cite{DE}, the rate function for a sequence of random variables is unique, it suffices to show that $I^o=J$, where $I^o$ is from Corollary \ref{cor:ordinarycontrolratefunction}.
We note that by Remark \ref{remark:limitingequationactingontimedependantfunctions}, we can replace \ref{Po:controlledlimitingeqn} in definition of the multiscale ordinary rate function $I^o$ by $Z$ satisfying Equation \eqref{eq:contolledequationactingontimedependantfunctions}. We will also use the alternative form of $\bar{D}(x,\mu)$ provided by Equation \eqref{eq:alternativediffusion} in Remark \ref{remark:alternateformofdiffusion}.

First we show $J\leq I^o$. Take $Z$ such $I^o(Z)<\infty$. Then $P^o(Z)$ is non-empty, and for any $h\in P^o(Z)$ and, by Equation \eqref{eq:contolledequationactingontimedependantfunctions}, for any $\psi\in C^\infty_c(U\times\R)$:
\begin{align*}
&|F_{Z}(\psi)| = \biggl|\int_0^T \E\biggl[\int_\R  \biggl([\sigma(X_s,y,\mc{L}(X_s))+\tau_1(X_s,y,\mc{L}(X_s))\Phi_y(X_s,y,\mc{L}(X_s))]h_1(s,X_s,y)\\
&+\tau_2(X_s,y,\mc{L}(X_s))\Phi_y(X_s,y,\mc{L}(X_s))h_2(s,X_s,y) \biggr)\psi_x(s,X_s)\pi(dy;X_s,\mc{L}(X_s))\biggr]ds \biggr| \\
&\leq \biggl(\int_0^T \E\biggl[\int_\R  \biggl([\sigma(X_s,y,\mc{L}(X_s))+\tau_1(X_s,y,\mc{L}(X_s))\Phi_y(X_s,y,\mc{L}(X_s))]^2 \\
&\hspace{2cm}+ [\tau_2(X_s,y,\mc{L}(X_s))\Phi_y(X_s,y,\mc{L}(X_s))]^2\biggr)\pi(dy;X_s,\mc{L}(X_s))|\psi_x(s,X_s)|^2 \biggr]ds\biggr)^{1/2}\\
&\times \biggl(\int_0^T \E\biggl[\int_\R|h_1(s,X_s,y)|^2+|h_2(s,X_s,y)|^2\pi(dy;X_s,\mc{L}(X_s))\biggr]ds\biggr)^{1/2}\\
& = \sqrt{2}\biggl(\int_0^T \E\biggl[\bar{D}(X_s,\mc{L}(X_s))|\psi_x(s,X_s)|^2 \biggr]ds\biggr)^{1/2}\biggl(\int_0^T \E\biggl[\int_\R|h(s,X_s,y)|^2\pi(dy;X_s,\mc{L}(X_s))\biggr]ds\biggr)^{1/2}
\end{align*}
so in particular, if $\int_0^T \E\biggl[\bar{D}(X_s,\mc{L}(X_s))|\psi_x(s,X_s)|^2\biggr]ds=0$, then $F_{Z}(\psi)=0$.
Then, observing that $\psi \in C^\infty_c(U\times \R)$ if and only if for any $c\in \R\setminus\br{0}$, $c\psi \in C^\infty_c(U\times \R)$ and that $F_Z$ is linear, we have:
\begin{align*}
J(Z)& = \sup_{\psi\in C^\infty_c(U\times\R):\int_0^T \E\biggl[\bar{D}(X_s,\mc{L}(X_s))|\psi_x(s,X_s)|^2\biggr]ds\neq 0}\biggl\lbrace F_Z(\psi)-\int_0^T\E\biggl[\bar{D}(X_t,\mc{L}(X_t))|\psi_x(t,X_t)|^2\biggr]dt \biggr\rbrace\vee 0\\
&=\sup_{\psi\in C^\infty_c(U\times\R):\int_0^T \E\biggl[\bar{D}(X_s,\mc{L}(X_s))|\psi_x(s,X_s)|^2\biggr]ds\neq 0}\sup_{c\in\R}\biggl\lbrace c F_Z(\psi)-c^2\int_0^T\E\biggl[\bar{D}(X_t,\mc{L}(X_t))|\psi_x(t,X_t)|^2\biggr]dt \biggr\rbrace\vee 0\\
& = \sup_{\psi\in C^\infty_c(U\times\R):\int_0^T \E\biggl[\bar{D}(X_s,\mc{L}(X_s))|\psi_x(s,X_s)|^2\biggr]ds\neq 0}\frac{|F_Z(\psi)|^2}{4\int_0^T\E\biggl[\bar{D}(X_t,\mc{L}(X_t))|\psi_x(t,X_t)|^2\biggr]dt}.
\end{align*}
Returning to the above inequality and squaring both sides, we have
\begin{align*}
\frac{|F_Z(\psi)|^2}{2\int_0^T\E\biggl[\bar{D}(X_t,\mc{L}(X_t))|\psi_x(t,X_t)|^2\biggr]dt}\leq \int_0^T \E\biggl[\int_\R |h(s,X_s,y)|^2\pi(dy;X_s,\mc{L}(X_s))\biggr]ds,
\end{align*}
for all $\psi\in C^\infty_c(U\times\R)$ such that $\int_0^T\E\biggl[\bar{D}(X_s,\mc{L}(X_s))|\psi_x(s,X_s)|^2\biggr]ds\neq 0$ and all $h\in P^o(Z)$. So $J(Z)\leq I^{o}(Z)$.

Now we prove $J\geq I^o$.
Assume without loss of generality that $J(Z)\leq C<\infty$. Then, since
\begin{align*}
J(Z)&=\sup_{\psi\in C^\infty_c(U\times\R)}\sup_{c\in\R}\biggl\lbrace c F_Z(\psi)-c^2\int_0^T\E\biggl[\bar{D}(X_t,\mc{L}(X_t))|\psi_x(t,X_t)|^2\biggr]dt \biggr\rbrace= +\infty
\end{align*}
if there exists $\psi\in C^\infty_c(U\times\R)$ such that $F_Z(\psi)\neq 0$ and $\int_0^T\E\biggl[\bar{D}(X_t,\mc{L}(X_t))|\psi_x(t,X_t)|^2\biggr]dt=0$, we have
\begin{align*}
J(Z) =\sup_{\psi\in C^\infty_c(U\times\R):\int_0^T \E\biggl[\bar{D}(X_s,\mc{L}(X_s))|\psi_x(s,X_s)|^2\biggr]ds\neq 0}\frac{|F_Z(\psi)|^2}{4\int_0^T\E\biggl[\bar{D}(X_t,\mc{L}(X_t))|\psi_x(t,X_t)|^2\biggr]dt}.
\end{align*}
This shows that for all $\psi\in C^\infty_c(U\times\R)$
\begin{align}\label{eq:multiFZbounded}
\biggl|F_Z(\psi)\biggr|&\leq 2\sqrt{C}\biggl(\int_0^T\E\biggl[\bar{D}(X_t,\mc{L}(X_t))|\psi_x(t,X_t)|^2\biggr]dt\biggr)^{1/2}.
\end{align}

Now we borrow some notation from \cite{DG} (see p. 270-271).
We let for $t\in [0,T]$ $\nabla_t,(\cdot,\cdot)_t,$ and $|\cdot|_t$ be (formally) the Riemannian gradient, inner product, and Riemannian norm in the tangent space of the Riemannian structure on $\R$ induced by the diffusion matrix $t\mapsto \bar{D}(\cdot,\mc{L}(X_t))$, i.e.:
\begin{align*}
\nabla_t f& \coloneqq \bar{D}(\cdot,\mc{L}(X_t))\frac{df}{dx}, \qquad
(X,Y)_t  \coloneqq \bar{D}(\cdot,\mc{L}(X_t))^{-1}XY,\qquad
|X|_t\coloneqq (X,X)^{1/2}_t.
\end{align*}

In particular, note that
\begin{align*}
|\nabla_t f|^2_t & = \bar{D}(\cdot,\mc{L}(X_t))|\frac{df}{dx}|^2.
\end{align*}

Now, as on p.279 in \cite{DG}, we define $L^2[0,T]$ to be the Hilbert space of measurable maps $g:[0,T]\times\R\tto \R$ with finite norm
\begin{align*}
\norm{g} \coloneqq \biggl(\int_0^T \langle \mc{L}(X_t), |g(t,\cdot)|^2_t \rangle dt\biggr)^{1/2} = \biggl(\int_0^T \E[\bar{D}(X_t,\mc{L}(X_t))^{-1}|g(t,X_t)|^2]dt \biggr)^{1/2}
\end{align*}
and inner product
\begin{align*}
[g_1,g_2]& \coloneqq \int_0^T \langle \mc{L}(X_t),(g_1(t,\cdot),g_2(t,\cdot))_t\rangle dt = \int_0^T \E[\bar{D}(X_t,\mc{L}(X_t))^{-1} g_1(t,X_t)g_2(t,X_t) ]dt.
\end{align*}

Denote by $L^2_{\nabla}[0,T]$ the closure in $L^2[0,T]$ of the linear subset $L$ consisting of all maps $(s,x)\mapsto \nabla_s\psi(s,x)$, $\psi \in C^\infty_c(U\times\R)$. Then $F_Z$ can be viewed as a linear functional on $L$, and by the bound \eqref{eq:multiFZbounded}, is bounded. Then, by the Riesz Representation Theorem, there exists $\bar{h}\in L^2_\nabla[0,T]$ such that
\begin{align}\label{eq:FZreiszrep}
F_Z(\psi) = \int_0^T \langle \mc{L}(X_s),(\bar{h}(s,\cdot),\nabla_s\psi(s,\cdot))_s \rangle ds = \int_0^T \E[h(s,X_s)\psi_x(s,X_s)] ds, \psi \in C^\infty_c(U\times\R).
\end{align}

Note that actually, $L^2_{\nabla}$ must be considered not as a class of functions, but as a set of equivalence classes of functions agreeing $\nu_{\mc{L}(X_\cdot)}$-almost surely. This is of no consequence, however, since the bound \eqref{eq:multiFZbounded} ensures that $F_Z(\psi)=F_Z(\tilde{\psi})$ if $\psi_x$ and $\tilde{\psi}_x$ are in the same equivalence class (see p.279 in \cite{DG} and Appendix D.5 in \cite{FK} for a more thorough treatment of the space $L^2_{\nabla}[0,T]$ and its dual).

Consider $\tilde{h}(s,x,y):[0,T]\times\R\times\R\tto \R^2$ given by
\begin{align}\label{eq:tildeh}
\tilde{h}_1(s,x,y) &= \frac{1}{2\bar{D}(x,\mc{L}(X_s))}[\sigma(x,y,\mc{L}(X_s))+\tau_1(x,y,\mc{L}(X_s))\Phi_y(x,y,\mc{L}(X_s))]\bar{h}(s,x)\\
\tilde{h}_2(s,x,y) &= \frac{1}{2\bar{D}(x,\mc{L}(X_s))}\tau_2(x,y,\mc{L}(X_s))\Phi_y(x,y,\mc{L}(X_s))\bar{h}(s,x)\nonumber.
\end{align}

We have:
\begin{align}\label{eq:tildehL2norm}
&\int_{0}^T \E\biggl[\int_\R |\tilde{h}(s,X_s,y)|^2 \pi(dy;X_s,\mc{L}(X_s)) \biggr]ds\\
&= \int_{0}^T \E\biggl[ \frac{|\bar{h}(s,X_s)|^2}{4|\bar{D}(x,\mc{L}(X_s))|^2}\int_\R \biggl([\sigma(X_s,y,\mc{L}(X_s))+\tau_1(X_s,y,\mc{L}(X_s))\Phi_y(X_s,y,\mc{L}(X_s))]^2 \nonumber\\
&\hspace{4cm}+[\tau_2(X_s,y,\mc{L}(X_s))\Phi_y(X_s,y,\mc{L}(X_s))]^2\biggr) \pi(dy;X_s,\mc{L}(X_s)) \biggr]ds\nonumber\\
& = \frac{1}{2}\int_{0}^T \E\biggl[\bar{D}(x,\mc{L}(X_s))^{-1}|\bar{h}(s,X_s)|^2 \biggr]ds<\infty.\nonumber
\end{align}

Moreover, for $\psi \in C^\infty_c(U\times\R)$:
\begin{align*}
&\int_0^T \E\biggl[\int_\R  \biggl([\sigma(X_s,y,\mc{L}(X_s))+\tau_1(X_s,y,\mc{L}(X_s))\Phi_y(X_s,y,\mc{L}(X_s))]\tilde{h}_1(s,X_s,y)\\
&+\tau_2(X_s,y,\mc{L}(X_s))\Phi_y(X_s,y,\mc{L}(X_s))\tilde{h}_2(s,X_s,y) \biggr)\psi_x(s,X_s)\pi(dy;X_s,\mc{L}(X_s))\biggr]ds \\
&=\int_0^T \E\biggl[\int_\R  \biggl([\sigma(X_s,y,\mc{L}(X_s))+\tau_1(X_s,y,\mc{L}(X_s))\Phi_y(X_s,y,\mc{L}(X_s))]^2\\
&+[\tau_2(X_s,y,\mc{L}(X_s))\Phi_y(X_s,y,\mc{L}(X_s))]^2\biggr)\pi(dy;X_s,\mc{L}(X_s)) \frac{\bar{h}(s,X_s)}{2\bar{D}(X_s,\mc{L}(X_s)}\psi_x(s,X_s)\biggr]ds \\
& = \int_0^T\E\biggl[\bar{h}(s,X_s) \psi_x(s,X_s)\biggr]ds\\
& = F_Z(\psi) \text{ by Equation }\eqref{eq:FZreiszrep}.
\end{align*}

Thus, $\tilde{h}\in P^o(Z)$ by definition. Take a sequence $\br{\tilde{\psi}^n} \subset L$ such that $\tilde{\psi}^n\tto \bar{h}$ in $L^2[0,T]$. By virtue of $\tilde{\psi}^n\in L$, we have for each $n$, there is $\psi^n \in C^\infty_c(U\times\R)$ such that $\tilde{\psi}^n(s,x) = \nabla_s\psi^n(s,x) = \bar{D}(x,\mc{L}(X_s))\psi^n_x(s,x)$.

In particular, we have $|\tilde{\psi}^n|_t^2\tto |\bar{h}|_t^2$, so
\begin{align*}
\int_0^T \E\biggl[\bar{D}(X_t,\mc{L}(X_t))|\psi^n_x(t,X_t)|^2 \biggr]dt \tto \int_0^T \E\biggl[\bar{D}(X_t,\mc{L}(X_t))^{-1}|\bar{h}(t,X_t)|^2 \biggr]dt
\end{align*}
and $(\tilde{\psi}^n,\bar{h})_{t}\tto |\bar{h}|_t^2$, so
\begin{align*}
\int_0^T \E\biggl[\psi^n_x(t,X_t)h(t,X_t)\biggr]dt \tto \int_0^T \E\biggl[\bar{D}(X_t,\mc{L}(X_t))^{-1}|\bar{h}(t,X_t)|^2 \biggr]dt.
\end{align*}

Note that if $\int_0^T \E\biggl[\bar{D}(X_t,\mc{L}(X_t))^{-1}|\bar{h}(t,X_t)|^2 \biggr]dt= 0$, we have by Equation \eqref{eq:tildehL2norm}, that the relation holds $\int_{0}^T \E\biggl[\int_\R |\tilde{h}(s,X_s,y)|^2 \pi(dy) \biggr]ds=0$, and hence $I^o(Z)=0$, so the desired bound is trivial.

Assuming then that $\int_0^T \E\biggl[D(X_t,\mc{L}(X_t))^{-1}|\bar{h}(t,X_t)|^2 \biggr]dt\neq 0$, we may choose a subsequence of $\br{\psi^n_x}$ such that  $\int_0^T \E\biggl[\bar{D}(X_t,\mc{L}(X_t))|\psi^n_x(t,X_t)|^2 \biggr]dt \neq 0,\forall n$. Then:
\begin{align*}
J(Z)&\geq \frac{1}{4}\frac{\biggl(\int_0^T \E\biggl[\psi^n_x(s,X_s)\bar{h}(s,X_s)\biggr]ds\biggr)^2}{{\int_0^T\E\biggl[\bar{D}(X_t,\mc{L}(X_t))|\psi^n_x(t,X_t)|^2\biggr]dt}} \text{ for all }n\in\bb{N}\\
&\tto \frac{1}{4} \int_0^T \E\biggl[D(X_t,\mc{L}(X_t))^{-1}|\bar{h}(t,X_t)|^2 \biggr]dt \text{ as }n\toinf.
\end{align*}

By Equation \eqref{eq:tildehL2norm},
\begin{align*}
\frac{1}{4} \int_0^T \E\biggl[D(X_t,\mc{L}(X_t))^{-1}|\bar{h}(t,X_t)|^2 \biggr]dt = \frac{1}{2}\int_0^T\E\biggl[\int_\R |\tilde{h}(s,X_s,y)|^2\pi(dy;X_s,\mc{L}(X_s))\biggl]ds,
\end{align*}
so since $\tilde{h}\in P^o(Z)$:
\begin{align*}
J(Z)&\geq \frac{1}{2}\int_0^T\E\biggl[\int_\R |\tilde{h}(s,X_s,y)|^2\pi(dy;X_s,\mc{L}(X_s))\biggl]ds\geq I^o(Z).
\end{align*}
\end{proof}

Now we are ready to prove Proposition \ref{prop:DGformofratefunction}.
\begin{proof}[Proof of Proposition \ref{prop:DGformofratefunction}]
As noted, the form of the rate function proved in Lemma \ref{lemma:Eqn4.21DGform} is analogous to that of Equation (4.21) in \cite{DG}. We follow the proof of Lemma 4.8 in \cite{DG}, making changes to account for the multiscale structure and the entry of $\mc{L}(X_s)$ rather than $Z_s$ in the subtracted term in Equation (4.24), which comes the fact that we are looking at moderate deviations rather than large deviations. We also use the specific information about the optimal control from the proof of Lemma \ref{lemma:Eqn4.21DGform}.

Once again, it is sufficient to show $I^o=I^{DG}$, or equivalently, $J=I^{DG}$. First we show that $I^o=J\leq I^{DG}$. Let $Z\in C([0,T];\mc{S}_{-w})$ be such that $I^{DG}(Z)<\infty$. Note that
\begin{align*}
&\sup_{\phi\in C^\infty_c(\R):\E[\bar{D}(X_t,\mc{L}(X_t))|\phi'(X_t)|^2]\neq 0}\biggl\lbrace  \langle \dot{Z}_t-\bar{L}^*_{\mc{L}(X_t)}Z_t,\phi\rangle - \E\biggl[\bar{D}(X_t,\mc{L}(X_t))|\phi'(X_t)|^2\biggr]\biggr\rbrace \nonumber\\
&=\sup_{\phi\in C^\infty_c(\R):\E[\bar{D}(X_t,\mc{L}(X_t))|\phi'(X_t)|^2]\neq 0} \sup_{c\in \R}\biggl \lbrace  c\langle \dot{Z}_t-\bar{L}^*_{\mc{L}(X_t)}Z_t,\phi\rangle - c^2\E\biggl[\bar{D}(X_t,\mc{L}(X_t))|\phi'(X_t)|^2\biggr]\biggr\rbrace \nonumber\\
& =\frac{1}{4}\sup_{\phi\in C^\infty_c(\R):\E[\bar{D}(X_t,\mc{L}(X_t))|\phi'(X_t)|^2]\neq 0}\frac{\biggl|\langle \dot{Z}_t-\bar{L}^*_{\mc{L}(X_t)}Z_t,\phi\rangle\biggr|^2}{\E\biggl[\bar{D}(X_t,\mc{L}(X_t))|\phi'(X_t)|^2\biggr]}
\end{align*}
for all $t\in[0,T]$. So for any $\psi \in C^\infty_c(U\times\R)$:
\begin{align*}
I^{DG}(Z)&= \frac{1}{4}\int_0^T \sup_{\phi\in C^\infty_c(\R):\E[\bar{D}(X_t,\mc{L}(X_t))|\phi'(X_t)|^2]\neq 0}\frac{\biggl|\langle \dot{Z}_t-\bar{L}^*_{\mc{L}(X_t)}Z_t,\phi\rangle\biggr|^2}{\E\biggl[\bar{D}(X_t,\mc{L}(X_t))|\phi'(X_t)|^2\biggr]}dt\\
&=\int_0^T \sup_{\phi\in C^\infty_c(\R):\E[\bar{D}(X_t,\mc{L}(X_t))|\phi'(X_t)|^2]\neq 0}\biggl\lbrace  \langle \dot{Z}_t-\bar{L}^*_{\mc{L}(X_t)}Z_t,\phi\rangle - \E\biggl[\bar{D}(X_t,\mc{L}(X_t))|\phi'(X_t)|^2\biggr]\biggr\rbrace  dt\\
&\geq \int_0^T \langle \dot{Z}_t-\bar{L}^*_{\mc{L}(X_t)}Z_t,\psi(t,\cdot)\rangle- \E\biggl[\bar{D}(X_t,\mc{L}(X_t))|\psi_x(t,X_t)|^2\biggr] dt\\
& =  \langle Z_T,\psi(T,\cdot)\rangle-\int_0^T \langle Z_t,\dot{\psi}(t,\cdot)\rangle dt-\int_0^T \langle Z_t,\bar{L}_{\mc{L}(X_t)}\psi(t,\cdot)\rangle dt  - \E\biggl[\bar{D}(X_t,\mc{L}(X_t))|\psi_x(t,X_t)|^2\biggr]  dt,
\end{align*}
where in the last step we used Lemma 4.3 in \cite{DG}. Then taking the supremum over all $\psi\in C^\infty_c(U\times\R)$, we get $I^{DG}(Z)\geq J(Z)$, as desired.

Now we show that $I^{DG}\leq J=I^o$. Consider $Z\in C([0,T];\mc{S}_{-w})$ such that $J(Z)<\infty$.

In Lemma \ref{lemma:Eqn4.21DGform}, we proved for $\tilde{h}$ as in Equation \eqref{eq:tildeh}, $\tilde{h}\in P^o(Z)$. We also showed:
\begin{align*}
\frac{1}{2}\int_0^T\E\biggl[\int_{\R}|\tilde{h}(s,X_s,y)|^2\pi(dy)\biggr]ds \leq J(Z)=I^o(Z) \leq \frac{1}{2}\int_0^T\E\biggl[\int_{\R}|h(s,X_s,y)|^2\pi(dy)\biggr]ds,\forall h\in P^o(Z),
\end{align*}
so that in fact
\begin{align}\label{eq:tildehrepresentation}
J(Z)=I^o(Z)=\frac{1}{2}\int_0^T\E\biggl[\int_{\R}|\tilde{h}(s,X_s,y)|^2\pi(dy)\biggr]ds = \frac{1}{4}\int_{0}^T \E\biggl[\bar{D}(X_s,\mc{L}(X_s))^{-1}|\bar{h}(s,X_s)|^2 \biggr]ds,
\end{align}
where in the last inequality we used Equation \eqref{eq:tildehL2norm}.

Now, by the fact that $\tilde{h}\in P^o(Z)$, we have by Equation \eqref{eq:MDPlimitFIXEDordinary} that for all $0\leq s\leq t\leq T$ and $\phi\in C^\infty_c(\R)$:
\begin{align*}
&\langle Z_t,\phi\rangle - \langle Z_s,\phi\rangle \\
&= \int_s^t \langle Z_u,\bar{L}_{\mc{L}(X_u)}\phi(\cdot)\rangle du+\int_s^t \E\biggl[\int_\R  \biggl([\sigma(X_u,y,\mc{L}(X_u))+\tau_1(X_u,y,\mc{L}(X_u))\Phi_y(X_u,y,\mc{L}(X_u))]\tilde{h}_1(u,X_u,y) \\
&+ \tau_2(X_u,y,\mc{L}(X_u))\Phi_y(X_u,y,\mc{L}(X_u))\tilde{h}_2(u,X_u,y)\biggr) \pi(dy;X_u,\mc{L}(X_u))\phi'(X_u)\biggr]du\\
& = \int_s^t \langle Z_u,\bar{L}_{\mc{L}(X_u)}\phi(\cdot)\rangle du+\int_s^t \E\biggl[ \bar{h}(u,X_u)\phi'(X_u)\biggr]du
\end{align*}
where $\bar{h}$ is as in Equation \eqref{eq:FZreiszrep}, so by Definition \ref{def:absolutelycontinuous} and Lemma \ref{lem:barLbounded}, $Z$ is an absolutely continuous map from $[0,T]$ to $\mc{S}'$. Then, using Lemma \ref{lemma:DG4.2}, we have for each $\phi\in C^\infty_c(\R)$:
\begin{align}\label{eq:Zdotrep}
\langle \dot{Z}_t,\phi\rangle &= \E\biggl[ \bar{h}(t,X_t) \phi'(X_t)\biggr]+\langle Z_t,\bar{L}_{\mc{L}(X_t)}\phi(\cdot)\rangle.
\end{align}

Using a density argument, we can make sure this holds simultaneously for all $\phi \in C^\infty_c(\R)$ and Lebesgue almost every $t\in [0,T]$ (see p.280 of \cite{DG}). This gives:
\begin{align*}
I^{DG}(Z)&= \frac{1}{4}\int_0^T \sup_{\phi\in C^\infty_c(\R):\E[\bar{D}(X_t,\mc{L}(X_t))|\phi'(X_t)|^2]\neq 0}\frac{\biggl(\E\biggl[ \bar{h}(t,X_t) \phi'(X_t)\biggr]\biggr)^2}{\E\biggl[\bar{D}(X_t,\mc{L}(X_t))|\phi'(X_t)|^2\biggr]}dt
\end{align*}

For any $\phi\in C^\infty_c(\R)$ and $t\in [0,T]$ such that $\E\biggl[\bar{D}(X_t,\mc{L}(X_t))|\phi'(X_t)|^2\biggr]\neq 0$, we have
\begin{align*}
\frac{\biggl(\E\biggl[ \bar{h}(t,X_t) \phi'(X_t)\biggr]\biggr)^2}{\E\biggl[\bar{D}(X_t,\mc{L}(X_t))|\phi'(X_t)|^2\biggr]}& = \frac{\biggl(\E\biggl[ \bar{D}(X_t,\mc{L}(X_t))^{-1/2}\bar{h}(t,X_t) \bar{D}(X_t,\mc{L}(X_t))^{1/2}\phi'(X_t)\biggr]\biggr)^2}{\E\biggl[\bar{D}(X_t,\mc{L}(X_t))|\phi'(X_t)|^2\biggr]}\\
&\leq \E\biggl[ \bar{D}(X_t,\mc{L}(X_t))^{-1}|\bar{h}(t,X_t)|^2\biggr]
\end{align*}
so
\begin{align*}
I^{DG}(Z)&\leq  \frac{1}{4}\int_0^T \E\biggl[ \bar{D}(X_t,\mc{L}(X_t))^{-1}|\bar{h}(t,X_t)|^2\biggr]dt,
\end{align*}
and by Equation \eqref{eq:tildehrepresentation} we are done.
\end{proof}

As a corollary to the above result, we also get an alternative form of the rate function in the setting without multiscale structure. This provides us with rate functions with which it is more feasible to compare the likelihood of rare events for the fluctuation process \eqref{eq:fluctuationprocess} as $N\toinf$ in the multiscale and non-multiscale setting as opposed to the variational form given in Theorem \ref{theo:MDP} and Corollary \ref{cor:mdpnomulti}. This analysis is outside the scope of this paper, but is an interesting avenue for future research.

\begin{corollary}\label{corollary:dawsongartnerformnomulti}
In the setting of Corollary \ref{cor:mdpnomulti}, assume in addition $\sigma^2(x,\mu)>0$, for all $x\in\R$ and $\mu\in\mc{P}_2(\R)$. Consider $\tilde{I}^{DG}:  C([0,T];\mc{S}_{-\rho})\tto [0,+\infty]$ given by :
\begin{align}\label{eq:DGratefunctionnomultiscale}
\tilde{I}^{DG}(Z)\coloneqq \frac{1}{2}\int_0^T \sup_{\phi\in C^\infty_c(\R):\E[\sigma^2(X_t,\mc{L}(X_t))|\phi'(\tilde{X}_t)|^2]\neq 0}\frac{|\langle \dot{Z}_t-\tilde{L}^*_{\mc{L}(\tilde{X}_t)}Z_t,\phi\rangle|^2}{\E[\sigma^2(X_t,\mc{L}(X_t))|\phi'(\tilde{X}_t)|^2]}dt,
\end{align}
if $Z(0)=0$, $Z$ is absolutely continuous in the sense if Definition \ref{def:absolutelycontinuous}, and $Z\in C([0,T];\mc{S}_{-v})$, and $I^{DG}(Z)=+\infty$ otherwise. Here $\tilde{X}_t$ is as in Corollary \ref{cor:mdpnomulti}, $\dot{Z}$ is the time derivative of $Z$ in the distribution sense from Lemma \ref{lemma:DG4.2} and $\tilde{L}^*_{\mc{L}(\tilde{X}_s)}:\mc{S}_{-v}\tto \mc{S}_{-(v+2)}$ is the adjoint of $\tilde{L}_{\mc{L}(\tilde{X}_s)}:\mc{S}_{v+2}\tto \mc{S}_v$ given in Corollary \ref{cor:mdpnomulti} (using here Lemma \ref{lem:barLbounded}).

Then $\br{Z^N}_{N\in\bb{N}}$ satisfies a large deviation principle on the space $C([0,T];\mc{S}_{-\rho})$ with speed $a^{-2}(N)$ and good rate function $\tilde{I}^{DG}$.
\end{corollary}
\begin{proof}
This follows by the same proof as Proposition \ref{prop:DGformofratefunction}, removing the dependence of the control on $y$ and setting $\Phi\equiv 0$.

\end{proof}
\subsection{Examples: A Class of Aggregation-Diffusion Equations}\label{SS:Examples}

A common form for interacting particle systems which are widely used in many settings such as in biology, ecology, social sciences, economics, molecular dynamics, and in study of spatially homogeneous granular media (see e.g., \cites{MT,Garnier1,BCCP,KCBFL}) is:
\begin{align}\label{eq:nomultilangevin}
dX^{i,N}_t &= -V'(X^{i,N}_t)dt - \frac{1}{N}\sum_{j=1}^N W'(X^{i,N}_t-X^{j,N}_t)dt +\sigma dW^i_t, \quad X^{i,N}_0=\eta^x
\end{align}
where $V:\R\tto\R$ is a sufficiently smooth confining potential and $W:\R\tto \R$ is a sufficiently smooth interaction potential. The class of systems \eqref{eq:nomultilangevin} contains the system in the seminal paper \cite{Dawson}, where many mathematical aspects of a model for cooperative behavior in a bi-stable confining potential with attraction to the mean are explored. This leads us to our first example:
\begin{example}
Consider the system \eqref{eq:nomultilangevin}. Let $v,\rho$ be as in Corollary \ref{cor:mdpnomulti}. Suppose $V',W'\in C_b^{\rho+2}$, $W'\in \mc{S}_{\rho+2}$, and $\sigma>0$. Then  $\br{Z^N}_{N\in\bb{N}} = \br{a(N)\sqrt{N}[\frac{1}{N}\sum_{i=1}^N\delta_{X^{i,N}_\cdot}-\mc{L}(\tilde{X}_\cdot)]}_{N\in\bb{N}}$ satisfies a large deviation principle on the space $C([0,T];\mc{S}_{-\rho})$ with speed $a^{-2}(N)$ and good rate function $\tilde{I}^{DG}$ given by:
\begin{align*}
\tilde{I}^{DG}(Z)\coloneqq \frac{1}{2\sigma^2}\int_0^T \sup_{\phi\in C^\infty_c(\R):\E[|\phi'(\tilde{X}_t)|^2]\neq 0}\frac{|\langle \dot{Z}_t-\tilde{L}^*_{\mc{L}(\tilde{X}_t)}Z_t,\phi\rangle|^2}{\E[|\phi'(\tilde{X}_t)|^2]}dt,
\end{align*}
if $Z(0)=0$, $Z$ is absolutely continuous in the sense if Definition \ref{def:absolutelycontinuous}, and $Z\in C([0,T];\mc{S}_{-v})$, and $I^{DG}(Z)=+\infty$ otherwise.

Here $\tilde{X}_t$ satisfies:
\begin{align*}
d\tilde{X}_t & = -V'(\tilde{X}_t)dt- \bar{\E}[W'(x-\bar{X}_t)]|_{x=\tilde{X}_t}dt + \sigma dW_t,\quad \tilde{X}_0=\eta^x
\end{align*}
and $\tilde{L}_{\mc{L}(\tilde{X}_s)}:\mc{S}_{v+2}\tto \mc{S}_v$ acts on $\phi\in C^\infty_c(\R)$ by:
\begin{align*}
\tilde{L}_{\mc{L}(\tilde{X}_s)}\phi(x) & = -[V'(x)+\E[W'(x-\tilde{X}_s)]]\phi'(x)+\frac{\sigma^2}{2}\phi''(x)-\E[W'(\tilde{X}_s-x)\phi'(\tilde{X}_s)].
\end{align*}

We are denoting by $\bar{X}_t$ an independent copy of $\tilde{X}_t$ on another probability space $(\bar{\W},\bar{\F},\bar{\Prob})$, and by $\bar{\E}$ the expectation on that space.
\end{example}
\begin{proof}
Noting that $\sigma$ is constant and $\frac{\delta}{\delta m}c(x,\mu)[z] = -W'(x-z)$, the assumptions put forward in Corollary \ref{cor:mdpnomulti} can be directly verified. This example then immediately falls into the regime of Corollary \ref{corollary:dawsongartnerformnomulti}.
\end{proof}
In \cite{GP}, the authors make, among other modifications, a modification to $V$ in Equation \eqref{eq:nomultilangevin} so that it is a so-called rough-potential (see also \cite{Zwanzig} and \cite{BS} Section 5), by letting $V^\epsilon(x) = V_1(x)+V_2(x/\epsilon)$, where $V_2$ is sufficiently smooth and periodic. The system becomes:
\begin{align*}
dX^{i,\epsilon,N}_t = -[V_1'(X^{i,\epsilon,N}_t)+\frac{1}{\epsilon}V_2'(X^{i,\epsilon,N}_t/\epsilon)]dt - \frac{1}{N}\sum_{j=1}^N W'(X^{i,\epsilon,N}_t-X^{j,\epsilon,N}_t)dt +\sigma dW^i_t.
\end{align*}

Letting $Y^{i,\epsilon,N}_t=X^{i,\epsilon,N}_t/\epsilon$, we see this is a subclass of systems of the form \eqref{eq:slowfast1-Dold} with
\begin{align*}
f(x,y,\mu)&=b(x,y,\mu) = -V_2'(y),\quad
g(x,y,\mu)=c(x,y,\mu) = -V_1'(x)-\langle \mu , W'(x-\cdot)\rangle\\
\sigma(x,y,\mu) &= \tau_1(x,y,\mu)\equiv \sigma,\quad \tau_2\equiv 0,
\end{align*}

Keeping within our setting of a slow-fast system on $\R$, we consider a version of this system where the fast and slow dynamics are allowed to be different, and the fast system is not confined to the torus: \begin{align}\label{eq:slowfastLangevin}
dX^{i,\epsilon,N}_t &= -[V_1'(X^{i,\epsilon,N}_t)+\frac{1}{\epsilon}V_2'(Y^{i,\epsilon,N}_t)]dt - \frac{1}{N}\sum_{j=1}^N W_1'(X^{i,\epsilon,N}_t-X^{j,\epsilon,N}_t)dt +\sigma dW^i_t\\
dY^{i,\epsilon,N}_t & = -\frac{1}{\epsilon}[V_3'(X^{i,\epsilon,N}_t)+\frac{1}{\epsilon}V_4'(Y^{i,\epsilon,N}_t)]dt - \frac{1}{\epsilon}\frac{1}{N}\sum_{j=1}^N W_2'(X^{i,\epsilon,N}_t-X^{j,\epsilon,N}_t)dt + \frac{1}{\epsilon}\tau_1 dW^i_t+\frac{1}{\epsilon}\tau_2 dB^i_t\nonumber\\
(X^{i,\epsilon,N}_0,Y^{i,\epsilon,N}_0)& = (\eta^{x},\eta^{y}).\nonumber
\end{align}

This falls into the class of systems \eqref{eq:slowfast1-Dold} with
\begin{align*}
b(x,y,\mu) & = -V_2'(y),\quad
c(x,y,\mu) = -V_1'(x) - \langle \mu, W_1'(x-\cdot)\rangle,\quad
\sigma(x,y,\mu)\equiv  \sigma \\
f(x,y,\mu) &= -V_4'(y),\quad
g(x,y,\mu) = -V_3'(x) - \langle \mu, W_2'(x-\cdot)\rangle,\quad
\tau_1(x,y,\mu)\equiv  \tau_1,\quad
\tau_2(x,y,\mu)\equiv \tau_2.
\end{align*}
\begin{example}
Consider the system \eqref{eq:slowfastLangevin}.

Suppose $V_4(y)=\frac{\kappa}{2}y^2 + \tilde{\eta}(y)$ where $\kappa>0$ and $\tilde{\eta}\in C^2_b(\R)$ is even with $\norm{\tilde{\eta}''}_\infty<\kappa$, $V_1',V_3',W_1',W_2'\in C_b^{r+2}(\R)$, $W_1',W_2'\in \mc{S}_{r+2}$ where $r$ is as in Equation \eqref{eq:rdefinition}, $\sigma,\tau_2\neq 0$, $V_2$ is even, and $V_2'$ is Lipschitz continuous and $O(|y|^{1/2})$ as $|y|\toinf$. 

Then $\br{Z^N}_{N\in\bb{N}}= \br{a(N)\sqrt{N}[\frac{1}{N}\sum_{i=1}^N\delta_{X^{i,\epsilon,N}_\cdot}-\mc{L}(X_\cdot)]}_{N\in\bb{N}}$ satisfies a large deviation principle on the space $C([0,T];\mc{S}_{-r})$ with speed $a^{-2}(N)$ and good rate function $I^{DG}$ given by:
\begin{align*}
I^{DG}(Z)& = \frac{1}{2[\sigma^2+2\alpha a +2\sigma\tau_1\tilde{\alpha}]}\int_0^T \sup_{\phi\in C^\infty_c(\R):\E[|\phi'(X_t)|^2]\neq 0}\frac{\biggl|\langle \dot{Z}_t-\bar{L}^*_{\mc{L}(X_t)}Z_t,\phi\rangle\biggr|^2}{\E\biggl[|\phi'(X_t)|^2\biggr]}dt
\end{align*}
if $Z(0)=0$, $Z$ is absolutely continuous in the sense if Definition \ref{def:absolutelycontinuous}, and $Z\in C([0,T];\mc{S}_{-w})$, and $I^{DG}(Z)=+\infty$ otherwise. Here $X_t$ satisfies:
\begin{align*}
dX_t &= -[\tilde{\alpha}V_3'(X_t)+ V_1'(X_t)]dt -\bar{\E}[\tilde{\alpha}W_2'(x-\bar{X}_t)+W_1'(x-\bar{X}_t)]|_{x=X_t}dt+[\sigma^2+2\alpha a +2\sigma\tau_1\tilde{\alpha}]^{1/2}dW_t\\
X_0& = \eta^x \nonumber\\
\tilde{\alpha}&=\int_\R \Phi'(y)\pi(dy),\quad
\alpha = \int_\R [\Phi'(y)]^2\pi(dy),\quad
a = \frac{1}{2}[\tau_1^2 + \tau_2^2]
\end{align*}
and $\bar{L}_{\mc{L}(X_s)}:\mc{S}_{w+2}\tto \mc{S}_w$ acts on $\phi\in C^\infty_c(\R)$ by:
\begin{align*}
\bar{L}_{\mc{L}(X_s)}\phi(x) & \coloneqq -[\tilde{\alpha}V_3'(x)+ V_1'(x) + \E[\tilde{\alpha}W_2'(x-X_s)+W_1'(x-X_s)]]\phi'(x)+\frac{1}{2}[\sigma^2+2\alpha a +2\sigma\tau_1\tilde{\alpha}]\phi''(x)\\
&\qquad-\E[[\tilde{\alpha}W_2'(X_s-x)+W_1'(X_s-x)]\phi'(X_s)]\nonumber.
\end{align*}
Again, we are denoting by $\bar{X}_t$ an independent copy of $X_t$ on another probability space $(\bar{\W},\bar{\F},\bar{\Prob})$ and by $\bar{\E}$ the expectation on that space.
\end{example}
\begin{proof}
Once we show $\int_\R V_2'(y)\pi(dy)=0$ for $\pi$ as in Equation \eqref{eq:invariantmeasureold}, it follows that assumptions \ref{assumption:uniformellipticity} - \ref{assumption:2unifboundedlinearfunctionalderivatives} and \ref{assumption:limitingcoefficientsregularityratefunction} hold via Example \ref{example:noxmudependenceforphiandpi} in the appendix. Via Remark \ref{remark:barDnondegenerate}, we also have $\bar{D}>0,\forall x\in\R,\mu\in\mc{P}_2(\R)$. Then this example is an immediate corollary of Proposition \ref{prop:DGformofratefunction}.

We know in this setting that $\pi$ admits a density of the form $\pi(y) = C\exp\left(\frac{-V_4(y)}{a}\right)$,
where $C$ is a normalizing constant (see Equation \eqref{eq:explicit1Dpi} in the appendix). Then since $V_2,V_4$ are assumed even and hence $V_2' \pi$ is odd, the result holds.
\end{proof}
\section{Overview of the approach and formulation of the Controlled System}\label{S:ControlSystem}
We use the weak convergence approach of \cite{DE} in order to prove the large deviations principle for $Z^N$. As discussed in Section \ref{sec:mainresults}, we prove the large deviations principle via proving $Z^N$ satisfies the Laplace principle with speed $a^{-2}(N)$ and good rate function $I$ given by Equation \eqref{eq:proposedjointratefunction} (see, e.g. \cite{DE} Section 1.2). 

The method for this is to use the variational representation from \cite{BD} to get that for each $N\in\bb{N}$ and $F\in C_b(C([0,T];S_{-\tau}))$, $\tau\geq w$, where $w$ is as in Equation \eqref{eq:wdefinition},
\begin{align}\label{eq:varrepfunctionalsBM}
-a^2(N)\log \E \exp\biggl(-\frac{1}{a^2(N)}F(Z^N) \biggr)
& = \inf_{\tilde{u}^N}\E\biggl[\frac{1}{2}\frac{1}{N}\sum_{i=1}^N \int_0^T\left(|\tilde{u}^{N,1}_i(s)|^2+|\tilde{u}^{N,2}_i(s)|^2\right)ds +F(\tilde{Z}^N)\biggr]
\end{align}
where $\br{\tilde{u}^{N,k}_i}_{i\in\bb{N},k=1,2}$ are $\br{\F_t}$-progressively-measurable processes such that  \begin{align}\label{eq:controlassumptions0}
\sup_{N\in\bb{N}} \frac{1}{N}\E\biggl[\sum_{i=1}^N \int_0^T \left(|\tilde{u}^{N,1}_i(s)|^2 + |\tilde{u}^{N,2}_i(s)|^2\right)ds\biggr]<\infty.
\end{align}

One can see that in fact the results of \cite{BD} indeed imply the equality \eqref{eq:varrepfunctionalsBM} by following an argument along the same lines as Proposition 3.3. in \cite{BS}.

This bound on the controls can be improved when proving the Laplace principle Lower Bound \eqref{eq:LPupperbound} to:
\begin{align}\label{eq:controlassumptions}
\sup_{N\in\bb{N}} \frac{1}{N}\sum_{i=1}^N \int_0^T \left(|\tilde{u}^{N,1}_i(s)|^2 + |\tilde{u}^{N,2}_i(s)|^2\right)ds <\infty,\bb{\Prob}-\text{ almost surely.}
\end{align}
by the argument found in Theorem 4.4 of \cite{BD}.
Here $\tilde{Z}^N$ is given by, for $\phi\in C^\infty_c(\R):$
\begin{align}\label{eq:controlledempmeasure}
\langle\tilde{Z}^N_t,\phi\rangle &= a(N)\sqrt{N}(\langle\tilde{\mu}^{\epsilon,N}_t,\phi\rangle-\langle\mc{L}(X_t),\phi\rangle),\quad\text{with}\quad
\tilde{\mu}^{\epsilon,N}_t = \frac{1}{N}\sum_{i=1}^N \delta_{\tilde{X}^{i,\epsilon,N}_t},\quad t\in [0,T].
\end{align}
$\tilde{X}^{i,\epsilon,N}_t$ are solutions to:
\begin{align}\label{eq:controlledslowfast1-Dold}
&d\tilde{X}^{i,\epsilon,N}_t = \biggl[\frac{1}{\epsilon}b(\tilde{X}^{i,\epsilon,N}_t,\tilde{Y}^{i,\epsilon,N}_t,\tilde{\mu}^{\epsilon,N}_t)+ c(\tilde{X}^{i,\epsilon,N}_t,\tilde{Y}^{i,\epsilon,N}_t,\tilde{\mu}^{\epsilon,N}_t)+\sigma(\tilde{X}^{i,\epsilon,N}_t,\tilde{Y}^{i,\epsilon,N}_t,\tilde{\mu}^{\epsilon,N}_t)\frac{\tilde{u}^{N,1}_i(t)}{a(N)\sqrt{N}} \biggr]dt \\
&+ \sigma(\tilde{X}^{i,\epsilon,N}_t,\tilde{Y}^{i,\epsilon,N}_t,\tilde{\mu}^{\epsilon,N}_t)dW^i_t\nonumber\\
&d\tilde{Y}^{i,\epsilon,N}_t  = \frac{1}{\epsilon}\biggl[\frac{1}{\epsilon}f(\tilde{X}^{i,\epsilon,N}_t,\tilde{Y}^{i,\epsilon,N}_t,\tilde{\mu}^{\epsilon,N}_t)+ g(\tilde{X}^{i,\epsilon,N}_t,\tilde{Y}^{i,\epsilon,N}_t,\tilde{\mu}^{\epsilon,N}_t) +\tau_1(\tilde{X}^{i,\epsilon,N}_t,\tilde{Y}^{i,\epsilon,N}_t,\tilde{\mu}^{\epsilon,N}_t)\frac{\tilde{u}^{N,1}_i(t)}{a(N)\sqrt{N}}  \nonumber\\
&+\tau_2(\tilde{X}^{i,\epsilon,N}_t,\tilde{Y}^{i,\epsilon,N}_t,\tilde{\mu}^{\epsilon,N}_t)\frac{\tilde{u}^{N,2}_i(t)}{a(N)\sqrt{N}} \biggr]dt+ \frac{1}{\epsilon}\biggl[\tau_1(\tilde{X}^{i,\epsilon,N}_t,\tilde{Y}^{i,\epsilon,N}_t,\tilde{\mu}^{\epsilon,N}_t)dW^i_t+\tau_2(\tilde{X}^{i,\epsilon,N}_t,\tilde{Y}^{i,\epsilon,N}_t,\tilde{\mu}^{\epsilon,N}_t)dB^i_t\biggr]\nonumber\\
&(\tilde{X}^{i,\epsilon,N}_0,\tilde{Y}^{i,\epsilon,N}_0) =(\eta^{x},\eta^{y})\nonumber.
\end{align}

We couple the controls to the joint empirical measures of the fast and slow process by defining occupation measures $\br{Q^N}_{N\in\bb{N}}\subset M_T(\R^4)$ in the following way: for $A,B\in \mc{B}(\R)$ and $C\in \mc{B}(\R^2)$:
\begin{align}\label{eq:occupationmeasures}
Q^N(A\times B\times C\times [0,t])& = \frac{1}{N}\sum_{i=1}^N \int_0^t \delta_{\tilde{X}^{i,\epsilon,N}_s}(A)\delta_{\tilde{Y}^{i,\epsilon,N}_s}(B)\delta_{(\tilde{u}^{N,1}_i(s),\tilde{u}^{N,2}_i(s))}(C)ds.
\end{align}

The proof of the Inequalities \eqref{eq:LPupperbound} and \eqref{eq:LPlowerbound} are attained by identifying limit in distribution of $(\tilde{Z}^N,Q^N)$ as satisfying the limiting controlled Equation \eqref{eq:MDPlimitFIXED}. This identification of the limit is the subject of Section \ref{sec:identificationofthelimit}. In order to identify this limit, we first need to establish tightness of the sequence of random variables $\br{(\tilde{Z}^N,Q^N)}_{N\in\bb{N}}$, as done in Section \ref{sec:tightness}. The proof of tightness relies on a combination of Ergodic-Type Theorems for the system of controlled interacting particles \eqref{eq:controlledslowfast1-Dold} as proved in Section \ref{sec:ergodictheoremscontrolledsystem} and on establishing rates of averaging for fully coupled McKean-Vlasov Equations, as done in Subsection \ref{sec:averagingfullycoupledmckeanvlasov}. These rates of averaging are needed do to a novel coupling argument made in the proof of tightness (see Lemma \ref{lemma:Zboundbyphi4}) to the following system of IID slow-fast McKean-Vlasov Equations:
\begin{align}\label{eq:IIDparticles}
d\bar{X}^{i,\epsilon}_t &= \biggl[\frac{1}{\epsilon}b(\bar{X}^{i,\epsilon}_t,\bar{Y}^{i,\epsilon}_t,\mc{L}(\bar{X}^\epsilon_t))+ c(\bar{X}^{i,\epsilon}_t,\bar{Y}^{i,\epsilon}_t,\mc{L}(\bar{X}^\epsilon_t)) \biggr]dt + \sigma(\bar{X}^{i,\epsilon}_t,\bar{Y}^{i,\epsilon}_t,\mc{L}(\bar{X}^\epsilon_t))dW^i_t\\
d\bar{Y}^{i,\epsilon}_t & = \frac{1}{\epsilon}\biggl[\frac{1}{\epsilon}f(\bar{X}^{i,\epsilon}_t,\bar{Y}^{i,\epsilon}_t,\mc{L}(\bar{X}^\epsilon_t))+ g(\bar{X}^{i,\epsilon}_t,\bar{Y}^{i,\epsilon}_t,\mc{L}(\bar{X}^\epsilon_t)) \biggr]dt \nonumber\\
&+ \frac{1}{\epsilon}\biggl[\tau_1(\bar{X}^{i,\epsilon}_t,\bar{Y}^{i,\epsilon}_t,\mc{L}(\bar{X}^\epsilon_t))dW^i_t+\tau_2(\bar{X}^{i,\epsilon}_t,\bar{Y}^{i,\epsilon}_t,\mc{L}(\bar{X}^\epsilon_t))dB^i_t\biggr]\nonumber\\
(\bar{X}^{i,\epsilon}_0,\bar{X}^{i,\epsilon}_0)& = (\eta^{x},\eta^{y}),\nonumber
\end{align}
where $\bar{X}^\epsilon$ is any particle that has common law with the $\bar{X}^{i,\epsilon}$'s and $W^i,B^i$ are the same driving Brownian motions as in Equations \eqref{eq:slowfast1-Dold} and \eqref{eq:controlledslowfast1-Dold}.

We will also make use of the empirical measure on $N$ of the IID slow particles from Equation \eqref{eq:IIDparticles}:
\begin{align}\label{eq:IIDempiricalmeasure}
\bar{\mu}^{\epsilon,N}\coloneqq \frac{1}{N}\sum_{i=1}^N \delta_{\bar{X}^{i,\epsilon}_t}.
\end{align}
\begin{remark}\label{remark:ontheiidsystem}
Note that these IID particles are what we get from replacing $\mu^{\epsilon,N}$ by $\mc{L}(\bar{X}^\epsilon)$ in Equation \eqref{eq:slowfast1-Dold}. Using such an auxiliary process is a traditional proof method for tightness of fluctuation processes related to empirical measures; See \cite{HM} Theorem 1/Lemma 1, \cite{LossFromDefault} Section 8, \cite{DLR} Section 5.1, \cite{KX} Theorem 2.4/3.1, \cite{FM} Lemma 3.2/Proposition 3.5/Section 4 for examples of this general approach.. However, a key difference here form those proofs is that the IID particles are not copies of the limiting process \eqref{eq:LLNlimitold}, but instead are copies of the process we would obtain from keeping $\epsilon>0$ fixed and sending $N\toinf$. As seen in \cite{BS}, the limit in distribution as $N\toinf$,$\epsilon\downarrow 0$ of the empirical measure $\mu^{\epsilon,N}$ does not depend on the relative rates at which $\epsilon$ and $N$ go to their respective limits. Hence, we are able to treat each of the problems separately, and obtain a rate of convergence of $\tilde{\mu}^{\epsilon,N}$ from Equation \eqref{eq:controlledempmeasure} to $\bar{\mu}^{\epsilon,N}$ from Equation \eqref{eq:IIDempiricalmeasure} as $N\toinf$ in $L^2$ (see Lemma \ref{lemma:XbartildeXdifference}), and a rate of convergence of $\mc{L}(\bar{X}^{1,\epsilon}_t)$ from Equation \eqref{eq:IIDparticles} to $\mc{L}(X_t)$ uniformly as an element of $S_{-m}$, where $X_t$ is as in Equation \eqref{eq:LLNlimitold} and $m$ is as in Equation \eqref{eq:mdefinition}. The latter is a problem of independent interest in itself, and extends the current known results on averaging for SDEs and McKean-Vlasov SDEs, which can be found in, e.g. \cite{RocknerFullyCoupled} and \cite{RocknerMcKeanVlasov} respectively. The result is contained in Subsection \ref{sec:averagingfullycoupledmckeanvlasov} as Theorem \ref{theo:mckeanvlasovaveraging}, and its proof is the subject of the complimentary paper \cite{BezemekSpiliopoulosAveraging2022}.
\end{remark}

\section{Ergodic-Type Theorems for the Controlled System \eqref{eq:controlledslowfast1-Dold}}\label{sec:ergodictheoremscontrolledsystem}
In this section, we use the method of auxiliary Poisson equations to derive rates of averaging in the form of Ergodic-Type Theorems for the controlled particles \eqref{eq:controlledslowfast1-Dold}. These results are used in the proof of tightness of the controlled fluctuation process, as they allow us to couple the controlled particles \eqref{eq:controlledslowfast1-Dold} to the IID slow-fast McKean-Vlasov Equations \eqref{eq:IIDparticles} - see Lemma \ref{lemma:XbartildeXdifference}. They also allow us to identify a prelimit representation for the controlled fluctuations processes $\tilde{Z^N}$ from Equation \eqref{eq:controlledempmeasure} (see Lemma \ref{lemma:Lnu1nu2representation}), which informs the controlled limit proved in Section \ref{sec:identificationofthelimit}. In particular, Proposition \ref{prop:fluctuationestimateparticles1} is necessary to handle the terms $\frac{1}{\epsilon}b$ appearing in the drift of the slow particles $X^{i,N,\epsilon}$, $\tilde{X}^{i,N,\epsilon}$ in Equations \eqref{eq:slowfast1-Dold},\eqref{eq:controlledslowfast1-Dold}. This is where the terms involving the solution $\Phi$ to the Poisson Equation \eqref{eq:cellproblemold} in the limiting coefficients \eqref{eq:limitingcoefficients} come from. The same analysis is performed in averaging fully-coupled standard diffusions - see e.g. \cite{PV2} Theorem 4 and \cite{RocknerFullyCoupled} Lemma 4.4 - but here we must also account for the dependence of the coefficients on the empirical measure, and hence derivatives of $\Phi$ in its measure component appear in the remainder terms. One term involving the derivative in the measure component of $\Phi$ a priori seems to be $\mc{O}(1)$ in the limit, but is seen to vanish as $N\toinf$ in Proposition \ref{prop:purpleterm1}. Naturally such a term does not appear in the setting without measure dependence of the coefficients, and is unique to slow-fast interacting particle systems and slow-fast McKean-Vlasov SDEs. Thus the ``doubled Poisson equation'' construction (see Equation \eqref{eq:doublecorrectorproblem}) and the proof of Proposition \ref{prop:purpleterm1} are novel to this paper and the related paper \cite{BezemekSpiliopoulosAveraging2022}. Proposition \ref{prop:llntypefluctuationestimate1} is used to see that drift and diffusion coefficients which depend on the fast particles $\tilde{Y}^{i,\epsilon,N}$ from Equation \eqref{eq:controlledslowfast1-Dold} can be exchanged for those where dependence on $\tilde{Y}^{i,\epsilon,N}$ is replaced with integration against the invariant measure $\pi$ from Equation \eqref{eq:invariantmeasureold} at a cost of $\mc{O}(\epsilon)$. This method is employed when establishing rates of stochastic homogenization in the standard (one-particle) setting in e.g. \cite{Spiliopoulos2014Fluctuations} Lemma 4.1,\cite{MS} Lemma B.5, and \cite{RocknerFullyCoupled} Lemma 4.2. There again, our setting is different than the standard case in that we must compensate for the dependence of the empirical measure of the coefficients, which yields terms involving the derivative in the measure component of the auxiliary Poisson Equation \eqref{eq:driftcorrectorproblem}.

\begin{proposition}\label{prop:fluctuationestimateparticles1}
Consider $\psi\in C^{1,2}_b([0,T]\times \R)$. Under assumptions \ref{assumption:uniformellipticity} - \ref{assumption:multipliedpolynomialgrowth}, we have for any $t\in [0,T]$:
\begin{align*}
&\frac{a(N)}{\sqrt{N}}\sum_{i=1}^N\E\biggl[\sup_{t\in[0,T]}\biggl|\int_0^t \frac{1}{\epsilon}b(i)\psi(s,\tilde{X}^{i,\epsilon,N}_s)ds - \int_0^t  \biggl(\gamma_1(i)\psi(s,\tilde{X}^{i,\epsilon,N}_s)+D_1(i)\psi_x(s,\tilde{X}^{i,\epsilon,N}_s)\\
&+[\frac{\tau_1(i)}{a(N)\sqrt{N}}\tilde{u}^{N,1}_i(s)+\frac{\tau_2(i)}{a(N)\sqrt{N}}\tilde{u}^{N,2}_i(s)]\Phi_y(i)\psi(s,\tilde{X}^{i,\epsilon,N}_s)\biggr)ds-\int_0^t\tau_1(i)\Phi_y(i)\psi(s,\tilde{X}^{i,\epsilon,N}_s)dW_s^i\\
&-\int_0^t \tau_2(i)\Phi_y(i)\psi(s,\tilde{X}^{i,\epsilon,N}_s)dB_s^i-\int_0^t \frac{1}{N}\sum_{j=1}^N b(j)\partial_{\mu}\Phi(i)[j]\psi(s,\tilde{X}^{i,\epsilon,N}_s)ds\biggr|^2\biggr]\\
&\leq C[\epsilon^2 a(N)\sqrt{N}(1+T+T^2)+\frac{a(N)}{\sqrt{N}}T^2]\norm{\psi}^2_{C_b^{1,2}}
\end{align*}
where here $(i)$ denotes the argument $(\tilde{X}^{i,\epsilon,N}_s,\tilde{Y}^{i,\epsilon,N}_s,\tilde{\mu}^{\epsilon,N}_s)$ and similarly for $(j)$, $[j]$ denotes the argument $\tilde{X}^{j,\epsilon,N}_s$, and $\Phi$ is as in \eqref{eq:cellproblemold}. Here we recall the definitions of $\gamma_1,D_1$ from Equation \eqref{eq:limitingcoefficients}. 
\end{proposition}
\begin{proof}

Using Lemma \ref{lemma:Ganguly1DCellProblemResult} to gain appropriate differentiablity of $\Phi$, letting $\Phi^N:\R\times\R\times\R^N\tto \R$ be the empirical projection of $\Phi$ and applying standard It\^o's formula and Proposition \ref{prop:empprojderivatives}  to the composition $\Phi^N(\tilde{X}^{i,\epsilon,N}_s,\tilde{Y}^{i,\epsilon,N}_s,(\tilde{X}^{1,\epsilon,N}_s,...,\tilde{X}^{N,\epsilon,N}_s))$, we get:
\begin{align*}
&\int_0^t \frac{1}{\epsilon}b(i)\psi(s,\tilde{X}^{i,\epsilon,N}_s)ds - \int_0^t  \biggl(\gamma_1(i)\psi(s,\tilde{X}^{i,\epsilon,N}_s)+D_1(i)\psi_x(s,\tilde{X}^{i,\epsilon,N}_s)\\
&+[\frac{\tau_1(i)}{a(N)\sqrt{N}}\tilde{u}^{N,1}_i(s)+\frac{\tau_2(i)}{a(N)\sqrt{N}}\tilde{u}^{N,2}_i(s)]\Phi_y(i)\psi(s,\tilde{X}^{i,\epsilon,N}_s)\biggr)ds-\int_0^t\tau_1(i)\Phi_y(i)\psi(s,\tilde{X}^{i,\epsilon,N}_s)dW_s^i\\
&-\int_0^t \tau_2(i)\Phi_y(i)\psi(s,\tilde{X}^{i,\epsilon,N}_s)dB_s^i-\int_0^t \frac{1}{N}\sum_{j=1}^N b(j)\partial_{\mu}\Phi(i)[j]\psi(s,\tilde{X}^{i,\epsilon,N}_s)ds  = \sum_{k=1}^{8} \tilde{B}^{i,\epsilon,N}_k
\end{align*}
where:
\begin{align*}
\tilde{B}^{i,\epsilon,N}_1(t)& = \epsilon [\Phi(\tilde{X}^{i,\epsilon,N}_0,\tilde{Y}^{i,\epsilon,N}_0,\tilde{\mu}^{\epsilon,N}_0)\psi(0,\tilde{X}^{i,\epsilon,N}_0) - \Phi(\tilde{X}^{i,\epsilon,N}_t,\tilde{Y}^{i,\epsilon,N}_t,\tilde{\mu}^{\epsilon,N}_t)\psi(t,\tilde{X}^{i,\epsilon,N}_t)]\\
\tilde{B}^{i,\epsilon,N}_2(t)& = \frac{1}{N}\int_0^t \sigma(i)\tau_1(i)\partial_\mu \Phi_y (i)[i]\psi(s,\tilde{X}^{i,\epsilon,N}_s)ds\\
\tilde{B}^{i,\epsilon,N}_3(t)& =\epsilon \int_0^t \biggl(\Phi(i)\dot{\psi}(s,\tilde{X}^{i,\epsilon,N}_s)+c(i)[\Phi_x(i)\psi(s,\tilde{X}^{i,\epsilon,N}_s)+\Phi(i)\psi_x(s,\tilde{X}^{i,\epsilon,N}_s)]+\frac{\sigma^2(i)}{2}[\Phi_{xx}(i)\psi(s,\tilde{X}^{i,\epsilon,N}_s)\\
&\qquad+2\Phi_x(i)\psi_x(s,\tilde{X}^{i,\epsilon,N}_s)+\Phi(i)\psi_{xx}(s,\tilde{X}^{i,\epsilon,N}_s)]\biggr)ds\\
\tilde{B}^{i,\epsilon,N}_4(t)& =\epsilon\int_0^t\frac{\sigma^2(i)}{2}[\frac{2}{N}\partial_\mu \Phi(i)[i]\psi_x(s,\tilde{X}^{i,\epsilon,N}_s)+\frac{2}{N}\partial_\mu \Phi_x(i)[i]\psi(s,\tilde{X}^{i,\epsilon,N}_s)]ds\\
\tilde{B}^{i,\epsilon,N}_5(t)& =\epsilon\int_0^t\frac{1}{N}\sum_{j=1}^N \biggl\lbrace c(j)\partial_\mu \Phi(i)[j]\psi(s,\tilde{X}^{i,\epsilon,N}_s)+\frac{1}{2}\sigma^2(j)[\frac{1}{N}\partial^2_\mu \Phi(i)[j,j] +\partial_z\partial_\mu \Phi(i)[j]]\psi(s,\tilde{X}^{i,\epsilon,N}_s)  \biggr\rbrace ds \\
\tilde{B}^{i,\epsilon,N}_6(t)& = \epsilon \biggl[ \int_0^t \sigma(i)[\Phi_x(i)\psi(s,\tilde{X}^{i,\epsilon,N}_s)+\Phi(i)\psi_x(s,\tilde{X}^{i,\epsilon,N}_s)]dW^i_t + \frac{1}{N}\sum_{j=1}^N\biggl\lbrace \int_0^t \sigma(j)\partial_\mu \Phi(i)[j]\psi(s,\tilde{X}^{i,\epsilon,N}_s)dW^j_s \biggr\rbrace \biggr] \\
\tilde{B}^{i,\epsilon,N}_7(t)& = \epsilon \int_0^t \frac{\sigma(i)}{a(N)\sqrt{N}}\tilde{u}^{N,1}_i(s)[\Phi_x(i)\psi(s,\tilde{X}^{i,\epsilon,N}_s)+\Phi(i)\psi_x(s,\tilde{X}^{i,\epsilon,N}_s)]ds \\
\tilde{B}^{i,\epsilon,N}_{8}(t)& = \epsilon \int_0^t \frac{1}{N} \biggl\lbrace \sum_{j=1}^N \frac{\sigma(j)}{a(N)\sqrt{N}}\tilde{u}^{N,1}_j(s) \partial_\mu \Phi(i)[j]\psi(s,\tilde{X}^{i,\epsilon,N}_s)\biggr\rbrace ds.
\end{align*}

Via Lemma \ref{lemma:tildeYuniformbound}, the assumed linear growth of $b$ and $c$ in $y$ and boundedness of $\sigma$, and the assumed bound (\ref{eq:controlassumptions}) on the controls, one can check that indeed $\tilde{\mu}^N_t \in \mc{P}_2(\R)$ for each $t\in [0,T]$ and $N\in\bb{N}$, and so there is no issue with the domain of $\Phi$ and its derivatives being $\mc{P}_2(\R)$.

Then, by multiple applications of H\"older's inequality, and using the assumed uniform in $x,\mu$ polynomial growth in $y$ of $\Phi$ and its derivatives from Assumption \ref{assumption:multipliedpolynomialgrowth}:
\begin{align*}
&\frac{a(N)}{\sqrt{N}}\sum_{i=1}^N \E\biggl[\sup_{t\in[0,T]}|\tilde{B}^{i,\epsilon,N}_1(t)|^2\biggr]\leq  \epsilon^2 a(N)\sqrt{N}\norm{\psi}^2_\infty\\
&\frac{a(N)}{\sqrt{N}}\sum_{i=1}^N \E\biggl[\sup_{t\in[0,T]}|\tilde{B}^{i,\epsilon,N}_2(t)|^2\biggr]\leq C\frac{a(N)}{\sqrt{N}} \frac{1}{N^2}\sum_{i=1}^N T\E\biggl[\int_0^T|\partial_\mu\Phi_y(i)[i]|^2ds\biggr]\norm{\psi}^2_\infty \\
&\qquad \leq C\frac{a(N)}{\sqrt{N}} \frac{1}{N}\sum_{i=1}^N T\E\biggl[\int_0^T\norm{\partial_\mu\Phi_y(i)[\cdot]}_{L^2(\R,\tilde{\mu}^{N,\epsilon}_s)}^2ds\biggr]\norm{\psi}^2_\infty \\
&\qquad\leq C\frac{a(N)}{\sqrt{N}}T^2\biggl(1+ \frac{1}{N}\sum_{i=1}^N \sup_{t\in [0,T]}\E\biggl[|\tilde{Y}^{i,\epsilon,N}_t|^{2\tilde{q}_{\Phi_y}(1,0,0)} \biggr]\biggr)\norm{\psi}^2_\infty\\
&\frac{a(N)}{\sqrt{N}}\sum_{i=1}^N \E\biggl[\sup_{t\in[0,T]}|\tilde{B}^{i,\epsilon,N}_3(t)|^2\biggr]\leq \epsilon^2 a(N)\sqrt{N}T^2\biggl(1+ \frac{1}{N}\sum_{i=1}^N \sup_{t\in [0,T]}\E\biggl[|\tilde{Y}^{i,\epsilon,N}_t|^{2}+|\tilde{Y}^{i,\epsilon,N}_t|^{2q_{\Phi}(0,2,0)} \biggr]\biggr)\norm{\psi}^2_{C^{1,2}_b}\\
&\frac{a(N)}{\sqrt{N}}\sum_{i=1}^N \E\biggl[\sup_{t\in[0,T]}|\tilde{B}^{i,\epsilon,N}_4(t)|^2\biggr]\leq C\frac{a(N)}{\sqrt{N}} \frac{\epsilon^2}{N^2}\sum_{i=1}^N T\E\biggl[\int_0^T\biggl|\biggl(|\partial_\mu\Phi(i)[i]|+|\partial_\mu\Phi_x(i)[i]|\biggr)\biggr|^2ds\biggr](\norm{\psi}^2_\infty+\norm{\psi_x}^2_\infty)\\
&\qquad \leq C\frac{a(N)}{\sqrt{N}} \frac{\epsilon^2}{N}\sum_{i=1}^N T\E\biggl[\int_0^T\biggl|\biggl(\norm{\partial_\mu\Phi(i)[\cdot]}_{L^2(\R,\tilde{\mu}^{N,\epsilon}_s)}+\norm{\partial_\mu\Phi_x(i)[\cdot]}_{L^2(\R,\tilde{\mu}^{N,\epsilon}_s)}\biggr)\biggr|^2ds\biggr]
(\norm{\psi}^2_\infty+\norm{\psi_x}^2_\infty)\\
&\qquad\leq C\frac{a(N)}{\sqrt{N}}\epsilon^2T^2 \biggl(1+ \frac{1}{N}\sum_{i=1}^N \sup_{t\in [0,T]}\E\biggl[|\tilde{Y}^{i,\epsilon,N}_t|^{2(\tilde{q}_{\Phi}(1,0,0)\vee \tilde{q}_{\Phi}(1,1,0))} \biggr]\biggr)(\norm{\psi}^2_\infty+\norm{\psi_x}^2_\infty)
\end{align*}

Here for $\tilde{B}_1$, we used the assumed boundedness of $\Phi$ from \ref{assumption:multipliedpolynomialgrowth}. For $\tilde{B}_2$ we used the assumed polynomial growth in $y$ of $\partial_\mu\Phi$ from \ref{assumption:multipliedpolynomialgrowth} and the boundedness of $\sigma$ and $\tau_1$ from \ref{assumption:gsigmabounded} and \ref{assumption:uniformellipticity}. For $\tilde{B}_3$ we used the assumed polynomial growth in $y$ of $\Phi_{xx}$ and boundedness of $\Phi,\Phi_x$ from \ref{assumption:multipliedpolynomialgrowth} and the boundedness of $\sigma$ and the linear growth in $y$ of $c$ from \ref{assumption:gsigmabounded}. In $\tilde{B}_4$ we used the assumed polynomial growth in $y$ of $\partial_\mu\Phi$ and $\partial_\mu\Phi_x$ from \ref{assumption:multipliedpolynomialgrowth} and the assumed boundedness of $\sigma$ from \ref{assumption:gsigmabounded}.

For $\tilde{B}^{i,\epsilon,N}_5(t)$, we bound the two terms separately. For the first, we use the assumed linear growth in $y$ of $c$ and polynomial growth of $\partial_\mu\Phi$ in $y$ to get:
\begin{align*}
&\frac{a(N)}{\sqrt{N}}\sum_{i=1}^N \frac{\epsilon^2}{N^2}\E\biggl[\sup_{t\in[0,T]}\biggl|\int_0^t\sum_{j=1}^N c(j)\partial_\mu\Phi(i)[j]\psi(s,\tilde{X}^{i,\epsilon,N}_s)ds\biggr|^2\biggr]\\
& \leq \epsilon^2 a(N) \sqrt{N} \frac{T}{N}\sum_{i=1}^N \E\biggl[\int_0^T\norm{\partial_\mu\Phi(i)[\cdot]}_{L^2(\R,\tilde{\mu}^{\epsilon,N}_s)}^2\frac{1}{N}\sum_{j=1}^N |c(j)|^2ds\biggr]\norm{\psi}^2_\infty\\
&\leq C\epsilon^2 a(N) \sqrt{N}T^2\biggl(1+ \frac{1}{N}\sum_{i=1}^N \sup_{s\in [0,T]}\E\biggl[\frac{1}{N}\sum_{j=1}^N |\tilde{Y}^{i,\epsilon,N}_s|^{2}\biggr]+\frac{1}{N}\sum_{i=1}^N \sup_{s\in [0,T]}\E\biggl[|\tilde{Y}^{i,\epsilon,N}_s|^{2\tilde{q}_{\Phi}(1,0,0)}\biggr]\\
&+ \sup_{s\in [0,T]}\E\biggl[\frac{1}{N^2}\sum_{j=1}^N\sum_{i=1}^N |\tilde{Y}^{i,\epsilon,N}_s|^{2\tilde{q}_{\Phi}(1,0,0)} |\tilde{Y}^{j,\epsilon,N}_s|^{2} \biggr]\biggr)\norm{\psi}^2_\infty\\
&\leq C\epsilon^2 a(N) \sqrt{N}T^2\biggl(1+ \frac{1}{N}\sum_{i=1}^N \sup_{s\in [0,T]}\E\biggl[\frac{1}{N}\sum_{j=1}^N |\tilde{Y}^{i,\epsilon,N}_s|^{2}\biggr]+\frac{1}{N}\sum_{i=1}^N \sup_{s\in [0,T]}\E\biggl[|\tilde{Y}^{i,\epsilon,N}_s|^{2\tilde{q}_{\Phi}(1,0,0)}\biggr]\\
&+ \sup_{s\in [0,T]}\E\biggl[\biggl(\frac{1}{N}\sum_{i=1}^N |\tilde{Y}^{i,\epsilon,N}_s|^{2(\tilde{q}_{\Phi}(1,0,0)\vee 1)}\biggr)^2 \biggr]\biggr)\norm{\psi}^2_\infty.
\end{align*}
For the second, we have by boundedness of $\sigma$ and the assumed polynomial growth in $y$ of $\partial^2_\mu\Phi$ and $\partial_z\partial_\mu\Phi:$
\begin{align*}
&\frac{a(N)}{\sqrt{N}}\frac{\epsilon^2}{N^2}\sum_{i=1}^N\E\biggl[\sup_{t\in [0,T]} \biggl|\int_0^t\sum_{j=1}^N\frac{1}{2}\sigma^2(j)[\frac{1}{N}\partial^2_\mu \Phi(i)[j,j] +\partial_z\partial_\mu \Phi(i)[j]]\psi(s,\tilde{X}^{i,\epsilon,N}_s) ds\biggr|^2\biggr]\\
&\leq  \frac{a(N)}{\sqrt{N}}\epsilon^2CT\sum_{i=1}^N\E\biggl[\int_0^T\frac{1}{N}\sum_{j=1}^N\frac{1}{N^2}|\partial^2_\mu \Phi(i)[j,j]|^2 +|\partial_z\partial_\mu \Phi(i)[j]|^2 ds\biggr]\norm{\psi}^2_\infty\\
&\leq  \frac{a(N)}{\sqrt{N}}\epsilon^2CT\sum_{i=1}^N\E\biggl[\int_0^T\frac{1}{N}\norm{\partial^2_\mu \Phi(i)[\cdot,\cdot]}_{L^2(\R,\tilde{\mu}^{\epsilon,N}_s)\otimes L^2(\R,\tilde{\mu}^{\epsilon,N}_s)}^2 +\norm{\partial_z\partial_\mu \Phi(i)[\cdot]}_{L^2(\R,\tilde{\mu}^{\epsilon,N}_s)}^2 ds\biggr]\norm{\psi}^2_\infty\\
&\leq  C\epsilon^2a(N)\sqrt{N}T^2\biggl[1+\frac{1}{N}\sum_{i=1}^N\sup_{s\in[0,T]}\E\biggl[|\tilde{Y}^{i,\epsilon,N}_s|^{2(\tilde{q}_{\Phi}(2,0,0)\vee \tilde{q}_{\Phi}(1,0,1) )}\biggr]\biggr]\norm{\psi}^2_\infty.
\end{align*}

For the martingale terms, by Burkholder-Davis-Gundy inequality, the assumed boundedness of $\sigma$, $\Phi$, and $\Phi_x$ and assumed polynomial growth in $y$ of $\partial_\mu\Phi$:
\begin{align*}
&\frac{a(N)}{\sqrt{N}}\sum_{i=1}^N \E\biggl[\sup_{t\in[0,T]}|\tilde{B}^{i,\epsilon,N}_6(t)|^2\biggr]\leq C\epsilon^2 a(N)\sqrt{N}T(\norm{\psi}^2_\infty+\norm{\psi_x}^2_\infty)+ C\frac{a(N)}{\sqrt{N}}\sum_{i=1}^N \frac{\epsilon^2}{N^2} \sum_{j=1}^N \E\biggl[\int_0^T |\partial_\mu \Phi(i)[j]|^2ds\biggr]\norm{\psi}^2_\infty \\
& = C\epsilon^2 a(N)\sqrt{N}T(\norm{\psi}^2_\infty+\norm{\psi_x}^2_\infty) +C \frac{a(N)}{\sqrt{N}}\sum_{i=1}^N \frac{\epsilon^2}{N} \E\biggl[\int_0^T \norm{\partial_\mu \Phi(i)[\cdot]}^2_{L^2(\R,\tilde{\mu}^{\epsilon,N}_s)}ds\biggr]\norm{\psi}^2_\infty \\
&\leq C\epsilon^2a(N)\sqrt{N}T(\norm{\psi}^2_\infty+\norm{\psi_x}^2_\infty)+C\frac{\epsilon^2a(N)}{\sqrt{N}}T\biggl(1+ \frac{1}{N}\sum_{i=1}^N \sup_{t\in [0,T]}\E\biggl[|\tilde{Y}^{i,\epsilon,N}_t|^{2\tilde{q}_{\Phi}(1,0,0)} \biggr]\biggr)\norm{\psi}^2_\infty.
\end{align*}

By the bound \eqref{eq:controlassumptions0} and the assumed boundedness of $\Phi,\Phi_x$, we have also
\begin{align*}
&\frac{a(N)}{\sqrt{N}}\sum_{i=1}^N \E\biggl[\sup_{t\in[0,T]}|\tilde{B}^{i,\epsilon,N}_7(t)|^2\biggr]\leq \frac{a(N)}{\sqrt{N}}\sum_{i=1}^N \frac{\epsilon^2}{a^2(N)N}CT\E\biggl[\int_0^T |\tilde{u}^{N,1}_i(s)|^2ds\biggr](\norm{\psi}^2_\infty+\norm{\psi_x}^2_\infty)\\
&\leq \frac{\epsilon^2}{a(N)\sqrt{N}}CT(\norm{\psi}^2_\infty+\norm{\psi_x}^2_\infty).
\end{align*}

Finally, by the assumed boundedness of $\sigma$ and polynomial growth of $\partial_\mu \Phi$ in $y$:
\begin{align*}
&\frac{a(N)}{\sqrt{N}}\sum_{i=1}^N \E\biggl[\sup_{t\in[0,T]}|\tilde{B}^{i,\epsilon,N}_8(t)|^2\biggr]
\leq \frac{a(N)}{\sqrt{N}}\sum_{i=1}^N \frac{\epsilon^2}{a^2(N)N^3}C\E\biggl[\biggl|\sum_{j=1}^N\int_0^T  |\tilde{u}^{N,1}_j(s)| |\partial_\mu\Phi(i)[j]|ds\biggr|^2\biggr]\norm{\psi}^2_\infty\\
&\leq \frac{a(N)}{\sqrt{N}}\sum_{i=1}^N \frac{\epsilon^2}{a^2(N)N}C\E\biggl[\biggl(\frac{1}{N}\sum_{j=1}^N\int_0^T  |\tilde{u}^{N,1}_j(s)|^2ds\biggr) \biggl(\int_0^T \norm{\partial_\mu\Phi(i)[\cdot]}^2_{L^2(\R,\tilde{\mu}^{\epsilon,N}_s)}ds\biggr)\biggr]\norm{\psi}^2_\infty\\
&\leq C\frac{\epsilon^2}{a(N)\sqrt{N}}T\biggl(1+ \frac{1}{N}\sum_{i=1}^N \sup_{t\in [0,T]}\E\biggl[|\tilde{Y}^{i,\epsilon,N}_t|^{2\tilde{q}_{\Phi}(1,0,0)} \biggr]\biggr)\norm{\psi}^2_\infty
\end{align*}
where we use the bound \eqref{eq:controlassumptions} in the last step. The result follows from Lemmas \ref{lemma:tildeYuniformbound} and \ref{lemma:ytildesquaredsumbound}, using that the exponent of $|\tilde{Y}^{i,\epsilon,N}_t|$ in the expectation of all these bounds is less than or equal to 2 as imposed in Assumption \ref{assumption:multipliedpolynomialgrowth}. Lemma \ref{lemma:ytildesquaredsumbound} is used to handle the last term appearing in the bound of the first part of $\tilde{B}_5$.
\end{proof}
\begin{remark}\label{remark:onthescalingofa(N)}
Bounding the first term in $\tilde{B}_5$ in Proposition \ref{prop:fluctuationestimateparticles1} is the only place where Lemma \ref{lemma:ytildesquaredsumbound} is required in this manuscript. The proof of Lemma \ref{lemma:ytildesquaredsumbound} is where it is required that there exists $\rho\in (0,1)$ such that $a(N)\sqrt{N}\epsilon^\rho \tto \lambda \in (0,+\infty]$. Thus, if this term can be otherwise bounded (e.g. if $c$ or $\partial_\mu \Phi$ is uniformly bounded), one can relax this technical assumption on the scaling sequence $a(N)$ to $a(N)\sqrt{N}\epsilon\tto 0$. Moreover, $a(N)\sqrt{N}\epsilon\tto 0$ is needed so that the term $\tilde{B}_1$ in Proposition \ref{prop:fluctuationestimateparticles1} vanishes - without this, one cannot hope to prove tightness of $\br{\tilde{Z}^N}_{N\in\bb{N}}$, as in Proposition \ref{prop:tildeZNtightness} there would be an $\mc{O}(1)$ term which is not uniformly continuous with respect to time. If $b\equiv 0$ and hence there is no need for Proposition \ref{prop:fluctuationestimateparticles1}, it is possible to prove tightness even when $a(N)\sqrt{N}\epsilon\tto \lambda \in [0,\infty)$. Under this scaling, we expect to get a different formulation for the rate function in Theorem \ref{theo:MDP} when $\lambda>0$. This is an interesting avenue for future research which we do not pursue here for purposes of the presentation.
\end{remark}

\begin{proposition}\label{prop:purpleterm1}
In the setup of Proposition \ref{prop:fluctuationestimateparticles1}, assume in addition \ref{assumption:qF2bound}. Then
\begin{align*}
\frac{a(N)}{\sqrt{N}}\sum_{i=1}^N\E\biggl[\sup_{t\in[0,T]}\biggl|\int_0^t \frac{1}{N}\sum_{j=1}^N b(j)\partial_{\mu}\Phi(i)[j]\psi(s,\tilde{X}^{i,\epsilon,N}_s)ds\biggr|^2\biggr]&\leq C[\epsilon^2 a(N)\sqrt{N}(1+T+T^2)+\frac{a(N)}{N^{3/2}}T^2]\norm{\psi}^2_{C_b^{1,2}}.
\end{align*}
\end{proposition}
\begin{proof}

Recall the operator $L_{x,\mu}$ from Equation \eqref{eq:frozengeneratormold}. For fixed $x\in\R,\mu\in\mc{P}(\R)$, this is the generator of
\begin{align}\label{eq:frozenprocess1}
dY^{x,\mu}_t = f(x,Y^{x,\mu}_t,\mu)dt+\tau_1(x,Y^{x,\mu}_t,\mu)dW_t+\tau_2(x,Y^{x,\mu}_t,\mu)dB_t
\end{align}
for $W_t,B_t$ independent 1-D Brownian motions.

We introduce a new generator $L^2_{x,\bar{x},\mu}$ parameterized by $x,\bar{x}\in\R,\mu\in\mc{P}_2$ which acts on $\psi\in C^2_b(\R^2)$ by
\begin{align}\label{eq:2copiesgenerator}
L^2_{x,\bar{x},\mu}\psi(y,\bar{y}) &= f(x,y,\mu)\psi_y(y,\bar{y})+f(\bar{x},\bar{y},\mu)\psi_{\bar{y}}(y,\bar{y})\\
&+ \frac{1}{2}[\tau_1^2(x,y,\mu)+\tau_2^2(x,y,\mu)]\psi_{yy}(y,\bar{y})+\frac{1}{2}[\tau_1^2(\bar{x},\bar{y},\mu)+\tau_2^2(\bar{x},\bar{y},\mu)]\psi_{\bar{y}\bar{y}}(y,\bar{y}).\nonumber
\end{align}

This is the generator associated to the 2-dimensional process solving 2 independent copies of Equation \eqref{eq:frozenprocess1} where the same parameter $\mu$ enters both equations, but different $x,\bar{x}$ enter each equation, i.e.
\begin{align}\label{eq:frozenprocess2}
dY^{x,\mu}_t &= f(x,Y^{x,\mu}_t,\mu)dt+\tau_1(x,Y^{x,\mu}_t,\mu)dW_t+\tau_2(x,Y^{x,\mu}_t,\mu)dB_t\\
d\bar{Y}^{\bar{x},\mu}_t &= f(\bar{x},\bar{Y}^{\bar{x},\mu}_t,\mu)dt+\tau_1(\bar{x},\bar{Y}^{\bar{x},\mu}_t,\mu)d\bar{W}_t+\tau_2(\bar{x},\bar{Y}^{\bar{x},\mu}_t,\mu)d\bar{B}_t\nonumber.
\end{align}
for $W_t,B_t,\bar{W}_t,\bar{B}_t$ independent 1-D Brownian motions.

It is easy then to see that the unique distributional solution of the adjoint equation
\begin{align}
L^2_{x,\bar{x},\mu}\bar{\pi}(\cdot;x,\bar{x},\mu) &=0,\qquad
\int_{\R^2}\bar{\pi}(dy,d\bar{y};x,\bar{x},\mu)=1,\forall x,\bar{x}\in\R,\mu\in\mc{P}(\R) \nonumber
\end{align}
 is given by
\begin{align}\label{eq:doublefrozeninvariantmeasure}
\bar{\pi}(dy,d\bar{y};x,\bar{x},\mu) = \pi(dy;x,\mu)\otimes\pi(d\bar{y};\bar{x},\mu)
\end{align}
where $\pi$ is as in Equation \eqref{eq:invariantmeasureold}. We now consider $\chi(x,\bar{x},y,\bar{y},\mu):\R\times\R\times\R\times\R\times\mc{P}(\R)\tto \R$ solving
\begin{align}\label{eq:doublecorrectorproblem}
L^2_{x,\bar{x},\mu}\chi(x,\bar{x},y,\bar{y},\mu) &= -b(x,y,\mu)\partial_\mu \Phi(\bar{x},\bar{y},\mu)[x]\\{}
\int_{\R}\int_{\R}\chi(x,\bar{x},y,\bar{y},\mu)\pi(dy;x,\mu)\pi(d\bar{y},\bar{x},\mu)&=0.\nonumber
\end{align}

Note that by the centering condition, Equation \eqref{eq:centeringconditionold}, the right hand side of Equation \eqref{eq:doublecorrectorproblem} integrates against $\bar{\pi}$ from Equation \eqref{eq:doublefrozeninvariantmeasure} to $0$ for all $x,\bar{x},\mu$. Also, the second order coefficient in $L^2_{x,\bar{x},\mu}$ is uniformly elliptic by virtue of Assumption \ref{assumption:uniformellipticity}, and by virtue of Equation \eqref{eq:fdecayimplication}, there is $R_{f_2}>0$ and $\Gamma_2>0$ such that
\begin{align*}
\sup_{x,\bar{x},\mu}(f(x,y,\mu)y+f(\bar{x},\bar{y},\mu)\bar{y})\leq -\Gamma_2 (|y|^2+|\bar{y}|^2),\forall y,\bar{y} \text{ such that } \sqrt{y^2+\bar{y}^2}>R_{f_2}.
\end{align*}

Thus  indeed we have a unique solution to \eqref{eq:doublecorrectorproblem} by Theorem 1 in \cite{PV1} (which is a classical solution by assumption).
Applying It\^o's formula to $\chi^{N}(\tilde{X}^{j,\epsilon,N}_t,\tilde{X}^{i,\epsilon,N}_t,\tilde{Y}^{j,\epsilon,N}_t,\tilde{X}^{i,\epsilon,N}_t,(\tilde{X}^{1,\epsilon,N}_t,...,\tilde{X}^{N,\epsilon,N}_t))\psi(t,\tilde{X}^{i,\epsilon,N}_t)$, where $\chi^{N}:\R\times\R\times\R\times\R\times\R^N\tto \R$ is the empirical projection of $\chi$ and using Proposition \ref{prop:empprojderivatives}, we get
\begin{align*}
\int_0^t \frac{1}{N}\sum_{j=1}^N b(j)\partial_{\mu}\Phi(i)[j]\psi(s,\tilde{X}^{i,\epsilon,N}_s)ds = \frac{1}{N}\sum_{j=1}^N \sum_{k=1}^{13}\bar{B}^{i,j,\epsilon,N}_{k}(t)
\end{align*}
where  
\begin{align*}
\bar{B}^{i,j,\epsilon,N}_1(t)& = \epsilon^2 [\chi(\tilde{X}^{j,\epsilon,N}_0,\tilde{X}^{i,\epsilon,N}_0,\tilde{Y}^{j,\epsilon,N}_0,\tilde{Y}^{i,\epsilon,N}_0,\tilde{\mu}^{\epsilon,N}_0)\psi(0,\tilde{X}^{i,\epsilon,N}_0)-\chi(\tilde{X}^{j,\epsilon,N}_t,\tilde{X}^{i,\epsilon,N}_t,\tilde{Y}^{j,\epsilon,N}_t,\tilde{Y}^{i,\epsilon,N}_t,\tilde{\mu}^{\epsilon,N}_t)\psi(t,\tilde{X}^{i,\epsilon,N}_t)]\\
\bar{B}^{i,j,\epsilon,N}_2(t)& =\epsilon\int_0^t \biggl(b(j)\chi_x(i,j)\psi(s,i)+b(i)\biggl[\chi_{\bar{x}}(i,j)\psi(s,i)+\chi(i,j)\psi_{\bar{x}}(s,i) \biggr]+g(j)\chi_y(i,j)\psi(s,i)+g(i)\chi_{\bar{y}}(i,j)\psi(s,i)\\
&\qquad+\sigma(j)\tau_1(j)\chi_{xy}(i,j)\psi(s,i)+\sigma(i)\tau_1(i)\biggl[\chi_{\bar{x}\bar{y}}(i,j)\psi(s,i)+\chi_{\bar{y}}(i,j)\psi_{\bar{x}}(s,i)\biggr]\biggr)ds\\
\bar{B}^{i,j,\epsilon,N}_3(t)& = \epsilon\int_0^t \frac{1}{N}\sum_{k=1}^N b(k) \partial_\mu \chi(i,j)[k]\psi(s,i)ds\\
\bar{B}^{i,j,\epsilon,N}_4(t)& =\frac{\epsilon}{N}\int_0^t\biggl(\sigma(j)\tau_1(j)\partial_\mu\chi_{y}(i,j)[j]\psi(s,i)+\sigma(i)\tau_1(i)\partial_\mu\chi^{2}_{\bar{y}}(i,j)[i]\psi(s,i)\biggr)ds\\
\bar{B}^{i,j,\epsilon,N}_5(t)& = \epsilon^2\int_0^t \biggl(\chi(i,j)\dot{\psi}(s,i)+c(j)\chi_x(i,j)\psi(s,i)+c(i)\biggl[\chi_{\bar{x}}(i,j)\psi(s,i)+\chi(i,j)\psi_{\bar{x}}(s,i) \biggr]\\
&\quad+ \frac{1}{N}\sum_{k=1}^N \biggl\lbrace c(k)\partial_\mu\chi(i,j)[k]\biggr\rbrace\psi(s,i)+\frac{1}{2}\sigma^2(j)\chi_{xx}(i,j)\psi(s,i)\\
&\quad+\frac{1}{2}\sigma^2(i)\biggl[\chi_{\bar{x}\bar{x}}(i,j)\psi(s,i)+2\chi_{\bar{x}}(i,j)\psi_{\bar{x}}(s,i)+\chi(i,j)\psi_{\bar{x}\bar{x}}(s,i)\biggr]\\
&\quad+\frac{1}{2}\frac{1}{N}\sum_{k=1}^N\biggl\lbrace \sigma^2(k)\biggl[\partial_z\partial_\mu\chi(i,j)[k]+\frac{1}{N}\partial^2_\mu\chi^{2}(i,j)[k,k]\biggr]\biggr\rbrace\psi(s,i)+\frac{1}{N}\sigma^2(j)\partial_\mu\chi_{x}(i,j)[j]\psi(s,i)\\
&\quad+\frac{1}{N}\sigma^2(i)\biggl[\partial_\mu\chi_{\bar{x}}(i,j)[i]\psi(s,i)+\partial_\mu\chi(i,j)[i]\psi_{\bar{x}}(s,i)\biggr]\biggr)ds\\
\bar{B}^{i,j,\epsilon,N}_6(t)& = \epsilon\int_0^t \tau_1(j)\chi_y(i,j)\psi(s,i)dW_s^j+\epsilon\int_0^t\tau_2(j)\chi_y(i,j)\psi(s,i)dB_s^j+\epsilon\int_0^t\tau_1(i)\chi_{\bar{y}}(i,j)\psi(s,i)dW_s^i\\
&\quad+\epsilon\int_0^t\tau_2(i)\chi_{\bar{y}}(i,j)\psi(s,i)dB_s^i\\
\bar{B}^{i,j,\epsilon,N}_7(t)& = \epsilon^2\int_0^t \sigma(j)\chi_x(i,j)\psi(s,i)dW_s^j+\epsilon^2\int_0^t \sigma(i)\biggl[\chi_{\bar{x}}(i,j)\psi(s,i)+\chi(i,j)\psi_{\bar{x}}(s,i)\biggr]dW_s^i\\
&\quad+ \frac{\epsilon^2}{N}\sum_{k=1}^N \biggl\lbrace \int_0^t \sigma(k)\partial_\mu\chi(i,j)[k]\psi(s,i) dW_s^k \biggr\rbrace\\
\bar{B}^{i,j,\epsilon,N}_8(t)& = \epsilon^2\int_0^t\left(\frac{\sigma(j)\tilde{u}^{N,1}_j(s)}{\sqrt{N}a(N)} \chi_x(i,j)\psi(s,i)+\frac{\sigma(i)\tilde{u}^{N,1}_i(s)}{\sqrt{N}a(N)} \biggl[\chi_{\bar{x}}(i,j)\psi(s,i)+\chi(i,j)\psi_{\bar{x}}(s,i) \biggr]\right)ds\\
\bar{B}^{i,j,\epsilon,N}_{9}(t)& =\epsilon^2\int_0^t\frac{1}{N}\sum_{k=1}^N \biggl\lbrace\frac{\sigma(k)\tilde{u}^{N,1}_k(s)}{\sqrt{N}a(N)}\partial_\mu\chi^{2}(i,j)[k]\biggr\rbrace\psi(s,i)ds\\
\bar{B}^{i,j,\epsilon,N}_{10}(t)& =\epsilon\int_0^t\left(\biggl[\frac{\tau_1(j)\tilde{u}^{N,1}_j(s)}{\sqrt{N}a(N)}+\frac{\tau_2(j)\tilde{u}^{N,2}_j(s)}{\sqrt{N}a(N)}\biggr]\chi_y(i,j)\psi(s,i)+ \biggl[\frac{\tau_1(i)\tilde{u}^{N,1}_i(s)}{\sqrt{N}a(N)}+\frac{\tau_2(i)\tilde{u}^{N,2}_i(s)}{\sqrt{N}a(N)}\biggr]\chi_{\bar{y}}(i,j)\psi(s,i)\right) ds\\
\bar{B}^{i,j,\epsilon,N}_{11}(t)& = \1_{i=j}\epsilon^2\int_0^t \sigma(i)\sigma(j)\biggl[\chi_{x\bar{x}}(i,j)\psi(s,i)+\chi_x(i,j)\psi_{\bar{x}}(s,i)\biggr]ds\\
\bar{B}^{i,j,\epsilon,N}_{12}(t)& =\1_{i=j}\epsilon\int_0^t\left(\sigma(j)\tau_1(i)\chi_{x\bar{y}}(i,j)\psi(s,i)+\sigma(i)\tau_1(j)\biggl[\chi_{\bar{x}y}(i,j)\psi(s,i)+\chi_y(i,j)\psi_{\bar{x}}(s,i)\biggr]\right)ds \\
\bar{B}^{i,j,\epsilon,N}_{13}(t)& = \1_{i=j}\int_0^t\biggl[\tau_1(i)\tau_1(j)+\tau_2(i)\tau_2(j)\biggr]\chi_{y\bar{y}}(i,j)\psi(s,i)ds.
\end{align*}

Here we have introduced the notation $\chi(i,j)$ to denote $\chi(\tilde{X}^{j,\epsilon,N}_s,\tilde{X}^{i,\epsilon,N}_s,\tilde{Y}^{j,\epsilon,N}_s,\tilde{Y}^{i,\epsilon,N}_s,\tilde{\mu}^{\epsilon,N}_s)$, $\partial_\mu\chi(i,j)[k]$ to denote $\partial_\mu\chi(\tilde{X}^{j,\epsilon,N}_s,\tilde{X}^{i,\epsilon,N}_s,\tilde{Y}^{j,\epsilon,N}_s,\tilde{Y}^{i,\epsilon,N}_s,\tilde{\mu}^{\epsilon,N}_s)[\tilde{X}^{k,\epsilon,N}_s]$, and similarly for $\partial_\mu\chi(i,j)[k,k]$. We also use $\psi(s,i)$ to denote $\psi(s,\tilde{X}^{i,\epsilon,N}_s)$.

Using that $\sigma,\tau_1,\tau_2,$ and $g$ are bounded and \ref{assumption:qF2bound} on the growth of $\chi$ and its derivatives, the proof that
\begin{align*}
\frac{a(N)}{\sqrt{N}}\sum_{i=1}^N\E\biggl[\sup_{t\in[0,T]}\biggl|\frac{1}{N}\sum_{j=1}^N \sum_{k=1}^{12}\bar{B}^{i,j,\epsilon,N}_{k}(t)\biggr|^2\biggr]&\leq C\epsilon^2 a(N)\sqrt{N}(1+T+T^2)\norm{\psi}^2_{C_b^{1,2}}
\end{align*}
follows essentially in the same way as Proposition $\ref{prop:fluctuationestimateparticles1}$. For example, for $\bar{B}_2$, we can use the assumed linear growth in $y$ of $b$ and boundedness of $g$ and $\sigma$ from \ref{assumption:gsigmabounded}, boundedness of $\tau_1$ from \ref{assumption:uniformellipticity}, and boundedness of $\chi,\chi_x,\chi_{\bar{x}},\chi_y,\chi_{\bar{y}}$ and polynomial growth in $y$ of $\chi_{xy}$ and $\chi_{\bar{x}\bar{y}}$ to get:
\begin{align*}
\frac{a(N)}{\sqrt{N}}\sum_{i=1}^N\E\biggl[\sup_{t\in[0,T]}\biggl|\frac{1}{N}\sum_{j=1}^N \bar{B}^{i,j,\epsilon,N}_{2}(t)\biggr|^2\biggr]&\leq C\epsilon^2 a(N)\sqrt{N}T^2\biggl(1+ \frac{1}{N}\sum_{i=1}^N \sup_{t\in [0,T]}\E\biggl[|\tilde{Y}^{i,\epsilon,N}_t|^{2}+|\tilde{Y}^{i,\epsilon,N}_t|^{2q_{\chi_y}(0,1,0)} \biggr]\biggr)\\
&\hspace{7.5cm}\times(\norm{\psi}^2_\infty+\norm{\psi_x}^2_\infty)\\
&\leq C\epsilon^2 a(N)\sqrt{N}T^2\biggl(1+ \frac{1}{N}\sum_{i=1}^N \sup_{t\in [0,T]}\E\biggl[|\tilde{Y}^{i,\epsilon,N}_t|^{2}\biggr]\biggr)(\norm{\psi}^2_\infty+\norm{\psi_x}^2_\infty)\\
&\leq C\epsilon^2 a(N)\sqrt{N}T^2(\norm{\psi}^2_\infty+\norm{\psi_x}^2_\infty),
\end{align*}
where in the last step we used Lemma \ref{lemma:tildeYuniformbound}.

The other bounds follow similarly. We omit the details for brevity. To handle the last term, we see by boundedness of $\tau_1,\tau_2$ from \ref{assumption:uniformellipticity} and linear growth of $\chi_{y\bar{y}}$ from \ref{assumption:qF2bound}:
\begin{align*}
\frac{a(N)}{\sqrt{N}}\sum_{i=1}^N\E\biggl[\sup_{t\in[0,T]}\biggl|\frac{1}{N}\sum_{j=1}^N \bar{B}^{i,j,\epsilon,N}_{13}\biggr|^2\biggr] & = \frac{a(N)}{\sqrt{N}}\frac{1}{N^2}\sum_{i=1}^N\E\biggl[\sup_{t\in[0,T]}\biggl|\int_0^t \biggl[\tau_1^2(i)+\tau_2^2(i)\biggr]\chi_{y\bar{y}}(i,i)\psi(s,i)ds\biggr|^2  \biggr]\\
&\leq \frac{a(N)}{\sqrt{N}}\frac{1}{N^2} CT\sum_{i=1}^N \E\biggl[\int_0^T |\chi_{y\bar{y}}(i,i)|^2ds \biggr]\norm{\psi_\infty}^2\\
&\leq \frac{a(N)}{N^{3/2}}CT^2(1+\frac{1}{N}\sum_{i=1}^N\sup_{t\in[0,T]}\E\biggl[|\tilde{Y}^{i,\epsilon,N}_t|^{2}\biggr])\norm{\psi_\infty}^2\\
&\leq \frac{a(N)}{N^{3/2}}CT^2\norm{\psi_\infty}^2 \textrm{ (by Lemma \ref{lemma:tildeYuniformbound}).}
\end{align*}
\end{proof}

\begin{proposition}\label{prop:llntypefluctuationestimate1}
Assume \ref{assumption:uniformellipticity} - \ref{assumption:gsigmabounded}. Let $F$ be any function such that $\Xi$ satisfies assumption \ref{assumption:forcorrectorproblem}. Then for $\bar{F}(x,\mu)\coloneqq \int_\R F(x,y,\mu)\pi(dy;x,\mu)$, with $\pi$ as in Equation \eqref{eq:invariantmeasureold} and $\psi\in C^{1,2}_b([0,T]\times\R)$
\begin{align*}
&\frac{a(N)}{\sqrt{N}}\sum_{i=1}^N\E\biggl[\sup_{t\in[0,T]}\biggl|\int_0^t \biggl(F(\tilde{X}^{i,\epsilon,N}_s,\tilde{Y}^{i,\epsilon,N}_s,\tilde{\mu}^{\epsilon,N}_s)-\bar{F}(\tilde{X}^{i,\epsilon,N}_s,\tilde{\mu}^{\epsilon,N}_s)\biggr)\psi(s,\tilde{X}^{i,\epsilon,N}_s)dt\biggr|\biggr]\leq \nonumber\\
&\hspace{10cm}\leq C\epsilon a(N)\sqrt{N}(1+T+T^{1/2})\norm{\psi}_{C_b^{1,2}}.
\end{align*}
\end{proposition}
\begin{proof}
By Lemma \ref{lemma:Ganguly1DCellProblemResult}, we can consider $\Xi:\R\times\R\times\mc{P}(\R)\tto\R$ the unique classical solution to
\begin{align}\label{eq:driftcorrectorproblem}
L_{x,\mu}\Xi(x,y,\mu) &= -[F(x,y,\mu)-\int_{\R}F(x,y,\mu)\pi(dy;x,\mu)],\quad
\int_{\R}\Xi(x,y,\mu)\pi(dy;x,\mu)=0.
\end{align}
($\Xi$ and $F$ may also depend on $t\in [0,T]$, but we suppress this in the notation here). Applying It\^o's formula to $\Xi^N(\tilde{X}^{i,\epsilon,N}_t,\tilde{Y}^{i,\epsilon,N}_t,(\tilde{X}^{1,\epsilon,N}_t,...,\tilde{X}^{N,\epsilon,N}_t))\psi(t,\tilde{X}^{i,\epsilon,N}_t)$, where again $\Xi^N:\R\times\R\times\R^N\tto \R$ is the empirical projection of $\Xi$ and using Proposition \ref{prop:empprojderivatives}, we get:
\begin{align*}
&\int_0^t \biggl(F(\tilde{X}^{i,\epsilon,N}_s,\tilde{Y}^{i,\epsilon,N}_s,\tilde{\mu}^{\epsilon,N})-\bar{F}(\tilde{X}^{i,\epsilon,N}_s,\tilde{\mu}^{\epsilon,N}_s)\biggr)\psi(s,\tilde{X}^{i,\epsilon,N}_s)dt  = \sum_{k=1}^{10}C^{i,\epsilon,N}_k(t)\\
C^{i,\epsilon,N}_1(t)& = \epsilon^2 [\Xi(\tilde{X}^{i,\epsilon,N}_0,\tilde{Y}^{i,\epsilon,N}_0,\tilde{\mu}^{\epsilon,N}_0)\psi(0,\tilde{X}^{i,\epsilon,N}_0)-\Xi(\tilde{X}^{i,\epsilon,N}_t,\tilde{Y}^{i,\epsilon,N}_t,\tilde{\mu}^{\epsilon,N}_t)\psi(t,\tilde{X}^{i,\epsilon,N}_t)]\nonumber\\
C^{i,\epsilon,N}_2(t)& = \epsilon\int_0^t \left( b(i)[\Xi_x(i)\psi(s,i)+\Xi(i)\psi_x(s,i)]+g(i)\Xi_y(i)\psi(s,i)+\sigma(i)\tau_1(i)[\Xi_{xy}(i)\psi(s,i) +\Xi_y(i)\psi_x(s,i)] \right)ds\nonumber\\
C^{i,\epsilon,N}_3(t)& = \epsilon\int_0^t \frac{1}{N}\sum_{j=1}^N b(j)\partial_{\mu}\Xi(i)[j]\psi(s,i)ds \nonumber\\
C^{i,\epsilon,N}_4(t)& = \frac{\epsilon}{N}\int_0^t \sigma(i)\tau_1(i)\partial_\mu \Xi_y (i)[i]\psi(s,i)ds\nonumber\\
C^{i,\epsilon,N}_5(t)& =\epsilon^2 \int_0^t \biggl(\Xi(i)\dot{\psi}(s,i)+c(i)[\Xi_x(i)\psi(s,i)+\Xi(i)\psi_x(s,i)]\\
&+\frac{\sigma^2(i)}{2}[\Xi_{xx}(i)\psi(s,i)+2\Xi_x(i)\psi_x(s,i)+\Xi(i)\psi_{xx}(s,i)+\frac{2}{N}\partial_\mu\Xi(i)[i]\psi_x(s,i)+\frac{2}{N}\partial_\mu\Xi_x(i)[i]\psi(s,i)] \nonumber\\
&+ \frac{1}{N}\sum_{j=1}^N \biggl\lbrace c(j)\partial_\mu \Xi(i)[j]\psi(s,i)+\frac{1}{2}\sigma^2(j)[\frac{1}{N}\partial^2_\mu \Xi(i)[j,j] +\partial_z\partial_\mu \Xi(i)[j]]\psi(s,i)  \biggr\rbrace \biggr)ds \nonumber\\
C^{i,\epsilon,N}_6(t)& = \epsilon\int_0^t \tau_1(i)\Xi_y(i)\psi(s,i)dW^i_s + \epsilon\int_0^t \tau_2(i)\Xi_y(i)\psi(s,i)dB^i_s\nonumber\\
C^{i,\epsilon,N}_7(t)& = \epsilon^2 \biggl[ \int_0^t \sigma(i)[\Xi_x(i)\psi(s,i)+\Xi(i)\psi_x(s,i)]dW^i_t + \frac{1}{N}\sum_{j=1}^N\biggl\lbrace \int_0^t \sigma(j)\partial_\mu \Xi(i)[j]\psi(s,i)dW^j_s \biggr\rbrace \biggr] \nonumber\\
C^{i,\epsilon,N}_8(t)& = \epsilon^2 \int_0^t \frac{\sigma(i)}{a(N)\sqrt{N}}\tilde{u}^{N,1}_i(s)[\Xi_x(i)\psi(s,i)+\Xi(i)\psi_x(s,i)]ds \nonumber\\
C^{i,\epsilon,N}_{9}(t)& = \epsilon^2 \int_0^t \frac{1}{N} \biggl\lbrace \sum_{j=1}^N \frac{\sigma(j)}{a(N)\sqrt{N}}\tilde{u}^{N,1}_j(s) \partial_\mu \Xi(i)[j]\psi(s,i)\biggr\rbrace ds \nonumber\\
C^{i,\epsilon,N}_{10}(t)& = \epsilon\int_0^t [\frac{\tau_1(i)}{a(N)\sqrt{N}}\tilde{u}^{N,1}_i(s)+\frac{\tau_2(i)}{a(N)\sqrt{N}}\tilde{u}^{N,2}_i(s)]\Xi_y(i)\psi(s,i)ds.
\end{align*}
Then using that $\sigma,\tau_1,\tau_2,$ and $g$ are bounded and that $b,c$ grow at most linearly in $y$ uniformly in $x,\mu$, and the assumptions on the growth of $\Xi$ and its derivatives from \ref{assumption:forcorrectorproblem}, the proof follows in essentially the same way as Propositions \ref{prop:fluctuationestimateparticles1} and \ref{prop:purpleterm1}.

Since $\Xi$ is not necessarily bounded under Assumption \ref{assumption:forcorrectorproblem} ($\tilde{q}_{\Xi}(n,l,\bm{\beta})\leq 1,\forall (n,l,\bm{\beta})\in \tilde{\bm{\zeta}}_1$ allows $\Xi$ to grow linearly in $y$), we need to handle the first term in the following way:
\begin{align*}
&\frac{a(N)}{\sqrt{N}}\sum_{i=1}^N \epsilon^2 \E\biggl[\sup_{t\in[0,T]}\biggl|\Xi(\tilde{X}^{i,\epsilon,N}_0,\tilde{Y}^{i,\epsilon,N}_0,\tilde{\mu}^{\epsilon,N}_0)\psi(0,\tilde{X}^{i,\epsilon,N}_0)-\Xi(\tilde{X}^{i,\epsilon,N}_t,\tilde{Y}^{i,\epsilon,N}_t,\tilde{\mu}^{\epsilon,N}_t)\psi(t,\tilde{X}^{i,\epsilon,N}_t) \biggr|\biggr]\leq\\
&\leq C\epsilon^2a(N)\sqrt{N}\biggl[1+ \frac{1}{N}\sum_{i=1}^N \E\biggl[\sup_{t\in[0,T]}|\tilde{Y}^{i,\epsilon,N}_t|\biggr]\biggr]\norm{\psi}_\infty\leq C(\rho)\epsilon^2a(N)\sqrt{N}\biggl[1+ \epsilon^{-\rho}\biggr]\norm{\psi}_\infty
\end{align*}
for any $\rho\in (0,2)$ by Lemma \ref{lemma:ytildeexpofsup}. Taking any $\rho\in (0,1]$, the desired bound holds.

The only other  terms that are handled differently in a way that matters are $C^{i,\epsilon,N}_4(t)$, which corresponds to $\tilde{B}^{i,\epsilon,N}_2(t)$, where the difference of having a $\epsilon$ in front means that it is bounded by $\epsilon a(N)\leq \epsilon a(N)\sqrt{N}$, hence there being no need to include $a(N)/\sqrt{N}$ in the definition of $C(N)$, and $C^{i,\epsilon,N}_2(t), C^{i,\epsilon,N}_3(t),C^{i,\epsilon,N}_6(t),$ and $C^{i,\epsilon,N}_{10}(t)$, which were $O(1)$ in Lemma \ref{prop:fluctuationestimateparticles1} and hence were not shown to vanish. $C_2$ is handled as $\tilde{B}_3$ was, $C_3$ in the same way that $\tilde{B}_5$ was, $C_6$ in the same way that $\tilde{B}_6$ was, and $C_{10}$ in the same way that $\tilde{B}_7$ was.

\end{proof}

\section{Tightness of the Controlled Pair}\label{sec:tightness}
In this section we throughout fix any controls satisfying the bound \eqref{eq:controlassumptions} and prove tightness of the pair $(\tilde{Z}^N,Q^N)$ from Equations \eqref{eq:controlledempmeasure} and \eqref{eq:occupationmeasures} under those controls. We will establish tightness in the appropriate spaces by proving tightness for each of the marginals.

As discussed in Section \ref{subsec:notationandtopology}, in order to prove tightness of the controlled fluctuation process $\tilde{Z}^N$ in $C([0,T];\mc{S}_{-m})$ for some $m\in\bb{N}$ sufficiently large (see Equation \eqref{eq:mdefinition}), we will use the theory of Mitoma from \cite{Mitoma}. In particular, we need to prove uniform $m$-continuity for sufficiently large $m$ in the family of Hilbert norms \eqref{eq:familyofhilbertnorms}, and tightness of $\langle Z^N,\phi\rangle$ as a $C([0,T];\R)$-valued random variables in order to apply Theorem 3.1 and Remark R1) in \cite{Mitoma}. For the former, by definition we need some uniform in time control over the $\mc{S}_{-m}$-norm of $\tilde{Z}^N$. By Markov's inequality, it suffices to show that $\sup_{\phi\in\mc{S}:\norm{\phi}_{m}=1}\E[\sup_{t\in[0,T]}|\langle\tilde{Z}^N_t,\phi\rangle |]\leq C$ (see, e.g., the proof of \cite{BW} Theorem 4.7). As mentioned in Remark \ref{remark:ontheiidsystem}, we will do so in Lemma \ref{lemma:Zboundbyphi4} via triangle inequality and establishing an $L^2$ rate of convergence of the controlled particle system \eqref{eq:controlledslowfast1-Dold} to the IID particle system \eqref{eq:IIDparticles}, and a rate of convergence of the IID particle system \eqref{eq:IIDparticles} to the averaged McKean-Vlasov SDE \eqref{eq:LLNlimitold}. The convergence of the controlled particle system \eqref{eq:controlledslowfast1-Dold} to the IID particle system \eqref{eq:IIDparticles} is the subject of Subsection \ref{subsec:couplingcontrolledtoiid} and the convergence of the IID slow-fast McKean-Vlasov SDEs \eqref{eq:IIDparticles} to the averaged McKean-Vlasov SDE \eqref{eq:LLNlimitold} is the subject of Subsection \ref{sec:averagingfullycoupledmckeanvlasov}.

A major difference between the coupling arguments in the references listed in Remark \ref{remark:ontheiidsystem} and ours is that the IID system in the listed references were all equal in distribution to the law of the system which they are considering fluctuations from. This is not the case for us, since, as is well-known, we do not expect in general to have $L^2$ convergence of fully-coupled slow-fast diffusions to their averaged limit (see \cite{Bensoussan} Remark 3.4.4 for an illustrative example). In other words, Lemma \ref{lemma:XbartildeXdifference} cannot hold with $\bar{X}^{i,\epsilon}$ replaced by IID copies of the averaged limiting McKean-Vlasov Equation $X_t$ from Equation \eqref{eq:LLNlimitold}. We are thus exploiting here the fact that the limits $\epsilon\downarrow 0$ and $N\toinf$ commute, as shown in \cite{BS}, and hence we can use an IID system of Slow-Fast McKean-Vlasov SDEs as our intermediate process for our proof of tightness. This commutativity of the limits will hold so long as sufficient conditions for the propagation of chaos and stochastic averaging respectively hold for the system of SDEs \eqref{eq:controlledslowfast1-Dold} and the invariant measure $\pi$ from Equation \eqref{eq:invariantmeasureold} is unique for all $x\in\R,\mu\in\mc{P}(\R)$. Recall that the latter is a consequence of assumptions \ref{assumption:uniformellipticity} and \ref{assumption:retractiontomean}.

Tightness of $Q^N$ is contained in Subsection \ref{SS:QNtightness}, and is essentially a consequence of moment bounds on the controlled particles \eqref{eq:controlledslowfast1-Dold}, which again follow from the results of Section \eqref{sec:ergodictheoremscontrolledsystem}.

\subsection{On the Rate of Averaging for Fully-Coupled Slow-Fast McKean-Vlasov Diffusions}\label{sec:averagingfullycoupledmckeanvlasov}

Here we recall a result which allows us to establish closeness of the slow-fast McKean-Vlasov SDEs \eqref{eq:IIDparticles} to the averaged McKean-Vlasov SDE \eqref{eq:LLNlimitold}. This result will be used in the Lemma \ref{lemma:Zboundbyphi4}, which is a key ingredient in the proof of tightness of $\br{\tilde{Z}^N}_{N\in\bb{N}}$. Therein, the first term being bounded is essentially due to the propagation of chaos holding for the controlled particle system \eqref{eq:controlledslowfast1-Dold}, as captured by Lemma \ref{lemma:XbartildeXdifference}. For the second term, the particles being IID means it is sufficient to gain control over convergence of $\sup_{\phi\in\mc{S}:\norm{\phi}_{m}=1}|\E[\phi(\bar{X}^{1,\epsilon}_t)- \phi(X_t) ] |$ as $\epsilon\downarrow 0$. There are very few results in the current literature in this direction. The existing averaging results for Slow-Fast McKean Vlasov SDEs can be found in \cite{RocknerMcKeanVlasov}, \cite{HLL}, \cite{KSS} and in \cite{BezemekSpiliopoulosAveraging2022}. In  \cite{RocknerMcKeanVlasov}, \cite{HLL}, \cite{KSS}, only systems where $L^2$ rates of averaging can be found are considered. Moreover, even for standard diffusion processes (which do not depend on their law), the only result for rates of convergence in distribution in the sense we desire for the fully-coupled setting is found in \cite{RocknerFullyCoupled} Theorem 2.3. The fully coupled case for McKean-Vlasov diffusions is addressed in \cite{BezemekSpiliopoulosAveraging2022}.

We mention here the main result from \cite{BezemekSpiliopoulosAveraging2022} that will be used in our case. In particular, we wish to establish a rate of convergence in distribution of
\begin{align}\label{eq:slow-fastMcKeanVlasov}
\bar{X}^{\epsilon}_t &= \eta^x+\int_0^t\biggl[\frac{1}{\epsilon}b(\bar{X}^{\epsilon}_s,\bar{Y}^{\epsilon}_s,\mc{L}(\bar{X}^{\epsilon}_s))+ c(\bar{X}^{\epsilon}_s,\bar{Y}^{\epsilon}_s,\mc{L}(\bar{X}^{\epsilon}_s)) \biggr]ds + \int_0^t\sigma(\bar{X}^{\epsilon}_s,\bar{Y}^{\epsilon}_s,\mc{L}(\bar{X}^{\epsilon}_s))dW_s\\
\bar{Y}^{\epsilon}_t & = \eta^y+\int_0^t\frac{1}{\epsilon}\biggl[\frac{1}{\epsilon}f(\bar{X}^{\epsilon}_s,\bar{Y}^{\epsilon}_s,\mc{L}(\bar{X}^{\epsilon}_s))+ g(\bar{X}^{\epsilon}_s,\bar{Y}^{\epsilon}_s,\mc{L}(\bar{X}^{\epsilon}_s)) \biggr]dt\nonumber \\
&+ \frac{1}{\epsilon}\biggl[\int_0^t\tau_1(\bar{X}^{\epsilon}_s,\bar{Y}^{\epsilon}_s,\mc{L}(\bar{X}^{\epsilon}_s))dW_s+\int_0^t\tau_2(\bar{X}^{\epsilon}_s,\bar{Y}^{\epsilon}_s,\mc{L}(\bar{X}^{\epsilon}_s))dB_s\biggr]\nonumber,
\end{align}
to the solution of Equation \eqref{eq:LLNlimitold}. Note that a solution to Equation \eqref{eq:slow-fastMcKeanVlasov} is equal in distribution to the IID particles from Equation \eqref{eq:IIDparticles}. 
The following moment bound holds.
\begin{lemma}\label{lemma:barYuniformbound}
Assume \ref{assumption:uniformellipticity}- \ref{assumption:retractiontomean}, \ref{assumption:strongexistence}, and \ref{assumption:gsigmabounded}. Then for any $p\in\bb{N}$: 
\begin{align*}
\sup_{\epsilon>0}\sup_{t\in [0,T]}\E\biggl[|\bar{Y}^{\epsilon}_t|^{2p}\biggr]\leq C(p,T)+|\eta^y|^{2p}.
\end{align*}
\end{lemma}
\begin{proof}
The proof of this lemma is omitted as it follows very closely the proof of Lemma 4.1 in \cite{BezemekSpiliopoulosAveraging2022}.
\end{proof}

Then, the main result of \cite{BezemekSpiliopoulosAveraging2022} that is relevant for our purposes is Theorem \ref{theo:mckeanvlasovaveraging}.
\begin{theo}[Corollary 3.2 of \cite{BezemekSpiliopoulosAveraging2022}]\label{theo:mckeanvlasovaveraging}
Assume that assumptions \ref{assumption:uniformellipticity} - \ref{assumption:forcorrectorproblem} as well as \ref{assumption:limitingcoefficientsLionsDerivatives}-\ref{assumption:tildechi} hold. Then for $\phi\in C_{b,L}^4(\R)$, there is a constant $C=C(T)$ that is independent of $\phi$ such that
\begin{align*}
\sup_{s\in [0,T]}\biggl|\E[\phi(\bar{X}^{\epsilon}_s)]-\E[\phi(X_s)]\biggr|&\leq \epsilon C(T)|\phi|_{4},
\end{align*}
where $\bar{X}^\epsilon$ is as in Equation \eqref{eq:slow-fastMcKeanVlasov}, $X$ is as in Equation \eqref{eq:LLNlimitold}, and $|\cdot|_4$ is as in Equation \eqref{eq:boundedderivativesseminorm}.
\end{theo}

\begin{remark}
Though in \cite{BezemekSpiliopoulosAveraging2022} the assumptions are stated in terms of sufficient conditions on the limiting coefficients for the needed regularity of $\Phi,\chi,\Xi,\bar{\gamma},\bar{D}^{1/2},$ and $\tilde{\chi}$ in the proofs therein to hold (which is much easier to do in that situation since the lack of control eliminates the need for tracking specific rates of polynomial growth), it can be checked that the assumptions imposed on these functions by \ref{assumption:multipliedpolynomialgrowth}, \ref{assumption:qF2bound}, \ref{assumption:forcorrectorproblem}, \ref{assumption:limitingcoefficientsLionsDerivatives}, and \ref{assumption:tildechi} respectively are sufficient. See also Remark 2.6 therein.

In particular, in \cite{BezemekSpiliopoulosAveraging2022}, since specific rates of polynomial growth are not tracked, it is assumed the initial condition of $\bar{Y}^{\epsilon}$ has all moments bounded. This holds automatically here, since $\eta^y\in\R$ are deterministic. Then, due to Lemma \ref{lemma:barYuniformbound}, it is sufficient to show the derivatives of the Poisson equations which show up in the proof have polynomial growth in $y$ uniformly in $x,\mu,z$. In fact, the same Poisson equations are being used in Section 4 of \cite{BezemekSpiliopoulosAveraging2022} to gain ergodic-type theorems of the same nature as those of Section \ref{sec:ergodictheoremscontrolledsystem} here. The growth rates of $\Phi,\chi,\Xi$ as imposed in \ref{assumption:multipliedpolynomialgrowth}, \ref{assumption:qF2bound}, and \ref{assumption:forcorrectorproblem} are already required here for the ergodic-type theorems for the controlled system \eqref{eq:controlledslowfast1-Dold} found in Section \ref{sec:ergodictheoremscontrolledsystem}, and these conditions can be seen as more than sufficient for the results of \cite{BezemekSpiliopoulosAveraging2022} to go through. The solution $\tilde{\chi}$ to Equation \eqref{eq:tildechi} does not, however, appear elsewhere in this paper, despite appearing in Proposition 4.4. of \cite{BezemekSpiliopoulosAveraging2022}, which is fundamental to the result presented here as Theorem \ref{theo:mckeanvlasovaveraging}. This is why we can allow for the specified derivatives of $\tilde{\chi}$ (which are exactly those appearing in the proof of Proposition 4.4.) in Assumption \ref{assumption:tildechi} to have polynomial growth of any order.

Lastly, we remark that the regularity of the limiting coefficients imposed by \ref{assumption:limitingcoefficientsLionsDerivatives} is used not for ergodic-type theorems, but instead to establish regularity a Cauchy-Problem on Wasserstein space in Lemma 5.1 of \cite{BezemekSpiliopoulosAveraging2022}, which provides a refinement of Theorem 2.15 in \cite{CST}. As remarked therein, these assumptions can likely be relaxed via an alternative proof method, but as it stands these are the only results in this direction which provide sufficient regularity on the derivatives needed to prove Theorem \ref{theo:mckeanvlasovaveraging}.
\end{remark}
\subsection{Coupling of the Controlled Particles and the IID Slow-Fast McKean-Vlasov Particles}\label{subsec:couplingcontrolledtoiid}
Here we establish a coupling result, which we will use along with  Theorem \ref{theo:mckeanvlasovaveraging} in order to establish tightness for the controlled fluctuation process $\br{\tilde{Z}^N}$ from Equation \eqref{eq:controlledempmeasure}.  Recall the processes $(\tilde{X}^{i,\epsilon,N}_t,\tilde{Y}^{i,\epsilon,N}_t)$ that satisfy  (\ref{eq:controlledslowfast1-Dold}) and $(\bar{X}^{i,\epsilon}_t,\bar{Y}^{i,\epsilon}_t)$ that satisfies  (\ref{eq:IIDparticles}).
\begin{lemma}\label{lemma:tildeYbarYdifference}
Assume \ref{assumption:uniformellipticity}- \ref{assumption:qF2bound}, \ref{assumption:uniformLipschitzxmu}, and \ref{assumption:2unifboundedlinearfunctionalderivatives}. Then there exists $C>0$ such that for all $t\in[0,T]$:
\begin{align*}
\frac{1}{N}\sum_{i=1}^N\E\biggl[|\tilde{Y}^{i,\epsilon,N}_t-\bar{Y}^{i,\epsilon}_t|^2\biggr]&\leq C\biggl\lbrace\epsilon^2+ \frac{1}{N}+\frac{1}{Na^2(N)} + \sup_{s\in [0,t]}\E\biggl[ \frac{1}{N}\sum_{i=1}^N\biggl|\tilde{X}^{i,\epsilon,N}_s-\bar{X}^{i,\epsilon}_s\biggr|^2\biggr]\biggr\rbrace
\end{align*}
\end{lemma}
\begin{proof}
We set $\tau_1\equiv 0$, since terms involving $\tau_1$ can be handled in the same way as those involving $\tau_2$ in the proof. By It\^o's formula, and  given that the stochastic integrals are martingales (using Lemmas \ref{lemma:tildeYuniformbound} and \ref{lemma:barYuniformbound}),
\begin{align*}
&\frac{d}{dt}\E\biggl[|\tilde{Y}^{i,\epsilon,N}_t-\bar{Y}^{i,\epsilon}_t|^2\biggr]  = 2 \E\biggl[\frac{1}{\epsilon^2}\biggl(f(\tilde{X}^{i,\epsilon,N}_t,\tilde{Y}^{i,\epsilon,N}_t,\tilde{\mu}^{\epsilon,N}_t) - f(\bar{X}^{i,\epsilon}_t,\bar{Y}^{i,\epsilon}_t,\mc{L}(\bar{X}^\epsilon_t))\biggr)(\tilde{Y}^{i,\epsilon,N}_t-\bar{Y}^{i,\epsilon}_t) \\
&+ \frac{1}{2\epsilon^2}|\tau_2(\tilde{X}^{i,\epsilon,N}_t,\tilde{Y}^{i,\epsilon,N}_t,\tilde{\mu}^{\epsilon,N}_t)-\tau_2(\bar{X}^{i,\epsilon}_t,\bar{Y}^{i,\epsilon}_t,\mc{L}(\bar{X}^\epsilon_t))|^2 \\
&+ \frac{1}{\epsilon}\biggl(g(\tilde{X}^{i,\epsilon,N}_t,\tilde{Y}^{i,\epsilon,N}_t,\tilde{\mu}^{\epsilon,N}_t) - g(\bar{X}^{i,\epsilon}_t,\bar{Y}^{i,\epsilon}_t,\mc{L}(\bar{X}^\epsilon_t))\biggr)(\tilde{Y}^{i,\epsilon,N}_t-\bar{Y}^{i,\epsilon}_t) \\
&+\frac{1}{\epsilon\sqrt{N}a(N)}\tau_2(\tilde{X}^{i,\epsilon,N}_t,\tilde{Y}^{i,\epsilon,N}_t,\tilde{\mu}^{\epsilon,N}_t)\tilde{u}^{N,2}_i(t)(\tilde{Y}^{i,\epsilon,N}_t-\bar{Y}^{i,\epsilon}_t)\biggr]\\
& \leq 2\E\biggl[ \frac{1}{\epsilon^2}\biggl(f(\tilde{X}^{i,\epsilon,N}_t,\tilde{Y}^{i,\epsilon,N}_t,\tilde{\mu}^{\epsilon,N}_t) - f(\tilde{X}^{i,\epsilon,N}_t,\bar{Y}^{i,\epsilon}_t,\tilde{\mu}^{\epsilon,N}_t)\biggr)(\tilde{Y}^{i,\epsilon,N}_t-\bar{Y}^{i,\epsilon}_t) \\
&+\frac{1}{\epsilon^2}\biggl(f(\tilde{X}^{i,\epsilon,N}_t,\bar{Y}^{i,\epsilon}_t,\tilde{\mu}^{\epsilon,N}_t) - f(\bar{X}^{i,\epsilon}_t,\bar{Y}^{i,\epsilon}_t,\bar{\mu}^{\epsilon,N}_t)\biggr)(\tilde{Y}^{i,\epsilon,N}_t-\bar{Y}^{i,\epsilon}_t)\\
&+\frac{1}{\epsilon^2}\biggl(f(\bar{X}^{i,\epsilon}_t,\bar{Y}^{i,\epsilon}_t,\bar{\mu}^{\epsilon,N}_t) - f(\bar{X}^{i,\epsilon}_t,\bar{Y}^{i,\epsilon}_t,\mc{L}(\bar{X}^\epsilon_t))\biggr)(\tilde{Y}^{i,\epsilon,N}_t-\bar{Y}^{i,\epsilon}_t)\\
&+ \frac{1}{\epsilon^2}|\tau_2(\tilde{X}^{i,\epsilon,N}_t,\tilde{Y}^{i,\epsilon,N}_t,\tilde{\mu}^{\epsilon,N}_t)-\tau_2(\tilde{X}^{i,\epsilon,N}_t,\bar{Y}^{i,\epsilon}_t,\tilde{\mu}^{\epsilon,N}_t)|^2 \\
&+ \frac{2}{\epsilon^2}|\tau_2(\tilde{X}^{i,\epsilon,N}_t,\bar{Y}^{i,\epsilon}_t,\tilde{\mu}^{\epsilon,N}_t)-\tau_2(\bar{X}^{i,\epsilon}_t,\bar{Y}^{i,\epsilon}_t,\bar{\mu}^{\epsilon,N}_t)|^2 \\
&+ \frac{2}{\epsilon^2}|\tau_2(\bar{X}^{i,\epsilon}_t,\bar{Y}^{i,\epsilon}_t,\bar{\mu}^{\epsilon,N}_t)-\tau_2(\bar{X}^{i,\epsilon}_t,\bar{Y}^{i,\epsilon}_t,\mc{L}(\bar{X}^\epsilon_t))|^2 \\
&+ \frac{1}{\epsilon}\biggl(g(\tilde{X}^{i,\epsilon,N}_t,\tilde{Y}^{i,\epsilon,N}_t,\tilde{\mu}^{\epsilon,N}_t) - g(\bar{X}^{i,\epsilon}_t,\bar{Y}^{i,\epsilon}_t,\mc{L}(\bar{X}^\epsilon_t))\biggr)(\tilde{Y}^{i,\epsilon,N}_t-\bar{Y}^{i,\epsilon}_t)\\
&+\frac{1}{\epsilon\sqrt{N}a(N)}\tau_2(\tilde{X}^{i,\epsilon,N}_t,\tilde{Y}^{i,\epsilon,N}_t,\tilde{\mu}^{\epsilon,N}_t)\tilde{u}^{N,2}_i(t)(\tilde{Y}^{i,\epsilon,N}_t-\bar{Y}^{i,\epsilon}_t)\biggr]\\
& \leq 2 \E\biggl[-\frac{\beta}{2\epsilon^2}|\tilde{Y}^{i,\epsilon,N}_t-\bar{Y}^{i,\epsilon}_t|^2 \\
&+\frac{1}{2\eta\epsilon^2}\biggl|f(\tilde{X}^{i,\epsilon,N}_t,\bar{Y}^{i,\epsilon}_t,\tilde{\mu}^{\epsilon,N}_t) - f(\bar{X}^{i,\epsilon}_t,\bar{Y}^{i,\epsilon}_t,\bar{\mu}^{\epsilon,N}_t)\biggr|^2 +\frac{\eta}{2\epsilon^2}|\tilde{Y}^{i,\epsilon,N}_t-\bar{Y}^{i,\epsilon}_t|^2\\
&+\frac{1}{2\eta\epsilon^2}\biggl|f(\bar{X}^{i,\epsilon}_t,\bar{Y}^{i,\epsilon}_t,\bar{\mu}^{\epsilon,N}_t) - f(\bar{X}^{i,\epsilon}_t,\bar{Y}^{i,\epsilon}_t,\mc{L}(\bar{X}^\epsilon_t))\biggr|^2+\frac{\eta}{2\epsilon^2}|\tilde{Y}^{i,\epsilon,N}_t-\bar{Y}^{i,\epsilon}_t|^2\\
&+ \frac{2}{\epsilon^2}|\tau_2(\tilde{X}^{i,\epsilon,N}_t,\bar{Y}^{i,\epsilon}_t,\tilde{\mu}^{\epsilon,N}_t)-\tau_2(\bar{X}^{i,\epsilon}_t,\bar{Y}^{i,\epsilon}_t,\bar{\mu}^{\epsilon,N}_t)|^2 \\
&+ \frac{2}{\epsilon^2}|\tau_2(\bar{X}^{i,\epsilon}_t,\bar{Y}^{i,\epsilon}_t,\bar{\mu}^{\epsilon,N}_t)-\tau_2(\bar{X}^{i,\epsilon}_t,\bar{Y}^{i,\epsilon}_t,\mc{L}(\bar{X}^\epsilon_t))|^2 \\
&+ \frac{1}{2\eta}\biggl|g(\tilde{X}^{i,\epsilon,N}_t,\tilde{Y}^{i,\epsilon,N}_t,\tilde{\mu}^{\epsilon,N}_t) - g(\bar{X}^{i,\epsilon}_t,\bar{Y}^{i,\epsilon}_t,\mc{L}(\bar{X}^\epsilon_t))\biggr|^2+\frac{\eta}{2\epsilon^2}|\tilde{Y}^{i,\epsilon,N}_t-\bar{Y}^{i,\epsilon}_t|^2 \\
&+\frac{2}{\eta N a^2(N)}|\tau_2(\tilde{X}^{i,\epsilon,N}_t,\tilde{Y}^{i,\epsilon,N}_t,\tilde{\mu}^{\epsilon,N}_t)|^2|\tilde{u}^{N,2}_i(t)|^2\\
&+ \frac{\eta}{2\epsilon^2}|\tilde{Y}^{i,\epsilon,N}_t-\bar{Y}^{i,\epsilon}_t)|^2\biggr]
\end{align*}
for all $\eta>0$, where in the second inequality we used Equation \eqref{eq:rocknertyperetractiontomean} of Assumption \ref{assumption:retractiontomean}. Taking $\eta = \beta/8$ and using the boundedness of $g$ from Assumption \ref{assumption:gsigmabounded} and of $\tau_1,\tau_2$ from Assumption \ref{assumption:uniformellipticity}, we get:
\begin{align*}
\frac{d}{dt}\E\biggl[|\tilde{Y}^{i,\epsilon,N}_t-\bar{Y}^{i,\epsilon}_t|^2\biggr] &\leq -\frac{\beta}{2\epsilon^2}\E\biggl[|\tilde{Y}^{i,\epsilon,N}_t-\bar{Y}^{i,\epsilon}_t|^2\biggr] +\frac{C}{\epsilon^2}\E\biggl[\biggl|f(\tilde{X}^{i,\epsilon,N}_t,\bar{Y}^{i,\epsilon}_t,\tilde{\mu}^{\epsilon,N}_t) - f(\bar{X}^{i,\epsilon}_t,\bar{Y}^{i,\epsilon}_t,\bar{\mu}^{\epsilon,N}_t)\biggr|^2\\
&+\biggl|f(\bar{X}^{i,\epsilon}_t,\bar{Y}^{i,\epsilon}_t,\bar{\mu}^{\epsilon,N}_t) - f(\bar{X}^{i,\epsilon}_t,\bar{Y}^{i,\epsilon}_t,\mc{L}(\bar{X}^\epsilon_t))\biggr|^2\\
&+|\tau_2(\tilde{X}^{i,\epsilon,N}_t,\bar{Y}^{i,\epsilon}_t,\tilde{\mu}^{\epsilon,N}_t)-\tau_2(\bar{X}^{i,\epsilon}_t,\bar{Y}^{i,\epsilon}_t,\bar{\mu}^{\epsilon,N}_t)|^2 \\
&+|\tau_2(\bar{X}^{i,\epsilon}_t,\bar{Y}^{i,\epsilon}_t,\bar{\mu}^{\epsilon,N}_t)-\tau_2(\bar{X}^{i,\epsilon}_t,\bar{Y}^{i,\epsilon}_t,\mc{L}(\bar{X}^\epsilon_t))|^2  \biggr] \\
&+ \frac{C}{Na^2(N)}\E\biggl[|\tilde{u}^{N,2}_i(s)|^2\biggr] + C
\end{align*}

 Now using the global Lipschitz property of $f$ from Assumption \ref{assumption:retractiontomean} and of $\tau_1$ and $\tau_2$ from Assumption \ref{assumption:uniformLipschitzxmu} to handle the terms of the form $|f(\tilde{X}^{i,\epsilon,N}_t,\bar{Y}^{i,\epsilon}_t,\tilde{\mu}^{\epsilon,N}_t) - f(\bar{X}^{i,\epsilon}_t,\bar{Y}^{i,\epsilon}_t,\bar{\mu}^{\epsilon,N}_t)|^2$  and Assumption \ref{assumption:2unifboundedlinearfunctionalderivatives} with Lemma \ref{lemma:rocknersecondlinfunctderimplication} for the terms of the form
$|f(\bar{X}^{i,\epsilon}_t,\bar{Y}^{i,\epsilon}_t,\bar{\mu}^{\epsilon,N}_t) - f(\bar{X}^{i,\epsilon}_t,\bar{Y}^{i,\epsilon}_t,\mc{L}(\bar{X}^\epsilon_t))|^2,$ we have:
\begin{align*}
\frac{d}{dt}\E\biggl[|\tilde{Y}^{i,\epsilon,N}_t-\bar{Y}^{i,\epsilon}_t|^2\biggr] &\leq -\frac{\beta}{2\epsilon^2}\E\biggl[|\tilde{Y}^{i,\epsilon,N}_t-\bar{Y}^{i,\epsilon}_t|^2\biggr] +\frac{C}{\epsilon^2}\E\biggl[\biggl|\tilde{X}^{i,\epsilon,N}_t-\bar{X}^{i,\epsilon}_t\biggr|^2 + \frac{1}{N}\sum_{i=1}^N\biggl|\tilde{X}^{i,\epsilon,N}_t-\bar{X}^{i,\epsilon}_t\biggr|^2 \biggr]\\
&+\frac{C}{\epsilon^2 N}+ \frac{C}{Na^2(N)}\E\biggl[|\tilde{u}^{N,2}_i(s)|^2\biggr] + C.
\end{align*}
When using Lipschitz continuity of $f,\tau_1,$ and $\tau_2$, we are also using that
\begin{align*}
\bb{W}_2(\tilde{\mu}^{\epsilon,N}_t,\bar{\mu}^{\epsilon,N}_t)\leq \frac{1}{N}\sum_{i=1}^N \biggl|\tilde{X}^{i,\epsilon,N}_t-\bar{X}^{i,\epsilon}_t\biggr|^2
\end{align*}
by Equation \eqref{eq:empiricalmeasurewassersteindistance} in Appendix \ref{Appendix:LionsDifferentiation}.

Now using a comparison theorem, dividing by $\frac{1}{N}$ and summing from $i=1,...,N$, we get
\begin{align*}
&\frac{1}{N}\sum_{i=1}^N\E\biggl[|\tilde{Y}^{i,\epsilon,N}_t-\bar{Y}^{i,\epsilon}_t|^2\biggr]\leq C\biggl\lbrace e^{-\frac{\beta}{2\epsilon^2}t}\int_0^t e^{\frac{\beta}{2\epsilon^2}s} ds +\frac{1}{\epsilon^2}e^{-\frac{\beta}{2\epsilon^2}t}\int_0^t\E\biggl[ \frac{1}{N}\sum_{i=1}^N\biggl|\tilde{X}^{i,\epsilon,N}_s-\bar{X}^{i,\epsilon}_s\biggr|^2\biggr] e^{\frac{\beta}{2\epsilon^2}s} ds \\
&\quad +\frac{1}{N\epsilon^2}e^{-\frac{\beta}{2\epsilon^2}t}\int_0^t e^{\frac{\beta}{2\epsilon^2}s} ds +\frac{1}{Na^2(N)}e^{-\frac{\beta}{2\epsilon^2}t}\int_0^t\E\biggl[
\frac{1}{N}\sum_{i=1}^N|\tilde{u}^{N,2}_i(s)|^2\biggr] e^{\frac{\beta}{2\epsilon^2}s} ds    \biggr\rbrace \\
&\leq C\biggl\lbrace \epsilon^2 + \sup_{s\in [0,t]}\E\biggl[ \frac{1}{N}\sum_{i=1}^N\biggl|\tilde{X}^{i,\epsilon,N}_s-\bar{X}^{i,\epsilon}_s\biggr|^2\biggr]+ \frac{1}{N} + \frac{1}{Na^2(N)}\int_0^T\E\biggl[
\frac{1}{N}\sum_{i=1}^N|\tilde{u}^{N,2}_i(s)|^2\biggr] ds \biggr\rbrace
\end{align*}
and by the bound \eqref{eq:controlassumptions0}, we get:
\begin{align*}
\frac{1}{N}\sum_{i=1}^N\E\biggl[|\tilde{Y}^{i,\epsilon,N}_t-\bar{Y}^{i,\epsilon}_t|^2\biggr]&\leq C\biggl\lbrace\epsilon^2+ \frac{1}{N}+\frac{1}{Na^2(N)} + \sup_{s\in [0,t]}\E\biggl[ \frac{1}{N}\sum_{i=1}^N\biggl|\tilde{X}^{i,\epsilon,N}_s-\bar{X}^{i,\epsilon}_s\biggr|^2\biggr]\biggr\rbrace.
\end{align*}
\end{proof}

\begin{lemma}\label{lemma:XbartildeXdifference}
Under assumptions \ref{assumption:uniformellipticity}-\ref{assumption:uniformLipschitzxmu} and \ref{assumption:2unifboundedlinearfunctionalderivatives} we have
\begin{align*}
\sup_{s\in [0,T]}\E\biggl[ \frac{1}{N}\sum_{i=1}^N\biggl|\tilde{X}^{i,\epsilon,N}_s-\bar{X}^{i,\epsilon}_s\biggr|^2\biggr]&\leq C(T)[\epsilon^2+ \frac{1}{N}+\frac{1}{Na^2(N)}]
\end{align*}
\end{lemma}
\begin{proof}
Letting $(\bar{i})$ denote the argument $(\bar{X}^{i,\epsilon}_s,\bar{Y}^{i,\epsilon}_s,\mc{L}(\bar{X}^\epsilon_s))$, $(\tilde{i})$ denote the argument $(\tilde{X}^{i,\epsilon,N}_s,\tilde{Y}^{i,\epsilon,N}_s,\tilde{\mu}^{\epsilon,N}_s)$:
\begin{align*}
&\tilde{X}^{i,\epsilon,N}_t-\bar{X}^{i,\epsilon}_t = \frac{1}{\epsilon}\int_0^t\left( b(\tilde{i})-b(\bar{i}) \right)ds + \int_0^t\left( c(\tilde{i}) - c(\bar{i}) \right)ds +\int_0^t \sigma(\tilde{i})\frac{\tilde{u}^{N,1}_i(s)}{a(N)\sqrt{N}} ds + \int_0^t\left( \sigma(\tilde{i}) - \sigma(\bar{i}) \right)dW^i_s \\
& = \int_0^t\left( \gamma(\tilde{i}) - \gamma(\bar{i})\right)ds + \int_0^t\left( [\sigma(\tilde{i}) + \tau_1(\tilde{i})\Phi_y(\tilde{i})] - [\sigma(\bar{i}) + \tau_1(\bar{i})\Phi_y(\bar{i})]\right) dW^i_s+\int_0^t\left( \tau_2(\tilde{i})\Phi_y(\tilde{i})-\tau_2(\bar{i})\Phi_y(\bar{i})\right)dB^i_s \\
&+ R^{i,\epsilon,N}_1(t)-R^{i,\epsilon,N}_2(t)+R^{i,\epsilon,N}_3(t)-R^{i,\epsilon,N}_4(t)+R^{i,\epsilon,N}_5(t),
\end{align*}
where here we recall $\Phi$ from Equation \eqref{eq:cellproblemold} and $\gamma,\gamma_1$ from Equation \eqref{eq:limitingcoefficients}, and:
\begin{align*}
R^{i,\epsilon,N}_1(t)& = \frac{1}{\epsilon}\int_0^t b(\tilde{i})ds - \int_0^t\left(  \gamma_1(\tilde{i})+[\frac{\tau_1(\tilde{i})}{a(N)\sqrt{N}}\tilde{u}^{N,1}_i(s)+\frac{\tau_2(\tilde{i})}{a(N)\sqrt{N}}\tilde{u}^{N,2}_i(s)]\Phi_y(\tilde{i})\right)ds\\
&-\int_0^t\tau_1(\tilde{i})\Phi_y(\tilde{i})dW_s^i-\int_0^t \tau_2(\tilde{i})\Phi_y(\tilde{i})dB_s^i-\int_0^t \frac{1}{N}\sum_{j=1}^N b(\tilde{X}^{j,\epsilon,N}_s,\tilde{Y}^{j,\epsilon,N}_s,\tilde{\mu}^{\epsilon,N}_s)\partial_{\mu}\Phi(\tilde{i})[\tilde{X}^{j,\epsilon,N}_s]ds\\
R^{i,\epsilon,N}_2(t)& =\frac{1}{\epsilon}\int_0^t b(\bar{i}) ds -\int_0^t \gamma_1(\bar{i})ds-\int_0^t\tau_1(\bar{i})\Phi_y(\bar{i})dW^i_s - \int_0^t \tau_2(\bar{i})\Phi_y(\bar{i})dB^i_s\\
&-\int_0^t \int_{\R^2}b(x,y,\mc{L}(\bar{X}^\epsilon_s))\partial_\mu \Phi(\bar{X}^{i,\epsilon}_s,\bar{Y}^{i,\epsilon}_s,\mc{L}(\bar{X}^\epsilon_s))[x]\mc{L}(\bar{X}^\epsilon_s,Y^\epsilon_s)(dx,dy)ds \\
R^{i,\epsilon,N}_3(t)& = \int_0^t \frac{1}{N}\sum_{j=1}^N b(\tilde{X}^{j,\epsilon,N}_s,\tilde{Y}^{j,\epsilon,N}_s,\tilde{\mu}^{\epsilon,N}_s)\partial_{\mu}\Phi(\tilde{i})[\tilde{X}^{j,\epsilon,N}_s]ds\\
R^{i,\epsilon,N}_4(t)& =\int_0^t \int_{\R^2}b(x,y,\mc{L}(\bar{X}^\epsilon_s))\partial_\mu \Phi(\bar{X}^{i,\epsilon}_s,\bar{Y}^{i,\epsilon}_s,\mc{L}(\bar{X}^\epsilon_s))[x]\mc{L}(\bar{X}^\epsilon_s,Y^\epsilon_s)(dx,dy)ds\\
R^{i,\epsilon,N}_5(t)& =\int_0^t \left(\sigma(\tilde{i})\frac{\tilde{u}^{N,1}_i(s)}{a(N)\sqrt{N}} +[\frac{\tau_1(\tilde{i})}{a(N)\sqrt{N}}\tilde{u}^{N,1}_i(s)+\frac{\tau_2(\tilde{i})}{a(N)\sqrt{N}}\tilde{u}^{N,2}_i(s)]\Phi_y(\tilde{i})\right) ds
\end{align*}
By Proposition \ref{prop:fluctuationestimateparticles1} with $\psi\equiv 1$, we have
\begin{align*}
\frac{1}{N}\sum_{i=1}^N \E\biggl[|R^{i,\epsilon,N}_1(t)|^2\biggr]&\leq C(T)[\epsilon^2+\frac{1}{N}],
\end{align*}
by Proposition 4.2 of \cite{BezemekSpiliopoulosAveraging2022} with $\psi\equiv 1$, we have
\begin{align*}
\frac{1}{N}\sum_{i=1}^N \E\biggl[|R^{i,\epsilon,N}_2(t)|^2\biggr]&\leq C(T)\epsilon^2,
\end{align*}
by Proposition \ref{prop:purpleterm1} with $\psi\equiv 1$, we have
\begin{align*}
\frac{1}{N}\sum_{i=1}^N \E\biggl[|R^{i,\epsilon,N}_3(t)|^2\biggr]&\leq C(T)[\epsilon^2+\frac{1}{N^2}],
\end{align*}
and by Proposition 4.3 of \cite{BezemekSpiliopoulosAveraging2022} with $\psi\equiv 1$, we have
\begin{align*}
\frac{1}{N}\sum_{i=1}^N \E\biggl[|R^{i,\epsilon,N}_4(t)|^2\biggr]&\leq C(T)\epsilon^2.
\end{align*}

Here we are using that, under the assumed regularity of $\Phi$ and $\chi$ imposed by Assumptions \ref{assumption:multipliedpolynomialgrowth} and \ref{assumption:qF2bound} respectively along with the result of Lemma \ref{lemma:barYuniformbound}, the norm can brought inside the expectation in Propositions 4.2 and 4.3 of \cite{BezemekSpiliopoulosAveraging2022} with little modification to the proofs.
Finally,  since under Assumption \ref{assumption:multipliedpolynomialgrowth} $\Phi_y$ is bounded:
\begin{align*}
\frac{1}{N}\sum_{i=1}^N\E\biggl[|R^{i,\epsilon,N}_5(t)|^2\biggr]& \leq C \frac{1}{a^2(N)N} \frac{1}{N}\sum_{i=1}^N\E\biggl[\int_0^T |\tilde{u}^{N,1}_i(s)|^2+|\tilde{u}^{N,2}_i(s)|^2 ds \biggr]
\leq C \frac{1}{a^2(N)N}
\end{align*}
by the assumed bound on the controls  \eqref{eq:controlassumptions0}. Now we see that, by It\^o Isometry:
\begin{align*}
&\frac{1}{N}\sum_{i=1}^N\E\biggl[|\tilde{X}^{i,\epsilon,N}_t-\bar{X}^{i,\epsilon}_t|^2\biggr] \leq \frac{1}{N}\sum_{i=1}^N\E\biggl[\biggl|\int_0^t\left(\gamma(\tilde{X}^{i,\epsilon,N}_s,\tilde{Y}^{i,\epsilon,N}_s,\tilde{\mu}^{\epsilon,N}_s)-\gamma(\bar{X}^{i,\epsilon}_s,\bar{Y}^{i,\epsilon}_s,\mc{L}(\bar{X}^\epsilon_s))\right)ds\biggr|^2\biggr] \\
&+ \frac{1}{N}\sum_{i=1}^N\E\biggl[\int_0^t\biggl|[\sigma+\tau_1\Phi_y](\tilde{X}^{i,\epsilon,N}_s,\tilde{Y}^{i,\epsilon,N}_s,\tilde{\mu}^{\epsilon,N}_s) - [\sigma+\tau_1\Phi_y](\bar{X}^{i,\epsilon}_s,\bar{Y}^{i,\epsilon}_s,\mc{L}(\bar{X}^\epsilon_s))\biggr|^2  ds\biggr]\\
&+ \frac{1}{N}\sum_{i=1}^N\E\biggl[\int_0^t\biggl|[\tau_2\Phi_y](\tilde{X}^{i,\epsilon,N}_s,\tilde{Y}^{i,\epsilon,N}_s,\tilde{\mu}^{\epsilon,N}_s) - [\tau_2\Phi_y](\bar{X}^{i,\epsilon}_s,\bar{Y}^{i,\epsilon}_s,\mc{L}(\bar{X}^\epsilon_s))\biggr|^2  ds\biggr] + R^N(t)
\end{align*}
where $R^N(t)\leq C(T)[\epsilon^2 + \frac{1}{N}]$. We handle the terms from the martingales first.
\begin{align*}
&\frac{1}{N}\sum_{i=1}^N\E\biggl[\int_0^t\biggl|[\sigma+\tau_1\Phi_y](\tilde{X}^{i,\epsilon,N}_s,\tilde{Y}^{i,\epsilon,N}_s,\tilde{\mu}^{\epsilon,N}_s) - [\sigma+\tau_1\Phi_y](\bar{X}^{i,\epsilon}_s,\bar{Y}^{i,\epsilon}_s,\mc{L}(\bar{X}^\epsilon_s))\biggr|^2  ds\biggr]\\
&\leq \frac{C}{N}\sum_{i=1}^N\biggl\lbrace \E\biggl[\int_0^t\biggl|[\sigma+\tau_1\Phi_y](\tilde{X}^{i,\epsilon,N}_s,\tilde{Y}^{i,\epsilon,N}_s,\tilde{\mu}^{\epsilon,N}_s) - [\sigma+\tau_1\Phi_y](\bar{X}^{i,\epsilon}_s,\bar{Y}^{i,\epsilon}_s,\bar{\mu}^{\epsilon,N})\biggr|^2  ds\biggr] \\
&+\E\biggl[\int_0^t\biggl|[\sigma+\tau_1\Phi_y](\bar{X}^{i,\epsilon}_s,\bar{Y}^{i,\epsilon}_s,\bar{\mu}^{\epsilon,N}) - [\sigma+\tau_1\Phi_y](\bar{X}^{i,\epsilon}_s,\bar{Y}^{i,\epsilon}_s,\mc{L}(\bar{X}^\epsilon_s))\biggr|^2  ds\biggr]\biggr\rbrace \\
&\leq\frac{C}{N}\sum_{i=1}^N\E\biggl[\int_0^t|\tilde{X}^{i,\epsilon,N}_s-\bar{X}^{i,\epsilon}_s|^2+|\tilde{Y}^{i,\epsilon,N}_s-\bar{Y}^{i,\epsilon}_s|^2 ds\biggr]+\frac{C}{N}\\
&\text{ by Lipschitz continuity from Assumption \ref{assumption:uniformLipschitzxmu} and Assumption \ref{assumption:2unifboundedlinearfunctionalderivatives} with Lemma \ref{lemma:rocknersecondlinfunctderimplication}}\\
&\leq C\biggl\lbrace\epsilon^2+ \frac{1}{N}+\frac{1}{Na^2(N)} + \int_0^t\sup_{\tau \in [0,s]}\E\biggl[ \frac{1}{N}\sum_{i=1}^N\biggl|\tilde{X}^{i,\epsilon,N}_\tau-\bar{X}^{i,\epsilon}_\tau\biggr|^2\biggr]ds\biggr\rbrace\\
&\text{ by Lemma \ref{lemma:tildeYbarYdifference} }.
\end{align*}
The exact same proof and bound applies to
\begin{align*}
\frac{1}{N}\sum_{i=1}^N\E\biggl[\int_0^t\biggl|[\tau_2\Phi_y](\tilde{X}^{i,\epsilon,N}_s,\tilde{Y}^{i,\epsilon,N}_s,\tilde{\mu}^{\epsilon,N}_s) - [\tau_2\Phi_y](\bar{X}^{i,\epsilon}_s,\bar{Y}^{i,\epsilon}_s,\mc{L}(\bar{X}^\epsilon_s))\biggr|^2  ds\biggr].
\end{align*}

In a similar manner:
\begin{align*}
&\frac{1}{N}\sum_{i=1}^N\E\biggl[\biggl|\int_0^t\left(\gamma(\tilde{X}^{i,\epsilon,N}_s,\tilde{Y}^{i,\epsilon,N}_s,\tilde{\mu}^{\epsilon,N}_s)-\gamma(\bar{X}^{i,\epsilon}_s,\bar{Y}^{i,\epsilon}_s,\mc{L}(\bar{X}^\epsilon_s))\right)ds\biggr|^2\biggr]\\
&\leq \frac{C(T)}{N}\sum_{i=1}^N\biggl\lbrace \E\biggl[\int_0^t\biggl|\gamma(\tilde{X}^{i,\epsilon,N}_s,\tilde{Y}^{i,\epsilon,N}_s,\tilde{\mu}^{\epsilon,N}_s) - \gamma(\bar{X}^{i,\epsilon}_s,\bar{Y}^{i,\epsilon}_s,\bar{\mu}^{\epsilon,N})\biggr|^2  ds\biggr] \\
&\quad+\E\biggl[\int_0^t\biggl|\gamma(\bar{X}^{i,\epsilon}_s,\bar{Y}^{i,\epsilon}_s,\bar{\mu}^{\epsilon,N}) - \gamma(\bar{X}^{i,\epsilon}_s,\bar{Y}^{i,\epsilon}_s,\mc{L}(\bar{X}^\epsilon_s))\biggr|^2  ds\biggr]\biggr\rbrace \\
&\leq\frac{C(T)}{N}\sum_{i=1}^N\E\biggl[\int_0^t|\tilde{X}^{i,\epsilon,N}_s-\bar{X}^{i,\epsilon}_s|^2+|\tilde{Y}^{i,\epsilon,N}_s-\bar{Y}^{i,\epsilon}_s|^2 ds\biggr]+\frac{C(T)}{N}\\
&\leq C(T)\biggl\lbrace\epsilon^2+ \frac{1}{N}+\frac{1}{Na^2(N)} + \int_0^t\sup_{\tau \in [0,s]}\E\biggl[ \frac{1}{N}\sum_{i=1}^N\biggl|\tilde{X}^{i,\epsilon,N}_\tau-\bar{X}^{i,\epsilon}_\tau\biggr|^2\biggr]ds\biggr\rbrace.
\end{align*}

Collecting these bounds, we have for all $p\in [0,T]$:
\begin{align*}
\frac{1}{N}\sum_{i=1}^N\E\biggl[|\tilde{X}^{i,\epsilon,N}_p-\bar{X}^{i,\epsilon}_p|^2\biggr]&\leq C_1(T)\int_0^p\sup_{\tau \in [0,s]}\E\biggl[ \frac{1}{N}\sum_{i=1}^N\biggl|\tilde{X}^{i,\epsilon,N}_\tau-\bar{X}^{i,\epsilon}_\tau\biggr|^2\biggr]ds + C_2(T)[\epsilon^2+ \frac{1}{N}+\frac{1}{Na^2(N)}]
\end{align*}
so
\begin{align*}
\sup_{p\in [0,t]}\frac{1}{N}\sum_{i=1}^N\E\biggl[|\tilde{X}^{i,\epsilon,N}_p-\bar{X}^{i,\epsilon}_p|^2\biggr]&\leq C(T)\left[\int_0^t\sup_{\tau \in [0,s]}\E\biggl[ \frac{1}{N}\sum_{i=1}^N\biggl|\tilde{X}^{i,\epsilon,N}_\tau-\bar{X}^{i,\epsilon}_\tau\biggr|^2\biggr]ds + [\epsilon^2+ \frac{1}{N}+\frac{1}{Na^2(N)}]\right]
\end{align*}
and by Gronwall's inequality:
\begin{align*}
\sup_{p\in [0,T]}\frac{1}{N}\sum_{i=1}^N\E\biggl[|\tilde{X}^{i,\epsilon,N}_p-\bar{X}^{i,\epsilon}_p|^2\biggr]&\leq C(T)[\epsilon^2+ \frac{1}{N}+\frac{1}{Na^2(N)}].
\end{align*}
\end{proof}
\subsection{Tightness of $\tilde{Z}^N$}\label{subsec:tightnesstildeZN}
We now have the tools to prove tightness of $\br{\tilde{Z}^N}$ from Equation \eqref{eq:controlledempmeasure}. We first prove a uniform-in-time bound on $\langle \tilde{Z}^N_t,\phi\rangle$ in terms of $|\cdot|_4$ (recall Equation \eqref{eq:boundedderivativesseminorm}) in Lemma \ref{lemma:Zboundbyphi4}. Then, using the results from Section \ref{sec:ergodictheoremscontrolledsystem}, we provide a prelimit representation for $\tilde{Z}^N$ which is a priori $\mc{O}(1)$ in $\epsilon , N$ in Lemma \ref{lemma:Lnu1nu2representation}. Combining these two lemmas, we are then able to establish tightness via the methods of \cite{Mitoma} in Proposition \ref{prop:tildeZNtightness}.
\begin{lemma}\label{lemma:Zboundbyphi4}
Under Assumptions \ref{assumption:uniformellipticity}-  \ref{assumption:2unifboundedlinearfunctionalderivatives}, there exists $C$ independent of $N$ such that for all $\phi\in C_{b,L}^4(\R)$
\begin{align*}
\sup_{N\in\bb{N}}\sup_{t\in [0,T]}\E\biggl[|\langle \tilde{Z}^N_t,\phi\rangle |^2 \biggr]&\leq C(T)|\phi|^2_4.
\end{align*}
\end{lemma}
\begin{proof}
Let $\bar{\mu}^{\epsilon,N}_t$ be as in Equation \eqref{eq:IIDempiricalmeasure}. Then, by triangle inequality
\begin{align*}
\E\biggl[|\langle \tilde{Z}^N_t,\phi\rangle |^2 \biggr] 
&\leq 2a^2(N)N\E\biggl[|\langle \tilde{\mu}^{\epsilon,N}_t,\phi\rangle -  \langle \bar{\mu}^{\epsilon,N}_t,\phi\rangle|^2 \biggr] + 2a^2(N)N\E\biggl[|\langle \bar{\mu}^{\epsilon,N}_t,\phi\rangle -  \langle \mc{L}(X_t),\phi\rangle|^2 \biggr].
\end{align*}

For the first term:
\begin{align*}
a^2(N)N\E\biggl[|\langle \tilde{\mu}^{\epsilon,N}_t,\phi\rangle -  \langle \bar{\mu}^{\epsilon,N}_t,\phi\rangle|^2 \biggr] & = a^2(N)N\E\biggl[\biggl|\frac{1}{N}\sum_{i=1}^N \phi(\tilde{X}^{i,\epsilon,N}_t)-\phi(\bar{X}^{i,\epsilon}_t)\biggr|^2 \biggr]\\
&\leq a^2(N)N\frac{1}{N}\sum_{i=1}^N \E\biggl[\biggl|\phi(\tilde{X}^{i,\epsilon,N}_t)-\phi(\bar{X}^{i,\epsilon}_t)\biggr|^2 \biggr]\\
&\leq  a^2(N)N\frac{1}{N}\sum_{i=1}^N  \E\biggl[\biggl|\tilde{X}^{i,\epsilon,N}_t-\bar{X}^{i,\epsilon}_t\biggr|^2 \biggr]\norm{\phi'}^2_\infty\\
&\leq C(T)[\epsilon^2 a^2(N)N+a^2(N)+1]\norm{\phi'}^2_\infty\\
&\leq C(T)\norm{\phi'}^2_\infty,
\end{align*}
where in the second to last inequality we used Lemma \ref{lemma:XbartildeXdifference}.

For the second term, we use that $\br{\bar{X}^{i,\epsilon}}_{i\in\bb{N}}$ are IID to see:
\begin{align*}
a^2(N)N\E\biggl[|\langle \bar{\mu}^{\epsilon,N}_t,\phi\rangle -  \langle \mc{L}(X_t),\phi\rangle|^2 \biggr] & = a^2(N)N\E\biggl[\biggl|\frac{1}{N}\sum_{i=1}^N \phi(\bar{X}^{i,\epsilon}_t)  - \E[\phi(X_t)]\biggr|^2 \biggr]\\
& = a^2(N)\biggl\lbrace (N-1)(\E[\phi(\bar{X}^{\epsilon}_t] - \E[\phi(X_t)])^2 \\
&+ \E\biggl[\biggl|\phi(\bar{X}^{\epsilon}_t) - \E[\phi(X_t)]\biggr|^2\biggr] \biggr\rbrace \\
& \leq a^2(N)\biggl\lbrace N(\E[\phi(\bar{X}^{\epsilon}_t)] - \E[\phi(X_t)])^2 + 4\norm{\phi}_\infty^2 \biggr\rbrace \\
&\leq a^2(N)N \epsilon^2 C(T)|\phi|^2_4+4a^2(N)\norm{\phi}^2_\infty,
\end{align*}
where in the last inequality we used Theorem \ref{theo:mckeanvlasovaveraging}. This bound vanishes as $N\toinf$.
\end{proof}

\begin{lemma}\label{lemma:Lnu1nu2representation}
Assume \ref{assumption:uniformellipticity}-\ref{assumption:forcorrectorproblem}, \ref{assumption:limitingcoefficientsLionsDerivatives}, and \ref{assumption:limitingcoefficientsregularity}.
Define $\bar{L}_{\nu_1,\nu_2}$ to be the operator parameterized by $\nu_1,\nu_2\in\mc{P}_2(\R)$ which acts on $\phi\in C^2_b(\R)$ by
\begin{align}\label{eq:Lnu1nu2}
\bar{L}_{\nu_1,\nu_2}\phi(x) & = \bar{\gamma}(x,\nu_2)\phi'(x)+\bar{D}(x,\nu_2)\phi''(x)\\
&+\int_\R \int_0^1 \frac{\delta}{\delta m}\bar{\gamma}(z,(1-r)\nu_1+r\nu_2)[x]\phi'(z)+ \frac{\delta}{\delta m}\bar{D}(z,(1-r)\nu_1+r\nu_2)[x]\phi''(z) dr\nu_1(dz).\nonumber
\end{align}
Here we recall $\bar{\gamma},\bar{D}$ from Equation \eqref{eq:averagedlimitingcoefficients}, the linear functional derivative from Definition \ref{def:LinearFunctionalDerivative}, $\Phi$ from Equation \eqref{eq:cellproblemold}, and the occupation measures $Q^N$ from Equation \eqref{eq:occupationmeasures}.
Then we have the representation: for $\phi\in C^\infty_c(\R)$ and $t\in[0,T]$:
\begin{align*}
\langle \tilde{Z}^N_t,\phi\rangle &=\int_0^t \langle \tilde{Z}^N_s,\bar{L}_{\mc{L}(X_s),\tilde{\mu}^{\epsilon,N}_s}\phi(\cdot)\rangle ds+\int_{\R\times\R\times\R^2\times[0,t]} \sigma(x,y,\tilde{\mu}^{\epsilon,N}_s)z_1 \phi'(x)Q^N(dx,dy,dz,ds)\\
&+\int_{\R\times\R\times\R^2\times[0,t]} [\tau_1(x,y,\tilde{\mu}^{\epsilon,N}_s)z_1+\tau_2(x,y,\tilde{\mu}^{\epsilon,N}_s)z_2]\Phi_y(x,y,\tilde{\mu}^{\epsilon,N}_s)\phi'(x)Q^N(dx,dy,dz,ds) +R^N_t(\phi)
\end{align*}
where
\begin{align*}
\E\biggl[\sup_{t\in[0,T]}\biggl| R^N_t(\phi)\biggr|\biggr]\leq \bar{R}(N,T)|\phi|_4,
\end{align*}
$\bar{R}(N,T)\tto 0$ as $N\toinf$, and $\bar{R}(N,T)$ is independent of $\phi$.
\end{lemma}
\begin{proof}
Recall $\tilde{\mu}^{N,\epsilon}$ from Equation \eqref{eq:controlledempmeasure}, $X_t$ from Equation \eqref{eq:LLNlimitold}, and that $\tilde{Z}^N=a(N)\sqrt{N}[\tilde{\mu}^{N,\epsilon}-\mc{L}(X)].$

By It\^o's formula,
\begin{align*}
\langle \tilde{\mu}^{\epsilon,N}_t,\phi\rangle & = \phi(x) + \frac{1}{N}\sum_{i=1}^N \int_0^t\left( \frac{1}{\epsilon}b(i)\phi'(\tilde{X}^{i,\epsilon,N}_s) + c(i)\phi'(\tilde{X}^{i,\epsilon,N}_s) + \frac{1}{2}\sigma^2(i)\phi''(\tilde{X}^{i,\epsilon,N}_s) + \sigma(i)\frac{\tilde{u}^{N,1}_i(s)}{a(N)\sqrt{N}}\phi'(\tilde{X}^{i,\epsilon,N}_s)\right)ds \\
&\quad+ \int_0^t \sigma(i)\phi'(\tilde{X}^{i,\epsilon,N}_s)dW^i_s
\end{align*}
where $(i)$ denotes the argument $(\tilde{X}^{i,\epsilon,N}_s,\tilde{Y}^{i,\epsilon,N}_s,\tilde{\mu}^{\epsilon,N}_s)$ and
\begin{align*}
\langle \mc{L}(X_s),\phi\rangle & = \phi(x)+\E\biggl[\int_0^t \left(\bar{\gamma}(X_s,\mc{L}(X_s))\phi'(X_s)+ \bar{D}(X_s,\mc{L}(X_s))\phi''(X_s)\right)ds + \int_0^t\sqrt{2}\bar{D}^{1/2}(X_s,\mc{L}(X_s)\phi'(X_s)dW_s \biggr]\\
& = \phi(x)+\E\biggl[\int_0^t\left(\bar{\gamma}(X_s,\mc{L}(X_s))\phi'(X_s)+ \bar{D}(X_s,\mc{L}(X_s))\phi''(X_s)\right)ds\biggr]
\end{align*}
since $\bar{D}^{1/2}$ is bounded as per Assumption \ref{assumption:limitingcoefficientsLionsDerivatives}. Then
\begin{align*}
&\langle \tilde{Z}^N_t,\phi\rangle = \frac{a(N)}{\sqrt{N}}\sum_{i=1}^N\biggl\lbrace\int_0^t \left(\frac{1}{\epsilon}b(i)\phi'(\tilde{X}^{i,\epsilon,N}_s) + c(i)\phi'(\tilde{X}^{i,\epsilon,N}_s) + \frac{1}{2}\sigma^2(i)\phi''(\tilde{X}^{i,\epsilon,N}_s) + \sigma(i)\frac{\tilde{u}^{N,1}_i(s)}{a(N)\sqrt{N}}\phi'(\tilde{X}^{i,\epsilon,N}_s)\right)ds \\
&+ \int_0^t \sigma(i)\phi'(\tilde{X}^{i,\epsilon,N}_s)dW^i_s-\E\biggl[\int_0^t \left(\bar{\gamma}(X_s,\mc{L}(X_s))\phi'(X_s)+ \bar{D}(X_s,\mc{L}(X_s))\phi''(X_s)\right)ds\biggr] \biggr\rbrace \\
& = \frac{a(N)}{\sqrt{N}}\sum_{i=1}^N\biggl\lbrace \int_0^t\left(\gamma(i)\phi'(\tilde{X}^{i,\epsilon,N}_s) -\E\biggl[ \bar{\gamma}(X_s,\mc{L}(X_s))\phi'(X_s)\biggr]\right)ds + \int_0^t\left( D(i)\phi''(\tilde{X}^{i,\epsilon,N}_s) -\E\biggl[ \bar{D}(X_s,\mc{L}(X_s))\phi''(X_s)\biggr]\right)ds \\
&+ R^i_1(t)+R^i_2(t)+M^i(t)\biggr\rbrace+\int_{\R\times\R\times\R^2\times[0,t]} \sigma(x,y,\tilde{\mu}^{\epsilon,N}_s)z_1 \phi'(x)Q^N(dx,dy,dz,ds)\\
&+\int_{\R\times\R\times\R^2\times[0,t]}\left( [\tau_1(x,y,\tilde{\mu}^{\epsilon,N}_s)z_1+\tau_2(x,y,\tilde{\mu}^{\epsilon,N}_s)z_2]\Phi_y(x,y,\tilde{\mu}^{\epsilon,N}_s)\phi'(x)\right)Q^N(dx,dy,dz,ds)
\end{align*}
where here we recall $\gamma_1,\gamma,D,D_1$ from Equation \eqref{eq:limitingcoefficients} and:
\begin{align*}
R^i_1(t) &\coloneqq \int_0^t \biggl(\frac{1}{\epsilon}b(i)\phi'(\tilde{X}^{i,\epsilon,N}_s)- \int_0^t  \gamma_1(i)\phi'(\tilde{X}^{i,\epsilon,N}_s)+D_1(i)\phi''(\tilde{X}^{i,\epsilon,N}_s)] \\
&+[\frac{\tau_1(i)}{a(N)\sqrt{N}}\tilde{u}^{N,1}_i(s)+\frac{\tau_2(i)}{a(N)\sqrt{N}}\tilde{u}^{N,2}_i(s)]\Phi_y(i)\phi'(\tilde{X}^{i,\epsilon,N}_s)\biggr)ds\\
&-\int_0^t\tau_1(i)\Phi_y(i)\phi'(\tilde{X}^{i,\epsilon,N}_s)dW^i_s - \int_0^t \tau_2(i)\Phi_y(i)\phi'(\tilde{X}^{i,\epsilon,N}_s)dB^i_s-\int_0^t \frac{1}{N}\sum_{j=1}^N b(j)\partial_{\mu}\Phi(i)[j]\phi'(\tilde{X}^{i,\epsilon,N}_s)ds\\
R^i_2(t)&\coloneqq \int_0^t \frac{1}{N}\sum_{j=1}^N b(j)\partial_{\mu}\Phi(i)[j]\phi'(\tilde{X}^{i,\epsilon,N}_s)ds \\
M^i(t)&\coloneqq \int_0^t[\tau_1(i)\Phi_y(i)+\sigma(i)]\phi'(\tilde{X}^{i,\epsilon,N}_s)dW^i_s + \int_0^t \tau_2(i)\Phi_y(i)\phi'(\tilde{X}^{i,\epsilon,N}_s)dB^i_s.
\end{align*}

For $R^i_1(t)$, we have via Proposition \ref{prop:fluctuationestimateparticles1} that
\begin{align*}
\E\biggl[\sup_{t\in [0,T]}\frac{a(N)}{\sqrt{N}}\sum_{i=1}^N|R^i_1(t)| \biggr]&\leq C[\epsilon a(N)\sqrt{N}(1+T+T^{1/2})+a(N)T]|\phi|_3.
\end{align*}

For $R^i_2(t)$, we have via Proposition \ref{prop:purpleterm1} that
\begin{align*}
\E\biggl[\sup_{t\in [0,T]}\frac{a(N)}{\sqrt{N}}\sum_{i=1}^N|R^i_2(t)| \biggr]&\leq C[\epsilon a(N)\sqrt{N}(1+T+T^{1/2})+\frac{a(N)}{\sqrt{N}}T]|\phi|_3.
\end{align*}

For $M^i(t)$, we have by Burkholder-Davis-Gundy inequality that
\begin{align*}
\E\biggl[\sup_{t\in[0,T]}\biggl|\frac{a(N)}{\sqrt{N}}\sum_{i=1}^N M^i(t) \biggr|\biggr]&\leq C\frac{a(N)}{\sqrt{N}}\biggl\lbrace\E\biggl[\biggl(\sum_{i=1}^N\int_0^T[\tau_1(i)\Phi_y(i)+\sigma(i)]^2|\phi'(i)|^2 ds\biggr)^{1/2}\biggr] \\
&\qquad+\E\biggl[\biggl(\sum_{i=1}^N\int_0^T |\tau_2(i)\Phi_y(i)|^2|\phi'(i)|^2 ds\biggr)^{1/2}\biggr]  \biggr\rbrace \\
&\leq Ca(N)T^{1/2}|\phi|_1
\end{align*}
since the integrand it bounded by Assumptions \ref{assumption:uniformellipticity},\ref{assumption:gsigmabounded}, and \ref{assumption:multipliedpolynomialgrowth}.
Thus we have
\begin{align*}
&\langle \tilde{Z}^N_t,\phi\rangle  =\nonumber\\
&= \frac{a(N)}{\sqrt{N}}\sum_{i=1}^N\biggl\lbrace \int_0^t\left(\gamma(i)\phi'(\tilde{X}^{i,\epsilon,N}_s) -\E\biggl[ \bar{\gamma}(X_s,\mc{L}(X_s))\phi'(X_s)\biggr]\right)ds \nonumber\\
&+ \int_0^t\left( D(i)\phi''(\tilde{X}^{i,\epsilon,N}_s) -\E\biggl[ \bar{D}(X_s,\mc{L}(X_s))\phi''(X_s)\biggr]\right)ds \biggr\rbrace\\
&+\int_{\R\times\R\times\R^2\times[0,t]} \sigma(x,y,\tilde{\mu}^{\epsilon,N}_s)z_1 \phi'(x)Q^N(dx,dy,dz,ds)\\
&+\int_{\R\times\R\times\R^2\times[0,t]} [\tau_1(x,y,\tilde{\mu}^{\epsilon,N}_s)z_1+\tau_2(x,y,\tilde{\mu}^{\epsilon,N}_s)z_2]\Phi_y(x,y,\tilde{\mu}^{\epsilon,N}_s)\phi'(x)Q^N(dx,dy,dz,ds) +R_t^{3,N}(\phi) \\
\end{align*}
where $\E\biggl[\sup_{t\in[0,T]}|R_t^{3,N}(\phi)|\biggr]\leq C(T)\epsilon a(N)\sqrt{N}\vee a(N)|\phi|_3$. We rewrite this as:
\begin{align*}
\langle \tilde{Z}^N_t,\phi\rangle &= \frac{a(N)}{\sqrt{N}}\sum_{i=1}^N\biggl\lbrace \int_0^t\left(\bar{\gamma}(\tilde{X}^{i,\epsilon,N}_s,\tilde{\mu}^{i,\epsilon,N}_s)\phi'(\tilde{X}^{i,\epsilon,N}_s) -\E\biggl[ \bar{\gamma}(X_s,\mc{L}(X_s))\phi'(X_s)\biggr]\right)ds \\
&+ \int_0^t \left(\bar{D}(\tilde{X}^{i,\epsilon,N}_s,\tilde{\mu}^{i,\epsilon,N}_s)\phi''(\tilde{X}^{i,\epsilon,N}_s) -\E\biggl[ \bar{D}(X_s,\mc{L}(X_s))\phi''(X_s)\biggr]\right)ds +R_4^i(t)+R_5^i(t)\biggr\rbrace\\
&+\int_{\R\times\R\times\R^2\times[0,t]} \sigma(x,y,\tilde{\mu}^{\epsilon,N}_s)z_1 \phi'(x)Q^N(dx,dy,dz,ds)\\
&+\int_{\R\times\R\times\R^2\times[0,t]} [\tau_1(x,y,\tilde{\mu}^{\epsilon,N}_s)z_1+\tau_2(x,y,\tilde{\mu}^{\epsilon,N}_s)z_2]\Phi_y(x,y,\tilde{\mu}^{\epsilon,N}_s)\phi'(x)Q^N(dx,dy,dz,ds) +R_t^{3,N}(\phi)
\end{align*}
where
\begin{align*}
R_4^i(t) & = \int_0^t [\gamma(\tilde{X}^{i,\epsilon,N}_s,\tilde{Y}^{i,\epsilon,N}_s,\tilde{\mu}^{\epsilon,N}_s) - \bar{\gamma}(\tilde{X}^{i,\epsilon,N}_s,\tilde{\mu}^{N,\epsilon}_s)]\phi'(\tilde{X}^{i,\epsilon,N}_s)ds\\
R_5^i(t) & = \int_0^t [D(\tilde{X}^{i,\epsilon,N}_s,\tilde{Y}^{i,\epsilon,N}_s,\tilde{\mu}^{\epsilon,N}_s) - \bar{D}(\tilde{X}^{i,\epsilon,N}_s,\tilde{\mu}^{\epsilon,N}_s)]\phi''(\tilde{X}^{i,\epsilon,N}_s)ds.
\end{align*}

By Proposition \ref{prop:llntypefluctuationestimate1} (using here Assumption \ref{assumption:forcorrectorproblem}):
\begin{align*}
\E\biggl[\sup_{t\in [0,T]}\frac{a(N)}{\sqrt{N}}\sum_{i=1}^N|R^i_4(t)| \biggr]& \leq C \epsilon a(N)\sqrt{N} (1+T^{1/2}+T)|\phi|_3
\end{align*}
and
\begin{align*}
\E\biggl[\sup_{t\in [0,T]}\frac{a(N)}{\sqrt{N}}\sum_{i=1}^N|R^i_5(t)| \biggr]&\leq C \epsilon a(N)\sqrt{N} (1+T^{1/2}+T)|\phi|_4.
\end{align*}

Now, we arrive at
\begin{align*}
\langle \tilde{Z}^N_t,\phi\rangle &= \frac{a(N)}{\sqrt{N}}\sum_{i=1}^N\biggl\lbrace \int_0^t\left(\bar{\gamma}(\tilde{X}^{i,\epsilon,N}_s,\tilde{\mu}^{i,\epsilon,N}_s)\phi'(\tilde{X}^{i,\epsilon,N}_s) -\E\biggl[ \bar{\gamma}(X_s,\mc{L}(X_s))\phi'(X_s)\biggr]\right)ds \\
&+ \int_0^t \left(\bar{D}(\tilde{X}^{i,\epsilon,N}_s,\tilde{\mu}^{i,\epsilon,N}_s)\phi''(\tilde{X}^{i,\epsilon,N}_s) -\E\biggl[ \bar{D}(X_s,\mc{L}(X_s))\phi''(X_s)\biggr]\right)ds\biggr\rbrace\\
&+\int_{\R\times\R\times\R^2\times[0,t]} \sigma(x,y,\tilde{\mu}^{\epsilon,N}_s)z_1 \phi'(x)Q^N(dx,dy,dz,ds)\\
&+\int_{\R\times\R\times\R^2\times[0,t]} [\tau_1(x,y,\tilde{\mu}^{\epsilon,N}_s)z_1+\tau_2(x,y,\tilde{\mu}^{\epsilon,N}_s)z_2]\Phi_y(x,y,\tilde{\mu}^{\epsilon,N}_s)\phi'(x)Q^N(dx,dy,dz,ds) +R_t^{N}(\phi)
\end{align*}
where $\E\biggl[\sup_{t\in[0,T]}|R_t^{N}(\phi)|\biggr]\leq C(T)[\epsilon a(N)\sqrt{N}\vee a(N)]|\phi|_4$. By Assumption \ref{assumption:limitingcoefficientsregularity}, $\bar{\gamma}$ and $\bar{D}$ have well-defined linear functional derivatives (see Definition \ref{def:LinearFunctionalDerivative}). Then we can rewrite $\langle \tilde{Z}^N_t,\phi\rangle$ as
\begin{align*}
&\langle \tilde{Z}^N_t,\phi\rangle
= \int_0^t\langle \tilde{Z}^N_s, \bar{\gamma}(\cdot,\tilde{\mu}^{\epsilon,N}_s)\phi'(\cdot) + \bar{D}(\cdot,\tilde{\mu}^{\epsilon,N}_s)\phi''(\cdot)\rangle ds\\
&+a(N)\sqrt{N}\biggl[\int_0^t \langle \mc{L}(X_s),[\bar{\gamma}(\cdot,\tilde{\mu}^{\epsilon,N}_s)-\bar{\gamma}(\cdot ,\mc{L}(X_s))]\phi'(\cdot) + [\bar{D}(\cdot,\tilde{\mu}^{\epsilon,N}_s)-\bar{D}(\cdot ,\mc{L}(X_s))]\phi''(\cdot)\rangle ds\biggr]\\
&+\int_{\R\times\R\times\R^2\times[0,t]} \sigma(x,y,\tilde{\mu}^{\epsilon,N}_s)z_1 \phi'(x)Q^N(dx,dy,dz,ds)\\
&+\int_{\R\times\R\times\R^2\times[0,t]} [\tau_1(x,y,\tilde{\mu}^{\epsilon,N}_s)z_1+\tau_2(x,y,\tilde{\mu}^{\epsilon,N}_s)z_2]\Phi_y(x,y,\tilde{\mu}^{\epsilon,N}_s)\phi'(x)Q^N(dx,dy,dz,ds) +R^{N}_t(\phi)\\
&= \int_0^t\langle \tilde{Z}^N_s, \bar{\gamma}(\cdot,\tilde{\mu}^{\epsilon,N}_s)\phi'(\cdot) + \bar{D}(\cdot,\tilde{\mu}^{\epsilon,N}_s)\phi''(\cdot)\rangle ds\\
&+a(N)\sqrt{N}\biggl[\int_0^t \langle \mc{L}(X_s),\biggl[\int_0^1\int_\R \frac{\delta}{\delta m}\bar{\gamma}(\cdot,(1-r)\mc{L}(X_s)+r\tilde{\mu}^{\epsilon,N}_s)[y](\tilde{\mu}^{N,\epsilon}_s(dy)-\mc{L}(X_s)(dy))dr\biggr]\phi'(\cdot)\\
& + \biggl[\int_0^1\int_\R \frac{\delta}{\delta m}\bar{D}(\cdot,(1-r)\mc{L}(X_s)+r\tilde{\mu}^{\epsilon,N}_s)[y](\tilde{\mu}^{N,\epsilon}_s(dy)-\mc{L}(X_s)(dy))dr\biggr]\phi''(\cdot)\rangle ds\biggr]\\
&+\int_{\R\times\R\times\R^2\times[0,t]} \sigma(x,y,\tilde{\mu}^{\epsilon,N}_s)z_1 \phi'(x)Q^N(dx,dy,dz,ds)\\
&+\int_{\R\times\R\times\R^2\times[0,t]} [\tau_1(x,y,\tilde{\mu}^{\epsilon,N}_s)z_1+\tau_2(x,y,\tilde{\mu}^{\epsilon,N}_s)z_2]\Phi_y(x,y,\tilde{\mu}^{\epsilon,N}_s)\phi'(x)Q^N(dx,dy,dz,ds) +R^{N}_t(\phi)\\
&=\int_0^t \langle \tilde{Z}^N_s,\bar{L}_{\mc{L}(X_s),\tilde{\mu}^{\epsilon,N}_s}\phi(\cdot)\rangle ds+\int_{\R\times\R\times\R^2\times[0,t]} \sigma(x,y,\tilde{\mu}^{\epsilon,N}_s)z_1 \phi'(x)Q^N(dx,dy,dz,ds)\\
&+\int_{\R\times\R\times\R^2\times[0,t]} [\tau_1(x,y,\tilde{\mu}^{\epsilon,N}_s)z_1+\tau_2(x,y,\tilde{\mu}^{\epsilon,N}_s)z_2]\Phi_y(x,y,\tilde{\mu}^{\epsilon,N}_s)\phi'(x)Q^N(dx,dy,dz,ds) +R^{N}_t(\phi).
\end{align*}
\end{proof}
\begin{proposition}\label{prop:tildeZNtightness}
Under Assumptions \ref{assumption:uniformellipticity}-\ref{assumption:limitingcoefficientsregularity}, $\br{\tilde{Z}^N}_{N\in\bb{N}}$ is tight as a sequence of $C([0,T];\mc{S}_{-m})$-valued random variables, where $m$ is as in Equation \eqref{eq:wdefinition}.
\end{proposition}
\begin{proof}
By Remark R.1 on p.997 of \cite{Mitoma}, it suffices to prove tightness of $\br{\langle \tilde{Z}^N,\phi\rangle}$ as a sequence of $C([0,T];\R)$-valued random variables for each $\phi \in \mc{S}$, along with uniform $7$-continuity as defined in the same remark. By the argument found in the proof of \cite{BW} Theorem 4.7, to show the latter it suffices to prove:
\begin{align}\label{eq:implies7cont}
\sup_{N\in\bb{N}}\E\biggl[\sup_{t\in[0,T]}\biggl|\langle \tilde{Z}^N_t,\phi\rangle\biggr|\biggr]& \leq C(T)\norm{\phi}_7,\forall \phi\in \mc{S}.
\end{align}
After these two results are shown, we will have established tightness of $\tilde{Z}^N$ as $C([0,T];\mc{S}_{-w})$-valued random variables for $m>7$ such that the canonical embedding $\mc{S}_{-7}\tto \mc{S}_{-m}$ is Hilbert-Schmidt. We start with showing tightness of $\br{\langle \tilde{Z}^N,\phi\rangle}$. By Lemma \ref{lemma:Lnu1nu2representation}, we write for any $\phi\in\mc{S}$:
\begin{align*}
\langle \tilde{Z}^N_t,\phi\rangle & = A^N_t(\phi)+R^N_t(\phi)\\
A^N_t(\phi)&\coloneqq \int_0^t \langle \tilde{Z}^N_s,\bar{L}_{\mc{L}(X_s),\tilde{\mu}^{\epsilon,N}_s}\phi(\cdot)\rangle ds+\int_{\R\times\R\times\R^2\times[0,t]} \sigma(x,y,\tilde{\mu}^{\epsilon,N}_s)z_1 \phi'(x)Q^N(dx,dy,dz,ds)\\
&+\int_{\R\times\R\times\R^2\times[0,t]} [\tau_1(x,y,\tilde{\mu}^{\epsilon,N}_s)z_1+\tau_2(x,y,\tilde{\mu}^{\epsilon,N}_s)z_2]\Phi_y(x,y,\tilde{\mu}^{\epsilon,N}_s)\phi'(x)Q^N(dx,dy,dz,ds)
\end{align*}
where for each $\phi$, $R^N(\phi)\tto 0$ in $C([0,T];\R)$ as $N\toinf$. Thus, to prove tightness of $\br{\langle \tilde{Z}^N,\phi\rangle}$, it is sufficient to prove tightness of $\br{A^N(\phi)}$. 
We note that for any and $0\leq \tau<t\leq T$:
\begin{align*}
A^N_t(\phi) - A^N_\tau(\phi)& = B^N_{t,\tau}(\phi) + C^N_{t,\tau}(\phi) +D^N_{t,\tau}(\phi)\\
B^N_{t,\tau}(\phi)&\coloneqq \int_\tau^t \langle \tilde{Z}^N_s,\bar{L}_{\mc{L}(X_s),\tilde{\mu}^{\epsilon,N}_s}\phi(\cdot)\rangle ds\\
C^N_{t,\tau}(\phi)& = \int_{\R\times\R\times\R^2\times[\tau,t]} \sigma(x,y,\tilde{\mu}^{\epsilon,N}_s)z_1 \phi'(x)Q^N(dx,dy,dz,ds)\\
D^N_{t,\tau}(\phi)&=\int_{\R\times\R\times\R^2\times[\tau,t]} [\tau_1(x,y,\tilde{\mu}^{\epsilon,N}_s)z_1+\tau_2(x,y,\tilde{\mu}^{\epsilon,N}_s)z_2]\Phi_y(x,y,\tilde{\mu}^{\epsilon,N}_s)\phi'(x)Q^N(dx,dy,dz,ds).
\end{align*}
Then we have for $\delta>0$,
\begin{align*}
\E\biggl[\sup_{|t-\tau|\leq\delta}\biggl|B^N_{t,\tau}(\phi)\biggr|\biggr]&\leq \delta^{1/2}\E\biggl[\int_0^T\biggl| \langle \tilde{Z}^N_s,\bar{L}_{\mc{L}(X_s),\tilde{\mu}^{\epsilon,N}_s}\phi(\cdot)\rangle\biggr|^2ds\biggr]^{1/2} \\
&\leq \delta^{1/2}T^{1/2}\sup_{s\in [0,T]}\E\biggl[\biggl| \langle \tilde{Z}^N_s,\bar{L}_{\mc{L}(X_s),\tilde{\mu}^{\epsilon,N}_s}\phi(\cdot)\rangle\biggr|^2\biggr]^{1/2}\\
&\leq \delta^{1/2}T^{1/2}\sup_{s\in [0,T]}\sup_{\nu_1,\nu_2\in\mc{P}_2(\R)}\E\biggl[\biggl| \langle \tilde{Z}^N_s,\bar{L}_{\nu_1,\nu_2}\phi(\cdot)\rangle\biggr|^2\biggr]^{1/2}\\
&\leq C(T)\delta^{1/2}|\phi|_6
\end{align*}
by boundedness of the first 5 derivatives in $x$ of $\bar{\gamma},\bar{D}$, and of the first 5 derivatives in $z$ of $\frac{\delta}{\delta m}\bar{\gamma},\frac{\delta}{\delta m}\bar{D}$ from Assumption \ref{assumption:limitingcoefficientsregularity}, the definition of $\bar{L}_{\nu_1,\nu_2}$ from Equation \eqref{eq:Lnu1nu2}, and Lemma \ref{lemma:Zboundbyphi4}. Also, we see: 
\begin{align*}
&\E\biggl[\sup_{|t-\tau|\leq\delta}\biggl|D^N_{t,\tau}(\phi)\biggr|\biggr] \leq\nonumber\\
&\leq C \frac{1}{N}\sum_{i=1}^N\E\biggl[\sup_{|t-\tau|\leq\delta}\int_\tau^t \left([|\tilde{u}^{N,1}_i(s)|+|\tilde{u}^{N,2}_i(s)|]|\Phi_y(\tilde{X}^{i,\epsilon,N}_s,\tilde{Y}^{i,\epsilon,N}_s,\tilde{\mu}^{\epsilon,N}_s)| \right)ds \biggr]|\phi|_1\\
&\leq C \frac{1}{N}\E\biggl[\sum_{i=1}^N\int_0^T |\tilde{u}^{N,1}_i(s)|^2+|\tilde{u}^{N,2}_i(s)|^2ds\biggr]^{1/2}\E\biggl[\sup_{|t-\tau|\leq\delta}\sum_{i=1}^N\int_\tau^t|\Phi_y(\tilde{X}^{i,\epsilon,N}_s,\tilde{Y}^{i,\epsilon,N}_s,\tilde{\mu}^{\epsilon,N}_s)|^2 ds\biggr]^{1/2}|\phi|_1 \\
&\leq C \E\biggl[\sup_{|t-\tau|\leq\delta}\frac{1}{N}\sum_{i=1}^N\int_\tau^t|\Phi_y(\tilde{X}^{i,\epsilon,N}_s,\tilde{Y}^{i,\epsilon,N}_s,\tilde{\mu}^{\epsilon,N}_s)|^2 ds\biggr]^{1/2}|\phi|_1\text{ by the bound \eqref{eq:controlassumptions0}}\\
&\leq C\delta^{1/2}C(T)|\phi|_1 \text{ by the boundedness of $\Phi_y$ from Assumption \ref{assumption:multipliedpolynomialgrowth}.}
\end{align*}
The proof that $\E\biggl[\biggl|B^N_{t,\tau}(\phi)\biggr|\biggr]\leq C\delta^{1/2}C(T)|\phi|_1$ holds in the same way.

So by the Arzel\`a-Ascoli tightness criterion on classical Wiener space (see, e.g. Theorem 4.10 in Chapter 2 of \cite{KS}), we have $\br{A^N(\phi)}$ and hence $\br{\langle \tilde{Z}^N,\phi\rangle}$ are tight as a sequence of $C([0,T];\R)$-valued random variables for each $\phi$.

Now we see by the same argument (fixing $\tau=0$) and the fact that, as shown in Lemma \ref{lemma:Lnu1nu2representation},

 $\E\biggl[\sup_{t\in[0,T]}\biggl| R^N_t(\phi)\biggr|\biggr]\leq \bar{R}(N,T)|\phi|_4$ with $\bar{R}(N,T)\tto 0$ as $N\toinf$:
\begin{align*}
\sup_{N\in\bb{N}}\E\biggl[\sup_{t\in[0,T]}\biggl|\langle \tilde{Z}^N_t,\phi\rangle\biggr|\biggr]&\leq C(T)|\phi|_6 \leq C(T)\norm{\phi}_7
\end{align*}
for all $\phi\in\mc{S}$, where here we used the inequality \eqref{eq:sobolembedding}. Thus the bound \eqref{eq:implies7cont} holds, and tightness is established.
\end{proof}

\subsection{Tightness of $Q^N$}\label{SS:QNtightness}
The proof of tightness of $\br{Q^N}$ from Equation \eqref{eq:occupationmeasures} is standard, see \cites{MS,BDF,BS}. We see that since the occupation measures $Q^N$ involve $\br{\tilde{X}^{i,\epsilon,N}}_{N\in\bb{N}}$ from Equation \eqref{eq:controlledslowfast1-Dold} as part of their definition, we will need some kind of uniform control on their expectation. Thus, we begin with a lemma: 
\begin{lemma}\label{lemma:tildeXuniformL2bound}
Under assumptions \ref{assumption:uniformellipticity}-\ref{assumption:qF2bound} and \ref{assumption:uniformLipschitzxmu}, we have
$\sup_{t\in[0,T]}\sup_{N\in\bb{N}}\frac{1}{N}\sum_{i=1}^N \E\biggl[|\tilde{X}^{i,\epsilon,N}_t|^2 \biggr]<\infty.$
\end{lemma}
\begin{proof}
Using that
\begin{align*}
\tilde{X}^{i,\epsilon,N}_t  &= \eta^x + \int_0^t \left(\frac{1}{\epsilon}b(i)+ c(i)\right)ds +\int_0^t \sigma(i)\frac{\tilde{u}^{N,1}_i(s)}{a(N)\sqrt{N}} ds + \int_0^t \sigma(i)dW^i_s\\
& = \eta^x + \int_0^t \frac{1}{\epsilon}b(i)ds -\biggl\lbrace\int_0^t  \gamma_1(i)ds+\int_0^t\tau_1(i)\Phi_y(i)dW^i_s + \int_0^t \tau_2(i)\Phi_y(i)dB^i_s+\int_0^t \frac{1}{N}\sum_{j=1}^N b(j)\partial_{\mu}\Phi(i)[j]ds\\
&+\int_0^t[\frac{\tau_1(i)}{a(N)\sqrt{N}}\tilde{u}^{N,1}_i(s)+\frac{\tau_2(i)}{a(N)\sqrt{N}}\tilde{u}^{N,2}_i(s)]\Phi_y(i)ds\biggr\rbrace + \int_0^t \gamma(i)ds + \int_0^t \sigma(i)dW^i_s\\
&+\int_0^t\tau_1(i)\Phi_y(i)dW^i_s + \int_0^t \tau_2(i)\Phi_y(i)dB^i_s+\int_0^t\frac{1}{N}\sum_{j=1}^N b(j)\partial_{\mu}\Phi(i)[j]ds\\
&+\int_0^t \left(\sigma(i)\frac{\tilde{u}^{N,1}_i(s)}{a(N)\sqrt{N}} +[\frac{\tau_1(i)}{a(N)\sqrt{N}}\tilde{u}^{N,1}_i(s)+\frac{\tau_2(i)}{a(N)\sqrt{N}}\tilde{u}^{N,2}_i(s)]\Phi_y(i)\right)ds,
\end{align*}
where here once again the argument $(i)$ is denoting $(\tilde{X}^{i,\epsilon,N}_s,\tilde{Y}^{i,\epsilon,N}_s,\tilde{\mu}^{\epsilon,N}_s)$ and similarly for $j$, and we recall $\Phi$ from Equation \eqref{eq:cellproblemold} and $\gamma_1,\gamma$ from Equation \eqref{eq:limitingcoefficients}.

So, by It\^o Isometry and boundedness of $\sigma$ from \ref{assumption:gsigmabounded}, of $\tau_1$ and $\tau_2$ from \ref{assumption:uniformellipticity}, and of $\Phi_y$ from \ref{assumption:multipliedpolynomialgrowth}:
\begin{align*}
\frac{1}{N}\sum_{i=1}^N \E\biggl[|\tilde{X}^{i,\epsilon,N}_t |^{2}\biggr]&\leq C(|\eta^x|^{2}+T) + \frac{C}{N}\sum_{i=1}^N \biggl\lbrace R^{i,N}_1(t)+R^{i,N}_2(t)+R^{i,N}_3(t)+R^{i,N}_4(t)\biggr\rbrace\\
R^{i,N}_1(t)&\coloneqq\E\biggl[\biggl|\int_0^t \frac{1}{\epsilon}b(i)ds -\biggl\lbrace\int_0^t  \gamma_1(i)ds+\int_0^t\tau_1(i)\Phi_y(i)dW^i_s + \int_0^t \tau_2(i)\Phi_y(i)dB^i_s\\
&+\int_0^t \frac{1}{N}\sum_{j=1}^N b(j)\partial_{\mu}\Phi(i)[j]ds+\int_0^t[\frac{\tau_1(i)}{a(N)\sqrt{N}}\tilde{u}^{N,1}_i(s)+\frac{\tau_2(i)}{a(N)\sqrt{N}}\tilde{u}^{N,2}_i(s)]\Phi_y(i)ds\biggr\rbrace  \biggr|^{2}\biggr]\\
R^{i,N}_2(t)&\coloneqq \E\biggl[\biggl|\int_0^t \gamma(i)ds\biggr|^2\biggr] \\
R^{i,N}_3(t)&=\E\biggl[\biggl|\int_0^t \frac{1}{N}\sum_{j=1}^N b(j)\partial_{\mu}\Phi(i)[j]ds \biggr|^{2}\biggr]\\
R^{i,N}_4(t)&=\E\biggl[\biggl|\int_0^t \left(\sigma(i)\frac{\tilde{u}^{N,1}_i(s)}{a(N)\sqrt{N}} +[\frac{\tau_1(i)}{a(N)\sqrt{N}}\tilde{u}^{N,1}_i(s)+\frac{\tau_2(i)}{a(N)\sqrt{N}}\tilde{u}^{N,2}_i(s)]\Phi_y(i)ds\right)\biggr|^2\biggr]
\end{align*}
Then applying Proposition \ref{prop:fluctuationestimateparticles1} with $\psi = 1$, we have
\begin{align*}
\frac{1}{N}\sum_{i=1}^NR^{i,N}_1(t)&\leq C[\epsilon^2(1+T+T^2)+\frac{1}{N}T^2]
\end{align*}
Using Assumption \ref{assumption:uniformLipschitzxmu}:
\begin{align*}
\frac{1}{N}\sum_{i=1}^NR^{i,N}_2(t)\leq \frac{1}{N}\sum_{i=1}^NT\E\biggl[\int_0^t |\gamma(i)|^2ds\biggr]&\leq CT\int_0^t\frac{1}{N}\sum_{i=1}^N\E\biggl[ |\tilde{X}^{i,\epsilon,N}_s|^2+|\tilde{Y}^{i,\epsilon,N}_s|^2+\frac{1}{N}\sum_{j=1}^N|\tilde{X}^{j,\epsilon,N}_s|^2\biggr]ds\\
&\leq CT^2+CT\int_0^t\frac{1}{N}\sum_{i=1}^N\E\biggl[ |\tilde{X}^{i,\epsilon,N}_s|^2\biggr]ds
\end{align*}
by Lemma \ref{lemma:tildeYuniformbound}.
Applying Proposition \ref{prop:purpleterm1} with $\psi=1$:
\begin{align*}
\frac{1}{N}\sum_{i=1}^NR^{i,N}_3(t)\leq C[\epsilon^2(1+T+T^2)+\frac{1}{N^2}T^2]
\end{align*}
Using the boundedness of $\sigma$ from \ref{assumption:gsigmabounded}, of $\tau_1$ and $\tau_2$ from \ref{assumption:uniformellipticity}, and of $\Phi_y$ from \ref{assumption:multipliedpolynomialgrowth} and the bound \eqref{eq:controlassumptions}:
\begin{align*}
\frac{1}{N}\sum_{i=1}^NR^{i,N}_4(t)&\leq \frac{CT}{a^2(N)N}\frac{1}{N}\sum_{i=1}^N \E\biggl[\int_0^T \left(|\tilde{u}^{N,1}_i(s)|^2+|\tilde{u}^{N,2}_i(s)|^2\right)ds]\leq \frac{CT}{a^2(N)N}.
\end{align*}

Then, by Gronwall's inequality:
\begin{align*}
\frac{1}{N}\sum_{i=1}^N \E\biggl[|\tilde{X}^{i,\epsilon,N}_t|^2 \biggr]&\leq C(T)[1+\epsilon^2+\frac{1}{N}+\frac{1}{N^2}+\frac{1}{a^2(N)N}]\leq C(T),
\end{align*}
since all the above terms which depend on $N,\epsilon$ in the first bound vanish as $N\toinf$.
Since this holds uniformly in $N$ and $t$, we are done.
\end{proof}

Now we can prove tightness of the occupation measures.
\begin{proposition}\label{prop:QNtightness}
Under assumptions \ref{assumption:uniformellipticity}-\ref{assumption:qF2bound} and \ref{assumption:uniformLipschitzxmu}, $\br{Q^N}_{N\in\bb{N}}$ is tight as a sequence of $M_T(\R^4)$-valued random variables (recall this space of measures introduced above Equation \eqref{eq:rigorousLotimesdt}).
\end{proposition}
\begin{proof}
Consider the function $G:\mc{P}(\R\times\R\times\R^2\times [0,T])\tto \R$ given by
\begin{align*}
G(\theta)= \int_{\R\times\R\times\R^2\times [0,T]}\left(|z|^2+|y|^2+|x|^2\right)\theta(dx,dy,dz,ds).
\end{align*}

Then we have $G$ is bounded below, and considering a given level set
$A_L\coloneqq\br{\theta\in M_T(\R^4):G(\theta)\leq L},$
it follows by Chebyshev's inequality that $\sup_{\theta\in A_L} \theta((K^\epsilon_L)^c)\leq \epsilon$
where $K^\epsilon_L$ is the compact subset of $\R^4\times[0,T]$
\begin{align*}
K^\epsilon_L\coloneqq \br{(x,y,z)\in \R\times\R\times\R^2:|x|^2+|y|^2+|z|^2\leq \frac{L}{\epsilon}}\times [0,T].
\end{align*}

We also see that any collection of measures on $\R^4\times [0,T]$ which is in $M_T(\R^4)$ is uniformly bounded in the total variation norm, and that for $\br{\theta^N}\subset A_L$ such that $\theta^N\tto \theta$ in $M_T(\R^4)$ (recalling here that we are using the topology of weak convergence), by a version of Fatou's lemma (see Theorem A.3.12 in \cite{DE})
\begin{align*}
G(\theta)\leq \liminf_{N\toinf}G(\theta^N)\leq L,
\end{align*}
so $\theta\in A_L$. Via Prokhorov's theorem, $A_L$ is precompact, and we have shown that $A_L$ is closed, and hence $G$ has compact level sets. Thus $G$ is a tightness function (see \cite{DE} p.309), and it suffices to prove
\begin{align*}
\sup_{N\in\bb{N}}\E\biggl[G(Q^N)\biggr] = \sup_{N\in\bb{N}}\frac{1}{N}\sum_{i=1}^N \E\biggl[\int_0^T \left(|\tilde{X}^{i,\epsilon,N}_s|^2+|\tilde{Y}^{i,\epsilon,N}_s|^2 + |\tilde{u}^{N,1}_i(s)|^2+|\tilde{u}^{N,2}_i(s)|^2\right)ds\biggr]<\infty
\end{align*}
to see that $\br{Q^N}$ is a tight sequence of $\mc{M}_R(\R^4)-$valued random variables. This follows immediately from the bound \eqref{eq:controlassumptions0} and Lemmas \ref{lemma:tildeYuniformbound} and \ref{lemma:tildeXuniformL2bound}.
\end{proof}
\section{Identification of the Limit}\label{sec:identificationofthelimit}
Now having established tightness of $\br{(\tilde{Z}^N,Q^N)}_\bb{N}$, we take any sub-sequence that converges in distribution as $C([0,T];\mc{S}_{-m})\times M_T(\R^4)$-valued random variables, and call the random variable which is its limit $(Z,Q)$. We will show that $Q\in P^*(Z)$, and that this uniquely characterizes the distribution of $(Z,Q)$ for a given choice of controls in the construction of $Q^N.$ We will at times apply Skorokhod's Representation Theorem to without loss of generality pose the problem on a probability space such that this subsequence converges to $(Z,Q)$ almost surely. We also do not distinguish from the subsequence and the original sequence in the notation, nor the original probability space and that invoked by Skorokhod's Representation Theorem.
We begin with two lemmas which allow us to identify convergence of the controlled empirical measure $\tilde{\mu}^N$ from \eqref{eq:controlledempmeasure} to the law of the averaged McKean-Vlasov equation \eqref{eq:LLNlimitold}:

\begin{lemma}\label{lemma:barXuniformbound}
In the setting of Proposition \ref{lemma:tildeXuniformL2bound}, we have for any $p\geq 1$:
\begin{align*}
\sup_{\epsilon>0}\sup_{t\in[0,T]}\E\biggl[|\bar{X}^{\epsilon}_t |^{p}\biggr]\leq |\eta^x|^p+C(T,p)[1+|\eta^y|^p].
\end{align*}
Here $\bar{X}^{\epsilon}$ is as in Equation \eqref{eq:slow-fastMcKeanVlasov}. That is, it is equal in distribution to the IID particles from Equation \eqref{eq:IIDparticles}.
\end{lemma}
\begin{proof}
This follows in the same way as Lemmas \ref{lemma:XbartildeXdifference} and \ref{lemma:tildeXuniformL2bound}, using Lemma \ref{lemma:barYuniformbound} and the ergodic-type Theorems from Section 4 of \cite{BezemekSpiliopoulosAveraging2022}. We omit the proof for brevity.
\end{proof}
\begin{lemma}\label{lemma:W2convergenceoftildemu}
Assume \ref{assumption:uniformellipticity}-\ref{assumption:2unifboundedlinearfunctionalderivatives}. Let $\tilde{\mu}^{\epsilon,N}_t$ be as in Equation \eqref{eq:controlledempmeasure}, with controls satisfying \eqref{eq:controlassumptions}. Then
\begin{align*}
\E\biggl[\bb{W}_2(\tilde{\mu}^{\epsilon,N}_t,\mc{L}(X_t)) \biggr]\tto 0 \text{ as }N\toinf,\forall t\in [0,T],
\end{align*}
where $X_t$ is as in Equation \eqref{eq:LLNlimitold}. In particular, decomposing $Q^N$ from Equation \eqref{eq:occupationmeasures} as $Q^N(dx,dy,dz,dt) = Q^N_t(dx,dy,dz)dt$,  for any $t\in [0,T]$, the first marginal of $Q^N_t$ converges to $\mc{L}(X_t)$ in probability as a sequence of $\mc{P}_2(\R)$-valued random variables.
\end{lemma}
\begin{proof}
Firstly, we note by Theorem \ref{theo:mckeanvlasovaveraging}, $\mc{L}(\bar{X}^\epsilon_t)\tto\mc{L}(X_t)$ in $\mc{P}(\R)$ (using here that $C^\infty_c(\R)$ is convergence determining - see \cite{EK} Proposition 3,4.4). In addition, 
by Lemma \ref{lemma:barXuniformbound}, we have $\sup_{\epsilon>0}\int_\R |x|^{p}\mc{L}(\bar{X}^\epsilon_t)(dx)<\infty$, for some $p>2$. Thus, we have by uniform integrability, $\E\biggl[|\bar{X}^\epsilon_t|^2\biggr]\tto \E\biggl[|X_t|^2\biggr]$ as $\epsilon\downarrow 0$, so $\bb{W}_2(\mc{L}(\bar{X}^\epsilon_t),\mc{L}(X_t))\tto 0$ as $\epsilon \downarrow 0$ (Theorem 5.5 in \cite{CD}). 
By Lemma \ref{lemma:XbartildeXdifference}, we also have
\begin{align*}
\E\biggl[\bb{W}_2(\tilde{\mu}^{\epsilon,N}_t,\bar{\mu}^{\epsilon,N}_t)\biggr]&\leq \E\biggl[ \frac{1}{N}\sum_{i=1}^N\biggl|\tilde{X}^{i,\epsilon,N}_t-\bar{X}^{i,\epsilon}_t\biggr|^2\biggr] \tto 0 \text{ as }N\toinf,
\end{align*}
where $\bar{\mu}^{\epsilon,N}$ is as in Equation \eqref{eq:IIDempiricalmeasure}.
Also, by Glivenko-Cantelli Convergence in the Wasserstein Distance (see, e.g. Section 5.1.2 in \cite{CD}):
\begin{align*}
\E\biggl[\bb{W}_2(\bar{\mu}^{\epsilon,N}_t,\mc{L}(\bar{X}^\epsilon_t))\biggr]\tto 0 \text{ as }N\toinf.
\end{align*}

So by the triangle inequality (see, e.g. the proof of \cite{CD} Proposition 5.3), we have:
\begin{align*}
\E\biggl[\bb{W}_2(\tilde{\mu}^{\epsilon,N}_t,\mc{L}(X_t)) \biggr]&\leq \E\biggl[\bb{W}_2(\tilde{\mu}^{\epsilon,N}_t,\bar{\mu}^{\epsilon,N}_t)\biggr] + \E\biggl[\bb{W}_2(\bar{\mu}^{\epsilon,N}_t,\mc{L}(\bar{X}^\epsilon_t))\biggr] + \bb{W}_2(\mc{L}(\bar{X}^\epsilon_t),\mc{L}(X_t)) \tto 0 \text{ as }N\toinf.
\end{align*}

The latter statement of the Lemma now follows from the construction of $Q^N$ and Markov's inequality.
\end{proof}
Now we can use the prelimit representation for the controlled fluctuation process $\tilde{Z}^N$ from Lemma \ref{lemma:Lnu1nu2representation} in order to identify the limiting behavior of $(\tilde{Z}^N,Q^N)$.

\begin{proposition}\label{prop:limitsatisfiescorrectequations}
Under assumptions \ref{assumption:uniformellipticity} - \ref{assumption:limitingcoefficientsregularity}, $(Z,Q)$ satisfies Equation \eqref{eq:MDPlimitFIXED} with probability 1.
\end{proposition}
\begin{proof}
We now invoke the Skorokhod's Representation Theorem as previously discussed. By a standard density argument, we can simply show that Equation \eqref{eq:MDPlimitFIXED} holds with probability 1 for each $\phi\in C^\infty_c(\R)$ and $t\in[0,T]$. This is using the fact that there exists a countable, dense collection of smooth, compactly supported functions in $\mc{S}_{m}$ (this follows from, e.g. Corollary 2.1.2 in \cite{Rauch}).

We note that by almost sure convergence of $\tilde{Z}^N$ to $Z$, we have for each $t\in[0,T]$ and $\phi\in C^\infty_c(\R)$, $\langle \tilde{Z}^N_t,\phi\rangle \tto \langle Z_t,\phi\rangle$ with probability 1. We also note that the prelimit representation given in Lemma \ref{lemma:Lnu1nu2representation} can be written solely in terms of $Q^N$ and $Z^N$ by replacing $\tilde{\mu}^{\epsilon,N}_s$ by the first marginal of $Q^N_s$. We can therefore take $\tilde{\mu}^{\epsilon,N}_s$ to also live on the new probability space from Skorokhod's Representation Theorem, and on that space we still have the convergence of $\tilde{\mu}^{\epsilon,N}_t$ to $\mc{L}(X_t)$ in probability proved in Lemma \ref{lemma:W2convergenceoftildemu}. Thus, by the representation provided by Lemma \ref{lemma:Lnu1nu2representation}, we only need to show the limits in probability:
\begin{align}
&\label{eq:problimit1}\int_0^t \langle \tilde{Z}^N_s,\bar{L}_{\mc{L}(X_s),\tilde{\mu}^{\epsilon,N}_s}\phi(\cdot)\rangle ds \tto (N\rightarrow\infty) \int_0^t \langle Z_s,\bar{L}_{\mc{L}(X_s)}\phi(\cdot)\rangle ds\\
&\label{eq:problimit2}\int_{\R\times\R\times\R^2\times[0,t]} \left(\sigma(x,y,\tilde{\mu}^{\epsilon,N}_s)z_1 \phi'(x)+[\tau_1(x,y,\tilde{\mu}^{\epsilon,N}_s)z_1+\tau_2(x,y,\tilde{\mu}^{\epsilon,N}_s)z_2]\Phi_y(x,y,\tilde{\mu}^{\epsilon,N}_s)\phi'(x)\right)Q^N(dx,dy,dz,ds) \\
&\tto (N\rightarrow\infty)\nonumber\\
&\int_{\R\times\R\times\R^2\times[0,t]} \left(\sigma(x,y,\mc{L}(X_s))z_1 \phi'(x)+ [\tau_1(x,y,\mc{L}(X_s))z_1+\tau_2(x,y,\mc{L}(X_s))z_2]\Phi_y(x,y,\mc{L}(X_s))\phi'(x)\right)Q(dx,dy,dz,ds), \nonumber
\end{align}
where $\bar{L}_{\nu_1,\nu_2}$ is as in Equation \eqref{eq:Lnu1nu2} and $\bar{L}_\nu$ is as in Equation \eqref{eq:MDPlimitFIXED}.
By boundedness and continuity of $\bar{\gamma},\bar{D}$ from assumption \ref{assumption:limitingcoefficientsregularity} (see Definition \ref{def:LinearFunctionalDerivative}), along with Lemma \ref{lemma:W2convergenceoftildemu}, we have for each $s\in [0,T]$ and $\phi \in C^\infty_c(\R)$, the limit in probability
\begin{align*}
\bar{L}_{\mc{L}(X_s),\tilde{\mu}^{\epsilon,N}_s}\phi(\cdot) &\tto \bar{L}_{\mc{L}(X_s)}\phi(\cdot) \text{ in }\mc{S}_{m}
\end{align*}
holds via the continuous mapping theorem. Thus, for each $s\in [0,T]$ and $\phi \in C^\infty_c(\R)$, the limit in probability
\begin{align*}
\langle \tilde{Z}^N_s,\bar{L}_{\mc{L}(X_s),\tilde{\mu}^{\epsilon,N}_s}\phi(\cdot)\rangle \tto \langle Z_s,\bar{L}_{\mc{L}(X_s)}\phi(\cdot)\rangle
\end{align*}
holds. We have, then, for all $t\in [0,T]$:
\begin{align*}
&\lim_{N\toinf}\E\biggl[\biggl|\int_0^t \left(\langle \tilde{Z}^N_s,\bar{L}_{\mc{L}(X_s),\tilde{\mu}^{\epsilon,N}_s}\phi(\cdot)\rangle -\langle Z_s,\bar{L}_{\mc{L}(X_s)}\phi(\cdot)\rangle\right) ds \biggr|\biggr]\leq\nonumber\\
 &\hspace{4cm}\leq\lim_{N\toinf}\E\biggl[\int_0^t \biggl|\langle \tilde{Z}^N_s,\bar{L}_{\mc{L}(X_s),\tilde{\mu}^{\epsilon,N}_s}\phi(\cdot)\rangle -\langle Z_s,\bar{L}_{\mc{L}(X_s)}\phi(\cdot)\rangle\biggr| ds \biggr],
\end{align*}
and we have by Lemma \ref{lemma:Zboundbyphi4} that
\begin{align*}
\sup_{N\in\bb{N}}\int_0^t\E\biggl[\biggl| \langle \tilde{Z}^N_s,\bar{L}_{\mc{L}(X_s),\tilde{\mu}^{\epsilon,N}_s}\phi(\cdot)-\bar{L}_{\mc{L}(X_s)}\phi(\cdot)\rangle  \biggr|^2\biggr]ds<\infty,
\end{align*}
so by uniform integrability we can pass to the limit to get
\begin{align*}
&\lim_{N\toinf}\E\biggl[\int_0^t \biggl|\langle \tilde{Z}^N_s,\bar{L}_{\mc{L}(X_s),\tilde{\mu}^{\epsilon,N}_s}\phi(\cdot)\rangle -\langle Z_s,\bar{L}_{\mc{L}(X_s)}\phi(\cdot)\rangle\biggr| ds \biggr] =\nonumber\\
&\hspace{5cm} =\E\biggl[\int_0^t \lim_{N\toinf}\biggl|\langle \tilde{Z}^N_s,\bar{L}_{\mc{L}(X_s),\tilde{\mu}^{\epsilon,N}_s}\phi(\cdot)\rangle -\langle Z_s,\bar{L}_{\mc{L}(X_s)}\phi(\cdot)\rangle\biggr| ds \biggr] = 0.
\end{align*}

Similarly, we can use that by the Lemma \ref{lemma:Zboundbyphi4} and Fatou's lemma:
\begin{align*}
\sup_{N\in\bb{N}}\int_0^t\E\biggl[\biggl| \langle \tilde{Z}^N_s,\bar{L}_{\mc{L}(X_s)}\phi(\cdot)\rangle- \langle Z_s,\bar{L}_{\mc{L}(X_s)}\phi(\cdot)\rangle\biggr|^2\biggr]ds<\infty
\end{align*}
and to get
\begin{align*}
\lim_{N\toinf}\E\biggl[\biggl|\int_0^t \left(\langle \tilde{Z}^N_s,\bar{L}_{\mc{L}(X_s)}\phi(\cdot)\rangle-\langle Z_s,\bar{L}_{\mc{L}(X_s)}\phi(\cdot)\rangle\right) ds \biggr|\biggr] & \leq \E\biggl[\int_0^t \lim_{N\toinf}\biggl|\langle \tilde{Z}^N_s,\bar{L}_{\mc{L}(X_s)}\phi(\cdot)\rangle-\langle Z_s,\bar{L}_{\mc{L}(X_s)}\phi(\cdot)\rangle\biggr| ds \biggr]\\
& = 0.
\end{align*}
Then, by Markov's inequality, we establish \eqref{eq:problimit1}. The limit \eqref{eq:problimit2} follows immediately from the integrand being bounded by $C[|z_1|+|z_2|]$ and continuous in $\bb{W}_2$, along with the assumed bound on the controls \eqref{eq:controlassumptions} (see, e.g., \cite{DE} Theorem A.3.18).
\end{proof}

\begin{proposition}\label{prop:viabilityoflimit}
Under assumptions \ref{assumption:uniformellipticity} - \ref{assumption:limitingcoefficientsregularity}, $Q\in P^*(Z)$ with Probability 1.
\end{proposition}
\begin{proof}
By Proposition \ref{prop:limitsatisfiescorrectequations}, \ref{PZ:limitingequation} in the definition of $P^*(Z)$ holds. It remains to show \ref{PZ:L2contolbound}-\ref{PZ:fourthmarginallimitnglaw}. 

\ref{PZ:fourthmarginallimitnglaw} is immediate from the fact that the last marginal of $Q$ is Lebesgue measure by the definition of $M_T(\R^4)$ above Equation \eqref{eq:rigorousLotimesdt}, and the first marginal of $\tilde{Q}^N_s$ is $\tilde{\mu}^{\epsilon,N}_s$, which converges in $\mc{P}_2(\R)$ and hence $\mc{P}(\R)$ to $\mc{L}(X_s)$ by Lemma \ref{lemma:W2convergenceoftildemu}.

\ref{PZ:L2contolbound} follows from the version of Fatou's lemma from Theorem A.3.12 in \cite{DE}, since $\int_{\R\times\R\times\R^2\times[0,T]}|z_1|^2+|z_2|^2Q^N(dx,dy,dz,dt)$ is a non-negative random variable, and
\begin{align*}
\E\biggl[\int_{\R\times\R\times\R^2\times[0,T]}\left(|z_1|^2+|z_2|^2\right) Q(dx,dy,dz,dt) \biggr]&\leq \liminf_{N\toinf}\E\biggl[\int_{\R\times\R\times\R^2\times[0,T]}\left(|z_1|^2+|z_2|^2\right) Q^N(dx,dy,dz,dt) \biggr]\\
&\leq \sup_{N\in\bb{N}}\int_0^T\E\biggl[\frac{1}{N}\sum_{i=1}^N |\tilde{u}^{N,1}_i(s)|^2+|\tilde{u}^{N,2}_i(s)|^2\biggl]ds<\infty
\end{align*}
by the assumed bound \eqref{eq:controlassumptions0}.

Lastly, to see \ref{PZ:secondmarginalinvtmeasure}, take $\psi\in C^\infty_c(U\times\R)$ and $\phi\in C^\infty_c(\R)$. Here $U$ is an open interval in $\R$ containing $[0,T]$. Then
applying It\^o's formula (recalling here $\tilde{X}^{i,\epsilon,N},\tilde{Y}^{i,\epsilon,N}$ from Equation \eqref{eq:controlledslowfast1-Dold}):
\begin{align*}
\phi(\tilde{Y}^{i,\epsilon,N}_T)\psi(T,\tilde{X}^{i,\epsilon,N}_T) &= \phi(\tilde{Y}^{i,\epsilon,N}_0)\psi(0,\tilde{X}^{i,\epsilon,N}_0) + \int_0^T \biggl(\dot{\psi}(s,\tilde{X}^{i,\epsilon,N}_s)\phi(\tilde{Y}^{i,\epsilon,N}_s)\\
&+ \frac{1}{\epsilon^2}\biggl[f(i)\phi'(\tilde{Y}^{i,\epsilon,N}_s)+\frac{1}{2}(\tau_1^2(i)+\tau_2^2(i))\phi''(\tilde{Y}^{i,\epsilon,N}_s)\biggr]\psi(s,\tilde{X}^{i,\epsilon,N}_s) \\
&+\frac{1}{\epsilon}\biggl[g(i)+\tau_1(i)\frac{\tilde{u}^{N,1}_i(s)}{a(N)\sqrt{N}}+\tau_2(i)\frac{\tilde{u}^{N,1}_2(s)}{a(N)\sqrt{N}}\biggr]\phi'(\tilde{Y}^{i,\epsilon,N}_s)\psi(s,\tilde{X}^{i,\epsilon,N}_s)\\
&+\biggl[c(i)+\sigma(i)\frac{\tilde{u}^{N,1}_i(s)}{a(N)\sqrt{N}}\biggr]\phi(\tilde{Y}^{i,\epsilon,N}_s)\psi_x(s,\tilde{X}^{i,\epsilon,N}_s)+\frac{1}{2}\sigma^2(i)\phi(\tilde{Y}^{i,\epsilon,N}_s)\psi_{xx}(s,\tilde{X}^{i,\epsilon,N}_s)\\
&+\frac{1}{\epsilon}b(i)\phi(\tilde{Y}^{i,\epsilon,N}_s)\psi_x(s,\tilde{X}^{i,\epsilon,N}_s) + \frac{1}{\epsilon}\sigma(i)\tau_1(i)\phi'(\tilde{Y}^{i,\epsilon,N}_s)\psi_x(s,\tilde{X}^{i,\epsilon,N}_s)\biggr)ds \\
&+ \frac{1}{\epsilon}\int_0^T \tau_1(i)\phi'(\tilde{Y}^{i,\epsilon,N}_s)\psi(s,\tilde{X}^{i,\epsilon,N}_s)dW^i_s +\frac{1}{\epsilon}\int_0^T \tau_2(i)\phi'(\tilde{Y}^{i,\epsilon,N}_s)\psi(s,\tilde{X}^{i,\epsilon,N}_s)dB^i_s \\
&+ \int_0^T \sigma(i)\phi(\tilde{Y}^{i,\epsilon,N}_s)\psi_x(s,\tilde{X}^{i,\epsilon,N}_s)dW^i_s
\end{align*}
where $(i)$ denotes the argument $(\tilde{X}^{i,\epsilon,N}_s,\tilde{Y}^{i,\epsilon,N}_s,\tilde{\mu}^{\epsilon,N}_s)$. So recalling the definition of $L_{x,\mu}$ from Equation \eqref{eq:frozengeneratormold}, multiplying both sides by $\frac{\epsilon^2}{N}$ and summing,
\begin{align*}
&\int_{\R\times\R\times\R^2\times[0,T]}L_{x,\tilde{\mu}^{\epsilon,N}_s}\phi(y)\psi(s,x)Q^N(dx,dy,dz,ds)=\\
 &=\frac{1}{N}\sum_{i=1}^N\biggl\lbrace \epsilon^2[\phi(\tilde{Y}^{i,\epsilon,N}_0)\psi(0,\tilde{X}^{i,\epsilon,N}_0)-\phi(\tilde{Y}^{i,\epsilon,N}_T)\psi(T,\tilde{X}^{i,\epsilon,N}_T)]\\
&+ \epsilon^2\int_0^T\biggl( \dot{\psi}(s,\tilde{X}^{i,\epsilon,N}_s)\phi(\tilde{Y}^{i,\epsilon,N}_s)+\biggl[c(i)+\sigma(i)\frac{\tilde{u}^{N,1}_i(s)}{a(N)\sqrt{N}}\biggr]\phi(\tilde{Y}^{i,\epsilon,N}_s)\psi_x(s,\tilde{X}^{i,\epsilon,N}_s)+\frac{1}{2}\sigma^2(i)\phi(\tilde{Y}^{i,\epsilon,N}_s)\psi_{xx}(s,\tilde{X}^{i,\epsilon,N}_s)\biggr)ds\\
&+\epsilon\int_0^T\biggl(\biggl[g(i)+\tau_1(i)\frac{\tilde{u}^{N,1}_i(s)}{a(N)\sqrt{N}}+\tau_2(i)\frac{\tilde{u}^{N,1}_2(s)}{a(N)\sqrt{N}}\biggr]\phi'(\tilde{Y}^{i,\epsilon,N}_s)\psi(s,\tilde{X}^{i,\epsilon,N}_s)\\
&+b(i)\phi(\tilde{Y}^{i,\epsilon,N}_s)\psi_x(s,\tilde{X}^{i,\epsilon,N}_s) + \sigma(i)\tau_1(i)\phi'(\tilde{Y}^{i,\epsilon,N}_s)\psi_x(s,\tilde{X}^{i,\epsilon,N}_s)\biggr)ds \\
&+ \epsilon\int_0^T \tau_1(i)\phi'(\tilde{Y}^{i,\epsilon,N}_s)\psi(s,\tilde{X}^{i,\epsilon,N}_s)dW^i_s +\epsilon\int_0^T \tau_2(i)\phi'(\tilde{Y}^{i,\epsilon,N}_s)\psi(s,\tilde{X}^{i,\epsilon,N}_s)dB^i_s \\
&+ \epsilon^2\int_0^T \sigma(i)\phi(\tilde{Y}^{i,\epsilon,N}_s)\psi_x(s,\tilde{X}^{i,\epsilon,N}_s)dW^i_s.
\end{align*}

Since all terms in the right hand side are bounded other than $b$ and $c$, which grow at most linearly in $y$ as per Assumption \ref{assumption:gsigmabounded}, we see after using the bound \eqref{eq:controlassumptions0} that the right hand side is bounded in square expectation by
\begin{align*}
C(T)\epsilon^2(1+\sup_{N\in\bb{N}}\frac{1}{N}\sum_{i=1}^N\sup_{s\in[0,T]}\E\biggl[|\tilde{Y}^{i,\epsilon,N}_s|^2 \biggr])\leq C(T)\epsilon^2
\end{align*}
by Lemma \ref{lemma:tildeYuniformbound}, and hence converges to $0$ in probability.

We can see also by the fact that $\phi$ and $\psi$ are compactly supported and the coefficients in $L_{x,\mu}$ are continuous in $(x,y,\bb{W}_2)$ by assumptions \ref{assumption:uniformellipticity} and \ref{assumption:retractiontomean}, we can use the definition of convergence in $M_T(\R^4)$ and Lemma \ref{lemma:W2convergenceoftildemu} to see the left hand side converges in probability to $$\int_{\R\times\R\times\R^2\times[0,T]}L_{x,\mc{L}(X_s)}\phi(y)\psi(s,x)Q(dx,dy,dz,ds)$$ (see, e.g., \cite{DE} Theorem A.3.18). Thus, using that $Q$ satisfies \ref{PZ:fourthmarginallimitnglaw}, $$\int_{\R\times\R\times\R^2\times[0,T]}L_{x,\mc{L}(X_s)}\phi(y)\psi(s,x)Q(dx,dy,dz,ds)=\int_0^T \int_{\R}\int_{\R}L_{x,\mc{L}(X_s)}\phi(y)\psi(s,x)\lambda(dy;x,s)\mc{L}(X_s)(dx)ds=0$$
for some stochastic kernel $\lambda$ almost surely. 
Then noting that by boundedness of the coefficients and the derivatives of $\phi$, we have $(s,x)\mapsto \int_\R L_{x,\mc{L}(X_s)}\phi(y) \lambda(dy;x,s)$ is in $L^1_{\text{loc}}([0,T]\times\R,\nu_{\mc{L}(X_\cdot)})$ for all $\phi$, and thus by Corollary 22.38 (2) in \cite{Driver}, for each $\phi$, we have
\begin{align*}
\int_\R L_{x,\mc{L}(X_s)}\phi(y) \lambda(dy;x,s)=0
\end{align*}
$\nu_{\mc{L}(X_\cdot)}$- almost surely. By a standard density argument (see \cite{BS} Section 6.2.1), we have by letting
\begin{align*}
A = \br{(s,x):\int_\R L_{x,\mc{L}(X_s)}\phi(y) \lambda(dy;x,s)=0,\forall \phi\in C^\infty_c(\R)},
\end{align*}
$\nu_{\mc{L}(X_\cdot)}(A\times [0,T])=\int_0^T \int_\R \1_A \mc{L}(X_s)(dx)ds=1$. This then characterizes $\lambda(dy;x,s)$ as $\nu_{\mc{L}(X_\cdot)}-$ almost surely satisfying $L_{x,\mc{L}(X_s)}^*\lambda(\cdot;x,s)=0$ in the distributional sense, and by definition of stochastic kernels $\int_\R \lambda(dy;x,s)=1,\forall x,s$, so $\lambda(dy;x,s)$ is an invariant measure associated to $L_{x,\mc{L}(X_s)}$. Since such an invariant measure is unique under assumptions \ref{assumption:uniformellipticity} and \ref{assumption:retractiontomean} by \cite{PV1} Proposition 1, we have in fact $\lambda(dy;x,s) = \pi(dy;x,\mc{L}(X_s))$ $\nu_{\mc{L}(X_\cdot)}-$ almost surely.

\end{proof}
\subsection{Weak-Sense Uniqueness}
In order to prove the Laplace Principle Lower bound \eqref{eq:LPlowerbound} in Section \ref{sec:lowerbound} and compactness of level sets in Proposition \ref{prop:goodratefunction}, we will need to be able to identify a given $Z\in C([0,T];\mc{S}_{-w/r})$ using only the information that $Z$ solves the limiting controlled Equation \eqref{eq:MDPlimitFIXED} for some fixed $Q$. Hence, in this subsection, we prove an appropriate notion of weak-sense uniqueness for Equation \eqref{eq:MDPlimitFIXED}. Recall the space spaces $\mc{S}_{p},\mc{S}_{-p}$, and the related norms from the beginning of Section \ref{subsec:notationandtopology}.

\begin{lemma}\label{lemma:KurtzAppendixAnalogues}
Let $p\in\bb{N}$ and consider $\phi\in \mc{S}_{p+2}$, $F\in C_b^p(\R)$, and $G\in \mc{S}_p$. Then for any $\mu\in\mc{P}(\R)$, we have:
\begin{enumerate}
\item $\langle \phi, F\phi' \rangle_p \leq C\norm{\phi}^2_p$
\item $\langle \phi, F\phi''\rangle_p \leq C\norm{\phi}^2_p-\int_\R (1+x^2)^p |\phi^{(p+1)}(x)|^2 F(x)dx$
\item $\norm{\int_\R G(\cdot)\phi^{(k)}(z)\mu(dz)}_p\leq C\norm{\phi}_{k+1}$, for $k\leq p-1$.
\end{enumerate}
\end{lemma}
\begin{proof}
The proof of (1) follows by the same integration by parts argument as A1) in the Appendix of \cite{KX}. Part 2 follows by the same integration by parts argument as A2) in the Appendix of \cite{KX}. It becomes evident upon reading those proofs that $w_p\coloneqq (1+x^2)^p$ can be replaced by any $w_p$ such that $w^{-1}_pD^kw_p$ is bounded for all $k\leq p$. The proof of 3 is similar to the proof of A4)  in the Appendix of \cite{KX}. We recall it here:
\begin{align*}
\norm{\int_\R G(\cdot)\phi^{(k)}(z)\mu(dz)}_p& = \biggl(\sum_{j=0}^p \int_{\R}(1+x^2)^{2p}\biggl|\int_\R G^{(j)}(x)\phi^{(k)}(z)\mu(dz)\biggr|^2 dx\biggr)^{1/2}\\
&\leq \norm{G}_p \biggl(\int_\R |\phi^{(k)}(z)|^2\mu(dz)\biggr)^{1/2} \text{ by H\"older's inequality}\\
&\leq \norm{G}_p |\phi|_{k}\\
&\leq \norm{G}_p \norm{\phi}_{k+1}.
\end{align*}
\end{proof}

\begin{lemma}\label{lem:barLbounded}
Under assumption \ref{assumption:limitingcoefficientsregularity}, for any $p\in \br{1,...,w+2}$, where $w$ is as in Equation \eqref{eq:wdefinition}, and any $s\in [0,T]$, $\bar{L}_{\mc{L}(X_s)}$ as given in Equation \eqref{eq:MDPlimitFIXED}, where $X_s$ is as in Equation \eqref{eq:LLNlimitold}, is a bounded linear map from $\mc{S}_{p+2}$ to $\mc{S}_p$. In particular, there exists $c_p$ such that for all $s\in [0,T]$ and $\phi \in \mc{S}_{p+2}$,
\begin{align*}
\norm{\bar{L}_{\mc{L}(X_s)}\phi}_p\leq c_p \norm{\phi}_{p+2}.
\end{align*}
The same holds with $w$ replaced by $r$ from Equation \eqref{eq:rdefinition} if we in addition assume \ref{assumption:limitingcoefficientsregularityratefunction}.
\end{lemma}

\begin{proof}
Linearity is clear. For $\phi\in \mc{S}_{p+2}$ and $s\in [0,T]$,
\begin{align*}
\norm{\bar{\gamma}(\cdot,\mc{L}(X_s))\phi'(\cdot)}^2_p& =  \sum_{k=0}^p \int_\R (1+x^2)^{2p}\biggl([\bar{\gamma}(x,\mc{L}(X_s))\phi'(x)]^{(k)}\biggr)^2dx\\
&\leq c_p \sum_{k=0}^p \int_\R (1+x^2)^{2p}\biggl(\phi^{(k+1)}(x)\biggr)^2dx \text{ by Assumption \ref{assumption:limitingcoefficientsregularity}}\\
& \leq c_p \norm{\phi}^2_{p+1}.
\end{align*}
In the same way, we can see
$\norm{\bar{D}(\cdot,\mc{L}(X_s))\phi''(\cdot)}^2_p \leq c_p \norm{\phi}^2_{p+2}.$
In addition, we have
\begin{align*}
&\norm{\int_\R \frac{\delta}{\delta m}\bar{\gamma}(z,\mc{L}(X_s))[\cdot]\phi'(z)\mc{L}(X_s)(dz)}^2_{p}  = \sum_{k=0}^p \int_\R (1+x^2)^{2p}\biggl(\frac{\partial^k}{\partial x^k}\biggl[\int_\R \frac{\delta}{\delta m}\bar{\gamma}(z,\mc{L}(X_s))[x]\phi'(z)\mc{L}(X_s)(dz)\biggr]\biggr)^2dx \\
&\qquad\leq \int_{\R} \norm{\frac{\delta}{\delta m}\bar{\gamma}(z,\mc{L}(X_s))[\cdot]}^2_{p}\mc{L}(X_s)(dz)|\phi|^2_1 \text{ by Jensen's inequality and Tonelli's Theorem}\\
&\qquad\leq c_p \norm{\phi}_2^2 \text{ by Assumption \ref{assumption:limitingcoefficientsregularity} and the inequality \eqref{eq:sobolembedding}}.
\end{align*}
Again, in the same way, we can see
\begin{align*}
\norm{\int_\R \frac{\delta}{\delta m}\bar{D}(z,\mc{L}(X_s))[\cdot]\phi''(z)\mc{L}(X_s)(dz)}^2_{p} & \leq c_p \norm{\phi}_3^2,
\end{align*}
so by definition of $\bar{L}_\nu$, the result follows.
\end{proof}

\begin{lemma}\label{lemma:4.32BW}
Under Assumption \ref{assumption:limitingcoefficientsregularity}, we have for any $p\in \br{1,...,w}$ and $F\in \mc{S}_{-p}$, where $w$ is as in Equation \eqref{eq:wdefinition},
\begin{align*}
\sup_{s\in[0,T]}\langle F,\bar{L}^*_{\mc{L}(X_s)}F\rangle_{-(p+2)}\leq \norm{F}^2_{-(p+2)}
\end{align*}
where $\bar{L}^*_{\mc{L}(X_s)}:\mc{S}_{-p}\tto \mc{S}_{-(p+2)}$ is the adjoint of $\bar{L}_{\mc{L}(X_s)}:\mc{S}_{p+2}\tto \mc{S}_{p}$ given in Equation \eqref{eq:MDPlimitFIXED} (using here Lemma \ref{lem:barLbounded}). The same holds if instead we further assume \ref{assumption:limitingcoefficientsregularityratefunction} and replace $w$ with $r$ from Equation \eqref{eq:rdefinition}.
\end{lemma}
\begin{proof}
By the Riesz representation theorem we can take $\phi \in \mc{S}_{p}$ such that for all $\psi\in \mc{S}_{p}$, $\langle F,\psi\rangle=\langle \phi,\psi\rangle_{p}$ and $\norm{F}_{-p}=\norm{\phi}_{p}$. By a density argument, we may assume in fact that $\phi \in \mc{S},$ $\langle F,\psi\rangle=\langle \phi,\psi\rangle_{p+2}$, and $\norm{\phi}_{p+2}=\norm{F}_{-(p+2)}$. Then for any $s\in[0,T]$, $\langle F,\bar{L}^*_{\mc{L}(X_s)}F\rangle_{-(p+2)} = \langle F,\bar{L}_{\mc{L}(X_s)}\phi\rangle = \langle \phi,\bar{L}_{\mc{L}(X_s)}\phi\rangle_{p+2}.$ Then,
\begin{align*}
&\langle \phi,\bar{L}_{\mc{L}(X_s)}\phi\rangle_{p+2}  = \nonumber\\
&=\langle \phi,\bar{\gamma}(\cdot,\mc{L}(X_s))\phi'(\cdot)\rangle_{p+2}+\langle \phi,\bar{D}(\cdot,\mc{L}(X_s))\phi''(\cdot)\rangle_{p+2}+\langle \phi,\int_\R \frac{\delta}{\delta m}\bar{\gamma}(z,\mc{L}(X_s))[\cdot]\phi'(z)\mc{L}(X_s)(dz)\rangle_{p+2} \\
&+ \langle \phi,\int_\R \frac{\delta}{\delta m}\bar{D}(z,\mc{L}(X_s))[\cdot]\phi''(z)\mc{L}(X_s)(dz)\rangle_{p+2} \\
&\leq \langle \phi,\bar{\gamma}(\cdot,\mc{L}(X_s))\phi'(\cdot)\rangle_{p+2}+\langle \phi,\bar{D}(\cdot,\mc{L}(X_s))\phi''(\cdot)\rangle_{p+2}+\norm{\phi}_{p+2}\norm{\int_\R \frac{\delta}{\delta m}\bar{\gamma}(z,\mc{L}(X_s))[\cdot]\phi'(z)\mc{L}(X_s)(dz)}_{p+2} \\
&+ \norm{\phi}_{p+2}\norm{\int_\R \frac{\delta}{\delta m}\bar{D}(z,\mc{L}(X_s))[\cdot]\phi''(z)\mc{L}(X_s)(dz)}_{p+2}\text{by Cauchy Schwarz}\\
&\leq C\biggl\lbrace \norm{\phi}^2_{p+2} + \norm{\phi}_{p+2}\norm{\phi}_{2}+\norm{\phi}_{p+2}\norm{\phi}_{3} \biggr\rbrace \text{ by Lemma \ref{lemma:KurtzAppendixAnalogues} and Assumption \ref{assumption:limitingcoefficientsregularity}}\\
&\leq C\norm{\phi}^2_{p+2}\\
& = C\norm{F}^2_{-(p+2)}.
\end{align*}
The proof follows in the same way if we replace $w$ with $r$.
\end{proof}

\begin{proposition}\label{proposition:weakuninqueness}
Under Assumption \ref{assumption:limitingcoefficientsregularity}, for any $(Z,Q)$ and $(\tilde{Z},Q)$ such that $Q\in P^*(Z)$ and $Q\in P^*(\tilde{Z})$, $Z=\tilde{Z}$ as elements of $C([0,T];\mc{S}_{-w})$. If we assume \ref{assumption:limitingcoefficientsregularityratefunction} instead of \ref{assumption:limitingcoefficientsregularity}, $Z=\tilde{Z}$ as elements of $C([0,T];\mc{S}_{-r})$.
\end{proposition}

\begin{proof}
Consider $\eta = Z-\tilde{Z}$. Then by virtue of \ref{PZ:limitingequation} in the definition of $P^*$, $\eta$ almost surely satisfies
\begin{align*}
\langle \eta_t,\phi\rangle = \int_0^t \langle \eta_s,\bar{L}_{\mc{L}(X_s)}\phi(\cdot)\rangle ds
\end{align*}
for all $t\in [0,T]$ and $\phi \in \mc{S}_w$. Let $\br{\phi_j^{w+2}}_{j\in\bb{N}}$ be an orthonormal basis for $\mc{S}_{-(w+2)}$. By chain rule, we have
\begin{align*}
\langle \eta_t,\phi_j^{w+2}\rangle^2 & = 2\int_0^t \langle  \eta_s,\phi_j^{w+2}\rangle\langle\eta_s,\bar{L}_{\mc{L}(X_s)}\phi_j^{w+2}(\cdot)\rangle ds.
\end{align*}

Summing through $j$, we have using Parseval's identity, Riesz representation theorem, and linearity of $\eta_s$ and $\bar{L}_{\mc{L}(X_s)}$ that
\begin{align*}
\norm{\eta(t)}_{-(w+2)} & = 2 \int_0^t \langle \eta(s),\bar{L}^*_{\mc{L}(s)}\eta(s)ds\rangle_{-(w+2)}ds\leq C \int_0^t \norm{\eta(s)}_{-(w+2)} ds \text{ by Lemma \ref{lemma:4.32BW}}
\end{align*}
so by Gronwall's inequality, $\norm{\eta(t)}_{-(w+2)}=0,\forall t\in [0,T]$, so $\norm{\eta(t)}_{-w}=0,\forall t\in [0,T]$, and hence $Z=\tilde{Z}$.
The proof follows in the same way if we replace $w$ with $r$.
\end{proof}
\begin{remark}
By \ref{PZ:secondmarginalinvtmeasure} and \ref{PZ:fourthmarginallimitnglaw} in the definition of $P^*$, we have that for any $Q\in P^*(Z)$ that disintegrating $Q(dx,dy,dz,ds)=\kappa(dz;x,y,s)\lambda(dy;x,s)Q_{(1,4)}(dx,ds)$, $\lambda(dy;x,s) = \pi(dy;x,\mc{L}(X_s))$ and $Q_{(1,4)}(dx,ds) = \mc{L}(X_s)(dx)ds=\nu_{\mc{L}(X_\cdot)}(dx,ds)$. Thus any $Q,\tilde{Q}\in P^*(Z)$ only differentiate in their control stochastic kernels, $\kappa(dz;x,y,s)$ and $\tilde{\kappa}(dz;x,y,s)$. These are, of course, entirely determined by the choice of controls in the construction of $Q^N$. In other words, keeping in mind the result of Proposition \ref{prop:viabilityoflimit}, the choice of controls in the prelimit system \eqref{eq:controlledslowfast1-Dold} determine uniquely the limit in distribution of $\tilde{Z}^N.$ 
\end{remark}

\section{Laplace principle Lower Bound}\label{sec:upperbound/compactnessoflevelsets}
We now can prove the Laplace principle Lower Bound \eqref{eq:LPupperbound}.
\begin{proposition}\label{prop:LPUB}
Under assumptions \ref{assumption:uniformellipticity} - \ref{assumption:limitingcoefficientsregularity}, Equation \eqref{eq:LPupperbound} holds.
\end{proposition}
\begin{proof}
Take $\tau\geq w$, with $w$ as in Equation \eqref{eq:wdefinition}, $F\in C_b(C([0,T];\mc{S}_{-\tau}))$ and $\eta>0$. By Equation \eqref{eq:varrepfunctionalsBM} there exists $\br{\tilde{u}^N}_{N\in\bb{N}}$ such that for all $N$,
\begin{align*}
-a^2(N)\log \E \exp\biggl(-\frac{1}{a^2(N)}F(Z^N) \biggr)\geq \E\biggl[\frac{1}{2}\frac{1}{N}\sum_{i=1}^N \int_0^T|\tilde{u}^{N,1}_i(s)|^2+|\tilde{u}^{N,2}_i(s)|^2ds +F(\tilde{Z}^N)\biggr]-\eta.
\end{align*}
Where $\tilde{Z}^N$ is as in Equation \eqref{eq:controlledempmeasure} and is controlled by $\br{\tilde{u}^N}_{N\in\bb{N}}$. Then letting $Q^N$ be as in Equation $\eqref{eq:occupationmeasures}$ with this choice of controls (recalling that we can assume the almost-sure bound \eqref{eq:controlassumptions} on the controls by the argument found in Theorem 4.4 of \cite{BD}), we have
\begin{align*}
\E\biggl[\frac{1}{2}\frac{1}{N}\sum_{i=1}^N \int_0^T\left(|\tilde{u}^{N,1}_i(s)|^2+|\tilde{u}^{N,2}_i(s)|^2\right)ds +F(\tilde{Z}^N)\biggr] = \E\biggl[\frac{1}{2}\int_{\R\times\R\times\R^2\times[0,T]}\left(|z_1|^2+|z_2|^2\right)Q^N(dxdydzds) +F(\tilde{Z}^N)\biggr]
\end{align*}
so by the version of Fatou's lemma from Theorem A.3.12 in \cite{DE}, we have
\begin{align*}
&\liminf_{N\toinf}-a^2(N)\log \E \exp\biggl(-\frac{1}{a^2(N)}F(Z^N) \biggr)\\
&\geq \liminf_{N\toinf}\E\biggl[\frac{1}{2}\int_{\R\times\R\times\R^2\times[0,T]}\left(|z_1|^2+|z_2|^2\right)Q^N(dxdydzds) +F(\tilde{Z}^N)\biggr]-\eta\\
&\geq \E\biggl[\liminf_{N\toinf}\frac{1}{2}\int_{\R\times\R\times\R^2\times[0,T]}\left(|z_1|^2+|z_2|^2\right)Q^N(dxdydzds) +F(\tilde{Z}^N)\biggr]-\eta\\
&\geq \E\biggl[\frac{1}{2}\int_{\R\times\R\times\R^2\times[0,T]}\left(|z_1|^2+|z_2|^2\right)Q(dxdydzds) +F(Z)\biggr]-\eta \\
&\geq \inf_{Z\in C([0,T];\mc{S}_{-m})}\biggl\lbrace\inf_{Q\in P^*(Z)}\biggl\lbrace\frac{1}{2}\int_{\R\times\R\times\R^2\times[0,T]}\left(|z_1|^2+|z_2|^2\right)Q(dxdydzds)\biggr\rbrace +F(Z)\biggr\rbrace-\eta\\
&\geq \inf_{Z\in C([0,T];\mc{S}_{-w})}\biggl\lbrace\inf_{Q\in P^*(Z)}\biggl\lbrace\frac{1}{2}\int_{\R\times\R\times\R^2\times[0,T]}\left(|z_1|^2+|z_2|^2\right)Q(dxdydzds)\biggr\rbrace +F(Z)\biggr\rbrace-\eta\\
& = \inf_{Z\in C([0,T];\mc{S}_{-w})}\lbrace I(Z) +F(Z)\rbrace-\eta,
\end{align*}
where in the second-to-last inequality we used Proposition \ref{prop:viabilityoflimit}. So Equation \eqref{eq:LPupperbound} is established. 
\end{proof}

\section{Laplace principle Upper Bound and Compactness of Level Sets}\label{sec:lowerbound}
We now prove the Laplace principle Upper Bound and, under the additional assumption of \ref{assumption:limitingcoefficientsregularityratefunction}, compactness of level sets.
\begin{proposition}\label{prop:LPlowerbound}
Under assumptions \ref{assumption:uniformellipticity} - \ref{assumption:limitingcoefficientsregularity}, the Laplace principle Upper Bound \eqref{eq:LPlowerbound} holds.
\end{proposition}
\begin{proof}
We use the ordinary formulation $I^o$ from Equation \eqref{eq:proposedjointratefunctionordinary}. We take $\eta>0$, $w$ as in Equation \eqref{eq:wdefinition}, $\tau\geq w$, $F\in C_b(C([0,T];\mc{S}_{-\tau}))$, and $Z^*$ such that
\begin{align*}
I(Z^*)+F(Z^*)\leq \inf_{Z\in C([0,T];\mc{S}_{-w})}\E\biggl[I(Z) +F(Z)\biggr]+\frac{\eta}{2}.
\end{align*}
Then we can find $h\in P^o(Z^*)$ such that
\begin{align*}
\frac{1}{2}\int_{0}^T \E\biggl[\int_\R |h(s,X_s,y)|^2 \pi(dy;X_s,\mc{L}(X_s)) \biggr]ds\leq I(Z^*)+\frac{\eta}{2}.
\end{align*}
Then since $\nu(\Gamma\times A\times B)\coloneqq \int_\Gamma \int_\R \1_A(x)\int_\R \1_B(y) \pi(dy;x,\mc{L}(X_s))\mc{L}(X_s)(dx)ds$, $\Gamma\in\mc{B}(U),A,B\in\mc{B}(\R)$ is a finite Borel measure on $U\times\R\times \R$ for all $x\in\R$, by Corollary 22.38 (1) in \cite{Driver}, we can take $\br{\psi^k_j}_{k\in\bb{N}}\subset C^\infty_c(U\times\R\times\R)$ such that $\psi^k_j\tto h_j$ in $L^2(U\times\R\times\R,\nu)$ for $j\in \br{1,2}$. Here we let $U$ be any open interval containing $[0,T]$ and assume $\nu(\Gamma\times A\times B)$ is $0$ when $\Gamma \cap [0,T]=\emptyset$.

Then letting $\tilde{u}^{N}_{i,k}(s,\omega) = \psi^k(s,\tilde{X}^{i,\epsilon,N,k}_s(\omega),\tilde{Y}^{i,\epsilon,N,k}_s(\omega))$, where $(\tilde{X}^{i,\epsilon,N,k}_s,\tilde{Y}^{i,\epsilon,N,k}_s)$ are as in Equation \eqref{eq:controlledslowfast1-Dold} but controlled by $\frac{\tilde{u}^{N}_{i,k}(s)}{a(N)\sqrt{N}}$,
\begin{align*}
\sup_{N\in\bb{N}}\int_0^T\E\biggl[\frac{1}{N}\sum_{i=1}^N |\tilde{u}^{N}_{i,k}(s)|^2\biggr]ds & = \sup_{N\in\bb{N}}\int_0^T\E\biggl[\frac{1}{N}\sum_{i=1}^N |\psi_k(s,\tilde{X}^{i,\epsilon,N}_s,\tilde{Y}^{i,\epsilon,N}_s)|^2\biggr]ds\leq T \norm{\psi_k}^2_\infty
\end{align*}
for each $k\in\bb{N}$, and in fact
\begin{align*}
\int_0^T\frac{1}{N}\sum_{i=1}^N |\tilde{u}^{N}_{i,k}(s)|^2ds\leq T \norm{\psi_k}^2_\infty
\end{align*}
for each $k\in\bb{N}$ (so the supposition \eqref{eq:controlassumptions} holds with this choice of controls).

Letting $(\tilde{Z}^{N,k},Q^{N,k})$ be as in Equations \eqref{eq:controlledempmeasure} and \eqref{eq:occupationmeasures} with this choice of controls, we want to establish that $(\tilde{Z}^{N,k},Q^{N,k})$ converges in distribution as a sequence of $C([0,T];\mc{S}_{-m})\times M_T(\R^4)$-valued random variables to $(\tilde{Z}^k,Q^k)$ as $N\toinf$, where $Q^k\in P^*(\tilde{Z}^k)$ (this is immediate since we prove this for all $L^2$ controls in Proposition \ref{prop:viabilityoflimit}) and such that
\begin{align}\label{eq:QNkdesiredform}
Q^{k}(A\times B\times C\times \Gamma) = \int_\Gamma \int_A \int_C \delta_{\psi^k(s,x,y)}(C)\pi(dy;x,\mc{L}(X_s))\mc{L}(X_s)(dx)ds,\forall A,B\in\mc{B}(\R),C\in\mc{B}(\R^2),\Gamma\in\mc{B}([0,T]).
\end{align}
By the weak-sense uniqueness established in Proposition \ref{proposition:weakuninqueness}, this determines each $\tilde{Z}^k$ almost surely to be the unique element of $C([0,T];\mc{S}_{-m})$ satisfying Equation \eqref{eq:MDPlimitFIXED} with $Q^k$ in the place of $Q$.

Then we will send $k\toinf$ and show $(\tilde{Z}^k,Q^k)$ converges to $(\tilde{Z},Q)$ in $C([0,T];\mc{S}_{-w})\times M_T(\R^4)$, where $Q\in P^*(\tilde{Z})$ and
\begin{align}\label{eq:QNdesiredform}
Q(A\times B\times C\times \Gamma) = \int_\Gamma \int_A \int_C \delta_{h(s,x,y)}(C)\pi(dy;x,\mc{L}(X_s))\mc{L}(X_s)(dx)ds,\forall A,B\in\mc{B}(\R),C\in\mc{B}(\R^2),\Gamma\in\mc{B}([0,T]).
\end{align}

Then by the weak-sense uniqueness established in Proposition \ref{proposition:weakuninqueness}, we have $\tilde{Z}\overset{d}{=}Z^*$. By reverse Fatou's lemma:
\begin{align*}
&\limsup_{N\toinf}-a^2(N)\log \E \exp\biggl(-\frac{1}{a^2(N)}F(Z^N) \biggr) \\
& = \limsup_{N\toinf}\inf_{\tilde{u}^N}\E\biggl[\frac{1}{2}\frac{1}{N}\sum_{i=1}^N \int_0^T\left(|\tilde{u}^{N,1}_i(s)|^2+|\tilde{u}^{N,2}_i(s)|^2\right)ds +F(\tilde{Z}^N)\biggr] \text{ by Equation \eqref{eq:varrepfunctionalsBM}}\\
&\leq \limsup_{N\toinf}\E\biggl[\frac{1}{2}\frac{1}{N}\sum_{i=1}^N \int_0^T\left(|\tilde{u}^{N,1}_{i,k}(s)|^2+|\tilde{u}^{N,2}_{i,k}(s)|^2\right)ds +F(\tilde{Z}^{N,k})\biggr],\forall k\in\bb{N}\\
&= \limsup_{N\toinf}\E\biggl[\frac{1}{2}\int_{\R\times\R\times\R^2\times[0,T]}\left(z_1^2+z_2^2 \right)Q^{N,k}(dx,dy,dz,ds) +F(\tilde{Z}^{N,k})\biggr],\forall k\in\bb{N}\\
&\leq \E\biggl[\frac{1}{2}\int_{\R\times\R\times\R^2\times[0,T]}\left(z_1^2+z_2^2 \right)Q^k(dx,dy,dz,ds) +F(\tilde{Z}^k)\biggr],\forall k\in\bb{N}\\
& = \frac{1}{2}\int_{0}^T \E\biggl[\int_\R |\psi^k(s,X_s,y)|^2 \pi(dy;X_s,\mc{L}(X_s)) \biggr]ds +\E\biggl[F(\tilde{Z}^k)\biggr],\forall k\in\bb{N}
\end{align*}

Then sending $k\toinf$ and using the $L^2$ convergence of $\psi^k$ to $h$ and the boundedness $F$ and convergence of $\tilde{Z}^k$ to $Z^*$, we get
\begin{align*}
\limsup_{N\toinf}-a^2(N)\log \E \exp\biggl(-\frac{1}{a^2(N)}F(Z^N) \biggr) &\leq \frac{1}{2}\int_{0}^T \E\biggl[\int_\R |h(s,X_s,y)|^2 \pi(dy;X_s,\mc{L}(X_s)) \biggr]ds +\E\biggl[F(Z^*)\biggr]\\
&\leq I(Z^*)+F(Z^*)+\frac{\eta}{2}\\
&\leq \inf_{Z\in C([0,T];\mc{S}_{-w})}\E\biggl[I(Z) +F(Z)\biggr] +\eta
\end{align*}
so Equation \eqref{eq:LPlowerbound} will be established.

Looking at the proof of Proposition \ref{prop:limitsatisfiescorrectequations}, to see $(\tilde{Z}^{N,k},Q^{N,k})$ converges to $(\tilde{Z}^k,Q^k)$ where $Q^k\in P^*(\tilde{Z}^k)$ satisfies Equation \eqref{eq:QNkdesiredform}. we just need to establish that

\begin{align*}
&\int_{\R\times\R\times\R^2\times[0,t]} \sigma(x,y,\tilde{\mu}^{\epsilon,N,k}_s)z_1 \phi'(x)Q^{N,k}(dx,dy,dz,ds)\\
&+\int_{\R\times\R\times\R^2\times[0,t]} [\tau_1(x,y,\tilde{\mu}^{\epsilon,N,k}_s)z_1+\tau_2(x,y,\tilde{\mu}^{\epsilon,N,k}_s)z_2]\Phi_y(x,y,\tilde{\mu}^{\epsilon,N,k}_s)\phi'(x)Q^{N,k}(dx,dy,dz,ds)
\end{align*}
converges in distribution to
\begin{align*}
&\int_0^t\E\biggl[\int_\R\sigma(X_s,y,\mc{L}(X_s))\psi^k_1(s,X_s,y) \phi'(X_s)\pi(dy;X_s,\mc{L}(X_s))\biggr]ds\\
&+\int_0^t\E\biggl[\int_\R \left([\tau_1(X_s,y,\mc{L}(X_s))\psi^k_1(s,X_s,y) +\tau_2(X_s,y,\mc{L}(X_s))\psi^k_2(s,X_s,y) ]\Phi_y(X_s,y,\mc{L}(X_s))\phi'(X_s)\right)\pi(dy;X_s,\mc{L}(X_s))\biggr]ds
\end{align*}
for all $\phi \in C^\infty_c(\R)$ and $t\in [0,T]$. Fix $k\in\bb{N}$ and $\phi$ and $t$. We have
\begin{align*}
&\int_{\R\times\R\times\R^2\times[0,t]} \sigma(x,y,\tilde{\mu}^{\epsilon,N,k}_s)z_1 \phi'(x)Q^{N,k}(dx,dy,dz,ds)\\
&+\int_{\R\times\R\times\R^2\times[0,t]} [\tau_1(x,y,\tilde{\mu}^{\epsilon,N,k}_s)z_1+\tau_2(x,y,\tilde{\mu}^{\epsilon,N,k}_s)z_2]\Phi_y(x,y,\tilde{\mu}^{\epsilon,N,k}_s)\phi'(x)Q^{N,k}(dx,dy,dz,ds)\\
&=\int_0^t \frac{1}{N}\sum_{i=1}^N \sigma(\tilde{X}^{i,\epsilon,N,k}_s,\tilde{Y}^{i,\epsilon,N,k}_s,\tilde{\mu}^{\epsilon,N,k}_s)\psi^k_1(s,\tilde{X}^{i,\epsilon,N,k}_s,\tilde{Y}^{i,\epsilon,N,k}_s) \phi'(\tilde{X}^{i,\epsilon,N,k}_s)ds\\
&+\int_0^t \frac{1}{N}\sum_{i=1}^N  \left[\tau_1(\tilde{X}^{i,\epsilon,N,k}_s,\tilde{Y}^{i,\epsilon,N,k}_s,\tilde{\mu}^{\epsilon,N,k}_s)\psi^k_1(s,\tilde{X}^{i,\epsilon,N,k}_s,\tilde{Y}^{i,\epsilon,N,k}_s)+\right.\nonumber\\
&\qquad\left.+\tau_2(\tilde{X}^{i,\epsilon,N,k}_s,\tilde{Y}^{i,\epsilon,N,k}_s,\tilde{\mu}^{\epsilon,N,k}_s)\psi^k_2(s,\tilde{X}^{i,\epsilon,N,k}_s,\tilde{Y}^{i,\epsilon,N,k}_s)\right]
\Phi_y(\tilde{X}^{i,\epsilon,N,k}_s,\tilde{Y}^{i,\epsilon,N,k}_s,\tilde{\mu}^{\epsilon,N,k}_s\phi'(\tilde{X}^{i,\epsilon,N,k}_s)ds
\end{align*}

Then using Proposition \ref{prop:llntypefluctuationestimate1} with
\begin{align*}
F(s,x,y,\mu) = \sigma(x,y,\mu)\psi^k_1(s,x,y) + [\tau_1(x,y,\mu)\psi^k_1(s,x,y)+\tau_2(x,y,\mu)\psi^k_2(s,x,y)]\Phi_y(x,y,\mu)
\end{align*}
using that $s$ only appears as a parameter, in the same way as $x$, so that the same proof holds (using also the assumed bound on the time derivative of $\Xi$ in \ref{assumption:forcorrectorproblem}), we get that
\begin{align*}
&\E\biggl[\biggl|\int_0^t \frac{1}{N}\sum_{i=1}^N \sigma(\tilde{X}^{i,\epsilon,N,k}_s,\tilde{Y}^{i,\epsilon,N,k}_s,\tilde{\mu}^{\epsilon,N,k}_s)\psi^k_1(s,\tilde{X}^{i,\epsilon,N,k}_s,\tilde{Y}^{i,\epsilon,N,k}_s) \phi'(\tilde{X}^{i,\epsilon,N,k}_s)ds\\
&+\int_0^t \frac{1}{N}\sum_{i=1}^N  \left[\tau_1(\tilde{X}^{i,\epsilon,N,k}_s,\tilde{Y}^{i,\epsilon,N,k}_s,\tilde{\mu}^{\epsilon,N,k}_s)\psi^k_1(s,\tilde{X}^{i,\epsilon,N,k}_s,\tilde{Y}^{i,\epsilon,N,k}_s)+\right.\nonumber\\
&\qquad\left.+\tau_2(\tilde{X}^{i,\epsilon,N,k}_s,\tilde{Y}^{i,\epsilon,N,k}_s,\tilde{\mu}^{\epsilon,N,k}_s)\psi^k_2(s,\tilde{X}^{i,\epsilon,N,k}_s,\tilde{Y}^{i,\epsilon,N,k}_s)\right]
\Phi_y(\tilde{X}^{i,\epsilon,N,k}_s,\tilde{Y}^{i,\epsilon,N,k}_s,\tilde{\mu}^{\epsilon,N,k}_s)\phi'(\tilde{X}^{i,\epsilon,N,k}_s)ds \\
&- \int_0^t \frac{1}{N}\sum_{i=1}^N \int_\R \sigma(\tilde{X}^{i,\epsilon,N,k}_s,y,\tilde{\mu}^{\epsilon,N,k}_s)\psi^k_1(s,\tilde{X}^{i,\epsilon,N,k}_s,y) \phi'(\tilde{X}^{i,\epsilon,N,k}_s)\\
&\quad+ [\tau_1(\tilde{X}^{i,\epsilon,N,k}_s,y,\tilde{\mu}^{\epsilon,N,k}_s)\psi^k_1(s,\tilde{X}^{i,\epsilon,N,k}_s,y)+\tau_2(\tilde{X}^{i,\epsilon,N,k}_s,y,\tilde{\mu}^{\epsilon,N,k}_s)\psi^k_2(s,\tilde{X}^{i,\epsilon,N,k}_s,y)]\times\\
&\hspace{5cm}\times\Phi_y(\tilde{X}^{i,\epsilon,N,k}_s,y,\tilde{\mu}^{\epsilon,N,k}_s)\psi(s,\tilde{X}^{i,\epsilon,N,k}_s)\pi(dy;\tilde{X}^{i,\epsilon,N}_s,\tilde{\mu}^{\epsilon,N}_s)ds\biggr|\biggr]\\
&\leq C(T)\epsilon
\end{align*}
Then noting that
\begin{align*}
&\int_0^t \frac{1}{N}\sum_{i=1}^N \int_\R \sigma(\tilde{X}^{i,\epsilon,N,k}_s,y,\tilde{\mu}^{\epsilon,N,k}_s)\psi^k_1(s,\tilde{X}^{i,\epsilon,N,k}_s,y) \phi'(\tilde{X}^{i,\epsilon,N,k}_s)\\
&+ [\tau_1(\tilde{X}^{i,\epsilon,N,k}_s,y,\tilde{\mu}^{\epsilon,N,k}_s)\psi^k_1(s,\tilde{X}^{i,\epsilon,N,k}_s,y)+\tau_2(\tilde{X}^{i,\epsilon,N,k}_s,y,\tilde{\mu}^{\epsilon,N,k}_s)\psi^k_2(s,\tilde{X}^{i,\epsilon,N,k}_s,y)]\times\\
&\hspace{5cm}\times\Phi_y(\tilde{X}^{i,\epsilon,N,k}_s,y,\tilde{\mu}^{\epsilon,N,k}_s)\phi'(\tilde{X}^{i,\epsilon,N,k}_s)\pi(dy;\tilde{X}^{i,\epsilon,N}_s,\tilde{\mu}^{\epsilon,N}_s)ds\\
& = \int_0^t \int_\R \int_\R \biggl(\sigma(x,y,\tilde{\mu}^{\epsilon,N,k}_s)\psi^k_1(s,x,y) \phi'(x)+ [\tau_1(x,y,\tilde{\mu}^{\epsilon,N,k}_s)\psi^k_1(s,x,y)+\tau_2(x,y,\tilde{\mu}^{\epsilon,N,k}_s)\psi^k_2(s,x,y)]\times\\
&\hspace{5cm}\times\Phi_y(x,y,\tilde{\mu}^{\epsilon,N,k}_s)\phi'(x)\biggr)\pi(dy;x,\tilde{\mu}^{\epsilon,N}_s)\tilde{\mu}^{\epsilon,N,k}_s(dx)ds
\end{align*}
and using that the integrand of the first two integrals above is bounded by Assumptions \ref{assumption:uniformellipticity},\ref{assumption:gsigmabounded},and \ref{assumption:multipliedpolynomialgrowth} and continuous in $\bb{W}_2$ by Assumption \ref{assumption:2unifboundedlinearfunctionalderivatives}, along with the convergence of $\tilde{\mu}^{\epsilon,N}_s$ to $\mc{L}(X_s)$ from Lemma \ref{lemma:W2convergenceoftildemu}, we have by dominated convergence theorem (invoking here Skorokhod's representation theorem to assume $\tilde{\mu}^{\epsilon,N}_s$ to $\mc{L}(X_s)$ almost surely as in Proposition \ref{prop:limitsatisfiescorrectequations}) and Theorem A.3.18 in \cite{DE}:

\begin{align*}
&\lim_{N\toinf}\E\biggl[\biggl|\int_0^T \int_\R \int_\R \biggl(\sigma(x,y,\tilde{\mu}^{\epsilon,N,k}_s)\psi^k_1(s,x,y) \phi'(x)+ [\tau_1(x,y,\tilde{\mu}^{\epsilon,N,k}_s)\psi^k_1(s,x,y)+\tau_2(x,y,\tilde{\mu}^{\epsilon,N,k}_s)\psi^k_2(s,x,y)]\\
&\hspace{5cm}\times\Phi_y(x,y,\tilde{\mu}^{\epsilon,N,k}_s)\phi'(x)\biggr)\pi(dy;x,\tilde{\mu}^{\epsilon,N}_s)\tilde{\mu}^{\epsilon,N,k}_s(dx)ds \\
&- \int_0^T \int_\R \int_\R \biggl(\sigma(x,y,\mc{L}(X_s))\psi^k_1(s,x,y) \phi'(x)+ [\tau_1(x,y,\mc{L}(X_s))\psi^k_1(s,x,y)+\tau_2(x,y,\mc{L}(X_s))\psi^k_2(s,x,y)]\\
&\hspace{5cm}\times\Phi_y(x,y,\mc{L}(X_s))\phi'(x)\biggr)\pi(dy;x,\mc{L}(X_s))\mc{L}(X_s)(dx)ds \biggr|\biggr]=0
\end{align*}
so by triangle inequality, the desired convergence is shown.

Now we seek to establish that $(\tilde{Z}^k,Q^k)$ converges to $(\tilde{Z},Q)$ in $C([0,T];\mc{S}_{-w})\times M_T(\R^4)$ where $Q\in P^*(\tilde{Z})$ and $Q$ satisfies \eqref{eq:QNdesiredform}.

We first prove precompactness. We have since $\psi^k\tto h$ in $L^2(U\times\R\times\R,\nu)$,
\begin{align*}
&\sup_{k\in\bb{N}}\int_{\R\times\R\times\R^2\times [0,T]} \left(z_1^2+z_2^2\right) Q^k(dx,dy,dz,ds) =\nonumber\\
&\quad= \sup_{k\in\bb{N}}\int_0^T\int_\R\int_\R \left(|\psi^k_1(s,x,y)|^2+|\psi^k_2(s,x,y)|^2\right) \pi(dy;x,\mc{L}(X_s))\mc{L}(X_s)(dx)ds <\infty.
\end{align*}
Moreover, by \ref{PZ:secondmarginalinvtmeasure} and \ref{PZ:fourthmarginallimitnglaw}
\begin{align*}
\int_{\R\times\R\times\R^2\times [0,T]} \left(y^2+|x|^{2}\right) Q^k(dx,dy,dz,ds) = \int_0^T\E\biggl[\int_{\R}y^2\pi(dy;X_s,\mc{L}(X_s)) + |X_s|^2 \biggr] ds <\infty,\forall k\in\bb{N},
\end{align*}
where here we have used that $\pi(\cdot;x,\mu)$ from Equation \eqref{eq:invariantmeasureold} has bounded moments of all orders uniformly in $x\in\R,\mu\in\mc{P}_2(\R)$ and that for $X_s$ from Equation \eqref{eq:LLNlimitold}, $\sup_{s\in [0,T]}\E[|X_s|^2]<\infty,$ which follows easily from the fact that $\bar{\gamma}$ and $\bar{D}$ are bounded as per Assumption \ref{assumption:limitingcoefficientsregularity}. Thus, via the same tightness function used for Proposition \ref{prop:QNtightness}, $\br{Q^k}_{k\in\bb{N}}$ is tight in $M_T(\R^4)$.

To see that $\br{\tilde{Z}^k}_{k\in\bb{N}}$ is precompact, we use that for each $k$ $(\tilde{Z}^k,Q^k)$ must satisfy Equation \eqref{eq:MDPlimitFIXED}. That is, for $\phi\in \mc{S}$ and $t\in [0,T]$:
\begin{align*}
\langle \tilde{Z}^k_t,\phi\rangle & = \int_0^t \langle \tilde{Z}^k_s,\bar{L}_{\mc{L}(X_s)}\phi\rangle ds + \int_0^t \langle B^k_s, \phi \rangle ds\\
\langle B^k_t, \phi\rangle &\coloneqq \int_{\R\times\R\times\R^2} \biggl(\sigma(x,y,\mc{L}(X_s))z_1 \phi'(x)\\
&\hspace{2cm}+ [\tau_1(x,y,\mc{L}(X_s))z_1+\tau_2(x,y,\mc{L}(X_s))z_2]\Phi_y(x,y,\mc{L}(X_s))\phi'(x)\biggr)Q_t^k(dx,dy,dz).
\end{align*}
Here $Q^k_t\in\mc{P}(\R^4)$ is such that $Q^k(dx,dy,dz,dt)=Q^k_t(dx,dy,dz)dt$. We can see that $B^k_t\in \mc{S}_{-(m+2)}$ for almost every $t\in [0,T]$, and in fact
\begin{align*}
\sup_{k\in\bb{N}}\int_0^T\norm{B^k_s}^2_{-(m+2)}ds&= \sup_{k\in\bb{N}}\int_0^T\sup_{\norm{\phi}_{m+2}=1}|\langle B^k_s,\phi \rangle|^2 ds\\
&\leq \sup_{k\in\bb{N}}\int_0^T\sup_{\norm{\phi}_{m+2}=1}\biggl\lbrace\int_{\R\times\R\times\R^2} \left(|z_1|^2+|z_2|^2\right)Q^k_s(dx,dy,dz)|\phi|^2_1\biggr\rbrace ds\\
&\leq \sup_{k\in\bb{N}}\int_0^T\sup_{\norm{\phi}_{m+2}=1}\biggl\lbrace\int_{\R\times\R\times\R^2} \left(|z_1|^2+|z_2|^2\right)Q^k_s(dx,dy,dz)\norm{\phi}^2_{m+2}\biggr\rbrace ds\\
& = \sup_{k\in\bb{N}}\int_0^T\int_{\R\times\R\times\R^2} \left(|z_1|^2+|z_2|^2\right)Q^k_s(dx,dy,dz) ds<\infty
\end{align*}

Thus, by the proof of Theorem 2.5.2 in \cite{KalX}, it suffices to show that for fixed $\phi \in \mc{S}$, $\langle \tilde{Z}^k_t,\phi\rangle$ is relatively compact in $C([0,T];\R)$, and $\tilde{Z}^k$ is uniformly $(m+2)$-continuous to get precompactness of $\tilde{Z}^k$ in $C([0,T];\mc{S}_{-w})$ for $w>m+2$ sufficiently large that the canonical embedding $\mc{S}_{-m-2}\tto \mc{S}_{-w}$ is Hilbert-Schmidt (see Equation \eqref{eq:wdefinition}).

We have that, in the same way as the proof of Proposition \ref{proposition:weakuninqueness} (using here that $\tilde{Z}^k\in C([0,T];\mc{S}_{-m})$), 
\begin{align*}
\norm{\tilde{Z}^k_t}_{-(m+2)}^2& =  2\int_0^t \langle \tilde{Z}^k_s,L^*_{\mc{L}(X_s)}\tilde{Z}^k_s\rangle_{-(m+2)} ds + 2\int_0^t \langle \tilde{Z}^k_s, B^k_s \rangle_{-(m+2)} ds\\
&\leq C\int_0^t \norm{\tilde{Z}^k_s}^2_{-(m+2)}ds +2\int_0^t \norm{\tilde{Z}^k_s}_{-(m+2)}\norm{B^k_s}_{-(m+2)} ds \text{ by Cauchy Schwarz and Lemma \ref{lemma:4.32BW}}\\
&\leq C\biggl\lbrace \int_0^t \norm{\tilde{Z}^k_s}^2_{-(m+2)}ds + \int_0^t \norm{B^k_s}^2_{-(m+2)} ds\biggr\rbrace
\end{align*}
so by Gronwall's inequality,
\begin{align*}
\sup_{k\in\bb{N}}\sup_{t\in [0,T]}\norm{\tilde{Z}^k_t}^2_{-(m+2)}\leq C(T).
\end{align*}

This gives then that for $t_1,t_2\in [0,T]$ and $\phi\in\mc{S}$:
\begin{align*}
|\langle \tilde{Z}^k_{t_2},\phi\rangle - \langle \tilde{Z}^k_{t_1},\phi\rangle| &\leq 2|t_2-t_1|\biggl\lbrace\int_0^T |\langle \tilde{Z}^k_s,\bar{L}_{\mc{L}(X_s)}\phi\rangle|^2ds +\int_0^T |\langle B^k_s,\phi\rangle|^2ds   \biggr\rbrace\\
&\leq 2|t_2-t_1|\biggl\lbrace\int_0^T \norm{\tilde{Z}^k_s}_{-(m+2)}^2\norm{\bar{L}_{\mc{L}(X_s)}\phi}^2_{m+2}ds +\int_0^T \norm{B^N_s}^2_{-(m+2)}\norm{\phi}^2_{m+2}ds   \biggr\rbrace \\
&\leq 2|t_2-t_1|C(T)\norm{\phi}^2_{m+4} \text{ by Lemma \ref{lem:barLbounded},}
\end{align*}
and precompactness of $\br{\tilde{Z}^k}_{k\in\bb{N}}$ is established.

Taking a convergent subsequence, which we do not relabel in the notation, we call its limit $(Z,Q)$. The fact that \ref{PZ:L2contolbound}-\ref{PZ:fourthmarginallimitnglaw} in the definition of $P^*(Z)$ are satisfied follows in the exact same way as  Proposition \ref{prop:goodratefunction}. It thus remains to show that $(Z,Q)$ satisfies Equation \eqref{eq:MDPlimitFIXED} with $Q$ given in Equation \eqref{eq:QNdesiredform}. At this point, by Proposition \ref{proposition:weakuninqueness}, we will have the limit is uniquely identified for every subsequence, and hence the lemma is proved. By a density argument, it suffices to show that for each $\phi\in C^\infty_c(\R)$ and $t\in[0,T]$,
\begin{align*}
&\lim_{k\toinf}\langle \tilde{Z}^k_t,\phi\rangle = \int_0^t \langle Z_s,\bar{L}_{\mc{L}(X_s)}\phi\rangle ds +\int_{0}^{t}\E\biggl[\int_\R \biggl(\sigma(x,y,\mc{L}(X_s))h_1(s,X_s,y) \phi'(X_s)\\
&+[\tau_1(X_s,y,\mc{L}(X_s))h_1(s,X_s,y)+\tau_2(X_s,y,\mc{L}(X_s))h_2(s,X_s,y)]\Phi_y(X_s,y,\mc{L}(X_s))\phi'(X_s)\biggr)\pi(dy;X_s,\mc{L}(X_s))\biggr]ds.
\end{align*}

We have by dominated convergence theorem, $L^2$ convergence of $\psi^k$ to $h$, and that under Assumption \ref{assumption:limitingcoefficientsregularity} $\bar{L}_{\mc{L}(X_s)}\phi\in \mc{S}_{w},\forall s\in [0,T]$:
\begin{align*}
&\lim_{k\toinf}\langle \tilde{Z}^k_t,\phi\rangle = \lim_{k\toinf}\biggl\lbrace\int_0^t \langle \tilde{Z}^k_s,\bar{L}_{\mc{L}(X_s)}\phi\rangle ds +\int_{0}^{t}\langle B^k_s,\phi\rangle ds \biggr\rbrace \\
& = \int_0^t \lim_{k\toinf} \langle \tilde{Z}^k_s,\bar{L}_{\mc{L}(X_s)}\phi\rangle ds +\lim_{k\toinf}\int_{0}^{t}\E\biggl[\int_\R\biggl( \sigma(x,y,\mc{L}(X_s))\psi^k_1(s,X_s,y) \phi'(X_s)\\
&+[\tau_1(X_s,y,\mc{L}(X_s))\psi^k_1(s,X_s,y)+\tau_2(X_s,y,\mc{L}(X_s))\psi^k_2(s,X_s,y)]\Phi_y(X_s,y,\mc{L}(X_s))\phi'(X_s)\biggr)\pi(dy;X_s,\mc{L}(X_s))\biggr]ds \\
& = \int_0^t \langle Z_s,\bar{L}_{\mc{L}(X_s)}\phi\rangle ds +\int_{0}^{t}\E\biggl[\int_\R \biggl(\sigma(x,y,\mc{L}(X_s))h_1(s,X_s,y) \phi'(X_s)\\
&+[\tau_1(X_s,y,\mc{L}(X_s))h_1(s,X_s,y)+\tau_2(X_s,y,\mc{L}(X_s))h_2(s,X_s,y)]\Phi_y(X_s,y,\mc{L}(X_s))\phi'(X_s)\biggr)\pi(dy;X_s,\mc{L}(X_s))\biggr]ds
\end{align*}
as desired.
\end{proof}
\begin{proposition}\label{prop:goodratefunction}
Under assumptions \ref{assumption:uniformellipticity} - \ref{assumption:2unifboundedlinearfunctionalderivatives} and \ref{assumption:limitingcoefficientsregularityratefunction}, $I$ given in Theorems \ref{theo:Laplaceprinciple}/\ref{theo:MDP} is a good rate function on $C([0,T];\mc{S}_{-r})$ for $r>w+2$ as in Equation \eqref{eq:rdefinition}.
\end{proposition}
\begin{proof}
We need to show that for any $L>0$,
\begin{align*}
\Theta_L \coloneqq \br{Z\in C([0,T];\mc{S}_{-r}):I(Z)\leq L}
\end{align*}
is compact in $C([0,T];\mc{S}_{-r})$.

Let $\br{Z^N}_{N\in\bb{N}}\subset \Theta_L$. Then by the form of $I$, for each $N\in\bb{N}$, there exists $Q^N\in P^*(Z^N)$ such that
\begin{align*}
\frac{1}{2}\int_{\R\times\R\times\R^2\times [0,T]} \left(z_1^2+z_2^2\right) Q^N(dx,dy,dz,ds)\leq L+\frac{1}{N}
\end{align*}
and by \ref{PZ:secondmarginalinvtmeasure} and \ref{PZ:fourthmarginallimitnglaw}, we have as with the $Q^k$'s in the proof of Proposition \ref{prop:LPUB}
\begin{align*}
\sup_{N\in\bb{N}}\int_{\R\times\R\times\R^2\times [0,T]} \left(y^2+|x|^{2}\right) Q^N(dx,dy,dz,ds) <\infty.
\end{align*}
Thus by the same tightness function used for Proposition \ref{prop:QNtightness}, $\br{Q^N}_{N\in\bb{N}}$ is tight in $M_T(\R^4)$.

Taking a subsequence of $\br{Q^N}$ which converges to some $Q\in M_T(\R^4)$ (which we do not relabel in the notation), define $Z\in C([0,T];\mc{S}_{-w})$ to be the unique solution to Equation \eqref{eq:MDPlimitFIXED} with this choice of $Q$. Here we are using that by the proof of Proposition \ref{prop:LPlowerbound} such a solution exists and that by Proposition \ref{proposition:weakuninqueness} it is unique - see the discussion before Lemma 4.10 in \cite{BW}. We claim that $(Z^N,Q^N)$ converges to $(Z,Q)$ in $C([0,T];\mc{S}_{-r})\times M_T(\R^4)$ and $Q\in P^*(Z)$. At this point we will have that since $Z^N$ has a limit, $\Theta_L$ is precompact, and by the version of Fatou's lemma from Theorem A.3.12 in \cite{DE}:
\begin{align*}
I(Z)&\leq \frac{1}{2}\int_{\R\times\R\times\R^2\times [0,T]} \left(z_1^2+z_2^2\right) Q(dx,dy,dz,ds) \leq \liminf_{N\toinf} \frac{1}{2}\int_{\R\times\R\times\R^2\times [0,T]} \left(z_1^2+z_2^2\right) Q^N(dx,dy,dz,ds)\leq L
\end{align*}
so $\Theta_L$ is closed, and hence compact. Note that we have $I(Z^N)<\infty$ implies $Z^N\in C([0,T];\mc{S}_{-w}),\forall N\in\bb{N}$ and by definition $Z\in C([0,T];\mc{S}_{-w})$. Thus if we could show convergence of $Z^N\tto Z$ in $C([0,T];\mc{S}_{-w})$, we would have compactness of level sets of $I$ as a rate function on $C([0,T];\mc{S}_{-w})$. However, such convergence is not immediately obvious, hence the need for the additional assumption \ref{assumption:limitingcoefficientsregularityratefunction}.

To see that $\br{Z^N}_{N\in\bb{N}}$ is precompact, we have that since $Q^N\in P^*(Z^N)$, for each $N$ $(Z^N,Q^N)$ must satisfy Equation \eqref{eq:MDPlimitFIXED}. That is, for $\phi\in \mc{S}$ and $t\in [0,T]$:
\begin{align*}
\langle Z^N_t,\phi\rangle & = \int_0^t \langle Z^N_s,\bar{L}_{\mc{L}(X_s)}\phi\rangle ds + \int_0^t \langle B^N_s, \phi \rangle ds\\
\langle B^N_t, \phi\rangle &\coloneqq \int_{\R\times\R\times\R^2} \biggl(\sigma(x,y,\mc{L}(X_s))z_1 \phi'(x)\\
&\hspace{2cm}+ [\tau_1(x,y,\mc{L}(X_s))z_1+\tau_2(x,y,\mc{L}(X_s))z_2]\Phi_y(x,y,\mc{L}(X_s))\phi'(x)\biggr)Q_t^N(dx,dy,dz).
\end{align*}
Thus precompactness of $\br{Z^N}_{N\in\bb{N}}$ in $C([0,T];\mc{S}_{-r})$ follows in the exact same way as precompactness of $\br{\tilde{Z}^k}_{k\in\bb{N}}$ in $C([0,T];\mc{S}_{-w})$ in the proof of Proposition \ref{prop:LPlowerbound}, but replacing $m$ by $w$. Note that there we knew that $\tilde{Z}^k$ was in $C([0,T];\mc{S}_{-m})$ for each $k$, where here we only know $Z^N\in C([0,T];\mc{S}_{-w})$ for each $N$.
Along the way, we get:
\begin{align*}
\sup_{N\in\bb{N}}\sup_{t\in [0,T]}\norm{Z^N_t}^2_{-(w+2)}\leq C(T).
\end{align*}

To see that $Q\in P^*(Z)$, we identify the point-wise limit of $\langle Z^N,\phi\rangle$ to satisfy the desired equation, i.e. \eqref{eq:MDPlimitFIXED} with our specific choice of $Q$. This uniquely characterizes the limit along the whole sequence by Lemma \ref{proposition:weakuninqueness}. This gives \ref{PZ:limitingequation}. \ref{PZ:L2contolbound} follows immediately from Fatou's lemma. \ref{PZ:secondmarginalinvtmeasure} and \ref{PZ:fourthmarginallimitnglaw} follow from convergence of the measure implying convergence of the marginals and uniqueness of the decomposition into stochastic kernels (see \cite{DE} Theorems A.4.2 and A.5.4).

To see \eqref{eq:MDPlimitFIXED} with our specific choice of $Q$ holds, we may by a density argument consider fixed $\phi\in C^\infty_c(\R)$ and $t\in [0,T]$. Then:
\begin{align*}
\langle Z_t,\phi\rangle &= \lim_{N\toinf}\langle Z^N_t,\phi\rangle  = \lim_{N\toinf} \biggl\lbrace\int_0^t \langle Z^N_s,\bar{L}_{\mc{L}(X_s)}\phi\rangle ds + \int_0^t \langle B^N_s, \phi \rangle ds\biggr\rbrace\\
& = \int_0^t \lim_{N\toinf} \langle Z^N_s,\bar{L}_{\mc{L}(X_s)}\phi\rangle ds +\lim_{N\toinf} \int_0^t \langle B^N_s, \phi \rangle ds \\
&\text{ (by boundedness of }\sup_{N\in\bb{N}}\sup_{t\in [0,T]}\norm{Z^N_t}^2_{-(w+2)} \text{ and Dominated Convergence Theorem)}\\
& = \int_0^t \langle Z_s,\bar{L}_{\mc{L}(X_s)}\phi\rangle ds + \int_0^t \langle B_s, \phi \rangle ds, \text{ since under assumption \ref{assumption:limitingcoefficientsregularityratefunction} $\bar{L}_{\mc{L}(X_s)}\phi\in \mc{S}_{r},\forall s\in [0,T]$.}
\end{align*}
Here
\begin{align*}
\langle B_t, \phi \rangle&\coloneqq\int_{\R\times\R\times\R^2} \biggl(\sigma(x,y,\mc{L}(X_s))z_1 \phi'(x)\\
&\hspace{4cm}+ [\tau_1(x,y,\mc{L}(X_s))z_1+\tau_2(x,y,\mc{L}(X_s))z_2]\Phi_y(x,y,\mc{L}(X_s))\phi'(x)\biggr)Q_t(dx,dy,dz),
\end{align*}
and to pass to the second limit, we use that the integrand appearing in $\int_0^t \langle B^N_s, \phi \rangle ds$ is bounded by $C[|z_1|+|z_2|]$, and hence is uniformly integrable with respect to $Q^N$.
\end{proof}

\section{Conclusions and Future Work}\label{S:Conclusions}

In this paper we have derived a moderate deviations principle for the empirical measure of a fully coupled multiscale system of weakly interacting particles in the joint limit as number of particles increases and averaging due to the multiscale structure takes over. Using weak convergence methods we have derived a variational form of the rate function  and have rigorously shown that the rate function can take equivalent forms analogous to the one derived in the seminal paper \cite{DG}.

In this paper we have assumed that the particles are in dimension one. It is of great interest to extend this work in the multidimensional case. One source of difficulty here is that in higher dimensions we would probably have to consider a different space for the fluctuation process to live in (see, e.g. \cite{FM} and \cite{LossFromDefault}). This is because in higher dimensions the result that for each $v$, there is $w\geq v$ such that $\mc{S}_{-v}\tto \mc{S}_{-w}$ is Hilbert-Schmidt breaks down, and the bound \eqref{eq:sobolembedding} no longer holds true. See \cite{DLR} Section 5.1 for a further discussion of this. The trade-off with using these alternative spaces is that they often require higher moments of the particles and limiting McKean-Vlasov Equation in order to establish tightness - see, e.g. Section 4.7 in \cite{FM}, where the proofs depend crucially on Lemma 3.1 (even in one dimension this would require having bounded $8$'th moments of the controlled particles $\tilde{X}^{i,N,\epsilon}$, with the required number of moments increasing with the dimension). This would seem to require strong assumptions on the coefficients in Equation \eqref{eq:slowfast1-Dold} even in the absence of multiscale structure, since the controls are a priori only bounded in $L^2$.

Another potentially interesting direction is to derive the moderate deviations principle for the stochastic current. See \cite{Orrieri} for some related results in the direction of large deviations for an interacting particle system in the joint mean field and small-noise limit.  Also, we are hopeful that the results of this paper can also be used for the construction of provably-efficient importance sampling schemes for the computation of rare events for statistics of weakly interacting diffusions that are relevant to the moderate-deviations scaling. Lastly, as we also mentioned in the introduction, we believe that the results of this paper can be used to study dynamical questions related to phase transitions in the spirit of \cite{Dawson}.

{}\appendix
\section{A List of Technical Notation}\label{sec:notationlist}
Here we provide a list of frequently used notation for the various processes, spaces, operators, ect. used throughout this manuscript for convenient reference. Other, more standard notation is introduced following Equation \eqref{eq:sobolembedding} in Section \ref{subsec:notationandtopology}.
\begin{itemize}
\item $\epsilon$ is the scale separation parameter which decreases to $0$ as $N\toinf$. $N$ is the number of particles. $a(N)$ is moderate deviations the scaling sequence such that $a(N)\downarrow 0$ and $a(N)\sqrt{N}\toinf$.
\item $(X^{i,\epsilon,N},Y^{i,\epsilon,N})$ is the slow-fast system of particles from Equation \eqref{eq:slowfast1-Dold}.
\item $\mu^{\epsilon,N}$ from Equation \eqref{eq:empiricalmeasures} is the empirical measure on the slow particles $X^{i,\epsilon,N}$ .
\item $X_t$ is the limiting averaged McKean-Vlasov Equation from Equation \eqref{eq:LLNlimitold}. $\mc{L}(X_t)$ denotes its Law.
\item $Z^N$ is the fluctuations process from Equation \eqref{eq:fluctuationprocess} for which we derive a large deviations principle.
\item $(\tilde{X}^{i,\epsilon,N},\tilde{Y}^{i,\epsilon,N})$ are the controlled slow-fast interacting particles from Equation \eqref{eq:controlledslowfast1-Dold}.
\item $\tilde{\mu}^{\epsilon,N}$ is the empirical measure on the controlled slow particles $\tilde{X}^{i,\epsilon,N}$ from Equation \eqref{eq:controlledempmeasure}.
\item $\tilde{Z}^N$ is the controlled fluctuations process from Equation \eqref{eq:controlledempmeasure}.
\item $Q^N$ are the occupation measures from Equation \eqref{eq:occupationmeasures}.
\item $(\bar{X}^{i,\epsilon},\bar{Y}^{i,\epsilon})$ are the IID slow-fast McKean-Vlasov Equations from Equation \eqref{eq:IIDparticles}. $\bar{X}^\epsilon$ is a random process with law Equal to that of the $\bar{X}^{i,\epsilon}$'s.
\item $\bar{\mu}^{\epsilon,N}$ from Equation \eqref{eq:IIDempiricalmeasure} is the empirical measure on $N$ of the IID slow particles $\bar{X}^{i,\epsilon}$.
\item $\mc{P}_2(\R)$ is the space of square integrable probability measures with the 2-Wasserstein metric $\bb{W}_2$ (Definition \ref{def:lionderivative}).
\item $M_T(\R^d)$ is the space of measures $Q$ on $\R^d\times [0,T]$ such that $Q(\R^d\times[0,t]) = t,\forall t\in [0,T]$ equipped with the topology of weak convergence.
\item For $p\in\bb{N}$, $\mc{S}_p$ is the completion of $\mc{S}$ with respect to $\norm{\cdot}_p$ (see Equation \eqref{eq:familyofhilbertnorms}) and $\mc{S}_{-p}=\mc{S}_p'$ the dual space of $\mc{S}_p$. We prove tightness of $\br{\tilde{Z}^N}_{N\in\bb{N}}$ in $C([0,T];\mc{S}_{-m})$ for the choice of $m$ found in Equation \eqref{eq:mdefinition}, the Laplace Principle on $C([0,T];\mc{S}_{-w})$ for the choice of $w$ found in Equation \eqref{eq:wdefinition}, and compactness of level sets of the rate function on $C([0,T];\mc{S}_{-r})$ for the choice of $r$ found in Equation \eqref{eq:rdefinition}.
\item For $n\in\bb{N}$, $|\cdot|_n$ is the sup norm defined in Equation \eqref{eq:boundedderivativesseminorm}, which is related to $\norm{\cdot}_{n+1}$ via Equation \eqref{eq:sobolembedding}.
\item For $G:\mc{P}_2(\R)\tto \R$ and $\nu\in\mc{P}_2(\R)$, $\partial_\mu G(\nu)[\cdot]:\R\tto \R$ denotes the Lions derivative of $G$ at the point $\nu$ (Definition \ref{def:lionderivative}) and $\frac{\delta}{\delta m}G(\nu)[\cdot]:\R\tto \R$ denotes the Linear Functional Derivative of $G$ at the point $\nu$ (Definition \ref{def:LinearFunctionalDerivative}).
\item For $G:\R\times \mc{P}_2(\R)\tto \R$, we use $D^{(n,l,\beta)}G$ to denote multiple derivatives of $G$ in space and measure in the multi-index notation of Definition \ref{def:multiindexnotation}. Spaces (denoted by $\mc{M}$ with some sub or super-scripts) containing functions with different regularity of such mixed derivatives are found in Definition \ref{def:lionsderivativeclasses}. When $G:\R\times\R\times\mc{P}_2(\R)\tto \R$, polynomial growth of such derivatives in $G$'s second coordinate, denoted by $q_G(n,l,\beta)$ or $\tilde{q}_G(n,l,\beta)$, are defined as in Equations \eqref{eq:newqnotation} and \eqref{eq:tildeq}.
\item $L_{x,\mu}$ is the frozen generator associated to the fast particles from Equation \eqref{eq:frozengeneratormold}. $\pi$ denotes its unique associated invariant measure from Equation \eqref{eq:invariantmeasureold} and $\Phi$ denotes the solution to the associated Poisson Equation \eqref{eq:cellproblemold}.
\item For $\nu\in\mc{P}_2(\R)$ $\bar{L}_\nu$ is the linearized generator of the limiting averaged McKean-Vlasov Equation $X_t$ at $\nu$ and is defined in Equation \eqref{eq:MDPlimitFIXED}.
\item $\bar{\gamma},\bar{D}$ from Equation \eqref{eq:averagedlimitingcoefficients} are the drift and diffusion coefficients of the limiting averaged McKean-Vlasov Equation $X_t$, and are defined in terms of $\gamma_1,D_1,\gamma,D:\R\times\R\times\mc{P}_2(\R)\tto \R$ from Equation \eqref{eq:limitingcoefficients}.
\end{itemize}

\section{A priori Bounds on Moments of the Controlled Process \eqref{eq:controlledslowfast1-Dold}}\label{sec:aprioriboundsoncontrolledprocess}
In this Appendix, we fix any controls satisfying the bound \eqref{eq:controlassumptions} and provide moment bounds on the fast component of the controlled particles \eqref{eq:controlledslowfast1-Dold}. These are needed, among other places, to handle possible growth lack of boundedness in $y$ of functions appearing in the remainders in the ergodic-type theorems of Section \ref{sec:ergodictheoremscontrolledsystem}. 
\begin{lemma}\label{lemma:tildeYuniformbound}
Under assumptions \ref{assumption:uniformellipticity}- \ref{assumption:retractiontomean}, \ref{assumption:strongexistence}, and \ref{assumption:gsigmabounded}, we have there is $C\geq 0$ such that: 
\begin{align*}
\sup_{N\in\bb{N}}\frac{1}{N}\sum_{i=1}^N\sup_{t\in[0,T]}\E\biggl[ |\tilde{Y}^{i,\epsilon,N}_t|^2\biggr]\leq C+|\eta^{y}|^2.
\end{align*}

\end{lemma}
\begin{proof}
By It\^o's formula, we have, letting $C\geq 0$ be any constant independent of $N$ which may change from line to line and $(i)$ denote the argument $(\tilde{X}^{i,\epsilon,N}_s,\tilde{Y}^{i,\epsilon,N}_s,\tilde{\mu}^{\epsilon,N})$:
\begin{align*}
&\E\biggl[ |\tilde{Y}^{i,\epsilon,N}_t|^2\biggr]  = |\eta^y|^2 + \int_0^t \E\biggl[ \frac{1}{\epsilon^2}\biggl(2f(i)\tilde{Y}^{i,\epsilon,N}_s+\tau_1^2(i)+\tau_2^2(i)\biggr)\biggr]ds + \frac{2}{\epsilon}\int_0^t \E\biggl[g(i)\tilde{Y}^{i,\epsilon,N}_s\biggr]ds \\
&+ \frac{2}{\epsilon}\E\biggl[\int_0^t \biggl(\tau_1(i)\frac{\tilde{u}^{N,1}_i(s)}{a(N)\sqrt{N}}+\tau_2(i)\frac{\tilde{u}^{N,1}_i(s)}{a(N)\sqrt{N}}\biggr)\tilde{Y}^{i,\epsilon,N}_sds\biggr]+\frac{2}{\epsilon}\E\biggl[\int_0^t \tau_1(i)\tilde{Y}^{i,\epsilon,N}_s dW^i_s+\int_0^t \tau_2(i)\tilde{Y}^{i,\epsilon,N}_s dB^i_s\biggr]\\
&\leq |\eta^y|^2  -\frac{\beta}{\epsilon^2}\int_0^t\E\biggl[ |\tilde{Y}^{i,\epsilon,N}_s|^2\biggr]ds+\frac{Ct}{\epsilon^2} + \frac{2}{\epsilon}\int_0^t \E\biggl[g(i)\tilde{Y}^{i,\epsilon,N}_s\biggr]ds \\
&+ \frac{2}{\epsilon}\E\biggl[\int_0^t \biggl(\tau_1(i)\frac{\tilde{u}^{N,1}_i(s)}{a(N)\sqrt{N}}+\tau_2(i)\frac{\tilde{u}^{N,1}_i(s)}{a(N)\sqrt{N}}\biggr)\tilde{Y}^{i,\epsilon,N}_sds\biggr]+\frac{2}{\epsilon}\E\biggl[\int_0^t \tau_1(i)\tilde{Y}^{i,\epsilon,N}_s dW^i_s+\int_0^t \tau_2(i)\tilde{Y}^{i,\epsilon,N}_s dB^i_s\biggr]\text{ by }\eqref{eq:fdecayimplication}\\
&\leq |\eta^y|^2  -\frac{\beta}{\epsilon^2}\int_0^t\E\biggl[ |\tilde{Y}^{i,\epsilon,N}_s|^2\biggr]ds+\frac{Ct}{\epsilon^2} + \frac{2}{\epsilon}\int_0^t \E\biggl[|\tilde{Y}^{i,\epsilon,N}_s|^2\biggr]ds+\norm{g}^2_\infty \frac{t}{\epsilon} \\
&+ \frac{|\tau_1|^2_\infty \vee |\tau_2|^2_\infty}{\epsilon a^2(N)N}\E\biggl[\int_0^T  |\tilde{u}^{N,1}_i(s)|^2+|\tilde{u}^{N,1}_i(s)|^2\biggr]+\frac{2}{\epsilon}\E\biggl[\int_0^t \tau_1(i)\tilde{Y}^{i,\epsilon,N}_s dW^i_s+\int_0^t \tau_2(i)\tilde{Y}^{i,\epsilon,N}_s dB^i_s\biggr]
\end{align*}
at which point it becomes clear that applying Burkholder Davis Gundy inequality, taking $\epsilon$ small enough that the $\frac{-\beta}{\epsilon^2}$ term dominates, and using the bound \eqref{eq:controlassumptions}, that $\E\biggl[\int_0^T |\tilde{Y}^{i,\epsilon,N}_s|^2\biggr]<\infty$ for each $N$, so by boundedness of $\tau_1,\tau_2$ the stochastic integrals are true martingales, and hence vanish in expectation. Note that in the above we are using the boundedness of $\tau_1,\tau_2$ from Assumption \ref{assumption:uniformellipticity} and of $g$ from Assumption \ref{assumption:gsigmabounded}.

So, returning to the initial equality:
\begin{align*}
&\frac{d}{dt}\E\biggl[ |\tilde{Y}^{i,\epsilon,N}_t|^2\biggr]  =   \E\biggl[ \frac{1}{\epsilon^2}\biggl(2f(i)\tilde{Y}^{i,\epsilon,N}_t+\tau_1^2(i)+\tau_2^2(i)\biggr)\biggr] + \frac{2}{\epsilon}\E\biggl[g(i)\tilde{Y}^{i,\epsilon,N}_t\biggr]\\
&+ \frac{2}{\epsilon}\E\biggl[\biggl(\tau_1(i)\frac{\tilde{u}^{N,1}_i(t)}{a(N)\sqrt{N}}+\tau_2(i)\frac{\tilde{u}^{N,1}_i(t)}{a(N)\sqrt{N}}\biggr)\tilde{Y}^{i,\epsilon,N}_t\biggr]\\
&\leq -\frac{\beta}{\epsilon^2}\E\biggl[ |\tilde{Y}^{i,\epsilon,N}_t|^2\biggr] + \frac{C}{\epsilon^2} + \frac{4\norm{g}_\infty^2}{\beta} + \frac{\beta}{4 \epsilon^2}\E[|\tilde{Y}^{i,\epsilon,N}_t|^2] + \frac{4}{\beta a^2(N)N}\E\biggl[\biggl(\tau_1(i)\tilde{u}^{N,1}_i(t)+\tau_2(i)\tilde{u}^{N,1}_i(t)\biggr)^2 \biggr] \\
&+ \frac{\beta}{4\epsilon^2}\E[|\tilde{Y}^{i,\epsilon,N}_t|^2] \\
&\leq -\frac{\beta}{2\epsilon^2}\E\biggl[ |\tilde{Y}^{i,\epsilon,N}_t|^2\biggr]+ \frac{C}{\epsilon^2} + \frac{4\norm{g}_\infty^2}{\beta} +\frac{8 \norm{\tau_1}_\infty^2\vee \norm{\tau_2}^2_\infty}{\beta a^2(N)N} \E\biggl[|\tilde{u}^{N,1}_i(t)|^2+|\tilde{u}^{N,1}_i(t)|^2 \biggr]
\end{align*}

Where in the first inequality we used the consequence \eqref{eq:fdecayimplication} of Assumption \ref{assumption:retractiontomean}. Now, recalling that if $g'(s)\leq -\gamma g(s)+f(s),\forall s\in[0,t],$ then $g(t)\leq \int_0^t f(s) e^{-\gamma(t-s)}ds +e^{-\gamma t}g(0)$, we set $g(t) = \sum_{i=1}^N \E\biggl[ |\tilde{Y}^{i,\epsilon,N}_t|^2\biggr], \gamma = \frac{\beta}{2\epsilon^2}, f(t) = \frac{CN}{\epsilon^2} + \frac{4\norm{g}_\infty^2N}{\beta} +\frac{8 \norm{\tau_1}_\infty^2\vee \norm{\tau_2}^2_\infty}{\beta a^2(N)} \E\biggl[\frac{1}{N}\sum_{i=1}^N|\tilde{u}^{N,1}_i(t)|^2+|\tilde{u}^{N,1}_i(t)|^2 \biggr]$
and get
\begin{align*}
&\sum_{i=1}^N \E\biggl[ |\tilde{Y}^{i,\epsilon,N}_t|^2\biggr]\leq e^{-\frac{\beta}{2\epsilon^2}t}CN(1+\frac{1}{\epsilon^2})\int_0^t e^{\frac{\beta}{2\epsilon^2}s} ds \nonumber\\
&\qquad+ \frac{C}{a^2(N)}e^{-\frac{\beta}{2\epsilon^2}t}\int_0^t  \E\biggl[\frac{1}{N}\sum_{i=1}^N|\tilde{u}^{N,1}_i(s)|^2+|\tilde{u}^{N,1}_i(s)|^2 \biggr] e^{\frac{\beta}{2\epsilon^2}s} ds +N|\eta^y|^2 e^{-\frac{\beta}{2\epsilon^2}t} \\
&\leq e^{-\frac{\beta}{2\epsilon^2}t}CN(1+\frac{1}{\epsilon^2})\int_0^t e^{\frac{\beta}{2\epsilon^2}s} ds + \frac{C}{a^2(N)}\int_0^t  \E\biggl[\frac{1}{N}\sum_{i=1}^N|\tilde{u}^{N,1}_i(s)|^2+|\tilde{u}^{N,1}_i(s)|^2 \biggr] ds +N|\eta^y|^2 \\
&\leq e^{-\frac{\beta}{2\epsilon^2}t}CN(1+\frac{1}{\epsilon^2})\int_0^t e^{\frac{\beta}{2\epsilon^2}s} ds + \frac{C}{a^2(N)} +N|\eta^y|^2 \text{ by the bound }\eqref{eq:controlassumptions0}\\
& =  e^{-\frac{\beta}{2\epsilon^2}t}CN(1+\frac{1}{\epsilon^2}) \frac{2\epsilon^2}{\beta}[e^{\frac{\beta}{2\epsilon^2}t}-1] + \frac{C}{a^2(N)} +N|\eta^y|^2 \\
&\leq CN(1+\epsilon^2)+\frac{C}{a^2(N)}+N|\eta^y|^2
\end{align*}
where $C$ is a constant changing from line to line which is independent of $t$ and $N$. Then
\begin{align*}
\frac{1}{N}\sum_{i=1}^N\sup_{t\in[0,T]}\E\biggl[ |\tilde{Y}^{i,\epsilon,N}_t|^2\biggr]&\leq C(1+\epsilon^2+\frac{1}{a^2(N)N})+|\eta^y|^2 \\
&\leq C+|\eta^y|^2
\end{align*}
since $\epsilon\downarrow 0$ and $a(N)\sqrt{N}\tto \infty$.
\end{proof}

\begin{lemma}\label{lemma:ytildeexpofsup}
Under assumptions \ref{assumption:uniformellipticity}-\ref{assumption:retractiontomean},\ref{assumption:strongexistence},and \ref{assumption:gsigmabounded}, we have:
\begin{align*}
\frac{1}{N}\sum_{i=1}^N \E\biggl[\sup_{0\leq t\leq T}|\tilde{Y}^{i,\epsilon,N}_t|^2 \biggr]\leq |\eta^{y}|^2 + C(\rho)\biggl[1+\epsilon^{-\rho}+\frac{1}{a^2(N) N}\biggr]
\end{align*}
for all $\rho\in (0,2)$.
\end{lemma}
\begin{proof}
The proof follows along the lines of that of Lemma B.4 in \cite{JS} and thus it is omitted here. Note that this is where the near-Ornstein–Uhlenbeck structure assumed in \eqref{eq:fnearOU} plays an important role.
\end{proof}

\begin{lemma}\label{lemma:ytildesquaredsumbound}
Under assumptions \ref{assumption:uniformellipticity}-\ref{assumption:retractiontomean},\ref{assumption:strongexistence}, and \ref{assumption:gsigmabounded}, we have:
\begin{align*}
\sup_{N\in\bb{N}}\sup_{0\leq t\leq T}\E\biggl[\biggl(\frac{1}{N}\sum_{i=1}^N|\tilde{Y}^{i,\epsilon,N}_t|^2\biggr)^2 \biggr]<\infty.
\end{align*}
\end{lemma}
\begin{proof}
The proof is very similar to Lemma \ref{lemma:tildeYuniformbound}, but we need in addition to use the result of Lemma \ref{lemma:ytildeexpofsup}. Because of the similarities, we assume wlog that $\tau_2=g=0$ and label $\tau_1$ as $\tau$ and $\tilde{u}^{N,1}_i(s)$ and $\tilde{u}^N_i(s)$. Then for every $N\in\bb{N}$ and $t\in [0,T]$,
\begin{align*}
&\E\biggl[\biggl(\frac{1}{N}\sum_{i=1}^N|\tilde{Y}^{i,\epsilon,N}_t|^2\biggr)^2 \biggr]  = \nonumber\\
&=|\eta^{y}|^4 + \frac{1}{N}\sum_{i=1}^N \biggl\lbrace 4\int_0^t\E\biggl[\biggl(\frac{1}{N}\sum_{i=1}^N|\tilde{Y}^{i,\epsilon,N}_s|^2\biggr)\frac{1}{\epsilon^2}\biggl(f(i)\tilde{Y}^{i,\epsilon,N}_s+\tau^2(i)\biggr) \biggr]ds\\
&+4\int_0^t\E\biggl[\biggl(\frac{1}{N}\sum_{i=1}^N|\tilde{Y}^{i,\epsilon,N}_s|^2\biggr)\frac{1}{\epsilon a(N)\sqrt{N}}\tau(i)\tilde{u}^{N}_i(s)\tilde{Y}^{i,\epsilon,N}_s\biggr]ds \\
&+8\int_0^t\E\biggl[\frac{1}{\epsilon^2N}\tau^2(i)|\tilde{Y}^{i,\epsilon,N}_s|^2 \biggr]ds+\frac{4}{\epsilon}\E\biggl[\int_0^t\biggl(\frac{1}{N}\sum_{i=1}^N|\tilde{Y}^{i,\epsilon,N}_s|^2\biggr)\tau(i)\tilde{Y}^{i,\epsilon,N}_s dW^i_s \biggr]  \biggr\rbrace,
\end{align*}
where we used the initial conditions are IID. By the same method as in Lemma \ref{lemma:tildeYuniformbound}, we can see that the martingale term vanishes for each $N$ and $t$. Then we get
\begin{align*}
&\frac{d}{dt}\E\biggl[\biggl(\frac{1}{N}\sum_{i=1}^N|\tilde{Y}^{i,\epsilon,N}_t|^2\biggr)^2 \biggr]  = \frac{1}{N}\sum_{i=1}^N \biggl\lbrace 4\E\biggl[\biggl(\frac{1}{N}\sum_{i=1}^N|\tilde{Y}^{i,\epsilon,N}_t|^2\biggr)\frac{1}{\epsilon^2}\biggl(f(i)\tilde{Y}^{i,\epsilon,N}_t+\tau^2(i)\biggr) \biggr]\\
&\qquad+4\E\biggl[\biggl(\frac{1}{N}\sum_{i=1}^N|\tilde{Y}^{i,\epsilon,N}_t|^2\biggr)\frac{1}{\epsilon a(N)\sqrt{N}}\tau(i)\tilde{u}^{N}_i(t)\tilde{Y}^{i,\epsilon,N}_t\biggr] +8\E\biggl[\frac{1}{\epsilon^2N}\tau^2(i)|\tilde{Y}^{i,\epsilon,N}_t|^2 \biggr]\biggr\rbrace\\
&\leq \frac{1}{N}\sum_{i=1}^N \biggl\lbrace \E\biggl[\biggl(\frac{1}{N}\sum_{i=1}^N|\tilde{Y}^{i,\epsilon,N}_t|^2\biggr)\biggl(-\frac{\beta}{\epsilon^2}|\tilde{Y}^{i,\epsilon,N}_t|^2+\frac{2}{\epsilon^2}\tau^2(i)\biggr) \biggr]\\
&\qquad+4\E\biggl[\biggl(\frac{1}{N}\sum_{i=1}^N\tilde{Y}^{i,\epsilon,N}_t|^2\biggr)\frac{1}{\epsilon a(N)\sqrt{N}}\tau(i)\tilde{u}^{N}_i(t)\tilde{Y}^{i,\epsilon,N}_t \biggr] +8\E\biggl[\frac{1}{\epsilon^2N}\tau^2(i)|\tilde{Y}^{i,\epsilon,N}_t|^2 \biggr]\biggr\rbrace\\
&\qquad+\frac{C}{\epsilon^2}\E\biggl[\frac{1}{N}\sum_{i=1}^N|\tilde{Y}^{i,\epsilon,N}_t|^2\biggr]\\
&\leq   -\frac{\beta}{\epsilon^2}\E\biggl[\biggl(\frac{1}{N}\sum_{i=1}^N|\tilde{Y}^{i,\epsilon,N}_t|^2\biggr)^2 \biggr]+\frac{4}{\epsilon a(N)\sqrt{N}}\E\biggl[\biggl(\frac{1}{N}\sum_{i=1}^N|\tilde{Y}^{i,\epsilon,N}_t|^2\biggr)\biggl(\frac{1}{N}\sum_{i=1}^N\tau(i)\tilde{u}^{N}_i(t)\tilde{Y}^{i,\epsilon,N}_t\biggr) \biggr] \\
&\qquad+\frac{C}{\epsilon^2}\E\biggl[\frac{1}{N}\sum_{i=1}^N|\tilde{Y}^{i,\epsilon,N}_t|^2\biggr]\biggl[1+\frac{1}{N} \biggr]\\
&\leq   -\frac{\beta}{2\epsilon^2}\E\biggl[\biggl(\frac{1}{N}\sum_{i=1}^N|\tilde{Y}^{i,\epsilon,N}_t|^2\biggr)^2 \biggr]+\frac{C}{a^2(N)N}\E\biggl[\biggl(\frac{1}{N}\sum_{i=1}^N\tau(i)\tilde{u}^{N}_i(t)\tilde{Y}^{i,\epsilon,N}_t\biggr)^2 \biggr] \\
&\qquad+\frac{C}{\epsilon^2}\E\biggl[\frac{1}{N}\sum_{i=1}^N|\tilde{Y}^{i,\epsilon,N}_t|^2\biggr]\biggl[1+\frac{1}{N} \biggr].
\end{align*}
Here we again used the implication \eqref{eq:fdecayimplication} of Assumption \ref{assumption:retractiontomean} and the boundedness of $\tau$ from \ref{assumption:uniformellipticity}. Then by the comparison theorem from the proof of Lemma \ref{lemma:tildeYuniformbound}, we have for all $t\in [0,T],N\in \bb{N}$,
\begin{align*}
&\E\biggl[\biggl(\frac{1}{N}\sum_{i=1}^N|\tilde{Y}^{i,\epsilon,N}_t|^2\biggr)^2 \biggr]\leq \nonumber\\
&\leq|\eta^{y}|^4e^{-\frac{\beta}{2\epsilon^2}t}+\frac{C}{\epsilon^2}e^{-\frac{\beta}{2\epsilon^2}t}\biggl[1+\frac{1}{N} \biggr]\int_0^t \E\biggl[\frac{1}{N}\sum_{i=1}^N|\tilde{Y}^{i,\epsilon,N}_s|^2\biggr]e^{\frac{\beta}{2\epsilon^2}s} ds \\
&\qquad+ \frac{C}{a^2(N)N}e^{-\frac{\beta}{2\epsilon^2}t}\int_0^t \E\biggl[\biggl(\frac{1}{N}\sum_{i=1}^N\tau(i)\tilde{u}^{N}_i(s)\tilde{Y}^{i,\epsilon,N}_s\biggr)^2 \biggr]e^{\frac{\beta}{2\epsilon^2}s} ds\\
&\leq |\eta^{y}|^4 +\frac{C}{\epsilon^2}e^{-\frac{\beta}{2\epsilon^2}t}\biggl[1+\frac{1}{N} \biggr]\int_0^t e^{\frac{\beta}{2\epsilon^2}s} ds\sup_{s\in[0,T]}\E\biggl[\frac{1}{N}\sum_{i=1}^N|\tilde{Y}^{i,\epsilon,N}_s|^2\biggr] \\
&\qquad+ \frac{C}{a^2(N)N}e^{-\frac{\beta}{2\epsilon^2}t}\int_0^t \E\biggl[\biggl(\frac{1}{N}\sum_{i=1}^N\tau(i)\tilde{u}^{N}_i(s)\tilde{Y}^{i,\epsilon,N}_s\biggr)^2 \biggr]e^{\frac{\beta}{2\epsilon^2}s} ds\\
&\leq |\eta^{y}|^4 +C\biggl[1+\frac{1}{N} \biggr]\sup_{s\in[0,T]}\E\biggl[\frac{1}{N}\sum_{i=1}^N|\tilde{Y}^{i,\epsilon,N}_s|^2\biggr] \\
&\qquad+ \frac{C}{a^2(N)N}e^{-\frac{\beta}{2\epsilon^2}t}\int_0^t \E\biggl[\biggl(\frac{1}{N}\sum_{i=1}^N\tau(i)\tilde{u}^{N}_i(s)\tilde{Y}^{i,\epsilon,N}_s\biggr)^2 \biggr]e^{\frac{\beta}{2\epsilon^2}s} ds\\
&\leq |\eta^{y}|^4+C\biggl[1+\frac{1}{N} \biggr]\sup_{s\in[0,T]}\E\biggl[\frac{1}{N}\sum_{i=1}^N|\tilde{Y}^{i,\epsilon,N}_s|^2\biggr] \\
&\qquad+ \frac{C}{a^2(N)N}e^{-\frac{\beta}{2\epsilon^2}t}\int_0^t \E\biggl[\biggl(\frac{1}{N}\sum_{i=1}^N|\tilde{u}^{N}_i(s)|^2\biggr)\biggl(\frac{1}{N}\sum_{i=1}^N|\tilde{Y}^{i,\epsilon,N}_s|^2\biggr) \biggr]e^{\frac{\beta}{2\epsilon^2}s} ds\\
&\leq |\eta^{y}|^4 +C\biggl[1+\frac{1}{N} \biggr]\sup_{s\in[0,T]}\E\biggl[\frac{1}{N}\sum_{i=1}^N|\tilde{Y}^{i,\epsilon,N}_s|^2\biggr] \\
&\qquad+ \frac{C}{a^2(N)N}\E\biggl[\int_0^T \biggl(\frac{1}{N}\sum_{i=1}^N|\tilde{u}^{N}_i(s)|^2\biggr)ds\sup_{s\in[0,T]}\biggl(\frac{1}{N}\sum_{i=1}^N|\tilde{Y}^{i,\epsilon,N}_s|^2\biggr) \biggr]\\
&\leq |\eta^{y}|^4+C\biggl[1+\frac{1}{N} \biggr]\sup_{s\in[0,T]}\E\biggl[\frac{1}{N}\sum_{i=1}^N|\tilde{Y}^{i,\epsilon,N}_s|^2\biggr] \\
&\qquad+ \frac{C}{a^2(N)N}\E\biggl[\sup_{s\in[0,T]}\biggl(\frac{1}{N}\sum_{i=1}^N|\tilde{Y}^{i,\epsilon,N}_s|^2\biggr) \biggr]\text{ by the bound \eqref{eq:controlassumptions}} \\
&\leq C(\eta^y,\rho)[1+\frac{1}{N}+\frac{1}{a^2(N)N}+\frac{1}{a^4(N)N^2}+\frac{1}{a^2(N)N\epsilon^{\rho}}]
\end{align*}
for any $\rho\in (0,2)$ by Lemmas \ref{lemma:tildeYuniformbound} and \ref{lemma:ytildeexpofsup}, where the constant $C$ is independent of $N$. Then taking $\rho\in (0,2)$ such that such that $\epsilon^{\rho/2} a(N)\sqrt{N} \tto \lambda \in (0,\infty]$ and using $a(N)\sqrt{N}\tto \infty$, all the terms in the bound which depend on $N$ are bounded as $N\toinf$, so we get a bound independent of $N$ and $t$, and the result is proved.
\end{proof}

\section{Regularity of the Poisson Equations}\label{sec:regularityofthecellproblem}
As discussed in Remark \ref{remark:choiceof1Dparticles}, there is a current gap in the literature regarding rates of polynomial growth of derivatives of the Poisson equations used in Section \ref{sec:ergodictheoremscontrolledsystem}.
 Nevertheless, it is important to verify that the assumptions imposed on these solutions in Section \ref{subsec:notationandtopology} are non-empty. 
 For the reasons outlined in Remark \ref{remark:choiceof1Dparticles}, we handle the case of the 1D Poisson equations from Equations \eqref{eq:cellproblemold},\eqref{eq:driftcorrectorproblem} and the Multi-Dimensional Poisson Equations \eqref{eq:tildechi} and \eqref{eq:doublecorrectorproblem}, 
 separately in Subsections \ref{subsection:regularityofthe1Dpoissoneqn} and \ref{subsec:multidimpoissonequation} below. In Subsection \ref{subsec:suffconditionsoncoefficients} we provide specific examples where the Assumptions in Section \ref{subsec:notationandtopology} hold. 
\subsection{Results for the 1-Dimensional Poisson Equation}\label{subsection:regularityofthe1Dpoissoneqn}
Throughout this subsection we assume \ref{assumption:uniformellipticity} and \ref{assumption:retractiontomean}. Recall the frozen generator $L_{x,\mu}$ from Equation \eqref{eq:frozengeneratormold}, the invariant measure $\pi$ from Equation \eqref{eq:invariantmeasureold}, the multi-index derivative notation and associated spaces of functions from Definitions \ref{def:multiindexnotation} and \ref{def:lionsderivativeclasses}, and the definition of $a$ from Equation \eqref{eq:frozengeneratormold}.
\begin{lemma}\label{lemma:Ganguly1DCellProblemResult}
 Consider $B:\R\times\R\times\bb{P}_2(\R)\tto \R$ continuous such that
\begin{align*}
\int_{\R}B(x,y,\mu)\pi(dy;x,\mu)=0,\forall x\in \R,\mu\in \mc{P}_2(\R)
\end{align*}
and $|B(x,y,\mu)|=O(|y|^{q_{B}})$ for $q_{B}\in\R$ uniformly in $x,\mu$ as $|y|\tto\infty$.
Then there exists a unique classical solution $u:\R\times\R\times\mc{P}_2(\R)\tto \R$ to
\begin{align*}
L_{x,\mu}u(x,y,\mu)=B(x,y,\mu)
\end{align*}
such that $u$ is continuous in $(x,y,\bb{W}_2)$, $\int_{\R}u(x,y,\mu)\pi(dy;x,\mu)=0$, and $u$ has at most polynomial growth as $|y|\tto \infty$.

In addition,
\begin{align*}
|u(x,y,\mu)|&=O(|y|^{q_{B}}) \text{ for }q_{B}\neq 0; \text{ if }q_{B}=0,\text{ then }|u(x,y,\mu)|=O(\ln(|y|)\\
|u_y(x,y,\mu)|&=O(|y|^{q_{B}-1}),\quad |u_{yy}(x,y,\mu)|=O(|y|^{q_{B}})
\end{align*}
as $|y|\toinf$ uniformly in $x,y,\mu$.

Furthermore if $B(x,y,\mu)$ is Lipschitz continuous in $y$ uniformly in $x,\mu$ (so that necessarily $q_B \leq 1$), then so are $u,u_y,u_{yy}$.
\end{lemma}
\begin{proof}
This follows by Proposition A.4 in \cite{GS}, since in our setting Condition 2.1 (i) holds with $\alpha =1$ and assumption A.3 holds with $\theta =1$. The statement about Lipschitz continuity of $u$ and its derivatives follows from the Lipschitz continuity of $a,f$ under assumptions \ref{assumption:uniformellipticity} and \ref{assumption:retractiontomean} and Theorem 9.19 in \cite{GT}.
\end{proof}

\begin{lemma}\label{lemma:derivativetransferformulas}
Consider $h:\R\times\R\times\mc{P}_2(\R)\tto \R$. Suppose $h$ in jointly continuous in $(x,y,\bb{W}_2)$ and grows at most polynomially in $y$ uniformly in $x,\mu$. Then
\begin{align}\label{eq:muLipschitztransferformula}
&\int_{\R}h(x,y,\mu_1)\pi(dy;x,\mu_1) -\int_{\R}h(x,y,\mu_2)\pi(dy;x,\mu_2) =  \nonumber\\
&\qquad=\int_{\R} h(x,y,\mu_1)-h(x,y,\mu_2)-[\mc{L}_{x,\mu_1}-\mc{L}_{x,\mu_2}]v(x,y,\mu_2)\pi(dy;x,\mu_1)
\end{align}
for all $x\in \R,\mu_1,\mu_2\in \mc{P}_2(\R)$,
and
\begin{align}\label{eq:xLipschitztransferformula}
&\int_{\R}h(x_1,y,\mu)\pi(dy;x_1,\mu) -\int_{\R}h(x_2,y,\mu)\pi(dy;x_2,\mu)  =  \nonumber\\
&\qquad=  \int_{\R} h(x_1,y,\mu)-h(x_2,y,\mu)-[\mc{L}_{x_1,\mu}-\mc{L}_{x_2,\mu}]v(x_2,y,\mu)\pi(dy;x_1,\mu)
\end{align}
for all $x_1,x_2\in \R,\mu\in \mc{P}_2(\R)$.

Here $v$ solves
\begin{align}\label{eq:poissoneqfortransferformulas}
\mc{L}_{x,\mu}v(x,y,\mu)& = h(x,y,\mu)- \int_\R h(x,\bar{y},\mu)\pi(d\bar{y};x,\mu).
\end{align}

Consider also $\mc{L}^{(k,j,\bm{\alpha}(\bm{p}_k))}_{x,\mu}[z_{\bm{p}_k}$ the differential operator acting on $\phi \in C^2_b(\R)$ by
\begin{align*}
\mc{L}^{(k,j,\bm{\alpha}(\bm{p}_k))}_{x,\mu}[z_{\bm{p}_k}]\phi(y)=D^{(k,j,\bm{\alpha}(\bm{p}_k))}f(x,y,\mu)[z_{\bm{p}_k}]\phi'(y)+D^{(k,j,\bm{\alpha}(\bm{p}_k))}a(x,y,\mu)[z_{\bm{p}_k}]\phi''(y).
\end{align*}

Assume that for some complete collection of multi-indices $\bm{\zeta}$, that $h,a,f\in \mc{M}_{p}^{\bm{\zeta}}(\R\times\R\times \mc{P}_2(\R))$, and that $v_{y},v_{yy}\in \mc{M}_{p}^{\bm{\zeta}'}(\R^d\times\R^d\times \mc{P}_2(\R^d))$, where $\bm{\zeta}'$ is obtained from removing any multi-indices which contain the maximal first and second values from $\bm{\zeta}$. 
Then for any multi-index $(n,l,\bm{\beta})\in\bm{\zeta}$:
\begin{align}\label{eq:derivativetransferformula}
&D^{(n,l,\bm{\beta})}\int_{\R} h(x,y,\mu)\pi(dy;x,\mu)[z_1,...,z_n]= \int_{\R}\left(D^{(n,l,\bm{\beta})}h(x,y,\mu)[z_1,...,z_n] - \right.\\
&\quad\left.-\sum_{k=0}^n\sum_{j=0}^l\sum_{\bm{p}_k} C_{(\bm{p}_k,j,n,l)} \mc{L}^{(k,j,\bm{\alpha}(\bm{p}_k))}_{x,\mu}[z_{\bm{p}_k}] D^{(n-k,l-j,\bm{\alpha}(\bm{p}'_{n-k}))}v(x,y,\mu)[z_{\bm{p}'_{n-k}}]\right) \pi(dy;x,\mu)\nonumber
\end{align}
where here $\bm{p}_k\in \binom{\br{1,...,n}}{k}$ with $\bm{p}'_{n-k} = \br{1,...,n} \setminus \bm{p}_k$, for $\bm{p}_k = \br{p_1,...,p_k}$, the argument $[z_{\bm{p}_k}]$ denotes $[z_{p_1},...,z_{p_k}]$, and $\bm{\alpha}(\bm{p}_k)\in \bb{N}^k$ is determined by $\bm{\beta}=(\beta_1,...,\beta_n)$ by $\bm{\alpha}(\bm{p}_k) = (\alpha_1,...,\alpha_k)$, $\alpha_j = \beta_{p_j},j\in\br{1,...,k}$, and similarly for $\bm{\alpha}(\bm{p}'_{n-k})$. Also here $C_{(\bm{p}_0,0,n,l)}=0$, and $C_{(\bm{p}_k,j,n,l)}>0,C_{(\bm{p}_k,j,n,l)}\in\bb{N}$ for $(k,j)\in \bb{N}^2, (k,j)\neq (0,0)$, (see Remark \ref{Rem:C_Def} for the exact definition of these constants).


The same result holds replacing $D^{(n,l,\bm{\beta})}$ with $\bm{\delta}^{(n,l,\bm{\beta})}$ if in addition we assume $h,a,f\in \mc{M}_{\bm{\delta},p}^{\bm{\zeta}}(\R\times\R\times \mc{P}_2(\R))$,$v_{y},v_{yy}\in \mc{M}_{\bm{\delta},p}^{\bm{\zeta}'}(\R^d\times\R^d\times \mc{P}_2(\R^d))$. In this setting we will denote by $\mc{L}^{(k,j,\bm{\alpha}(\bm{p}_k)),\bm{\delta}}_{x,\mu}[z_{\bm{p}_k}]$ is the differential operator acting on $\phi \in C^2_b(\R)$ by
\begin{align*}
\mc{L}^{(k,j,\bm{\alpha}(\bm{p}_k)),\bm{\delta}}_{x,\mu}[z_{\bm{p}_k}]\phi(y)=\bm{\delta}^{(k,j,\bm{\alpha}(\bm{p}_k))}f(x,y,\mu)[z_{\bm{p}_k}]\phi'(y)+\bm{\delta}^{(k,j,\bm{\alpha}(\bm{p}_k))}a(x,y,\mu)[z_{\bm{p}_k}]\phi''(y).
\end{align*}

\end{lemma}
\begin{proof}
The proofs of (\ref{eq:muLipschitztransferformula}), (\ref{eq:xLipschitztransferformula}) and of \eqref{eq:derivativetransferformula} for the Lions derivatives is the content of Lemma A.4 in \cite{BezemekSpiliopoulosAveraging2022}.

For the linear functional derivatives, we can use the exact same proof as in Lemma A.4 of \cite{BezemekSpiliopoulosAveraging2022}. Though it is not immediately obvious from the definition of the linear functional derivative that standard properties of derivatives such as chain and product rule apply, we can use Proposition 5.44/Remark 5.47 along with the representation (5.50) from Proposition 5.51 in \cite{CD} to see that these properties are inherited from the Lions derivative (which is defined via lifting and using a Fr\'echet derivative). Note that the needed uniform in $\mu$ Lipschitz continuity assumption for the Lions derivatives of $a,b,h$ needed for Proposition 5.51 is already implied by the definition of $\mc{M}_{p}^{\bm{\zeta}}(\R\times\R\times \mc{P}_2(\R))$. 

Then we get
\begin{align*}
\frac{\delta}{\delta m}\int_{\R}h(x,y,\mu)\pi(y;x,\mu)dy[z]& = \int_{\R} \frac{\delta}{\delta m}h(x,y,\mu)[z]- \mc{L}^{(1,0,0),\bm{\delta}}_{x,\mu}[z]v(x,y,\mu)\pi(dy;x,\mu),
\end{align*}
for all $x,z\in\R,\mu\in \mc{P}_2$, and can induct on $l,n$ in the same way as is done in Lemma A.4 of \cite{BezemekSpiliopoulosAveraging2022}. The details are omitted for brevity.

\end{proof}

\begin{remark}\label{Rem:C_Def}
The non-negative integers $C_{(\bm{p}_k,j,n,l)}$ in the statement of Lemma \ref{lemma:derivativetransferformulas} can be iteratively computed according to the following rules:
\begin{align*}
C_{(\bm{p}_0,0,0,0)}&=0
\end{align*}
To go up in $l$ (taking an $x$ derivative), we have for any $l,n\in\bb{N}$, $\bb{N}\ni j\leq l+1$, $\bb{N}\ni k\leq n$, and $\bm{p}_k\in\binom{\br{1,...,n}}{k}$:
\begin{align*}
C_{(\bm{p}_k,j,n,l+1)}& =
\begin{cases}
C_{(\bm{p}_k,l,n,l)}, &\text{ if }j=l+1\\
C_{(\bm{p}_k,0,n,l)}, &\text{ if }j=0\\
C_{(\bm{p}_k,1,n,l)}+1, &\text{ if }j=1 \\
C_{(\bm{p}_k,j-1,n,l)} +C_{(\bm{p}_k,j,n,l)}, &\text{ otherwise}\\
\end{cases}
\end{align*}

To go up in $n$ (taking a measure derivative), we have for any $l,n\in\bb{N}$, $\bb{N}\ni j\leq l$, $\bb{N}\ni k\leq n+1$, and $\bm{p}_k\in\binom{\br{1,...,n+1}}{k}$
\begin{align*}
C_{(\bm{p}_k,j,n+1,l)}& =
\begin{cases}
C_{(\br{1,...,n},j,n,l)}&\text{ if }k=n+1\\
C_{(\bm{p}_0,j,n,l)}&\text{ if }k=0\\
\1_{\bm{p}_1 = \br{n+1}}+C_{(\bm{p}_1,j,n,l)}\1_{\bm{p}_1 \neq  \br{n+1}}&\text{ if }k=1\\
C_{(\bm{p}_{k}\setminus \br{n+1},j,n,l)} \1_{\br{n+1}\in\bm{p}_{k} } +C_{(\bm{p}_k,j,n,l)}\1_{\br{n+1}\not\in\bm{p}_{k}}, &\text{ otherwise}.\\
\end{cases}
\end{align*}
\end{remark}
\begin{lemma}\label{lemma:explicitrateofgrowthofderivativesinparameters1D}
{}Consider $B:\R\times\R\times\bb{P}_2(\R)\tto \R$ continuous such that
\begin{align*}
\int_{\R}B(x,y,\mu)\pi(dy;x,\mu)=0,\forall x\in \R,\mu\in \mc{P}_2(\R).
\end{align*}

Suppose that for some complete collection of multi-indices $\bm{\zeta}$ that $B,a,f\in \mc{M}_{p}^{\bm{\zeta}}(\R\times\R\times \mc{P}_2(\R))$.
Then for the unique classical solution $u:\R\times\R\times\mc{P}_2(\R)\tto \R$ to
\begin{align*}
L_{x,\mu}u(x,y,\mu)=B(x,y,\mu)
\end{align*}
such that $u$ is continuous in $(x,y,\bb{W}_2)$, $\int_{\R}u(x,y,\mu)\pi(dy;x,\mu)=0$, and $u$ has at most polynomial growth as $|y|\tto \infty$ (which exists by Lemma \ref{lemma:Ganguly1DCellProblemResult}),
\begin{enumerate}
\item $u,u_y,u_{yy}\in \mc{M}_{p}^{\bm{\zeta}}(\R\times\R\times \mc{P}_2(\R))$.

\item If $B,a,f\in \mc{M}_{\bm{\delta},p}^{\bm{\zeta}}(\R\times\R\times \mc{P}_2(\R))\cap \mc{M}_{p}^{\bm{\zeta}}(\R\times\R\times \mc{P}_2(\R))$, then $u,u_y,u_{yy}\in \mc{M}_{\bm{\delta},p}^{\bm{\zeta}}(\R\times\R\times \mc{P}_2(\R))$.

\item If $B,a,f\in \mc{M}_{p,L}^{\bm{\zeta}}(\R\times\R\times \mc{P}_2(\R))$, then $u,u_y,u_{yy}\in \mc{M}_{p,L}^{\bm{\zeta}}(\R\times\R\times \mc{P}_2(\R))$.
\end{enumerate}

Moreover, if we suppose that for all multi-indices $(n,l,\bm{\beta})\in \bm{\zeta}$, $q_f(n,l,\bm{\beta})\leq 1$ and $q_a(n,l,\bm{\beta})\leq 0$ (using here the notation of \eqref{eq:newqnotation}), we have control on the growth rate of the derivatives of $u$ in terms of those of $B$. In particular, for any $(n,l,\bm{\beta})\in \bm{\zeta}$:
\begin{align*}
q_u(n,l,\bm{\beta}) \leq \max\br{q_{B}(k,j,\bm{\alpha}(k)):\alpha(k)\in \binom{\bm{\beta}}{k},k\leq n,j\leq l},
\end{align*}%

when the right hand side is nonzero, and the corresponding term grows at most like $\ln(|y|)$ as $|y|\toinf$ when the left hand side is zero. In addition, $q_{u_y}(n,l,\bm{\beta}) \leq q_{u}(n,l,\bm{\beta})-1$, and $q_{u_{yy}}(n,l,\bm{\beta}) \leq  q_{u}(n,l,\bm{\beta})$, for all $(n,l,\bm{\beta})\in\bm{\zeta}$. 
\end{lemma}
\begin{proof}

For 1), the proof essentially uses the same tools and a similar method to Lemma A.2 in \cite{BezemekSpiliopoulosAveraging2022}, so we will only check this in the case for $(n,l,\bm{\beta})=(0,1,0)$ and then comment on how the rest of the terms follow. Importantly, Lemma A.2 in \cite{BezemekSpiliopoulosAveraging2022} only assumes existence and polynomial growth of derivatives of the solution $u$ up to one order less than the derivative obtained there.

The result for $(n,l,\bm{\beta})=(0,0,0)$ is just another way of writing Lemma \ref{lemma:Ganguly1DCellProblemResult}.

The differentiability and continuity of the derivatives is immediate via the explicit representation for $u$
\begin{align}
v(x,y,\mu)& = \int_{-\infty}^y \frac{1}{a(x,\bar{y},\mu)\pi(\bar{y};x,\mu)}\biggl[\int_{-\infty}^{\bar{y}}B(x,\tilde{y},\mu)\pi(\tilde{y};x,\mu)d\tilde{y} \biggr]d\bar{y}\nonumber\\
\label{eq:explicit1Dpi}\pi(y;x,\mu)& = \frac{Z(x,\mu)}{a(x,y,\mu)}\exp\biggl(\int_0^y \frac{f(x,\bar{y},\mu)}{a(x,\bar{y},\mu)}d\bar{y}\biggr)
\end{align}
where $Z^{-1}(x,\mu)\coloneqq \int_{\R}\frac{1}{a(x,y,\mu)}\exp\biggl(\int_0^y \frac{f(x,\bar{y},\mu)}{a(x,\bar{y},\mu)}d\bar{y}\biggr)dy$ is the normalizing constant.

 To obtain the rate of polynomial growth of $u_x$, we differentiate the equation that $u$ satisfies to get
\begin{align*}
L_{x,\mu}u_x(x,y,\mu)&=B_x(x,y,\mu)- f_x(x,y,\mu)u_y(x,y,\mu) - a_x(x,y,\mu)u_{yy}(x,y,\mu)\\
& = B_x(x,y,\mu) - L^{(0,1,0)}_{x,\mu}u(x,y,\mu)
\end{align*}
in the notation of Lemma A.2 in \cite{BezemekSpiliopoulosAveraging2022}. But by the centering condition on $B$, we have that letting $B=h$ in Lemma A.2 in \cite{BezemekSpiliopoulosAveraging2022}, $u=v$ in the statement of that same lemma. Thus we have
\begin{align*}
\int_{\R}\left(B_x(x,y,\mu) - L^{(0,1,0)}_{x,\mu}u(x,y,\mu)\right)\pi(dy;x,\mu)& = \frac{\partial}{\partial x}\int_{\R}B(x,y,\mu)\pi(dy;x,\mu)=0,
\end{align*}
and the inhomogeneity of the elliptic PDE that $u_x$ solves, in fact obeys the centering condition, and hence Lemma \ref{lemma:Ganguly1DCellProblemResult} applies. From the same lemma we already know that $q_{u,y}(0,0,0)=q_{B}(0,0,0)-1$ and $q_{yy}=q_B(0,0,0)$. This establishes that $u_x$ grows at most polynomially in $y$ uniformly in $x,\mu$. Under the additional assumptions that $q_{f}(0,1,0)\leq 1$ and $q_{a}(0,1,0)\leq 0$, we have the inhomogeneity is $O(|y|^{q_B(0,0,0)\vee q_{B}(0,1,0)})$. So by Lemma \ref{lemma:Ganguly1DCellProblemResult}, $q_{u,x}=q_{B}(0,0,0)\vee q_{B}(0,1,0),q_{u,x,y}=q_{B}(0,0,0)\vee q_{B}(0,1,0)-1,q_{u,x,y,y}=q_{B}(0,0,0)\vee q_{B}(0,1,0)$.

All of the bounds work in the same way, with the inhomogeneity of the elliptic PDE of the desired derivative of $u$ solves being the integrand of the expression for the corresponding derivative of $\bar{B}(x,y,\mu)$ from Lemma A.2 in \cite{BezemekSpiliopoulosAveraging2022}. Put explicitly:
\begin{align}\label{eq:formulasatisfiedbyderivatives2}
&L_{x,\mu}D^{(n,l,\bm{\beta})}u(x,y,\mu)[z_1,...,z_n]
 = D^{(n,l,\bm{\beta})}B(x,y,\mu)[z_1,...,z_n]- \nonumber\\
 &\hspace{4cm}- \sum_{k=0}^n\sum_{j=0}^l\sum_{\bm{p}_k} C_{(\bm{p}_k,j,n,l)} L^{(k,j,\bm{\alpha}(\bm{p}_k))}_{x,\mu}[z_{\bm{p}_k}] D^{(n-k,l-j,\bm{\alpha}(\bm{p}'_{n-k}))}u(x,y,\mu)[z_{\bm{p}'_{n-k}}],
\end{align}
where the constants $C_{(\bm{p}_k,j,n,l)}$ are defined inductively in Remark A.3 in \cite{BezemekSpiliopoulosAveraging2022} and $L^{(k,j,\bm{\alpha}(\bm{p}_k))}_{x,\mu}[z_{\bm{p}_k}]$ is the differential operator acting on $\phi \in C^2_b(\R)$ by
\begin{align*}
L^{(k,j,\bm{\alpha}(\bm{p}_k))}_{x,\mu}[z_{\bm{p}_k}]\phi(y)=D^{(k,j,\bm{\alpha}(\bm{p}_k))}f(x,y,\mu)[z_{\bm{p}_k}]\phi'(y)+D^{(k,j,\bm{\alpha}(\bm{p}_k))}a(x,y,\mu)[z_{\bm{p}_k}]\phi''(y).
\end{align*}
The first $y$ derivative of a lower order derivative in a parameter of $u$ in the inhomogeneity is always multiplied by a derivative of $f$, and so if that derivative of $f$ grows at most linearly in $y$, the growth of that term is at most that of that lower order derivative of $u$, and same for the second $y$ derivative in a parameter of $u$ in the inhomogeneity, which multiplied by a bounded lower order derivative of $a$. Thus it is clear the result follows by proceeding inductively on $n,l$.

The proof for 2) follows in the exact same way. We note here that Lemma A.2 in \cite{BezemekSpiliopoulosAveraging2022} holds for the linear functional derivatives $\bm{\delta}^{(n,l,\bm{\beta})}$ in place of the Lions derivatives $D^{(n,l,\bm{\beta})}$ if in addition we assume $h,a,f\in \mc{M}_{\bm{\delta},p}^{\bm{\zeta}}(\R\times\R\times \mc{P}_2(\R))$,$v_{y},v_{yy}\in \mc{M}_{\bm{\delta},p}^{\bm{\zeta}'}(\R^d\times\R^d\times \mc{P}_2(\R^d))$.

The proof for 3) is similar to step 4 in the proof of Theorem 2.1 in \cite{RocknerFullyCoupled}.
For the case $(n,l,\bm{\beta})=(0,0,0)$, we first note that
\begin{align*}
L_{x,\mu_1}[u(x,y,\mu_1) - u(x,y,\mu_2)] &=B(x,y,\mu_1) -  L_{x,\mu_1}u(x,y,\mu_2) \nonumber\\
&= B(x,y,\mu_1) - B(x,y,\mu_2) -[L_{x,\mu_1}-L_{x,\mu_2}]u(x,y,\mu_2).
\end{align*}
By the transfer formula in Lemma A.2 of \cite{BezemekSpiliopoulosAveraging2022} we have
\begin{align*}
&\int_{\R}\left(B(x,y,\mu_1) - B(x,y,\mu_2) -[L_{x,\mu_1}-L_{x,\mu_2}]u(x,y,\mu_2)\right)\pi(dy,x,\mu_1) = \nonumber\\
&\qquad= \int_{\R}B(x,y,\mu_1)\pi(dy;x,\mu_1) -\int_{\R}B(x,y,\mu_2)\pi(dy;x,\mu_2)=0,
\end{align*}
so in fact the inhomogeneity in the above Poisson equation is centered. Now, rather than using Lemma \ref{lemma:Ganguly1DCellProblemResult}, we apply \cite{PV1} Theorem 2 to get there is $k\in\R$ sufficiently large and $C>0$ such that for all $x\in\R,\mu_1,\mu_2\in\mc{P}_2(\R)$
\begin{align*}
\sup_{y\in\R}\frac{|u(x,y,\mu_1) - u(x,y,\mu_2)|}{(1+|y|)^k}&\leq C \sup_{y\in\R}\frac{\biggl| B(x,y,\mu_1) - B(x,y,\mu_2) -[L_{x,\mu_1}-L_{x,\mu_2}]u(x,y,\mu_2) \biggr|}{(1+|y|)^k}\\
&\leq C\bb{W}_2(\mu_1,\mu_2)
\end{align*}
by the Lipschitz assumptions on $B,f,a$. Thus for all $x,y\in\R,\mu_1,\mu_2\in\mc{P}_2(\R)$,
\begin{align*}
|u(x,y,\mu_1) - u(x,y,\mu_2)|\leq C\bb{W}_2(\mu_1,\mu_2)(1+|y|)^k.
\end{align*}
To see then that there are $k',k''\in\R,C',C''>0$ such that
\begin{align*}
|u_y(x,y,\mu_1) - u_y(x,y,\mu_2)|&\leq C\bb{W}_2(\mu_1,\mu_2)(1+|y|)^{k'}\\
|u_{yy}(x,y,\mu_1) - u_{yy}(x,y,\mu_2)|&\leq C\bb{W}_2(\mu_1,\mu_2)(1+|y|)^{k''},
\end{align*}
we can apply the result of \cite{GS} Lemma B.1 and Remark B.2, and the last line of Proposition A.4 in the same reference.

The proof with $\mu_1,\mu_2$ replaced by $x_1,x_2$ follows in the same way.

For the Lipschitz property in $z$, we first recall that for all $x,y,z\in\R,\mu\in\mc{P}_2(\R)$
\begin{align*}
L_{x,\mu}D^{(1,0,0)}u(x,y,\mu)[z] = D^{(1,0,0)}B(x,y,\mu)[z]-L^{(1,0,0)}_{x,\mu}[z]u(x,y,\mu)
\end{align*}
so
\begin{align*}
L_{x,\mu}\biggl[D^{(1,0,0)}u(x,y,\mu)[z_1]-D^{(1,0,0)}u(x,y,\mu)[z_2]\biggr]&=D^{(1,0,0)}B(x,y,\mu)[z_1]-L^{(1,0,0)}_{x,\mu}[z_1]u(x,y,\mu)\\
&-\biggl[D^{(1,0,0)}B(x,y,\mu)[z_2]-L^{(1,0,0)}_{x,\mu}[z_2]u(x,y,\mu)\biggr].
\end{align*}
By the transfer formula in Lemma A.2 of \cite{BezemekSpiliopoulosAveraging2022} we have for all $x,z\in\R,\mu\in\mc{P}_2(\R)$:
\begin{align*}
\int_{\R}\left(D^{(1,0,0)}B(x,y,\mu)[z]-L^{(1,0,0)}_{x,\mu}[z]u(x,y,\mu)\right)\pi(dy;x,\mu) = D^{(0,1,0)}\int_{\R}B(x,y,\mu)\pi(dy;x,\mu)[z] =0,
\end{align*}
so the inhomogeneity in the Poisson equation above in centered. Thus, using the same argument as for the other Lipschitz continuity as well as the fact that $D^{(1,0,0)}B,D^{(1,0,0)}f,D^{(1,0,0)}a$ are Lipschitz in $z$ and $u_y,u_{yy}$ grow at most polynomially in $y$, we get there is $K\in \R$ and $C>0$ such that
\begin{align*}
\biggl|D^{(1,0,0)}u(x,y,\mu)[z_1]-D^{(1,0,0)}u(x,y,\mu)[z_2]\biggr|\leq C|z_1-z_2| (1+|y|)^k,
\end{align*}
and similarly for $D^{(1,0,0)}u_y$ and $D^{(1,0,0)}u_{yy}$.

Then using the Poisson equation the derivatives satisfy given in Equation \eqref{eq:formulasatisfiedbyderivatives2}, we can iteratively use this same approach, along with the fact that products and sums of functions in $\mc{M}_{p,L}^{\bm{\zeta}}(\R\times\R\times \mc{P}_2(\R))$ remain in $\mc{M}_{p,L}^{\bm{\zeta}}(\R\times\R\times \mc{P}_2(\R))$, to achieve the full result.

\end{proof}

\begin{lemma}\label{lem:regularityofaveragedcoefficients}
Suppose that for some complete collection of multi-indices $\bm{\zeta}$ that $h,a,f\in \mc{M}_{p}^{\bm{\zeta}}(\R\times\R\times \mc{P}_2(\R))$. Then $\int_{\R} h(x,y,\mu)\pi(dy;x,\mu)\in \mc{M}_{b}^{\bm{\zeta}}(\R\times \mc{P}_2(\R))$. If in addition, $h,a,f\in \mc{M}_{\bm{\delta},p}^{\bm{\zeta}}(\R\times\R\times \mc{P}_2(\R))$, then $\int_{\R} h(x,y,\mu)\pi(dy;x,\mu)\in \mc{M}_{\bm{\delta},b}^{\bm{\zeta}}(\R\times \mc{P}_2(\R))$. Further, if we have that  $h,a,f\in \mc{M}_{p,L}^{\bm{\zeta}}(\R\times\R\times \mc{P}_2(\R))$, then $\int_{\R} h(x,y,\mu)\pi(dy;x,\mu)\in \mc{M}_{b,L}^{\bm{\zeta}}(\R\times \mc{P}_2(\R))$.
\end{lemma}
\begin{proof}
This follows via Lemmas A.2 in \cite{BezemekSpiliopoulosAveraging2022} and \ref{lemma:Ganguly1DCellProblemResult} in a similar way to Lemma \ref{lemma:explicitrateofgrowthofderivativesinparameters1D}. The details are omitted.

\end{proof}
\subsection{Result for the d-Dimensional Poisson Equation}\label{subsec:multidimpoissonequation}
{}\begin{lemma}\label{lemma:extendedrocknermultidimcellproblem}
Suppose $F:\R^d\times\R^d\times \mc{P}_2(\R)\tto \R^d,G:\R^d\times\R^d\times \mc{P}_2(\R)\tto \R,\tau:\R^d\times\R^d\times\mc{P}_2(\R)\tto \R^{d\times m}$
\begin{align*}
&|F(x_1,y_1,\mu_1)-F(x_2,y_2,\mu_2)|+|G(x_1,y_1,\mu_1)-G(x_2,y_2,\mu_2)|+\norm{\tau(x_1,y_1,\mu_1) - \tau(x_2,y_2,\mu_2)}\\
&\leq C[|x_1-x_2|+|y_1-y_2|+\bb{W}_2(\mu_1,\mu_2)],\forall x_1,x_2,y_1,y_2\in \R^d,\mu_1,\mu_2\in \mc{P}_2(\R^d),
\end{align*}
and there exists $\beta>0$ such that for all $x\in\R^d,\mu\in \mc{P}_2(\R)$:
\begin{align*}
2\langle F(x,y_1,\mu)-F(x,y_2,\mu),(y_1-y_2)\rangle+3\norm{\tau(x,y_1,\mu)-\tau(x,y_2,\mu)}^2 &\leq -\beta |y_1-y_2|^2.
\end{align*}
Here $\langle \cdot,\cdot\rangle$ is denoting the inner product on $\R^d$ and $\norm{\cdot}$ the matrix norm. Also assume that $\tau$ is bounded, and
\begin{align*}
|G(x,y,\mu)|,|F(x,y,\mu)|\leq C(1+|y|),\forall x\in \R^d,\mu\in \mc{P}_2(\R).
\end{align*}

Define the differential operator $\tilde{\mc{L}}_{x,\mu}$ which for each $x\in \R^d,\mu\in \mc{P}_2(\R)$ acts on $\phi\in C^2_b(\R)$ by
\begin{align*}
\tilde{\mc{L}}_{x,\mu}\phi(y)& \coloneqq F(x,y,\mu)\cdot \nabla \phi(y) + \frac{1}{2}  \tau \tau^\top (x,y,\mu): \nabla^2 \phi(y).
\end{align*}

Then there is a unique invariant measure $\nu(\cdot;x,\mu)$ associated to $\tilde{\mc{L}}_{x,\mu}$ for each $x,\mu$, and we assume the centering condition on $G$:
\begin{align*}
\int_{\R^d}G(x,y,\mu)\nu(dy;x,\mu)&=0,\forall x\in\R^d,\mu\in\mc{P}_2(\R).
\end{align*}

Finally, we assume the below derivatives all exist, are jointly continuous in $(x,y,\bb{W}_2)$ and auxiliary variables where applicable, and satisfy:
\begin{align*}
\sup_{x\in\R^d,\mu\in\mc{P}_2(\R)}\max\{|\partial_x G(x,\mu,y_1)-\partial_x G(x,\mu,y_2)|,|\partial_y G(x,\mu,y_1)-\partial_y G(x,\mu,y_2)|\}&\leq C|y_1-y_2|\\
\sup_{x\in\R^d,\mu\in\mc{P}_2(\R)}\max\{\norm{\partial^2_x G(x,\mu,y_1)-\partial^2_x G(x,\mu,y_2)},\norm{\partial^2_y G(x,\mu,y_1)-\partial^2_y G(x,\mu,y_2)}\}&\leq C|y_1-y_2|\\
\sup_{x\in\R^d,\mu\in\mc{P}_2(\R)}\norm{\partial_x\partial_y G(x,\mu,y_1)-\partial_x\partial_y G(x,\mu,y_2)}&\leq C|y_1-y_2|\\
\sup_{x\in\R^d,\mu\in\mc{P}_2(\R)}\norm{\partial_\mu \partial_x G(x,\mu,y_1)[\cdot]-\partial_\mu \partial_x G(x,\mu,y_2)[\cdot]}_{L^2(\R,\mu)}&\leq C|y_1-y_2|\\
\sup_{x\in\R^d,\mu\in\mc{P}_2(\R)}\norm{\partial_z\partial_\mu G(x,\mu,y_1)[\cdot]-\partial_z\partial_\mu G(x,\mu,y_2)[\cdot]}_{L^2(\R,\mu)}&\leq C|y_1-y_2|\\
\sup_{x\in\R^d,\mu\in\mc{P}_2(\R)}\norm{\partial_\mu \partial_x G(x,\mu,y_1)[\cdot]-\partial_\mu \partial_x G(x,\mu,y_2)[\cdot]}_{L^2(\R,\mu)}&\leq C|y_1-y_2|\\
\sup_{x\in\R^d,\mu\in\mc{P}_2(\R)}\norm{\partial_\mu \partial_y G(x,\mu,y_1)[\cdot]-\partial_\mu \partial_y G(x,\mu,y_2)[\cdot]}_{L^2(\R,\mu)}&\leq C|y_1-y_2|\\
\sup_{x\in\R^d,\mu\in\mc{P}_2(\R)}\norm{\partial^2_\mu G(x,\mu,y_1)[\cdot,\cdot]-\partial^2_\mu G(x,\mu,y_2)[\cdot,\cdot]}_{L^2(\R,\mu)\otimes L^2(\R,\mu)}&\leq C|y_1-y_2|\\
\sup_{x,y\in \R^d,\mu\in \mc{P}_2(\R)} \max\biggl\lbrace\norm{\partial_y\partial_x G(x,\mu,y)}, \norm{\partial^2_y G(x,\mu,y)}, \norm{\partial_\mu \partial_y G(x,\mu,y)[\cdot]}_{L^2(\R,\mu)}\biggr\rbrace&\leq C
\end{align*}
and same for $G$ replaced by $F$ and $\tau$, and in addition
\begin{align*}
&\sup_{x,y\in \R^d,\mu\in \mc{P}_2(\R)} \max\biggl\lbrace\norm{\partial^2_x F(x,\mu,y)},\norm{\partial^2_x \tau(x,\mu,y)},\norm{\partial_z\partial_\mu F(x,\mu,y)[\cdot]}_{L^2(\mu,\R)},\norm{\partial_z\partial_\mu \tau(x,\mu,y)[\cdot]}_{L^2(\mu,\R)},\\
&\qquad\norm{\partial_\mu\partial_x F(x,\mu,y)[\cdot]}_{L^2(\mu,\R)},\norm{\partial_\mu\partial_x \tau(x,\mu,y)[\cdot]}_{L^2(\mu,\R)},\norm{\partial^2_\mu F(x,\mu,y)[\cdot,\cdot]}_{L^2(\mu,\R)\otimes L^2(\mu,\R)},\nonumber\\
&\qquad\norm{\partial^2_\mu \tau(x,\mu,y)[\cdot,\cdot]}_{L^2(\mu,\R)\otimes L^2(\mu,\R)}\biggr\rbrace\leq C.
\end{align*}
Then the partial differential equation
\begin{align*}
\tilde{\mc{L}}_{x,\mu}\chi(x,y,\mu)& = -G(x,y,\mu)
\end{align*}
admits a unique classical solution $\chi:\R^d\times\R^d\times\mc{P}_2(\R)\tto \R$ which has all of the above derivatives, and
\begin{align*}
&\sup_{x\in\R^d,\mu\in\mc{P}_2(\R)}\max\biggl\lbrace|\chi(x,y,\mu)|,\norm{\partial_x \chi(x,y,\mu)},\norm{\partial_\mu \chi(x,y,\mu)[\cdot]}_{L^2(\R,\mu)},\norm{\partial^2_x \chi(x,y,\mu)},\norm{\partial_z\partial_\mu \chi(x,y,\mu)[\cdot]}_{L^2(\R,\mu)},\\
&\norm{\partial_\mu \partial_x\chi(x,y,\mu)[\cdot]}_{L^2(\R,\mu)},\norm{\partial^2_\mu \chi(x,y,\mu)[\cdot,\cdot]}_{L^2(\R,\mu)\otimes L^2(\R,\mu)}\biggr\rbrace \leq C(1+|y|),\forall y\in \R^2,\\
&\sup_{x,y\in\R^d,\mu\in\mc{P}_2(\R)}\max\biggl\lbrace\norm{\partial_y \chi(x,y,\mu)},\norm{\partial^2_y \chi(x,y,\mu)},\norm{\partial_x\partial_y \chi(x,y,\mu)},\norm{\partial_\mu \partial_y \chi(x,y,\mu)}_{L^2(\R,\mu)}\leq C.
\end{align*}
Moreover, if all listed derivatives of $F,G,\tau$ are jointly continuous in $(x,y,\bb{W}_2)$, then so are listed derivatives of $\chi$.

In the notation of Definition \ref{def:lionsderivativeclasses}, this conclusion reads $\chi \in \tilde{\mc{M}}^{\tilde{\bm{\zeta}}}_p(\R^2\times\R^2\times\mc{P}_2(\R))$, $\chi_y \in \tilde{\mc{M}}^{\tilde{\bm{\zeta}}_1}_p(\R^2\times\R^2\times\mc{P}_2(\R))$, $\chi_{yy} \in \tilde{\mc{M}}^{(0,0,0)}_p(\R^2\times\R^2\times\mc{P}_2(\R))$ with $\tilde{q}_\chi(n,l,\bm{\beta})\leq 1,\forall (n,l,\bm{\beta})\in \tilde{\bm{\zeta}}$, $\tilde{q}_{\chi_y}(n,l,\bm{\beta})\leq 0,\forall (n,l,\bm{\beta})\in \tilde{\bm{\zeta}}_1$, and $\tilde{q}_{\chi_{yy}}(0,0,0)\leq 0$ where $\tilde{\bm{\zeta}},\tilde{\bm{\zeta}}_1$ are as in Equation \eqref{eq:collectionsofmultiindices}.
\end{lemma}

\begin{proof}
The arguments here follow closely those in \cite{RocknerMcKeanVlasov}.
Existence and uniqueness for the invariant measure and strong solution from the Poisson equation are the subject of the beginning of Section 3.3 and Section 4.1 of \cite{RocknerMcKeanVlasov}. The bound for $\chi$, $\partial_y\chi$, $\partial_x\chi,\partial_\mu\chi,\partial^2_{x}\chi$, and $\partial_z\partial_\mu \chi$ is also the subject of Proposition 4.1/Section 6.3 of \cite{RocknerMcKeanVlasov}, where we made the modification that $\tau,F$ (their $g,f$ respectively) are bounded in $x,\mu$, from which one can see that the bound on the solution is also uniform in $x,\mu$.

Thus we just need to show the bounds for $\partial^2_y\chi$, $\partial_x\partial_y \chi$, $\partial_\mu\partial_y\chi$, $\partial_\mu \partial_x\chi$, and $\partial^2_\mu \chi$. The bounds for $\partial^2_y\chi$, $\partial_x\partial_y \chi$ and $\partial_\mu\partial_y\chi$ are established in the recent \cite{HLLS} Proposition 3.1.

For the mixed partial derivative in $x$ and $\mu$ and the second partial derivative in $\mu$, we can follow the proof of Proposition 4.1 of \cite{RocknerMcKeanVlasov}. The details are omitted here due to the similarity of the argument.
\end{proof}

\subsection{Some specific examples for which the assumptions of the paper hold}\label{subsec:suffconditionsoncoefficients}
\begin{proposition}\label{prop:suffcondfordoubledcellproblem}
Suppose \ref{assumption:uniformellipticity}- \ref{assumption:centeringcondition} hold. Let $\tilde{\bm{\zeta}},\tilde{\bm{\zeta}}_1$ be as in Equation \eqref{eq:collectionsofmultiindices}, and consider also:
\begin{align*}
\bm{\zeta}&\ni \br{(0,j_1,0),(1,j_2,j_3),(2,j_4,(j_5,0)),(3,0,0):j_1=0,1,2,j_2+j_3 \leq 2,j_4+j_5\leq 1}\\
\bm{\zeta}_1&\ni \br{(0,2,0),(1,j_1,j_2),(2,0,0):j_1+j_2\leq 1}.
\end{align*}
In addition, suppose:
\begin{enumerate}
\item For $h=\tau_1,\tau_2,b$:
\begin{align*}
|h(x_1,y_1,\mu_1)-h(x_2,y_2,\mu_2)|\leq C(|x_1-x_2|+|y_1-y_2|+\bb{W}_2(\mu_1,\mu_2)),\forall x_1,x_2,y\in\R,\mu_1,\mu_2\in \mc{P}_2(\R).
\end{align*}
\item $a,f,b\in \mc{M}^{\bm{\zeta}}_p(\R\times\R\times\mc{P}_2(\R))$, $a_y,f_y,b_y\in \mc{M}^{\tilde{\bm{\zeta}}}_p(\R\times\R\times\mc{P}_2(\R))$, $a_{yy},f_{yy},b_{yy}\in \mc{M}^{\tilde{\bm{\zeta}}_1}_p(\R\times\R\times\mc{P}_2(\R))$, and $a_{yyy},f_{yyy},b_{yyy}$ are bounded.
\item $q_a(n,l,\bm{\beta})\leq 0$ for all $(n,l,\bm{\beta})\in\bm{\zeta},q_{a_y}(n,l,\bm{\beta})\leq 0$ for all $(n,l,\bm{\beta})\in\tilde{\bm{\zeta}},$ and $q_{a_{yy}}(n,l,\bm{\beta})\leq 0$ for all  $(n,l,\bm{\beta})\in\tilde{\bm{\zeta}}_1$. In addition, $q_{a}(0,1,0)<0$.
\item $q_f(n,l,\bm{\beta})\leq 1$ for all $(n,l,\bm{\beta})\in\bm{\zeta},q_f(n,l,\bm{\beta})\leq 0$ for all $(n,l,\bm{\beta})\in\bm{\zeta}_1,$ $q_{f_y}(n,l,\bm{\beta})\leq 0$ for all $(n,l,\bm{\beta})\in\tilde{\bm{\zeta}},$ and $q_{f_{yy}}(n,l,\bm{\beta})\leq 0$ for all  $(n,l,\bm{\beta})\in\tilde{\bm{\zeta}}_1$. In addition, $q_{f}(0,1,0)<0$.
\item $q_b(n,l,\bm{\beta})< 0$ for all $(n,l,\bm{\beta})\in\bm{\zeta},q_{b_y}(n,l,\bm{\beta})\leq 0$ for all $(n,l,\bm{\beta})\in\tilde{\bm{\zeta}},$ and $q_{b_{yy}}(n,l,\bm{\beta})\leq 0$ for all  $(n,l,\bm{\beta})\in\tilde{\bm{\zeta}}_1$.
\end{enumerate}
Then assumptions \ref{assumption:qF2bound} and \ref{assumption:tildechi} hold.
\end{proposition}

\begin{proof}
We first want to show the assumptions of Lemma \ref{lemma:extendedrocknermultidimcellproblem} with $d=2,m=4$ hold with $F_1(x,y,\mu) = f(x_1,y_1,\mu),F_2=f(x_2,y_2,\mu)$, $\tau_{11}(x,y,\mu) = \tau_1(x_1,y_1,\mu)$,$\tau_{12}(x,y,\mu) = \tau_2(x_1,y_1,\mu),\tau_{23}(x,y,\mu) = \tau_1(x_2,y_2,\mu)$,
$\tau_{24}(x,y,\mu) = \tau_2(x_2,y_2,\mu)$, and $\tau_{ij}\equiv 0$ otherwise, and $G(x,y,\mu) = b(x_1,y_1,\mu)\partial_\mu \Phi(x_2,y_2,\mu)[x_1]$ or $G(x,y,\mu) = b(x_1,y_1,\mu)\Phi(x_2,y_2,\mu)$.

Under these assumptions we have $q_{\Phi}(n,l,\bm{\beta}),q_{\Phi_{yy}}(n,l,\bm{\beta})<0$ and $q_{\Phi_y}(n,l,\bm{\beta})<-1$ for all $(n,l,\bf{\beta})\in\bm{\zeta}$ via Lemma \ref{lemma:explicitrateofgrowthofderivativesinparameters1D}.

The first Lipschitz assumption follows by (1) and \ref{assumption:retractiontomean}. The retraction to mean assumption is immediate from \ref{assumption:retractiontomean}. We also have $F$ grows at most linearly in $|y|$ by \ref{assumption:retractiontomean}, and $G$ is in fact bounded by the above assumptions.

Checking the uniform Lipschitz in $y$ assumptions for the derivatives of $G$, we have, for example, for the $x$ derivative of the first choice, that:
\begin{align*}
b_x(x_1,y_1,\mu)\partial_\mu \Phi(x_2,y_2,\mu)[x_1]+b(x_1,y_1,\mu)\partial_z\partial_\mu \Phi(x_2,y_2,\mu)[x_1]
\end{align*}
and
\begin{align*}
b(x_1,y_1,\mu)\partial_\mu \Phi_x(x_2,y_2,\mu)[x_1]
\end{align*}
need to be Lipschitz in $y$ uniformly in $x\in\R^2,\mu\in \mc{P}_2(\R)$. To guarantee that the product of functions is Lipschitz without any more a priori information on the structure of each function, we must have that each function is Lipschitz and bounded. Since we make the assumptions that $q_b(0,j,0),q_{b_y}(0,j,0)\leq 0,j=0,1$ and assumptions on $b,f,a$ such that $q_{\Phi}(1,j_1,j_2),q_{\Phi_y}(1,j_1,j_2)\leq 0,j_1+j_2\leq 1$, this assumption holds. This is where the requirement that for many $(n,l,\bm{\beta})$, $q_{b}(n,l,\bm{\beta})<0$ is coming in to play.

If we differentiate $G$ in the same way to see all the Lipschitz and bounded assumptions needed on each of $b$, $\Phi$, and $\partial_\mu \Phi$'s derivatives, we see that the only term that requires special care under these assumptions is $\partial_\mu\Phi_{yy}(x_2,y_2,\mu)[x_1]$. But using the equation elliptic equation that $\partial_\mu \Phi(x,y,\mu)[z]$ satisfies for each $x,z\in\R,\mu\in\mc{P}_2(\R)$ and Lemma \ref{lemma:Ganguly1DCellProblemResult}, we see for the second derivative of $\partial_\mu\Phi(x,y,\mu)[z]$ to be uniformly Lipschitz in $y$, it is sufficient for
\begin{align*}
\partial_\mu b(x,y,\mu)[z]-\partial_\mu f(x,y,\mu)[z]\Phi_y(x,y,\mu) -\partial_\mu a(x,y,\mu)[z]\Phi_{yy}(x,y,\mu)
\end{align*}
to be uniformly Lipschitz in $y$. $\partial_\mu b(x,y,\mu)[z]$ is already assumed to have this property, and $\Phi_y(x,y,\mu),\Phi_{yy}(x,y,\mu)$ are bounded and uniformly Lipschitz in $y$ by assumption. Hence we just need in addition there that $q_{f}(1,0,0),q_{f_y}(1,0,0),q_{a}(1,0,0),q_{a_y}(1,0,0)\leq 0$ as assumed.

Clearly since we prove the Lipschitz property for each of the derivatives of $G$ by ensuring each component is Lipschitz and bounded, the needed boundedness assumption for the mixed derivatives in $y$ of $G$ also holds.

Now to apply Lemma \ref{lemma:extendedrocknermultidimcellproblem}, we just need to make sure that the needed Lipschitz and bounded assumptions on the derivatives of $F$ and $\tau$ hold. But these are implied by our assumptions on $f$ and $a$.

Now we just need to improve the result of Lemma \ref{lemma:extendedrocknermultidimcellproblem} to get $q_{\chi}(0,j,0)\leq 0,j=0,1$. We turn to \cite{PV1} Theorem 2. We have $q_{G}(0,0,0)<0$, so $q_{\chi}(0,0,0)\leq 0$ by a direct application of that Theorem. Then using that, as remarked in the proof of Lemma \ref{lemma:derivativetransferformulas}, the transfer formula for the $x$ derivatives in the $d$-dimensional case still hold in our setting and that $q_{\chi_y}(0,0,0),q_{\chi_{yy}}(0,0,0)\leq 0$ by Lemma \ref{lemma:extendedrocknermultidimcellproblem} and $q_{f}(0,1,0),q_{a}(0,1,0)<0$ by assumption, we can get the inhomogeneity for the Poisson equation which $\partial_x \chi$ satisfies also decays polynomially in $|y|$ as $|y|\toinf$ uniformly in $x\in\R^2,\mu\in\mc{P}_2(\R)$, so again by \cite{PV1} Theorem 2, $q_{\chi}(0,1,0)\leq 0$.

Note that, while the sufficient conditions posed here for Assumption \ref{assumption:qF2bound} automatically imply those for \ref{assumption:tildechi}, since Assumption \ref{assumption:tildechi} does not require specific polynomial growth, it can actually be proved under much weaker sufficient conditions - see Appendix A of \cite{BezemekSpiliopoulosAveraging2022}.
\end{proof}

\begin{proposition}\label{prop:suffcondrestofassumptions}
Suppose the conditions of Proposition \ref{prop:suffcondfordoubledcellproblem} and \ref{assumption:gsigmabounded} hold. Let $\bm{\zeta}$ and $\bm{\zeta}_1$ be as in Proposition \ref{prop:suffcondfordoubledcellproblem}, consider the collections of multi-indices from Equation \eqref{eq:collectionsofmultiindices}, and let, in addition:
\begin{align*}
\hat{\bm{\zeta}}_1 &\ni \br{(0,j_1,0),(1,j_2,j_3),(2,j_4,(j_5,j_6)),(3,j_7,(j_8,0,0))\\
&:j_1\in\br{0,1,...5},j_3\leq 4,j_2+j_3\leq 5,j_5+j_6\leq 2,j_4+j_5+j_6\leq 3,j_7+j_8\leq 1 }.\\
\mathring{\bm{\zeta}}& = \bar{\bm{\zeta}}\cup \bar{\bm{\zeta}}_{w+2}\\
\mathring{\bm{\zeta}}_1&\ni \br{(j,j_2,0),(1,j_2,j_3):j=0,1,2,j_2=0,1,j_3=1,...,w+2}.\\
\end{align*}
In addition, suppose:
\begin{enumerate}
\item For $h=\sigma,g,c$:
\begin{align*}
|h(x_1,y_1,\mu_1)-h(x_2,y_2,\mu_2)|\leq C(|x_1-x_2|+|y_1-y_2|+\bb{W}_2(\mu_1,\mu_2)),\forall x_1,x_2,y\in\R,\mu_1,\mu_2\in \mc{P}_2(\R).
\end{align*}
\item $g,\sigma,\tau_1,c \in \mc{M}_{p,L}^{\hat{\bm{\zeta}}}$ and $f,a,b\in \mc{M}_{p,L}^{\hat{\bm{\zeta}}_1}$
\item $q_a(0,3,0)\leq 0$, $q_f(0,3,0)\leq 1$, $q_b(0,3,0)\leq 3$
\item  $q_{\sigma}(n,l,\bm{\beta}),q_{c}(n,l,\bm{\beta})\leq 1,q_{g}(n,l,\bm{\beta}),q_{\tau_1}(n,l,\bm{\beta})\leq 2$ for $(n,l,\bm{\beta})\in\mathring{\bm{\zeta}}_1$ and $q_{\sigma}(n,l,\bm{\beta}),q_{c}(n,l,\bm{\beta})\leq 2,q_{g}(n,l,\bm{\beta}),q_{\tau_1}(n,l,\bm{\beta})\leq 3$ for $(n,l,\bm{\beta})\in \mathring{\bm{\zeta}}$.
\item $b,f,\tau_1,\tau_2\in \mc{M}_{\bm{\delta},p}^{\mathring{\bm{\zeta}}_1}$ and $c,\sigma,g\in \mc{M}_{\bm{\delta},p}^{\mathring{\bm{\zeta}}} $.
\item $b,f,a\in \mc{M}_{p}^{\bm{\zeta}_{x,w+3}},\sigma,\tau_1,c,g \in \mc{M}_{p}^{\bm{\zeta}_{x,w+2}}$, and
\begin{align*}
&\norm{\frac{\delta}{\delta m}b(x,y,\mu)[\cdot]}_{w+2},\norm{\frac{\delta}{\delta m}b_x(x,y,\mu)[\cdot]}_{w+2},\norm{\frac{\delta}{\delta m}f(x,y,\mu)[\cdot]}_{w+2},\norm{\frac{\delta}{\delta m}f_x(x,y,\mu)[\cdot]}_{w+2},\norm{\frac{\delta}{\delta m}a(x,y,\mu)[\cdot]}_{w+2},\\
&\norm{\frac{\delta}{\delta m}a_x(x,y,\mu)[\cdot]}_{w+2},\norm{\frac{\delta}{\delta m}g(x,y,\mu)[\cdot]}_{w+2},\norm{\frac{\delta}{\delta m}\sigma(x,y,\mu)[\cdot]}_{w+2},\norm{\frac{\delta}{\delta m}\tau_1(x,y,\mu)[\cdot]}_{w+2},\norm{\frac{\delta}{\delta m}c(x,y,\mu)[\cdot]}_{w+2}\\
& \leq C(1+|y|^k),
\end{align*}
uniformly in $x\in\R,\mu\in\mc{P}_2(\R)$ for some $k\in\bb{N}.$
\item There exists $\bar{\lambda}_->0$ such that $\bar{D}(x,\mu)\geq \bar{\lambda}_-$ for all $x\in\R,\mu\in\mc{P}_2$.
\end{enumerate}
Then assumptions \ref{assumption:strongexistence} and \ref{assumption:multipliedpolynomialgrowth} - \ref{assumption:limitingcoefficientsregularity}, hold, and \ref{assumption:limitingcoefficientsregularityratefunction} holds if we replace $w$ with $r$ in (6).
\end{proposition}
\begin{proof}
\ref{assumption:strongexistence} follows from the above Lipschitz properties, writing the system of SDEs \eqref{eq:slowfast1-Dold} in terms of the empirical projections of the coefficients and using standard strong existence and uniqueness results (see Proposition A.1 in \cite{BS}) for the weakly interacting system \eqref{eq:slowfast1-Dold} and applying Theorem 2.1 in \cite{Wang} to the IID McKean-Vlasov system \eqref{eq:IIDparticles}.

For the limiting system \eqref{eq:LLNlimitold}, we have, noting that $\hat{\bm{\zeta}}$ is the same $\hat{\bm{\zeta}}$ from assumption  \ref{assumption:limitingcoefficientsLionsDerivatives}, and $\hat{\bm{\zeta}}_1$ is just $\hat{\bm{\zeta}}$ with one extra $x$ derivative in all spacial components, that by Lemma \ref{lemma:explicitrateofgrowthofderivativesinparameters1D}, $\Phi,\Phi_x,\Phi_y,\Phi_{xy}\in \mc{M}_{p,L}^{\hat{\bm{\zeta}}}$, and thus under these assumptions $\gamma,D \in \mc{M}_{p,L}^{\hat{\bm{\zeta}}}$, and hence by Lemma \ref{lem:regularityofaveragedcoefficients} $\bar{\gamma},\bar{D}\in \mc{M}_{b,L}^{\hat{\bm{\zeta}}}$. Using that $x\mapsto \sqrt{x}$ is smooth with bounded derivatives of all orders on bounded sets in $[\bar{\lambda}_-,+\infty)$, we get via chain rule that in fact $\bar{D}^{1/2}\in \mc{M}_{b,L}^{\hat{\bm{\zeta}}}$. This immediately implies \ref{assumption:limitingcoefficientsLionsDerivatives}, and also yields by definition that $\bar{\gamma},\bar{D}$ are bounded and Lipschitz in $x,\bb{W}_2$. So again Theorem 2.1 in \cite{Wang} applies, and we gain strong existence and uniqueness of the averaged McKean-Vlasov SDE \eqref{eq:LLNlimitold}. Note that this is the only place where an assumption on the limiting coefficients, that is (7), is being used, and as per Remark \ref{remark:barDnondegenerate} this assumption holds in all but pathological cases.

For assumption \ref{assumption:multipliedpolynomialgrowth}, we already showed from the assumptions in Proposition \ref{prop:suffcondfordoubledcellproblem}, $q_{\Phi}(n,l,\bm{\beta}),q_{\Phi_{yy}}(n,l,\bm{\beta})<0$ and $q_{\Phi_y}(n,l,\bm{\beta})<-1$ for all $(n,l,\bm{\beta})\in\bm{\zeta}$, which is much stronger than what we require.

For \ref{assumption:qF2bound} and \ref{assumption:tildechi}, we already showed the result in Proposition \ref{prop:suffcondfordoubledcellproblem}.

For \ref{assumption:forcorrectorproblem}, we can check that for each of the 3 choices of $F$, that $q_F,q_{F}(n,l,\bm{\beta})\leq 1$ for $(n,l,\bm{\beta})\in \mathring{\bm{\zeta}}_1$ and $q_{F}(n,l,\bm{\beta})\leq 2$ for $(n,l,\bm{\beta})\in \mathring{\bm{\zeta}}$, so the result follows by Lemma \ref{lemma:explicitrateofgrowthofderivativesinparameters1D}. In particular, since one of our choices of $F$ is $\gamma$, which involves $\Phi_x$, we use that since $q_{b}(0,3,0)\leq 3, q_{f}(0,3,0)\leq 1,q_{a}(0,3,0)\leq 0$, Lemma \ref{lemma:explicitrateofgrowthofderivativesinparameters1D} implies $q_{\Phi_y}(0,3,0)\leq 2$.

For \ref{assumption:uniformLipschitzxmu}, we use that all the terms in the products involved in each of the functions are bounded and jointly Lipschitz. The Lipschitz properties in $x,y$ can be extrapolated from boundedness of each of the respective derivatives of the functions and in $\bb{W}_2$ follow from the boundedness of the Lions derivatives by \cite{CD} Remark 5.27.


For \ref{assumption:2unifboundedlinearfunctionalderivatives}, we have via Lemma \ref{lemma:explicitrateofgrowthofderivativesinparameters1D} that $\Phi,\Phi_x,\Phi_y,\Phi_{xy}\in \mc{M}_{\bm{\delta},p}^{\mathring{\bm{\zeta}}}$ (by construction of $\mathring{\bm{\zeta}}_1$), so that in fact all the listed functions are in $\mc{M}_{\bm{\delta},p}^{\mathring{\bm{\zeta}}}$. By Lemma \ref{lem:regularityofaveragedcoefficients}, this also implies the continuity for the Linear Functional Derivatives in \ref{assumption:limitingcoefficientsregularity}/\ref{assumption:limitingcoefficientsregularityratefunction} by definition.

And finally, for \ref{assumption:limitingcoefficientsregularity}/\ref{assumption:limitingcoefficientsregularityratefunction}, we have via Lemma \ref{lemma:explicitrateofgrowthofderivativesinparameters1D} that $\Phi,\Phi_x,\Phi_y,\Phi_{xy}\in \mc{M}_{\bm{\delta},p}^{\bm{\zeta}_{x,w+2}}$ (by construction of $\bm{\zeta}_{x,w+3}$). Then in fact $\gamma,D\in \mc{M}_{\bm{\delta},p}^{\bm{\zeta}_{x,w+2}}$, and by Lemma \ref{lem:regularityofaveragedcoefficients}, we get $\bar{\gamma},\bar{D}\in \mc{M}_{\bm{\delta},b}^{\bm{\zeta}_{x,w+2}}$.

For the regularity of the linear functional derivatives, via the equality \eqref{eq:derivativetransferformula} given by Lemma \ref{lemma:derivativetransferformulas}, we see it is sufficient to show
\begin{align*}
\norm{\frac{\delta}{\delta m}\Phi(x,y,\mu)[\cdot]}_{w+2},\norm{\frac{\delta}{\delta m}\Phi_y(x,y,\mu)[\cdot]}_{w+2},\norm{\frac{\delta}{\delta m}\Phi_x(x,y,\mu)[\cdot]}_{w+2},\norm{\frac{\delta}{\delta m}\Phi_{xy}(x,y,\mu)[\cdot]}_{w+2} \leq C(1+|y|^k),
\end{align*}
uniformly in $x\in\R,\mu\in\mc{P}_2(\R)$ for some $k\in\bb{N}.$ This follows as in the proof of the Lipschitz property in \ref{lemma:explicitrateofgrowthofderivativesinparameters1D}, iteratively using the equation \eqref{eq:formulasatisfiedbyderivatives2} and that the coefficient for the growth rate in $y$ can be written in terms of the inhomogeneity via \cite{PV1} Theorem 2 and the assumption (6), then applying \cite{GS} Lemma B1 / Remark B2 to get the result for the derivatives in $y$ as well.
\end{proof}
The examples that follow, examples \ref{example:contrivedfulldependencecase}-\ref{example:L2averaging}, present concrete cases covered by our assumptions.
\begin{example}\label{example:contrivedfulldependencecase}(A case with full dependence of the coefficients on $(x,y,\mu)$)
 Suppose $\tau_1,\tau_2,\sigma>0$ are constant with $\sigma$ large enough that $\bar{D}(x,\mu)>0,\forall x\in\R,\mu\in\mc{P}_2(\R)$, where $\bar{D}$ is as in Equation \eqref{eq:averagedlimitingcoefficients}, and the other coefficients take the form
\begin{align*}
b(x,y,\mu) &= q\biggl(y-\frac{1}{\kappa}\langle \mu,\phi_f\rangle\biggr)p_b(x),\quad
c(x,y,\mu) = r_c(y)+p_c(x)+\langle \mu,\phi_c\rangle\\
f(x,y,\mu) &= -\kappa y + \langle \mu,\phi_f\rangle,\quad
g(x,y,\mu) = r_g(y)+p_g(x)+\langle \mu,\phi_g\rangle,
\end{align*}
where here $\kappa>0$ and $\langle \phi, \mu\rangle \coloneqq \int_\R \phi(z)\mu(dz)$. Suppose also that $q \in C^\infty(\R)$ is odd, there is $\beta>0$ such that $|q(z)|,|q'(z)|,|q''(z)|,|q'''(z)|\leq C(1+|z|)^{-\beta}$, $\norm{r_c'}_\infty\leq C$, $|r_g|_{C_b^1(\R)}\leq C$, $\phi_c,\phi_g,\phi_f\in \mc{S}_{w+2}$, and $p_c,p_g \in C_b^{w+2}(\R),p_b\in C_b^{w+3}(\R)$. Then assumptions $\ref{assumption:uniformellipticity} - \ref{assumption:limitingcoefficientsregularity}$ hold, and \ref{assumption:limitingcoefficientsregularityratefunction} holds if $w$ is replaced by $r$.
\end{example}
\begin{proof}
\ref{assumption:uniformellipticity} and \ref{assumption:gsigmabounded} are immediate. \ref{assumption:retractiontomean} follows from noting that $\eta(x,y,\mu) =\eta(\mu)= \langle \mu,\phi_f\rangle$, so $\partial_\mu \eta(\mu)[z] = \phi'_f(z)$ (see Example 1 in Section 5.2.2 in \cite{CD}) and by Remark 5.27 in \cite{CD}, $|\eta(\mu_1)-\eta(\mu_2)|\leq \norm{\phi_f'}_\infty\bb{W}_2(\mu_1,\mu_2).$ In addition,
\begin{align*}
&2(f(x,y_1,\mu)-f(x,y_2,\mu))(y_1-y_2)+3|\tau_1(x,y_1,\mu)-\tau_1(x,y_2,\mu)|^2 +3|\tau_2(x,y_1,\mu)-\tau_2(x,y_2,\mu)|^2 \\
& = -2\kappa(y_1-y_2)^2.
\end{align*}

For \ref{assumption:centeringcondition}, we can find via the explicit form of $\pi$ in Lemma \ref{lemma:derivativetransferformulas} (or the fact that the frozen process is given by the Vasicek model and hence the transition density is an explicitly computable Gaussian) that
\begin{align*}
\pi(y;\mu)& = \sqrt{\frac{k}{2\pi a}}\exp\biggl(-\frac{\kappa}{2a}[y-\frac{1}{\kappa}\langle \mu,\phi_f\rangle]^2 \biggr),
\end{align*}
so
\begin{align*}
\int_{\R}b(x,y,\mu)\pi(dy;x,\mu)  &= p_b(x)\sqrt{\frac{k}{2\pi a}}\int_{\R}q\biggl(y-\frac{1}{\kappa}\langle \mu,\phi_f\rangle\biggr)\exp\biggl(-\frac{\kappa}{2a}[y-\frac{1}{\kappa}\langle \mu,\phi_f\rangle]^2 \biggr) dy \\
& = p_b(x)\sqrt{\frac{k}{2\pi a}}\int_{\R}q(y)\exp\biggl(-\frac{\kappa}{2a}y^2 \biggr) dy\\
& = 0,\forall x\in \R,\mu\in\mc{P}_2(\R)
\end{align*}
since the integrand is odd.

For the rest of the assumptions, we use Propositions \ref{prop:suffcondfordoubledcellproblem} and \ref{prop:suffcondrestofassumptions}. We have
\begin{align*}
\frac{\partial^{j+k}}{\partial y^j \partial x^k}b(x,y,\mu)& = q^{(j)}\biggl(y-\frac{1}{\kappa}\langle \mu,\phi_f\rangle\biggr)p^{(k)}_b(x)\\
\partial^{l}_\mu\frac{\partial^{j+k}}{\partial y^j \partial x^k}b(x,y,\mu)[z_1,z_2,...,z_l]& = q^{(j+l)}\biggl(y-\frac{1}{\kappa}\langle \mu,\phi_f\rangle\biggr)p^{(k)}_b(x)(-\kappa^{-1})^l\prod_{m=1}^l \phi'(z_m)\\
\frac{\delta^l}{\delta m^l}\frac{\partial^{j+k}}{\partial y^j \partial x^k}b(x,y,\mu)[z_1,z_2,...,z_l]& = q^{(j+l)}\biggl(y-\frac{1}{\kappa}\langle \mu,\phi_f\rangle\biggr)p^{(k)}_b(x)(-\kappa^{-1})^l\prod_{m=1}^l \phi(z_m)\\
\frac{\partial^l}{\partial z^l}\frac{\delta}{\delta m}\frac{\partial^{j+k}}{\partial y^j \partial x^k}b(x,y,\mu)[z]& = q^{(j+1)}\biggl(y-\frac{1}{\kappa}\langle \mu,\phi_f\rangle\biggr)p^{(k)}_b(x)\phi^{(l)}_f(z)\\
\partial_\mu f(x,y,\mu)[z] &= \phi'_f(z),\quad
\frac{\partial^l}{\partial z^l}\frac{\delta}{\delta m}f(x,y,\mu)[z] = \phi^{(l)}_f(z)\\
\frac{\partial^j}{\partial y^j}h(x,y,\mu) & = r^{(j)}_h(y),\quad
 \frac{\partial^k}{\partial x^k}h(x,y,\mu)  = p^{(j)}_h(x)\\
\partial_\mu h(x,y,\mu)[z] &= \phi'_h(z),\quad
\frac{\partial^l}{\partial z^l}\frac{\delta}{\delta m}h(x,y,\mu)[z] = \phi^{(l)}_h(z)
\end{align*}
for $h=c,g$ and $j,k,l\in\bb{N}$ such that the above derivatives are defined.

From this we can see that $b_y,b_x,\partial_\mu b(x,y,\mu)[z]$ are all uniformly bounded, and hence (1) in Proposition \ref{prop:suffcondfordoubledcellproblem} holds.

For (2)-(5), $a$ is constant, and all the considered derivatives of $f$ are uniformly $0$ except for $f_y = -\kappa$, $\partial^l_z\partial_\mu f(x,y,\mu)[z] = \phi^{(1+l)}_f(z),l=0,1,2$, all of which are uniformly bounded in $y$. All the involved derivatives of $b$ are seen to be bounded functions of $x,z_1,z_2,z_3$ multiplied by $q^{(j)}\biggl(y-\frac{1}{\kappa}\langle \mu,\phi_f\rangle\biggr)$ for $j\in \br{0,1,2,3}$, so since the translation is uniformly bounded in $\mu$, we see all of the listed $q_{b}(n,l,\bm{\beta}),q_{b_{y}}(n,l,\bm{\beta}),q_{b_{yy}}(n,l,\bm{\beta})<0$ . So the assumptions of \ref{prop:suffcondfordoubledcellproblem} hold.

Now turning to Proposition \ref{prop:suffcondrestofassumptions}, we have $h_y,h_x,\partial_\mu h(x,y,\mu)[z]$, $h=g,c$ are all uniformly bounded, and hence (1) holds.

(2) follows from observing that the desired derivatives in $x$ of $f$ and of $c,g$ are independent of $\mu$ and bounded Lipschitz in $x$. $\partial_z^l \partial_\mu p$ for $p=f,c,g$, $l=0,...,4$ only depends on $z$, and is Lipschitz for all $l$. All the listed derivatives of $b$ can easily be shown to be Lipschitz in $x,z$ uniformly in $y,\mu$ via the representations above, and since they take the form of bounded functions in $x,z_1,z_2,z_3,z_4$ multiplied by $q^{(j)}\biggl(y-\frac{1}{\kappa}\langle \mu,\phi_f\rangle\biggr)$ for $j\in \br{0,1,2,3,4}$, of which the Lions derivative is uniformly bounded, we have by Remark 5.27 in \cite{CD} that they are Lipschitz in $\bb{W}_2$ uniformly in $x,y,z$.

$q_b(0,3,0)<0$, so (3) holds.

The first and second derivatives of $c,g$ in $x$ are bounded, their first derivative in $\mu$ is bounded and its derivative in $z$ are bounded, and the rest of the derivatives of (4) are $0$. 

For (5)-(6) $f,c,g$, the listed derivatives do not depend on $y$ or $\mu$, and are uniformly bounded, with the linear functional derivatives in (6) being in $\mc{S}_{w+2}$ by assumption. For $b$, all the derivatives in $\tilde{\bm{\zeta}}_3$ are bounded by $\norm{q}_\infty\norm{p_b}_{C_b^{w+3}}$, the second linear functional derivatives are uniformly bounded by their representation above, and $\norm{\frac{\partial^k}{\partial x}\frac{\delta}{\delta m}b(x,y,\mu)[z]}_{w+2}\leq \norm{q'}_\infty\norm{p_b}_{C^k_b(\R)}\norm{\phi_f}_{w+2}\leq C,k=0,1$.\

Finally, (7) holds by supposition (noting that by the form provided for $\bar{D}$ in Equation \eqref{eq:alternativediffusion} and the fact that $\Phi$ does not depend on $\sigma$ that such a sufficiently large choice exists).

\end{proof}
\begin{example}\label{example:noxmudependenceforphiandpi}(A case where $\Phi$ and $\pi$ are independent of $x,\mu$)
Consider the case:
\begin{align*}
b(x,y,\mu) & = b(y),\quad
c(x,y,\mu) = c_1(x) + \langle \mu, c_2(x-\cdot)\rangle,\quad
\sigma(x,y,\mu)\equiv  \sigma\\
f(x,y,\mu) &= -\kappa y +\eta(y),\quad
g(x,y,\mu) = g_1(x) + \langle \mu, g_2(x-\cdot)\rangle, \quad
\tau_1(x,y,\mu)\equiv \tau_1,\quad
\tau_2(x,y,\mu)\equiv \tau_2.
\end{align*}

Suppose $\eta\in C^1_b(\R)$ with $\norm{\eta'}_\infty<\kappa$, $c_1,g_1,c_2,g_2\in C_b^{w+2}(\R)$, $c_2,g_2\in \mc{S}_{w+2}$, $\tau_1^2+\tau_2^2>0$, and $b$ is Lipschitz continuous, $O(|y|^{1/2})$, and satisfies the centering condition  \eqref{eq:centeringconditionold}. Then Assumptions \ref{assumption:uniformellipticity}-\ref{assumption:limitingcoefficientsregularity} hold. Furthermore, if this holds with $w$ replaced by $r$, then Assumption \ref{assumption:limitingcoefficientsregularityratefunction} holds.
\end{example}
\begin{proof}
Note that here $\Phi$ does not depend on $x$ or $\mu$, there is no need for Lemma \ref{prop:purpleterm1}, 
 which adds to the simplification of things (we don't need to check Assumption \ref{assumption:qF2bound}. In particular, there is no need for the extremely restrictive assumptions needed to apply Lemma \ref{lemma:extendedrocknermultidimcellproblem} since, as we will see, an application of Proposition A.2 from \cite{GS} is sufficient to handle Assumption \ref{assumption:tildechi}, and the rest of the Poisson Equations are 1-dimensional.

Assumptions \ref{assumption:uniformellipticity},\ref{assumption:retractiontomean}, \ref{assumption:centeringcondition}, and \ref{assumption:gsigmabounded} are immediate.

For \ref{assumption:strongexistence}, we have for $F(x,\mu)=c(x,\mu)$ or $g(x,\mu)$ $F_x(x,\mu) = F_1'(x)+\langle \mu,F'(x-\cdot)\rangle$  and $\partial_\mu F(x,\mu)[z] = -F_2'(x-z)$ are bounded, so for all coefficients joint Lipschitz continuity in $(x,y,\bb{W}_2)$ holds (using again Example 1 in Section 5.2.2 and Remark 5.27 in \cite{CD}), and the result holds in the same way as in Proposition \ref{prop:suffcondrestofassumptions}.

For \ref{assumption:multipliedpolynomialgrowth}, we just need $\Phi(y)$ grows at most linearly in $y$ and $\Phi'(y)$ is bounded. From Lemma \ref{lemma:Ganguly1DCellProblemResult}, we have in fact $\Phi$ is $O(|y|^{1/2})$ and $\Phi'$ is $O(|y|^{-1/2})$.

For \ref{assumption:forcorrectorproblem}, we have $\gamma(x,y,\mu) = [g_1(x)+\langle \mu,g_2(x-\cdot)\rangle]\Phi'(y)+c_1(x)+\langle \mu,c_2(x-\cdot)\rangle$ and $D(x,y,\mu) = D(y)=b(y)\Phi(y)+\sigma\tau_1\Phi'(y)+\frac{1}{2}\sigma^2$. Then for the case $F=\gamma$, $\Xi(x,y,\mu) = \tilde{\Xi}(y)[g_1(x)+\langle \mu,g_2(x-\cdot)\rangle]$ where
\begin{align*}
\mc{L}\tilde{\Xi}(y)=\Phi'(y)-\int_{\R}\Phi'(y)\pi(dy).
\end{align*}
$\Phi'$ is $O(|y|^{-1/2})$, so by Lemma \ref{lemma:Ganguly1DCellProblemResult}, $\tilde{\Xi}\in C^2_b(\R)$. Using $g_1$ and $g_2$ have two bounded derivatives, it is plain to see the result holds. A similar proof shows the result holds with $F=\sigma\psi_1(t,x,y)+[\tau_1\psi_1(t,x,y)+\tau_2\psi_2(t,x,y)]\Phi'(y),\psi_1,\psi_2\in C^\infty_c([0,T]\times\R\times \R)$. Since $\Phi$ and $b$ are $O(|y|^{1/2})$, $D$ is $O(|y|)$, so by Lemma \ref{lemma:Ganguly1DCellProblemResult} $\Xi(y)$ corresponding to $F(y)=D(y)$ is $O(|y|)$, with $\Xi'$ bounded.

For \ref{assumption:uniformLipschitzxmu}, we use the Lipschitz and boundedness properties of $\Phi'(y)$ from Lemma \ref{lemma:Ganguly1DCellProblemResult}. The result then follows by the previously noted Lipschitz properties of $c$ and $g$, and hence $\gamma$.

For \ref{assumption:limitingcoefficientsLionsDerivatives}, $\bar{D}$ is constant and $\bar{\gamma}(x,\mu) = \alpha g(x,\mu)+c(x,\mu)$, for $\alpha = \int_\R \Phi'(y)\pi(dy)$. The result thus follows from $g_1,c_1\in C^5_b(\R)$ and $g_2,c_2\in C^6_b(\R)$.

For \ref{assumption:tildechi}, we have $\tilde{\chi}(x,y,\mu)=\tilde{\chi}(y)$ grows linearly in $y$ and $\tilde{\chi}'(y)$ is $O(|y|^3)$ via Proposition A.2 of \cite{GS}. 

For \ref{assumption:2unifboundedlinearfunctionalderivatives}, none of the listed functions depend on $\mu$ other than $\gamma$, and $\frac{\delta}{\delta m}\gamma(x,y,\mu)[z] = g_2(x-z)\Phi'(y)+c_2(x-z)$ is bounded.

For \ref{assumption:limitingcoefficientsregularity}, $\bar{D}$ is constant and $\bar{\gamma}(x,\mu) = \alpha g(x,\mu)+c(x,\mu)$, for $\alpha = \int_\R \Phi'(y)\pi(dy)$, so the result follows from $g_1,g_2,c_1,c_2\in C^{w+2}_b(\R)$, and $g_2,c_2\in \mc{S}_{w+2}.$ The proof for extending to \ref{assumption:limitingcoefficientsregularityratefunction} follows in the same way, replacing $w$ by $r$.

\end{proof}

\begin{example}\label{example:L2averaging}(The case without full-coupling)
Consider the case where
\begin{align*}
b(x,y,\mu) &\equiv 0,\quad \sigma(x,y,\mu) = \sigma(x,\mu).
\end{align*}
In this setting, it is known that when also $g\equiv 0$ and $\tau_1\equiv 0$, under sufficient conditions on $c,\sigma,f$ and $\tau_2$, we can expect not only convergence in distribution of $\bar{X}^\epsilon \overset{d}{=}\bar{X}^{i,\epsilon}$ from Equation \eqref{eq:IIDparticles} to $X$ from Equation \eqref{eq:LLNlimitold}, but also convergence in $L^2$. It is easily seen that this also holds when $g,\tau_1\neq 0$ if they are sufficiently regular.

Note that in the limiting coefficients from Equation \eqref{eq:averagedlimitingcoefficients}, we have $\Phi\equiv 0$, so $\bar{\gamma}(x,\mu) = \bar{c}(x,\mu)$ and $\bar{D}(x,\mu) = \frac{1}{2}\sigma^2(x,\mu)$. In this setting, we can see immediately that there is no need for Assumptions \ref{assumption:multipliedpolynomialgrowth}, \ref{assumption:qF2bound}, and \ref{assumption:tildechi}. \ref{assumption:forcorrectorproblem} need only hold with $F=c$ and $F=\psi \in C^\infty_c([0,T]\times\R\times\R)$. We will see that, since we can gain the aforementioned $L^2$ averaging, there is no need for Theorem \ref{theo:mckeanvlasovaveraging}, and hence for Assumption \ref{assumption:limitingcoefficientsLionsDerivatives}.

Sufficient conditions for Theorem \ref{theo:MDP} to hold in this case are: \ref{assumption:uniformellipticity}- \ref{assumption:centeringcondition}, \ref{assumption:gsigmabounded}, \ref{assumption:uniformLipschitzxmu}, \ref{assumption:2unifboundedlinearfunctionalderivatives}, and
\begin{enumerate}
\item $c,a,f\in \mc{M}_{p}^{\tilde{\bm{\zeta}}}(\R\times\R\times \mc{P}_2(\R))$ with  $q_f(n,l,\bm{\beta})\leq 1$, $q_a(n,l,\bm{\beta})\leq 0$, $q_c(n,l,\bm{\beta})\leq 2$, $\forall (n,l,\bm{\beta})\in \tilde{\bm{\zeta}}$ and $q_c(n,l,\bm{\beta})\leq 1$, $\forall (n,l,\bm{\beta})\in \tilde{\bm{\zeta}}_1$.
\item $\sigma^2 \in\mc{M}_b^{\bm{\zeta}_{x,r+2}}(\R\times\mc{P}_2(\R))\cap \mc{M}_{\bm{\delta},b}^{\bar{\bm{\zeta}}_{r+2}}(\R\times\mc{P}_2(\R))$ and
$\sup_{x\in\R,\mu\in\mc{P}(R)}\norm{\frac{\delta}{\delta m}\sigma^2(x,\mu)[\cdot]}_{r+2}<\infty.$
\item $f,a,c \in \mc{M}_{p}^{\bm{\zeta}_{x,r+2}}$ and
\begin{align*}
&\norm{\frac{\delta}{\delta m}f(x,y,\mu)[\cdot]}_{r+2},\norm{\frac{\delta}{\delta m}a(x,y,\mu)[\cdot]}_{r+2},\norm{\frac{\delta}{\delta m}c(x,y,\mu)[\cdot]}_{r+2} \leq C(1+|y|^k),
\end{align*}
uniformly in $x\in\R,\mu\in\mc{P}_2(\R)$ for some $k\in\bb{N}.$
\end{enumerate}
In the above the referenced collections of multi-indices are from Equation \eqref{eq:collectionsofmultiindices}.
\end{example}
\begin{proof}
\ref{assumption:forcorrectorproblem} holds with $F=c$ and $\psi$ via an application of Lemmas \ref{lemma:explicitrateofgrowthofderivativesinparameters1D} and \ref{lem:regularityofaveragedcoefficients}, using (1).

Further, Lemma \ref{lem:regularityofaveragedcoefficients} gives $\bar{c}(x,\mu)$ is Lipschitz continuous, so \ref{assumption:strongexistence} holds in the same way as Examples \ref{example:contrivedfulldependencecase} and \ref{example:noxmudependenceforphiandpi} via the Lipschitz properties imposed on $c,\sigma,f,\tau_2$ from assumptions \ref{assumption:retractiontomean} and \ref{assumption:uniformLipschitzxmu}.

Since \ref{assumption:forcorrectorproblem} holds with $F=c$, one can see that Proposition 4.5 of \cite{BezemekSpiliopoulosAveraging2022} holds with $F=c$ and $\psi=1$ (noting that under these assumptions the norm may be moved inside the expectation with little change to the proof method). Then:
\begin{align*}
\E[|\bar{X}^\epsilon_t - X_t|^2] & \leq  C\biggl\lbrace\E[\biggl|\int_0^t c(\bar{X}^\epsilon_s,Y^\epsilon_s,\mc{L}(\bar{X}^\epsilon_s)) - \bar{c}(X_s,\mc{L}(X_s))ds\biggr|^2] + \E[\int_0^t|\sigma(\bar{X}^\epsilon_s,\mc{L}(\bar{X}^\epsilon_s))-\sigma(X_s,\mc{L}(X_s))|^2ds]\biggr\rbrace\\
&\leq C\biggl\lbrace\E[\biggl|\int_0^t c(\bar{X}^\epsilon_s,Y^\epsilon_s,\mc{L}(\bar{X}^\epsilon_s)) - \bar{c}(\bar{X}^\epsilon_s,\mc{L}(\bar{X}^\epsilon_s))ds\biggr|^2]+ \E[\int_0^t |\bar{c}(\bar{X}^\epsilon_s,\mc{L}(\bar{X}^\epsilon_s)) - \bar{c}(X_s,\mc{L}(X_s))|^2ds]\\
&+ \E[\int_0^t|\sigma(\bar{X}^\epsilon_s,\mc{L}(\bar{X}^\epsilon_s))-\sigma(X_s,\mc{L}(X_s))|^2ds]\biggr\rbrace\\
&\leq C\biggl\lbrace\E[\biggl|\int_0^t c(\bar{X}^\epsilon_s,Y^\epsilon_s,\mc{L}(\bar{X}^\epsilon_s)) - \bar{c}(\bar{X}^\epsilon_s,\mc{L}(\bar{X}^\epsilon_s))ds\biggr|^2]+ \E[\int_0^t |\bar{c}(\bar{X}^\epsilon_s,\mc{L}(\bar{X}^\epsilon_s)) - \bar{c}(X_s,\mc{L}(X_s))|^2ds]\\
&+ \E[\int_0^t|\sigma(\bar{X}^\epsilon_s,\mc{L}(\bar{X}^\epsilon_s))-\sigma(X_s,\mc{L}(X_s))|^2ds]\biggr\rbrace\\
&\leq C\epsilon^2 (1+t^2+t) +C\int_0^t \E[|\bar{X}^\epsilon_s - X_s|^2] + \bb{W}_2(\mc{L}(\bar{X}^\epsilon_s),\mc{L}(X_s))ds
\end{align*}
where in the last step we used Proposition 4.5 of \cite{BezemekSpiliopoulosAveraging2022} with $F=c$ and $\psi=1$, the assumed Lipschitz continuity of $\sigma$, and the inherited Lipschitz continuity of $\bar{c}$ via Lemma \ref{lem:regularityofaveragedcoefficients}. Bounding the 2-Wasserstein distance between the Laws by the difference in squared expectation of the processes, we get by Gr\"onwall's inequality:
\begin{align*}
\sup_{t\in[0,T]}\E[|\bar{X}^\epsilon_t - X_t|^2]&\leq C(T)\epsilon^2.
\end{align*}

Thus, in the proof of Lemma \ref{lemma:Zboundbyphi4}, we can circumnavigate using Theorem \ref{theo:mckeanvlasovaveraging} in the last line, and instead use
\begin{align*}
&a^2(N)N \biggl|\E[\phi(\bar{X}^{\epsilon}_t] - \E[\phi(X_t)] \biggr|^2+ 4a^2(N)\norm{\phi}^2_\infty \\
&\leq a^2(N)N C(T)\E[|\bar{X}^{\epsilon}_t-X_t|^2] \norm{\phi'}^2_\infty +2a(N)\norm{\phi}_\infty\\
&\leq a^2(N)N \epsilon^2 C(T)\norm{\phi'}^2_\infty +2a(N)\norm{\phi}_\infty.
\end{align*}
This shows we can circumnavigate the need for Assumption \ref{assumption:limitingcoefficientsLionsDerivatives}, and also shows that the bound from Lemma \ref{lemma:Zboundbyphi4} can be improved to $C(T)|\phi|_1^2$.

Lastly, Assumption \ref{assumption:limitingcoefficientsregularityratefunction} holds immediately for $\bar{D}=\frac{1}{2}\sigma^2$ using (2), and can be seen to hold for $\bar{\gamma}=\bar{c}$ using (3) and the same proof as in Proposition \ref{prop:suffcondrestofassumptions}.
\end{proof}

\begin{remark}
One can in fact see in the setting of Example \ref{example:L2averaging} that, as noted, the bound in Lemma \ref{lemma:Zboundbyphi4} can be improved to $C(T)|\phi|_1^2$, and further, that since $R_5^i$ in Lemma \ref{lemma:Lnu1nu2representation} is zero, the bound on $R^N_t(\phi)$ in the same Lemma can be improved to $\bar{R}(N,T)|\phi|_3$. Moreover, via the bound above, we can see via triangle inequality and Lemma \ref{lemma:XbartildeXdifference} that $\sup_{s\in [0,T]}\E\biggl[ \frac{1}{N}\sum_{i=1}^N\biggl|\tilde{X}^{i,\epsilon,N}_s-X^{i}_s\biggr|^2\biggr]\leq C(T)[\epsilon^2+ \frac{1}{N}+\frac{1}{Na^2(N)}],$ where $\br{X^i}_{i=1}^N$ are IID copies of the limiting McKean-Vlasov Equation \eqref{eq:LLNlimitold} driven by the same Brownian motions as the $\tilde{X}^{i,\epsilon,N}$'s. This allows for the proof of the Laplace Principle Upper Bound in Proposition \ref{prop:LPlowerbound} to go through in the same way as in Subsection 4.4 of \cite{BW}, and eliminates the need for the approximation argument therein. Thus, in fact, the rate function can be posed on $C([0,T];\mc{S}_{-\rho})$ and taken to be infinite outside of $C([0,T];\mc{S}_{-v})$ as in Corollary \ref{cor:mdpnomulti}. This also allows us to see that (2) and (3) in Example \ref{example:L2averaging} can be relaxed by replacing $r$ with $\rho$.
\end{remark}

\section{On Differentiation of Functions on Spaces of Measures}\label{Appendix:LionsDifferentiation}
We will need the following two definitions from \cite{CD}:

\begin{defi}
\label{def:lionderivative}
Given a function $u:\mc{P}_2(\R^d)\tto \R$, we may define a lifting of $u$ to $\tilde{u}:L^2(\tilde\W,\tilde\F,\tilde\Prob;\R^d)\tto \R$ via $\tilde u (X) = u(\mc{L}(X))$ for $X\in L^2(\tilde\W,\tilde\F,\tilde\Prob;\R^d)$. We assume $\tilde\W$ is a Polish space, $\tilde\F$ its Borel $\sigma$-field, and $\tilde\Prob$ is an atomless probability measure (since $\tilde\W$ is Polish, this is equivalent to every singleton having zero measure).

Here, denoting by $\mu(|\cdot|^r)\coloneqq \int_{\R^d}|x|^r \mu(dx)$ for $r>0$,
\begin{align*}
\mc{P}_2(\R^d) \coloneqq \br{ \mu\in \mc{P}(\R^d):\mu(|\cdot|^2)= \int_{\R^d}|x|^2 \mu(dx)<\infty}.
\end{align*}
$\mc{P}_2(\R^d)$ is a Polish space under the $L^2$-Wasserstein distance
\begin{align*}
\bb{W}_2 (\mu_1,\mu_2)\coloneqq \inf_{\pi \in\mc{C}_{\mu_1,\mu_2}} \biggl[\int_{\R^d\times\R^d} |x-y|^2 \pi(dx,dy)\biggr]^{1/2},
\end{align*}
where $\mc{C}_{\mu_1,\mu_2}$ denotes the set of all couplings of $\mu_1,\mu_2$.

We say $u$ is \textbf{L-differentiable} or \textbf{Lions-differentiable} at $\mu_0\in\mc{P}_2(\R^d)$ if there exists a random variable $X_0$ on some $(\tilde\W,\tilde\F,\tilde\Prob)$ satisfying the above assumptions, $\mc{L}(X_0)=\mu_0$ and $\tilde u$ is Fr\'echet differentiable at $X_0$.

The Fr\'echet derivative of $\tilde u$ can be viewed as an element of $L^2(\tilde\W,\tilde\F,\tilde\Prob;\R^d)$ by identifying $L^2(\tilde\W,\tilde\F,\tilde\Prob;\R^d)$ and its dual. From this, one can find that if $u$ is L-differentiable at $\mu_0\in\mc{P}_2(\R^d)$, there is a deterministic measurable function $\xi: \R^d\tto \R^d$ such that $D\tilde{u}(X_0)=\xi(X_0)$, and that $\xi$ is uniquely defined $\mu_0$-almost everywhere on $\R^d$. We denote this equivalence class of $\xi\in L^2(\R^d,\mu_0;\R^d)$ by $\partial_\mu u(\mu_0)$ and call $\partial_\mu u(\mu_0)[\cdot]:\R^d\tto \R^d$ the \textbf{Lions derivative} of $u$ at $\mu_0$. Note that this definition is independent of the choice of $X_0$ and $(\tilde\W,\tilde\F,\tilde\Prob)$. See \cite{CD} Section 5.2.

To avoid confusion when $u$ depends on more variables than just $\mu$, if $\partial_\mu u(\mu_0)$ is differentiable at $z_0\in\R^d$, we denote its derivative at $z_0$ by $\partial_z\partial_\mu u(\mu_0)[z_0]$.
\end{defi}
\begin{defi}
\label{def:fullyC2}
(\cite{CD} Definition 5.83) We say $u:\mc{P}_2(\R)\tto \R$ is \textbf{Fully} $\mathbf{C^2}$ if the following conditions are satisfied:
\begin{enumerate}
\item $u$ is $C^1$ in the sense of L-differentiation, and its first derivative has a jointly continuous version $\mc{P}_2(\R)\times \R\ni (\mu,z)\mapsto \partial_\mu u(\mu)[z]\in\R$.
\item For each fixed $\mu\in\mc{P}_2(\R)$, the version of $\R\ni z\mapsto \partial_\mu u(\mu)[z]\in\R$ from the first condition is differentiable on $\R$ in the classical sense and its derivative is given by a jointly continuous function $\mc{P}_2(\R)\times \R\ni (\mu,z)\mapsto \partial_z\partial_\mu u(\mu)[z]\in\R$.
\item For each fixed $z\in \R$, the version of $\mc{P}_2(\R)\ni \mu\mapsto \partial_\mu u(\mu)[z]\in \R$ in the first condition is continuously L-differentiable component-by-component, with a derivative given by a function $\mc{P}_2(\R)\times \R\times \R\ni(\mu,z,\bar{z})\mapsto \partial^2_\mu u(\mu)[z][\bar{z}]\in\R$ such that for any $\mu\in\mc{P}_2(\R)$ and $X\in L^2(\tilde\W,\tilde\F,\tilde\Prob;\R)$ with $\mc{L}(X)=\mu$, $\partial^2u(\mu)[z][X]
$ gives the Fr\'echet derivative at $X$ of $L^2(\tilde\W,\tilde\F,\tilde\Prob;\R)\ni X'\mapsto \partial_\mu u(\mc{L}(X'))[z]$ for every $z\in\R$. Denoting $\partial^2_\mu u(\mu)[z][\bar{z}]$ by $\partial^2_\mu u(\mu)[z,\bar{z}]$, the map $\mc{P}_2(\R)\times \R\times \R\ni(\mu,z,\bar{z})\mapsto \partial^2_\mu u(\mu)[z,\bar{z}]$ is also assumed to be continuous in the product topology.
\end{enumerate}
\end{defi}
\begin{remark}\label{remark:thirdLionsDerivative}
In this paper we will in fact also look at functions $u:\mc{P}_2(\R)\tto \R$ which are required to have $3$ Lions Derivatives. We will assume such functions are \textbf{Fully} $\mathbf{C^2}$, and satisfy:
\begin{enumerate}\setcounter{enumi}{3}
\item For each each fixed $\mu\in\mc{P}_2(\R)$ the version of  $\R\times \R \ni (z_1,z_2)\mapsto \partial^2_\mu u(\mu)[z_1,z_2]\in\R$ in Definition \ref{def:fullyC2} (3) is differentiable on $\R^2$ in the classical sense and its derivative is given by a jointly continuous function $\mc{P}_2(\R)\times \R\times\R\ni (\mu,z_1,z_2)\mapsto \nabla_z\partial^2_\mu u(\mu)[z_1,z_2] = (\partial_{z_1}\partial^2_\mu u(\mu)[z_1,z_2],\partial_{z_2}\partial^2_\mu u(\mu)[z_1,z_2])\in\R^2$.
\item For each fixed $(z_1,z_2)\in \R^2$, the version of $\mc{P}_2(\R)\ni \mu\mapsto \partial^2_\mu u(\mu)[z_1,z_2]\in \R$ in Definition \ref{def:fullyC2} (3) is continuously L-differentiable component-by-component, with a derivative given by a function $\mc{P}_2(\R)\times \R\times \R \times \R\ni(\mu,z_1,z_2,z_3)\mapsto \partial^3_\mu u(\mu)[z_1,z_2][z_3]\in\R$ such that for any $\mu\in\mc{P}_2(\R)$ and $X\in L^2(\tilde\W,\tilde\F,\tilde\Prob;\R)$ with $\mc{L}(X)=\mu$, $\partial^3u(\mu)[z_1,z_2][X]
$ gives the Fr\'echet derivative at $X$ of $L^2(\tilde\W,\tilde\F,\tilde\Prob;\R)\ni X'\mapsto \partial^2_\mu u(\mc{L}(X'))[z_1,z_2]$ for every $(z_1,z_2)\in\R^2$. Denoting $\partial^3_\mu u(\mu)[z_1,z_2][z_3]$ by $\partial^2_\mu u(\mu)[z_1,z_2,z_3]$, the map $\mc{P}_2(\R)\times \R\times \R\times \R\ni(\mu,z_1,z_2,z_3)\mapsto \partial^3_\mu u(\mu)[z_1,z_2,z_3]$ is also assumed to be continuous in the product topology.
\end{enumerate}
Though we don't require higher than 3 Lions derivatives in this paper, when we state general results for higher Lions derivatives in terms of the spaces from Definition \ref{def:lionsderivativeclasses}, we assume the analogous higher continuity.
\end{remark}

We will also make use of another notion of differentiation of functions of probability measures: the linear functional derivative.

\begin{defi}\label{def:LinearFunctionalDerivative}
(\cite{CD} Definition 5.43) Let $p:\mc{P}_2(\R)\tto \R$. We say $p$ has \textbf{Linear Functional Derivative} $\frac{\delta}{\delta m}p$ if there exists a function $(z,\mu)\ni \R\times\mc{P}_2(\R)\mapsto \frac{\delta}{\delta m} p(\mu)[z]\in \R$ continuous in the product topology on $\R\times\mc{P}_2(\R)$ such that for any bounded subset $\mc{K}\subset \mc{P}_2(\R)$, the function $\R\ni z\mapsto \frac{\delta}{\delta m}p(\mu)[z]$ is of at most quadratic growth uniformly in $\mu$ for $\mu\in\mc{K}$, and for all $\nu_1,\nu_2\in\mc{P}_2(\R^d):$
\begin{align*}
p(\nu_2)-p(\nu_1) = \int_0^1 \int_\R \frac{\delta}{\delta m}p((1-r)\nu_1+r\nu_2)[z](\nu_2(dz)-\nu_1(dz))dr.
\end{align*}
Note in particular that this implies that $p$ is continuous on $\mc{P}_2(\R)$.

The second linear functional derivative is said to exist if the linear functional derivative of $\frac{\delta}{\delta m}p(\mu)[z_1]$ as defined above exists for each $z_1\in \R$. For any bounded subset $\mc{K}\subset \mc{P}_2(\R)$, the function $(z_1,z_2)\ni \R\times\R\mapsto \frac{\delta}{\delta m}\biggl(\frac{\delta}{\delta m} p(\mu)[z_1]\biggr)[z_2]\coloneqq \frac{\delta^2}{\delta m^2} p(\mu)[z_1,z_2]\in \R$, is of at most quadratic growth uniformly in $\mu$ for $\mu\in\mc{K}$, $(z_1,z_2,\mu)\ni \R\times\R\times\mc{P}_2(\R)\mapsto \frac{\delta^2}{\delta m^2} p(\mu)[z_1,z_2]\in \R$ and is assumed to be continuous in the product topology on $\R\times\R\times\mc{P}_2(\R)$.
\end{defi}
\begin{remark}\label{remark:onLFDs}
See Section 5.4.1 of \cite{CD} for well-posedness of the above notion of differentiability and relation to Lions derivative. In particular, under sufficient regularity on $u:\mc{P}_2(\R)\tto \R$, $\partial_\mu u(\mu)[z] = \partial_z\frac{\delta}{\delta m}u(\mu)[z]$. For a formal understanding of the linear functional derivative as a Fr\'echet Derivative, see p.21 of \cite{CDLL}. Lastly, it is important to note that the linear functional derivative is only defined up to a constant by definition. This is usually not of importance, at it normally arises when studying fluctuations of measures. In particular, applying $\tilde{Z}^N_t$ as defined in \eqref{eq:fluctuationprocess} to a constant function, we of course get $0$ for any $N\in\bb{N}$ and $t\in [0,T]$, so shifting the linear functional derivative by a constant in Equation \eqref{eq:MDPlimitFIXED} does not change the representation of the limiting process. A common means of fixing this constant for concreteness is to require that $\langle\mu , \frac{\delta}{\delta m}u(\mu)[\cdot]\rangle=0,\forall \mu \in \mc{P}_2(\R)$ (see p.31 of \cite{CDLL} or Section 2.2 of \cite{DLR}. However, due to our choice of topology for the fluctuations process, correcting the constant for the linear functional derivative may break assumptions \ref{assumption:limitingcoefficientsregularity} and \ref{assumption:limitingcoefficientsregularityratefunction}. We thus interpret these assumptions to mean that there is a choice of constant when defining each of the linear functional derivatives of the functions in question which makes them satisfy the desired properties.
\end{remark}

We recall a useful connection between the Lions derivative as defined in \ref{def:lionderivative} and the empirical measure.
\begin{proposition}\label{prop:empprojderivatives}
For $g:\mc{P}_2(\R^d)\tto \R^d$ which is Fully $C^2$ in the sense of definition \ref{def:fullyC2}, we can define the empirical projection of $g$, as $g^N: (\R^d)^N\tto \R^d$ given by
\begin{align*}
g^N(\beta_1,...,\beta_N)\coloneqq g(\frac{1}{N}\sum_{i=1}^N \delta_{\beta_i}).
\end{align*}

Then $g^N$ is twice differentiable on $(\R^d)^N$, and for each $\beta_1,..,\beta_N\in\R^d$, $(i,j)\in \br{1,...,N}^2$, $l\in\br{1,...,d}$
\begin{align}
\label{eq:empfirstder}
\nabla_{\beta_i} g^N_l(\beta_1,...,\beta_N)= \frac{1}{N} \partial_\mu  g_l(\frac{1}{N}\sum_{i=1}^N \delta_{\beta_i}) [\beta_i]
\end{align}
and
\begin{align}
\label{eq:empsecondder}
\nabla_{\beta_i} \nabla_{\beta_j} g^N_l(\beta_1,...,\beta_N)= \frac{1}{N} \partial_z \partial_\mu  g_l(\frac{1}{N}\sum_{i=1}^N \delta_{\beta_i}) [\beta_i] \1_{i=j} + \frac{1}{N^2} \partial^2_\mu g_l(\frac{1}{N}\sum_{i=1}^N \delta_{\beta_i})[\beta_i,\beta_j].
\end{align}

\end{proposition}
\begin{proof}
This follows from Propositions 5.35 and 5.91 of \cite{CD}.

\end{proof}

Finally, we provide a Lemma which allows us to couple the interacting particles \eqref{eq:controlledslowfast1-Dold} to the IID McKean-Vlasov Equations \eqref{eq:IIDparticles} knowing only information about the growth of the linear functional derivatives of the coefficients.
\begin{lemma}\label{lemma:rocknersecondlinfunctderimplication}
Suppose $p:\R\times\R\times\mc{P}_2(\R)\tto \R$ satisfies
\begin{align*}
\sup_{x,z\in\R,\mu\in \mc{P}_2(\R)}|\frac{\delta}{\delta m}p(x,y,\mu)[z]|+\sup_{x,z,\bar{z}\in\R,\mu\in \mc{P}_2(\R)}|\frac{\delta^2}{\delta m^2}p(x,y,\mu)[z,\bar{z}]|\leq C(1+|y|^k)
\end{align*}
for some $C>0,k\in\bb{N}$ independent of $y\in\R$ and that $p(x,y,\cdot):\mc{P}_2(\R)\tto \R$ is Lipschitz continuous in $\bb{W}_2$ for all $x,y\in\R$. Assume \ref{assumption:uniformellipticity}-\ref{assumption:qF2bound} and \ref{assumption:uniformLipschitzxmu}. Then for $(\bar{X}^{i,\epsilon},\bar{Y}^{i,\epsilon})$ as in Equation \eqref{eq:IIDparticles} and $\bar{\mu}^{\epsilon,N}_t$ as in Equation \eqref{eq:IIDempiricalmeasure}, we have there exists $C>0$ independent of $N$ such that for all $t\in [0,T]$:
\begin{align*}
\E\biggl[\biggl|p(\bar{X}^{i,\epsilon}_t,\bar{Y}^{i,\epsilon}_t,\bar{\mu}^{\epsilon,N}_t)-p(\bar{X}^{i,\epsilon}_t,\bar{Y}^{i,\epsilon}_t,\mc{L}(\bar{X}^\epsilon_t))\biggr|^2\biggr]\leq \frac{C}{N-1}.
\end{align*}
Here $\bar{X}^\epsilon\overset{d}{=}\bar{X}^{i,\epsilon},\forall i,N\in\bb{N}$.
\end{lemma}
\begin{proof}
This follows using the same conditional expectation argument as on p.26 in \cite{DLR} and then following the proof of Lemma 5.10 in the same paper, but where we only require second order expansions rather than 4th. Since the argument and assumptions are slightly different, we present the proof here for completeness.

We first write
\begin{align*}
&\E\biggl[\biggl|p(\bar{X}^{i,\epsilon}_t,\bar{Y}^{i,\epsilon}_t,\bar{\mu}^{\epsilon,N}_t)-p(\bar{X}^{i,\epsilon}_t,\bar{Y}^{i,\epsilon}_t,\mc{L}(\bar{X}^\epsilon_t))\biggr|^2\biggr]\leq 2\E\biggl[\biggl|p(\bar{X}^{i,\epsilon}_t,\bar{Y}^{i,\epsilon}_t,\bar{\mu}^{\epsilon,N}_t)-p(\bar{X}^{i,\epsilon}_t,\bar{Y}^{i,\epsilon}_t,\bar{\mu}^{\epsilon,N,-i}_t)\biggr|^2\biggr]\\
&\qquad+2\E\biggl[\biggl|p(\bar{X}^{i,\epsilon}_t,\bar{Y}^{i,\epsilon}_t,\bar{\mu}^{\epsilon,N,-i}_t)-p(\bar{X}^{i,\epsilon}_t,\bar{Y}^{i,\epsilon}_t,\mc{L}(\bar{X}^\epsilon_t))\biggr|^2\biggr]\\
&\quad\leq C\E\biggl[|\bb{W}_2(\bar{\mu}^{\epsilon,N}_t,\bar{\mu}^{\epsilon,N,-i}_t)|^2\biggr]
+2\E\biggl[\biggl|p(\bar{X}^{i,\epsilon}_t,\bar{Y}^{i,\epsilon}_t,\bar{\mu}^{\epsilon,N,-i}_t)-p(\bar{X}^{i,\epsilon}_t,\bar{Y}^{i,\epsilon}_t,\mc{L}(\bar{X}^\epsilon_t))\biggr|^2\biggr]
\end{align*}
where here $\bar{\mu}^{\epsilon,N,-i}_t$ denotes $\bar{\mu}^{\epsilon,N}_t$ with the $i$'th particle removed, i.e.
\begin{align*}
\bar{\mu}^{\epsilon,N,-i}_t&\coloneqq \frac{1}{N-1}\sum_{j=1,j\neq i}^N \delta_{\bar{X}^{j,\epsilon}_t}.
\end{align*}
Recall the formula
\begin{align}\label{eq:empiricalmeasurewassersteindistance}
\bb{W}^p_p(\mu^N_x,\mu^N_y)& = \min_{\sigma}\frac{1}{N}\sum_{i=1}^N |x_i-\sigma(y)_i|^p
\end{align}
for $x,y\in\R^N$, $\mu^N_x = \frac{1}{N}\sum_{i=1}^N \delta_{x_i},\mu^N_x = \frac{1}{N}\sum_{i=1}^N \delta_{x_i}$, and where $\sigma:\R^N\tto \R^N$ denotes a permutation of the coordinates of a vector in $\R^N$ (see e.g. Equation 2.8 in \cite{Orrieri}). This suggests that the first term should be bounded due to the bound $C\E[|\bar{X}^{i,\epsilon}_t|^2]/N \leq C/N$ from Lemma \ref{lemma:barXuniformbound}.

To see this is indeed true, we take $\mu^N_x,\mu^{N,-i}_x$ for any $x\in \R^N$, where here $\mu^{N,-i}_x$ is defined in the same way as $\bar{\mu}^{\epsilon,N,-i}_t$, and see
\begin{align*}
\bb{W}_2^2(\mu^N_x,\mu^{N,-i}_x) &\leq \int_{\R^2}|x-y|^2 \gamma^N(dx,dy)\\
\gamma^N(dx,dy)&\coloneqq \frac{1}{N}\sum_{j=1,j\neq i}^N\biggl[\delta_{x_j}(dx)+\frac{1}{N-1}\delta_{x_i}(dx)\biggr]\delta_{x_j}(dy)
\end{align*}
We see that indeed $\gamma^N$ is a coupling between $\mu^N_x,\mu^{N,-i}_x$ since it is clearly non-negative,
\begin{align*}
\int_{\R^2}\gamma^N(dx,dy)& = \frac{1}{N}[N-1][1+\frac{1}{N-1}]=1,
\end{align*}
and for $f\in C_b(\R)$,
\begin{align*}
\int_{\R^2}f(x)\gamma^N(dx,dy)& =  \frac{1}{N}\sum_{j=1,j\neq i}^N\biggl[f(x_j)+\frac{1}{N-1}f(x_i)\biggr] =\biggl\lbrace\frac{1}{N}\sum_{j=1,j\neq i}^N f(x_j)\biggr\rbrace + \frac{1}{N}f(x_i) =   \int_\R f(y)\mu^N_x(dy)\\
\int_{\R^2}f(y)\gamma^N(dx,dy) & = \frac{1}{N}\biggl[1+\frac{1}{N-1}\biggr]\sum_{j=1,j\neq i}^Nf(x_j)=\frac{1}{N-1}\sum_{j=1,j\neq i}^Nf(x_j)=\int_\R f(y)\mu^{N,-i}_x(dy).
\end{align*}
So indeed
\begin{align*}
\bb{W}_2^2(\mu^N_x,\mu^{N,-i}_x) &\leq \int_{\R^2}|x-y|^2 \gamma^N(dx,dy) = \frac{1}{N}\sum_{j=1,j\neq i}^N \left\lbrace|x_j-x_j|^2+\frac{1}{N-1}|x_j-x_i|^2\right\rbrace\\
& = \frac{1}{N(N-1)}\sum_{j=1,j\neq i}^N|x_j-x_i|^2.
\end{align*}

Now, applying this to the first term we wish to bound,
\begin{align*}
C\E\biggl[|\bb{W}_2(\bar{\mu}^{\epsilon,N}_t,\bar{\mu}^{\epsilon,N,-i}_t)|^2\biggr]&\leq \frac{C}{N(N-1)}\sum_{j=1,j\neq i}^N\E\biggl[\biggl|\bar{X}^{j,\epsilon}_t-\bar{X}^{i,\epsilon}_t \biggr|^2 \biggr] = \frac{C}{N}\biggl[\E[|\bar{X}^{\epsilon}_t|^2]-\E[\bar{X}^{\epsilon}_t]^2\biggr]\\
&\leq \frac{C}{N}\E[|\bar{X}^{\epsilon}_t|^2]\leq \frac{C}{N}
\end{align*}
where in the equality we use that the $\bar{X}^{i,\epsilon}_t$'s are IID, and in the last bound we used Lemma \ref{lemma:barXuniformbound}.

Now we turn to the second term. We have by independence,
\begin{align*}
&\E\biggl[\biggl|p(\bar{X}^{i,\epsilon}_t,\bar{Y}^{i,\epsilon}_t,\bar{\mu}^{\epsilon,N,-i}_t)-p(\bar{X}^{i,\epsilon}_t,\bar{Y}^{i,\epsilon}_t,\mc{L}(\bar{X}^\epsilon_t))\biggr|^2\biggr] =\nonumber\\
& \qquad= \E\biggl[\E\biggl[\biggl|p(x,y,\bar{\mu}^{\epsilon,N,-i}_t)-p(x,y,\mc{L}(\bar{X}^\epsilon_t))\biggr|^2\biggr]\Bigg|_{(x,y) = (\bar{X}^{i,\epsilon}_t,\bar{Y}^{i,\epsilon}_t)}\biggr].
\end{align*}

We will show that for $q:\mc{P}_2(\R)\tto \R$ with two bounded Linear Functional Derivatives, that for $\br{\xi_i}_{i\in \bb{N}}$ IID with $\xi_i\sim \mu \in \mc{P}_2(\R)$, that letting $\xi = (\xi_1,...,\xi_N)$ and $\mu^N_\xi$ be as above with $\xi$ in the place of $x$:
\begin{align}\label{eq:generalboundedLFDresult}
\E\biggl[|q(\mu^N_\xi) - q(\mu)|^2\biggr] \leq \frac{C}{N}\biggl[\sup_{z\in\R,\mu\in \mc{P}_2(\R)}|\frac{\delta}{\delta m}q(\mu)[z]|^2+\sup_{z,\bar{z}\in\R,\mu\in \mc{P}_2(\R)}|\frac{\delta^2}{\delta m^2}q(\mu)[z,\bar{z}]|^2\biggr].
\end{align}

Applying this to the above equality, we have there is $k\in\bb{\N}$ such that
\begin{align*}
\E\biggl[\biggl|p(\bar{X}^{i,\epsilon}_t,\bar{Y}^{i,\epsilon}_t,\bar{\mu}^{\epsilon,N,-i}_t)-p(\bar{X}^{i,\epsilon}_t,\bar{Y}^{i,\epsilon}_t,\mc{L}(\bar{X}^\epsilon_t))\biggr|^2\biggr] &\leq \frac{C}{N-1}\E\biggl[1 +  |\bar{Y}^{\epsilon}_t|^{2k}  \biggr]\leq \frac{C}{N-1}
\end{align*}
by Lemma \ref{lemma:barYuniformbound}, and the result will have been proved. 

We now prove the bound \eqref{eq:generalboundedLFDresult}. By definition of the linear functional derivative, we have:
\begin{align*}
q(\mu^N_\xi) - q(\mu)& = \int_0^1 \int_\R \frac{\delta}{\delta m}q(r\mu^N_\xi+(1-r)\mu)[z](\mu^N_{\xi}(dz)-\mu(dz))dr = S_1+S_2
\end{align*}
where
\begin{align*}
S^N_1& = \int_\R\frac{\delta}{\delta m}q(\mu)[z](\mu^N_{\xi}(dz)-\mu(dz)),\\
S^N_2&  = \int_0^1 \int_\R \biggl[\frac{\delta}{\delta m}q(r\mu^N_\xi+(1-r)\mu)[z]-\frac{\delta}{\delta m}q(\mu)[z]\biggr](\mu^N_{\xi}(dz)-\mu(dz))dr.
\end{align*}
For $S_1$, we have by independence:
\begin{align*}
\E\biggl[|S_1|^2\biggr]& = \E\biggl[\biggl|\frac{1}{N}\sum_{i=1}^N \frac{\delta}{\delta m}q(\mu)[\xi_i]-\E\left[\frac{\delta}{\delta m}q(\mu)[\xi_1]\right]\biggr|^2\biggr]\\
& = \frac{1}{N}\biggl(\E\biggl[\biggl|\frac{\delta}{\delta m}q(\mu)[\xi_1]\biggr|^2\biggr]-\E\left[\frac{\delta}{\delta m}q(\mu)[\xi_1]\right]^2\biggr)\\
&\leq \frac{1}{N}\sup_{z\in \R,\mu\in \mc{P}_2(\R)}|\frac{\delta}{\delta m}q(\mu)[z]|^2.
\end{align*}

Now we set
\begin{align*}
\phi^i_r \coloneqq \frac{\delta}{\delta m}q(r\mu^N_\xi+(1-r)\mu)[\xi_i]-\frac{\delta}{\delta m}q(\mu)[\xi_i] - \tilde{\E}\biggl[\frac{\delta}{\delta m}q(r\mu^N_\xi+(1-r)\mu)[\tilde{\xi}]-\frac{\delta}{\delta m}q(\mu)[\tilde{\xi}] \biggr],
\end{align*}
where $\tilde{\xi}$ is an independent copy of the $\xi_i$'s, $r\in[0,1]$, and the expectation $\tilde{\E}$ is taken over the law of $\tilde{\xi}$.

Then we have $S^N_2 = \frac{1}{N}\sum_{i=1}^N \int_0^1 \phi^i_r dr$, and
\begin{align*}
\E\biggl[\biggl|S^N_2 \biggr|^2 \biggr] & \leq \frac{1}{N^2}\int_0^1 \E\biggl[\biggl|\sum_{i=1}^N \phi^i_r \biggr|^2 \biggr] dr  = \frac{1}{N^2}\int_0^1 \sum_{i=1}^N\sum_{j=1}^N \E\biggl[\phi^i_r\phi^j_r  \biggr] dr  = S^N_{2,1} + S^N_{2,2}
\end{align*}
where
\begin{align*}
S^N_{2,1} &  = \frac{1}{N^2}\int_0^1 \sum_{i=1}^N \E\biggl[|\phi^i_r|^2  \biggr] dr,\quad\text{and}\quad
S^N_{2,2}  = \frac{1}{N^2}\int_0^1 \sum_{i=1}^N\sum_{j=1,j\neq i}^N \E\biggl[\phi^i_r\phi^j_r  \biggr] dr.
\end{align*}

Observing that for all $i\in\bb{N},r\in[0,1]$ and $\omega\in\W$, $|\phi^i_r(\omega)|^2\leq C\sup_{z\in \R,\mu\in \mc{P}_2(\R)}|\frac{\delta}{\delta m}q(\mu)[z]|^2$, so we have
\begin{align*}
S^N_{2,1}\leq \frac{C}{N}\sup_{z\in \R,\mu\in \mc{P}_2(\R)}|\frac{\delta}{\delta m}q(\mu)[z]|^2.
\end{align*}

For $S^N_{2,2}$, we introduce the measures $\mu^{N,-(i_1,i_2)}_{\xi}\coloneqq \frac{1}{N-2}\sum_{j=1,j\neq i_1,i_2}^N \delta_{\xi_j}$ for $i_1,i_2\in \br{1,...,N}$, and let
\begin{align*}
\phi^{i,-(i_1,i_2)}_r & \coloneqq  \frac{\delta}{\delta m}q(r\mu^{N,-(i_1,i_2)}_\xi+(1-r)\mu)[\xi_i]-\frac{\delta}{\delta m}q(\mu)[\xi_i] - \tilde{\E}\biggl[\frac{\delta}{\delta m}q(r\mu^{N,-(i_1,i_2)}_\xi+(1-r)\mu)[\tilde{\xi}]-\frac{\delta}{\delta m}q(\mu)[\tilde{\xi}] \biggr].
\end{align*}

Then
\begin{align*}
\phi^i_r\phi^j_r& = [\phi^i_r-\phi^{i,-(i,j)}_r][\phi^j_r-\phi^{j,-(i,j)}_r]+\phi^{j,-(i,j)}_r[\phi^i_r-\phi^{i,-(i,j)}_r]+\phi^{i,-(i,j)}_r[\phi^j_r-\phi^{j,-(i,j)}_r]+\phi^{i,-(i,j)}_r\phi^{j,-(i,j)}_r,
\end{align*}
so
\begin{align*}
S^N_{2,2}& = S^N_{2,2,1}+S^N_{2,2,2}+S^N_{2,2,3}+S^N_{2,2,4}\\
S^N_{2,2,1}& = \frac{1}{N^2}\int_0^1 \sum_{i=1}^N\sum_{j=1,j\neq i}^N \E\biggl[\phi^{i,-(i,j)}_r\phi^{j,-(i,j)}_r  \biggr] dr\\
S^N_{2,2,2}& =\frac{1}{N^2}\int_0^1 \sum_{i=1}^N\sum_{j=1,j\neq i}^N \E\biggl[\phi^{j,-(i,j)}_r[\phi^i_r-\phi^{i,-(i,j)}_r]  \biggr] dr\\
S^N_{2,2,3}& =\frac{1}{N^2}\int_0^1 \sum_{i=1}^N\sum_{j=1,j\neq i}^N \E\biggl[\phi^{i,-(i,j)}_r[\phi^j_r-\phi^{j,-(i,j)}_r]  \biggr] dr\\
S^N_{2,2,4}& =\frac{1}{N^2}\int_0^1 \sum_{i=1}^N\sum_{j=1,j\neq i}^N \E\biggl[[\phi^i_r-\phi^{i,-(i,j)}_r][\phi^j_r-\phi^{j,-(i,j)}_r]  \biggr] dr.
\end{align*}
For $S^N_{2,2,1}$, we have
\begin{align*}
\E\biggl[\phi^{i,-(i,j)}_r\phi^{j,-(i,j)}_r  \biggr] & = \E\biggl[\E\biggl[\phi^{i,-(i,j)}_r\phi^{j,-(i,j)}_r|\xi_{k},k\neq i,j\biggr]  \biggr] = \E\biggl[\E\biggl[\phi^{i,-(i,j),x}_r\phi^{j,-(i,j),x}_r\biggr]\Bigg|_{x=\xi}  \biggr]\\
& = \E\biggl[\E\biggl[\phi^{i,-(i,j),x}_r\biggr]\Bigg|_{x=\xi}\E\biggl[\phi^{j,-(i,j),x}_r\biggr]\Bigg|_{x=\xi}  \biggr]
\end{align*}
where
\begin{align*}
\phi^{i,-(i_1,i_2),x}_r & \coloneqq  \frac{\delta}{\delta m}q(r\mu^{N,-(i_1,i_2)}_x+(1-r)\mu)[\xi_i]-\frac{\delta}{\delta m}q(\mu)[\xi_i] - \tilde{\E}\biggl[\frac{\delta}{\delta m}q(r\mu^{N,-(i_1,i_2)}_x+(1-r)\mu)[\tilde{\xi}]-\frac{\delta}{\delta m}q(\mu)[\tilde{\xi}] \biggr]
\end{align*}
and same for $j$. Then
\begin{align*}
\E\biggl[\phi^{i,-(i,j),x}_r\biggr]\Bigg|_{x=\xi} & = \biggl\lbrace\E\biggl[\frac{\delta}{\delta m}q(r\mu^{N,-(i_1,i_2)}_x+(1-r)\mu)[\xi_i]-\frac{\delta}{\delta m}q(\mu)[\xi_i]\biggr] \\
&- \tilde{\E}\biggl[\frac{\delta}{\delta m}q(r\mu^{N,-(i_1,i_2)}_x+(1-r)\mu)[\tilde{\xi}]-\frac{\delta}{\delta m}q(\mu)[\tilde{\xi}] \biggr]\biggr\rbrace\biggl|_{(x=\xi)} = 0
\end{align*}
since $\xi_i\overset{d}{=}\tilde{\xi}$, and same for $\E\biggl[\phi^{j,-(i,j),x}_r\biggr]\Bigg|_{x=\xi}$. Thus in fact, $S^N_{2,2,1}=0$.

To handle $S_{2,2,2}-S_{2,2,4}$, we need to see how to bound $|\phi^i_r - \phi^{i,-(i,j)}_r|$. We have that
\begin{align*}
\phi^i_r - \phi^{i,-(i,j)}_r & =  \frac{\delta}{\delta m}q(r\mu^N_\xi+(1-r)\mu)[\xi_i]-\frac{\delta}{\delta m}q(r\mu^{N,-(i,j)}_\xi+(1-r)\mu)[\xi_i]\\
&+ \tilde{\E}\biggl[\frac{\delta}{\delta m}q(r\mu^{N,-(i,j)}_\xi+(1-r)\mu)[\tilde{\xi}]-\frac{\delta}{\delta m}q(r\mu^N_\xi+(1-r)\mu)[\tilde{\xi}]\biggr] \\
& = r \int_0^1 \int_\R \frac{\delta^2}{\delta m^2}q(rs\mu^N_\xi + r(1-s)\mu^{N,-(i,j)}_\xi + (1-r)\mu)[\xi_i,\bar{z}][\mu^N_\xi(d\bar{z})-\mu^{N,-(i,j)}_\xi(d\bar{z})]ds\\
& + r\tilde{\E}\biggl[\int_0^1 \int_\R\frac{\delta^2}{\delta m^2}q(rs\mu^N_\xi + r(1-s)\mu^{N,-(i,j)}_\xi + (1-r)\mu)[\tilde{\xi},\bar{z}][\mu^N_\xi(d\bar{z})-\mu^{N,-(i,j)}_\xi(d\bar{z})]ds\biggr].
\end{align*}

Then using
\begin{align*}
\mu^N_x - \mu^{N,-(i,j)}_x & = \frac{1}{N}\sum_{k=1}^N \delta_{x_k} - \frac{1}{N-2}\sum_{k=1,k\neq i,j}^N \delta_{x_k} = \frac{1}{N}\delta_{x_i} + \frac{1}{N}\delta_{x_j} -\frac{2}{N(N-2)}\sum_{k=1,k\neq i,j}^N \delta_{x_k}
\end{align*}
and that $r\in [0,1]$, we get
\begin{align*}
|\phi^i_r(\omega) - \phi^{i,-(i,j)}_r(\omega)|&\leq \frac{4}{N}\sup_{z,\bar{z}\in\R,\mu\in\mc{P}_2(\R)}|\frac{\delta^2}{\delta m^2}q(\mu)[z,\bar{z}]|
\end{align*}
for all $\omega\in\W,r\in[0,1],i,j\in\bb{N}.$ This combined with the fact that $|\phi^{k,-(i,j)}_r(\omega)|\leq C\sup_{z\in \R,\mu\in \mc{P}_2(\R)}|\frac{\delta}{\delta m}q(\mu)[z]|,k=i,j$ for any $i,j\in\bb{N},r\in [0,1],\omega\in\W$ allows us to see:
\begin{align*}
S^N_{2,2,2}&\leq C\frac{N(N-1)}{N^2}\sup_{z\in \R,\mu\in \mc{P}_2(\R)}|\frac{\delta}{\delta m}q(\mu)[z]|\frac{1}{N}\sup_{z,\bar{z}\in\R,\mu\in\mc{P}_2(\R)}|\frac{\delta^2}{\delta m^2}q(\mu)[z,\bar{z}]|\\
&\leq \frac{C}{N}[\sup_{z\in \R,\mu\in \mc{P}_2(\R)}|\frac{\delta}{\delta m}q(\mu)[z]|^2 +\sup_{z,\bar{z}\in\R,\mu\in\mc{P}_2(\R)}|\frac{\delta^2}{\delta m^2}q(\mu)[z,\bar{z}]|^2  ]\\
S^N_{2,2,3}&\leq C\frac{N(N-1)}{N^2}\sup_{z\in \R,\mu\in \mc{P}_2(\R)}|\frac{\delta}{\delta m}q(\mu)[z]|\frac{1}{N}\sup_{z,\bar{z}\in\R,\mu\in\mc{P}_2(\R)}|\frac{\delta^2}{\delta m^2}q(\mu)[z,\bar{z}]|\\
&\leq \frac{C}{N}[\sup_{z\in \R,\mu\in \mc{P}_2(\R)}|\frac{\delta}{\delta m}q(\mu)[z]|^2 +\sup_{z,\bar{z}\in\R,\mu\in\mc{P}_2(\R)}|\frac{\delta^2}{\delta m^2}q(\mu)[z,\bar{z}]|^2  ]\\
S^N_{2,2,4} &\leq C\frac{N(N-1)}{N^2} \frac{1}{N^2}\sup_{z,\bar{z}\in\R,\mu\in\mc{P}_2(\R)}|\frac{\delta^2}{\delta m^2}q(\mu)[z,\bar{z}]|^2\\
&\leq \frac{C}{N^2}\sup_{z,\bar{z}\in\R,\mu\in\mc{P}_2(\R)}|\frac{\delta^2}{\delta m^2}q(\mu)[z,\bar{z}]|^2.
\end{align*}
So the bound \eqref{eq:generalboundedLFDresult} is proved.
\end{proof}
\begin{remark}
Note we could have polynomial growth in $x$ for the Linear Functional Derivatives as well and the result above would still hold, so long as we have sufficient bounded moments for $\bar{X}^{\epsilon}_t$. Also, the result is independent of the fact that the particles depend on $\epsilon$, and of the fact that the particles are one-dimensional. See Lemma 5.10 in \cite{DLR} and Theorem 2.11 in \cite{CST} for similar results in the higher-dimensional setting.
\end{remark}

\end{document}